\documentclass[11pt]{ociamthesis}
\newcommand{\color}[6]{}
\usepackage{amssymb, epsfig, graphics, amscd}
\input xy
\xyoption{all}
\usepackage{amsmath}
\usepackage{amssymb}
\usepackage{amsthm}
\usepackage {amsfonts}
\usepackage{graphicx}
\usepackage{hyperref}

\usepackage{fancyhdr}

\theoremstyle{plain}
\newtheorem{thm}{Theorem}[section]

\newtheorem{prop}[thm]{Proposition}
\newtheorem{cor}[thm]{Corollary}
\newtheorem{lem}[thm]{Lemma}
\newtheorem{req}[thm]{Requirement}
\newtheorem{que}[thm]{Question}

\newtheorem{defn}[thm]{Definition}
\theoremstyle{definition}
\newtheorem{rem}{Remark}[section]
\newtheorem{assum}[rem]{Assumption}
\newtheorem{cond}[rem]{Condition}
\newtheorem{examp}[rem]{Example}

\DeclareMathOperator{\OD}{\mathtt{OD}}
\DeclareMathOperator{\coni}{con}
\DeclareMathOperator{\sDet}{\mathtt{sDet}}
\DeclareMathOperator{\RHom}{\mathbf{R}Hom}
\DeclareMathOperator{\Free}{Free}
\DeclareMathOperator{\forget}{\mathbf{fib}}
\DeclareMathOperator{\nonu}{nu}
\DeclareMathOperator{\TC}{\mathbf{TC}}
\DeclareMathOperator{\TCC}{\mathbf{TCC}}
\DeclareMathOperator{\coalg}{coalg}
\DeclareMathOperator{\Constr}{\mathtt{Constr}}
\DeclareMathOperator{\ev}{\mathrm{ev}}

\DeclareMathOperator{\Coh}{Coh}
\DeclareMathOperator{\Tnu}{\mathbf{T}_{nu}}
\DeclareMathOperator{\cone}{\mathtt{cone}}
\DeclareMathOperator{\nt}{nt}
\DeclareMathOperator{\Spec}{Spec}
\DeclareMathOperator{\mm}{\mathbf{m}}
\DeclareMathOperator{\KK}{\mathfrak{K}^*}
\DeclareMathOperator{\id}{\mathbf{id}}
\DeclareMathOperator{\idmon}{\mathbf{1}}

\DeclareMathOperator{\MM}{\mathcal{M}}

\DeclareMathOperator{\homdual1}{\mathcal{H}^*}
\DeclareMathOperator{\Tot}{Tot}
\DeclareMathOperator{\nilp}{nilp}
\DeclareMathOperator{\Image}{Image}
\DeclareMathOperator{\triang}{\mathbf{triang}}
\DeclareMathOperator{\thick}{\mathbf{thick}}
\DeclareMathOperator{\Perf}{\mathsf{Perf}}
\DeclareMathOperator{\vect}{\mathbf{vect}_k}
\DeclareMathOperator{\Rep}{Rep}
\DeclareMathOperator{\Fun}{Fun}

\DeclareMathOperator{\dgvect}{\mathbf{dgvect}_k}
\DeclareMathOperator{\ob}{\mathbf{Ob}}
\DeclareMathOperator{\ext}{ext}
\DeclareMathOperator{\Ext}{Ext}
\DeclareMathOperator{\Hom}{Hom}
\DeclareMathOperator{\End}{End}

\DeclareMathOperator{\st}{\mathsf{st}}
\DeclareMathOperator{\Mot}{\mathsf{Mot}}
\DeclareMathOperator{\MF}{\mathsf{MF}}
\DeclareMathOperator{\Hilb}{\mathsf{Hilb}}
\DeclareMathOperator{\DT}{\mathsf{DT}}
\DeclareMathOperator{\Kth}{\mathsf{K}_0}
\DeclareMathOperator{\sest}{ss}
\DeclareMathOperator{\var}{var}
\DeclareMathOperator{\Var}{Var}
\DeclareMathOperator{\GL}{GL}
\DeclareMathOperator{\BGL}{BGL}

\DeclareMathOperator{\lmod}{-\mathsf{mod}}

\DeclareMathOperator{\lMod}{-\mathsf{Mod}}
\DeclareMathOperator{\lModn}{-\mathsf{Mod}_{norm}}
\DeclareMathOperator{\lModis}{-\mathsf{Mod}_{\infty,S}}

\DeclareMathOperator{\lmodi}{-\mathsf{mod}_{\infty}}
\DeclareMathOperator{\lmodinu}{-\mathsf{mod}_{\infty,nu}}
\DeclareMathOperator{\lModinu}{-\mathsf{Mod}_{\infty,nu}}
\DeclareMathOperator{\lModnu}{-\mathsf{Mod}_{nu}}
\DeclareMathOperator{\lModbs}{-\mathsf{Mod}_{\blacksquare}}
\DeclareMathOperator{\rModbs}{\mathsf{Mod}_{\blacksquare}-}
\DeclareMathOperator{\lModibs}{-\mathsf{Mod}_{\infty,\blacksquare}}

\DeclareMathOperator{\lmodinilp}{-\mathsf{mod}_{\infty,nilp}}
\DeclareMathOperator{\rmodinilp}{\mathsf{mod}_{\infty,nilp}-}

\DeclareMathOperator{\lModi}{-\mathsf{Mod}_{\infty}}
\DeclareMathOperator{\rmod}{\mathsf{mod}-}
\DeclareMathOperator{\rMod}{\mathsf{Mod}-}
\DeclareMathOperator{\rmodi}{\mathsf{mod}_{\infty}-}
\DeclareMathOperator{\rModi}{\mathsf{Mod}_{\infty}-}
\DeclareMathOperator{\rModnu}{\mathsf{Mod}_{nu}-}
\DeclareMathOperator{\Mat}{Mat}
\DeclareMathOperator{\biMods}{-\mathsf{Mod}_{S}-}
\DeclareMathOperator{\bimods}{-\mathsf{mod}_{S}-}

\DeclareMathOperator{\bimod}{-\mathsf{mod}-}
\DeclareMathOperator{\biMod}{-\mathsf{Mod}-}
\DeclareMathOperator{\bimodi}{-\mathsf{mod}_{\infty}-}
\DeclareMathOperator{\biModi}{-\mathsf{Mod}_{\infty}-}
\DeclareMathOperator{\biModn}{-\mathsf{Mod}_{norm}-}
\DeclareMathOperator{\biModbs}{-\mathsf{Mod}_{\blacksquare}-}
\DeclareMathOperator{\biModinu}{-\mathsf{Mod}_{\infty,nu}-}

\DeclareMathOperator{\amod}{\mathsf{mod}}
\DeclareMathOperator{\lcomod}{-\mathsf{coMod}}
\DeclareMathOperator{\coBimod}{-\mathsf{coBimod}-}

\DeclareMathOperator{\sut}{sut}

\DeclareMathOperator{\Di}{\mathbf{D}_{\infty}}
\DeclareMathOperator{\vir}{vir}
\DeclareMathOperator{\crit}{crit}

\DeclareMathOperator{\mf}{\mathsf{mf}}
\DeclareMathOperator{\tra}{\mathbf{tr}}
\DeclareMathOperator{\sw}{\mathbf{sw}}
\DeclareMathOperator{\cpct}{cpct}
\DeclareMathOperator{\HOM}{\mathcal{HOM}}
\DeclareMathOperator{\ZZ}{\mathbb{Z}}
\DeclareMathOperator{\Kcyc}{\mathfrak{K}^*_{\mathrm{cyc}}}
\DeclareMathOperator{\cyc}{cyc}

\DeclareMathOperator{\image}{Image}
\DeclareMathOperator{\coimage}{Coimage}

\DeclareMathOperator{\ch}{\mathbf{chains}}
\DeclareMathOperator{\cc}{\mathfrak{C}^*}
\DeclareMathOperator{\brackets}{\langle\bullet,\bullet\rangle}

\DeclareMathOperator{\CC}{\mathcal{C}}
\DeclareMathOperator{\FF}{\mathcal{F}}
\DeclareMathOperator{\VV}{\mathfrak{V}}
\DeclareMathOperator{\LL}{\mathbb{L}}
\DeclareMathOperator{\maM}{\mathbf{M}}
\DeclareMathOperator{\DD}{\mathcal{D}}

\DeclareMathOperator{\Aa}{\mathfrak{A}}
\DeclareMathOperator{\fram}{fr}
\DeclareMathOperator{\nc}{nc}
\DeclareMathOperator{\Ho}{H}
\DeclareMathOperator{\Zo}{Z}
\DeclareMathOperator{\alg}{alg}
\DeclareMathOperator{\PSL}{PSL}
\DeclareMathOperator{\fin}{fin}
\DeclareMathOperator{\Ker}{Ker}
\DeclareMathOperator{\tw}{\mathbf{tw}}

\DeclareMathOperator{\Coker}{Coker}
\DeclareMathOperator{\op}{op}

\newcommand{\rperf}{\Perf(\rModi\mathcal{C})}
\newcommand{\lperf}{\Perf(\mathcal{C}\lModi)}
\newcommand{\Cp}{\mathbb{C}}
\newcommand{\muhat}{\hat{\mu}}
\usepackage{amssymb, amsmath, epsfig, graphics, amscd}
\date {}
\author{Ben Davison}

\title{Orientation data in\\[1ex]     
        Motivic Donaldson--Thomas theory}   

\author{Ben Davison}             
\college{St Peter's College}  

\degree{Doctor of Philosophy}     
\degreedate{Trinity 2011}         

\begin{document}

\baselineskip=14pt plus1pt

\setcounter{secnumdepth}{2}
\setcounter{tocdepth}{1}
\mbox{}
\thispagestyle{empty}
\clearpage
\thispagestyle{empty}
\maketitle  






\mbox{}
\thispagestyle{empty}
\clearpage
\mbox{}
\thispagestyle{empty}
\clearpage
\begin{romanpages}          
\fancyhead[CO,CE]{\fancyplain{}{\textit{Contents}}}
\tableofcontents            
\cleardoublepage
\section*{Acknowledgements}
Firstly I must thank Bal\'{a}zs Szendr\H{o}i, for setting me onto such interesting mathematics and for patiently listening to the endless nonsense that preceded this thesis.  Thanks also must go to the Mathematics department at Oxford, and the surrounding mathematical community, such an incredible place for exchanging ideas and getting things done.  In particular I thank Tom Bridgeland, Richard Thomas, Alastair King and Dominic Joyce, for their generosity with their time and their ideas.  Also in Oxford I thank Tobias Barthel, for keeping me always on my toes, and Victoria Hoskins, for keeping me sane.  Next I must thank Maxim Kontsevich and Yan Soibelman, for producing the magnificently rich paper \cite{KS}, to which the current work owes so much.  Working out and understanding the contents of this paper has been a project that I have greatly enjoyed.  I thank Maxim Kontsevich in particular, for his input and comments.  My understanding, and the understanding of the material by the community in general, also owes much to Bernhard Keller, Michel Van den Bergh, Jim Bryan and Victor Ginzburg, all of whom I am very grateful to for stimulating conversations along the way.  Additional thanks must go to my examiners, Tom Bridgeland and Bernhard Keller for helping to beat this thesis into some shape.  I thank Sergey Mozgovoy, Kazushi Ueda, Ezra Getzler and Yukinobu Toda for again also showing such generosity with their ideas -- all of the above people have made the last few years a real pleasure.  I am extremely grateful also to the kind people at Northwestern, who have put me up twice, in such intellectually stimulating surroundings; the conversations I had with Kevin Costello did me enormous good.  Finally, I thank again my supervisor Bal\'{a}zs Szendr\H{o}i, for making the last few years so productive and so much fun, my friends in London and elsewhere for providing distractions, especially Nicholas, Harry, Adam, Naomi, and Tom, and my family, whose kind support over these years of my producing what must be, to them at least, incomprehensible gibberish, has kept me going.
\clearpage


\section*{Statement of Originality}
This thesis contains no material that has already been accepted, or is concurrently being submitted, for any degree or diploma or certificate or other qualification in this University or elsewhere. To the best of my knowledge and belief this thesis contains no material previously published or written by another person, except where due reference is made in the text.
\bigskip \bigskip

\hfill Ben Davison

\hfill 31st March 2011
\clearpage

\mbox{}
\thispagestyle{empty}
\clearpage

\end{romanpages}

\setlength{\headheight}{15pt}
\pagestyle{fancy}
\fancyhead{}
\renewcommand{\sectionmark}[1]{\markright{\thesection.\ #1}}
\renewcommand{\chaptermark}[1]{\markboth{Chapter \thechapter.\ #1}{}}
\fancyhead[CO]{\fancyplain{}{\textit{\rightmark}}}
\fancyhead[CE]{\fancyplain{}{\textit{\leftmark}}}

\chapter{Introduction}
\section{Motivation of the thesis}
\label{mott}
This thesis regards the concept of orientation data, as introduced in \cite{KS}.  This is a notion that arises when one tries to write down an `integration map' in motivic Donaldson--Thomas theory.  To see what one could mean by all this, it is worth bearing in mind what motivic Donaldson--Thomas theory is a refinement \textit{of}, i.e. old-fashioned (more than three years old...) Donaldson--Thomas theory.  The \textit{really} old-fashioned approach to this subject (more than six years old...) is that the Donaldson--Thomas count associated to a moduli space $\mathcal{M}$ of stable sheaves on a projective 3-Calabi-Yau variety is the degree of a zero-dimensional class in the Chow ring of $\mathcal{M}$ (this is what is known as the virtual fundamental class).  The existence of this class goes back to Richard Thomas's thesis \cite{RTthesis}, and further back to \cite{LT}.  The recourse to intersection theory here means that we need compactness.  Also, the virtual fundamental class itself is constructed from a symmetric perfect obstruction theory, a two term partial resolution of the cotangent complex of our moduli space by vector bundles, and so there is no indication at this stage that there is anything `motivic' going on -- if we cut the moduli space up along some constructible decomposition, and restrict our two term complex to each part, we will in general forget some gluing data for the perfect obstruction theory, as well as sacrificing compactness.
\bigbreak
As an indication that there \textit{is} something motivic going on, consider Behrend's result (see \cite{Behr09}), that for an arbitrary finite type scheme $X$ over $\Cp$ there is a constructible $\mathbb{Z}$-valued function $\nu_X$, such that if $X$ is equipped with a symmetric perfect obstruction theory, and is compact, then the degree of the resulting virtual fundamental class is given by the weighted Euler characteristic of $X$ with respect to the function $\nu_X$.  So if one knows what $\nu_X$ is for a moduli space $X$, one is free to constructibly chop up $X$ any way one likes, and calculate the contribution to the Donaldson--Thomas count from each piece.
\bigbreak
Behrend's theorem, in principle, makes it easier to calculate Donaldson--Thomas invariants.  On the one hand, if there is a torus action on a scheme $\mathcal{M}$, then the calculation of the weighted Euler characteristic reduces to the calculation of the weighted Euler characteristic at the fixed locus $F$ (still weighting by $\nu_{\mathcal{M}}$, \textit{not} $\nu_F$).  More generally, the weighted Euler characteristic is the right notion for understanding \textit{wall crossing}, which is a story we will proceed to sketch.  Wall crossing phenomena provide a key to understanding actual calculations of Donaldson--Thomas invariants, which on the face of it, and especially away from the world of toric varieties where we have localization results that reduce calculations to combinatorics in many cases are rather hard to get a grip on (see e.g. \cite{LocalizationFormula}, \cite{MNOP1}, \cite{ObstructionsHilbert} for the foundations of the approach via torus actions, and then e.g. \cite{Conifold}, \cite{Dominos}, \cite{ColouredYoungDiagrams}, \cite{MR}, \cite{Dav08}...  for a selection of applications).
\bigbreak
Say we have an Abelian category $\mathcal{S}$, a homomorphism $\theta:\Kth(\mathcal{S})\rightarrow \mathbb{Z}^n$ from the Grothendieck group, some notion of (semi)stability $\alpha$ for objects of $\mathcal{S}$ (see \cite{HL}), and moduli spaces $\mathcal{M}_{\alpha-\sest}^{\gamma}$ of semistable objects of class $\theta^{-1}(\gamma)$ for each $\gamma\in \mathbb{Z}^n$, then we define a partition function
\[
\mathcal{Z}_{\DT,\alpha}(\textbf{t}):=\sum_{\gamma\in\mathbb{Z}^n} \left(\int_{\mathcal{M}_{\alpha -\sest}^{\gamma}}\nu_{\mathcal{M}_{\alpha -\sest}^{\gamma}}d\chi\right)\textbf{t}^{\gamma},
\]
where $\gamma$ is considered as a multi-index.  We are intentionally being a little vague about these spaces $\mathcal{M}_{\alpha -\sest}^{\gamma}$ -- the experts will already have spotted that when there are `strictly semistables' these spaces are either stacks (and not schemes), making the right hand side undefined, or the spaces can be considered as (non-fine) moduli schemes, making the given definition the wrong one.  The solution to this problem lies in the work of Joyce and Song (see \cite{JS08}), and consists of modifying $\nu_{\mathcal{M}_{\alpha -\sest}^{\gamma}}$ at the points parameterising strictly semistable objects of $\mathcal{S}$.  One of the main results of \cite{JS08} concerns the variation of the partition function $\mathcal{Z}_{\DT,\alpha}(\textbf{t})$ as we vary $\alpha$ -- this is what is meant by wall crossing.  We can sketch an example of this variation in a specific case (this example is somehow paradigmatic, see Section 7 of \cite{KS}).  Let $A$ be some finitely presented algebra (in fact $A$ will be a 3-Calabi-Yau algebra, though we needn't worry about what this means at this point, see Chapter \ref{quiveralgebras}, or \cite{KLH}, \cite{kaj03} \cite{keller-intro}, \cite{KSnotes} for details).  We can consider $A$ as the quiver algebra for a quiver $Q$ with one vertex $v_0$, and finitely many edges, and finitely many relations, since $A$ is finitely presented.    We adjoin an extra vertex $v_{\infty}$ to $Q$, with one arrow from $v_{\infty}$ to the original vertex $v_0$, and no new relations.  Call the new algebra $\tilde{A}$, then we may fix a homomorphism $\theta$ from the Grothendieck group of finite-dimensional left $\tilde{A}$-modules to $\mathbb{Z}^2$, given by taking the dimension vector.
\bigbreak
Before going any further we must introduce the Hall algebra of stack functions $\st(\mathcal{M}_{\tilde{A}})$ for the category of left $\tilde{A}$-modules.  We have taken the presentation of this from \cite{TB10}, which in turn is a distillation of the much longer series (\cite{Jo06a}, \cite{Jo06b}, \cite{Jo07a}, \cite{Jo07b}).  This is generated, as a $\mathbb{Z}$-module, by symbols $[f:S\rightarrow \mathcal{M}_{\tilde{A}}]$, where $\mathcal{M}_{\tilde{A}}$ is the moduli stack of finite-dimensional left $\tilde{A}$-modules, and $S$ is a finite type stack.  We ask that all the stacks appearing in these symbols have affine stabilizer groups for all their geometric points (recall that a geometric point of a stack $X$ defined over $k$ is a morphism $\Spec(K)\rightarrow X$, where $K$ is an algebraically closed extension of $k$).  We impose the obvious cut and paste relations on this $\mathbb{Z}$-module, and identify $[f_1:S_1\rightarrow \mathcal{M}_{\tilde{A}}]=[f_2:S_2\rightarrow \mathcal{M}_{\tilde{A}}]$ if there exists a morphism $g:S_1\rightarrow S_1$ over $\mathcal{M}_{\tilde{A}}$ inducing an isomorphism of categories $S_1(\Spec(K))\cong S_2(\Spec(K))$ for every $K$ an algebraically closed extension of our ground field $k$.  Finally, we impose the relation that for every pair of Zariski fibrations of stacks $[f_1:Y_1\rightarrow X]$ and $[f_2:Y_2\rightarrow X]$, and every morphism $[g:X\rightarrow \mathcal{M}_{\tilde{A}}]$, with the fibres of $f_1$ and $f_2$ isomorphic, and with $X$ and the fibres of the $f_i$ of finite type, there is an identity 
\[
[g\circ f_1:Y_1\rightarrow \mathcal{M}_{\tilde{A}}]=[g\circ f_2:Y_2\rightarrow \mathcal{M}_{\tilde{A}}].
\]

The moduli stack $\mathcal{M}_{\tilde{A}}$ is locally finite type, and breaks into countably many components $\mathcal{M}_{\tilde{A},\gamma}$, one for each dimension vector $\gamma\in\mathbb{Z}^2$.  Going back to the general case, in which we are interested in objects of some category $\mathcal{S}$, where $\mathcal{M}$ is the moduli space of objects in $\mathcal{S}$, we need to impose some conditions on $\mathcal{M}$, the target of our functions, i.e. we ask that it be locally finite type, and that the stabilizer of any geometric point be an affine algebraic group.  One can readily verify that in the case at hand, each stack $\mathcal{M}_{\tilde{A},\gamma}$ is a global quotient stack of an affine finite type scheme by a product of general linear groups, so these conditions are met.  The $\mathbb{Z}$-module $\st(\mathcal{M}_{\tilde{A}})$ carries an action of $\Kth(\st_k)$, which is defined as a $\mathbb{Z}$-module the same way as we defined the underlying $\mathbb{Z}$-module of $\st(\mathcal{M}_{\tilde{A}})$, except that our base stack is now $\Spec(k)$.  The group $\Kth(\st_k)$ carries a ring structure, with multiplication given by Cartesian product.  The action of a class $[X]\in \Kth(\st_k)$ on $[f:S\rightarrow \mathcal{M}_{\tilde{A}}]$ is given by precomposing $f$ with the projection $S\times X\rightarrow S$.  We could instead have considered the $\mathbb{Z}$-module $\st(\mathcal{M}_{\tilde{A}})_{\var}$ generated by functions $f$ whose source is a variety.  This $\mathbb{Z}$-module is a module over $\Kth(\Var_k)$, which is defined to be the sub $\mathbb{Z}$-module of $\Kth(\st_k)$ generated by varieties.  Under our assumptions, there is a natural map from the localization
\begin{equation}
\label{halllocal}
\st(\mathcal{M}_{\tilde{A}})_{\var}[ [\GL_n(k)]^{-1},n\geq 1]\rightarrow \st(\mathcal{M}_{\tilde{A}})
\end{equation}
and it is an isomorphism (see \cite{kre99}).  It is given by sending the formal inverse $[\GL_n(k)]^{-1}$ to $\BGL_n(k)$, the stack theoretic quotient of $\Spec(k)$ by the trivial action of $\GL_n(k)$.
\bigbreak
The Hall algebra product is expressed precisely, and neatly, in the language of stacks.  Let $\mathcal{M}^{(2)}_{\tilde{A}}$ be the stack of short exact sequences of finite-dimensional $\tilde{A}$-modules.  Then projection of a short exact sequence onto its first and last term defines a map of stacks $\mathcal{M}_{\tilde{A}}^{(2)}\rightarrow \mathcal{M}_{\tilde{A}}\times \mathcal{M}_{\tilde{A}}$.  To give a family $\mathcal{F}_i$ of $\tilde{A}$-modules is the same as giving a morphism $f_i:S_i\rightarrow \mathcal{M}_{\tilde{A}}$, which can be considered as a stack function.  Then we obtain a diagram
\[
\xymatrix{
T\ar[r]\ar[d] &\mathcal{M}_{\tilde{A}}^{(2)}\ar[r]\ar[d]&\mathcal{M}_{\tilde{A}}\\
S_1\times S_2 \ar[r]&\mathcal{M}_{\tilde{A}}\times\mathcal{M}_{\tilde{A}}
}
\]
where the square is a pullback, and the rightmost map is obtained by sending a short exact sequence to its middle term.  The composition of the morphisms in the top row is again a stack function, and we define $[f_1:S_1\rightarrow \mathcal{M}_{\tilde{A}}]\star [f_2:S_2\rightarrow \mathcal{M}_{\tilde{A}}]:=[T\rightarrow \mathcal{M}_{\tilde{A}}]$.  Informally, $[f_1:S_1\rightarrow \mathcal{M}_{\tilde{A}}]\star [f_2:S_2\rightarrow \mathcal{M}_{\tilde{A}}]$ is obtained by first taking the stack parameterising short exact sequences
\[
\xymatrix{
0\ar[r]&M'\ar[r]& M\ar[r]&M''\ar[r]&0
}
\]
with $M'$ in the family parameterised by $\mathcal{F}_1$ and $M''$ in the family parameterised by $\mathcal{F}_2$, and allowing this stack to be the base for a family of $\tilde{A}$-modules by considering the middle term of such short exact sequences.  Note, however, that even in the case in which $S_1$ and $S_2$ above are varieties, $T$ will almost always be a stack, and not a variety, incorporating the fact that short exact sequences have stabilizers given by $\Hom(M'',M')$.
\bigbreak
One of the central ideas of the study of motivic Donaldson--Thomas theory is that there is a $\Kth(\st_k)$-linear map (an `integration map' $\Phi$) from the motivic Hall algebra to a twisted polynomial ring $\textbf{M}[x^{\gamma}|\gamma\in\Gamma]$ where 
\begin{equation}
\label{Mdef}
\textbf{M}:=\overline{\Mot}^{\hat{\mu}}(\Spec(k))[\mathbb{L}^{1/2},[\GL_n(k)]^{-1}\hbox{ }|\hbox{ }n\geq 1]
\end{equation}
is obtained from the `universal coefficient ring' of equivariant motives by adding formal inverses to the motives of the general linear groups, and a formal square root of the motive $\mathbb{L}$ representing the affine line.  The group $\hat{\mu}$ is the inverse limit of the finite groups of roots of unity.  The ring $\overline{\Mot}^{\hat{\mu}}(\Spec(k))$ is a quotient ring of $\Mot^{\hat{\mu}}(\Spec(k))$, which as a $\mathbb{Z}$-module is generated by $\hat{\mu}$-equivariant varieties, subject to the usual cut and paste relations and identifications under $\hat{\mu}$-equivariant morphisms inducing bijections of geometric points (we require also that the $\hat{\mu}$-equivariant varieties that generate our ring have $\hat{\mu}$-action factoring through some finite group of roots of unity).  To properly define $\Mot^{\hat{\mu}}(\Spec(k))$ we should impose the extra relations $[X]=[Y]\cdot \mathbb{L}^n$, for $X$ any $\hat{\mu}$-equivariant vector bundle of dimension $n$ on a variety $Y$ -- this relation is required in order to make the formula for the motivic vanishing cycle, in terms of an embedded resolution, well defined (see Section \ref{mvc}).  The equivalence relation imposed in order to pass further, from $\Mot^{\hat{\mu}}(\Spec(k))$ to $\overline{\Mot}^{\hat{\mu}}(\Spec(k))$, is such that all realizations of motives one might care to take (e.g. $\muhat$-equivariant Hodge polynomial) factor through the quotient ring (see \cite{KS}).  The twisting here comes from the rule $x_{\gamma_1}\cdot x_{\gamma_2}=\mathbb{L}^{\frac{\langle \gamma_1,\gamma_2\rangle}{2}} x_{\gamma_1+\gamma_2}$, where $\langle\bullet,\bullet\rangle$ is a skew symmetric form on $\Gamma$ (we will assume that our category $\mathcal{S}$ is 3-Calabi-Yau, and require that the pullback of $\langle\bullet,\bullet\rangle$ under $\theta$ is the Euler form).  The ring $\overline{\Mot}^{\hat{\mu}}(\Spec(k))$ comes with a product, such that realizations of motives (taking Euler characteristic, Serre polynomials, taking $\Kth$ class in the category of $\hat{\mu}$-equivariant mixed Hodge structures, etc) are ring homomorphisms (see \cite{Looi00}, \cite{And86}, \cite{DL01} for the statement regarding Hodge structures, from which the others follow).  The key requirement of a putative integration map is that it take the Hall algebra product to the induced twisted product in $\textbf{M}[x^{\gamma}|\gamma\in\Gamma]$.  This enables us to read off properties of generating series with coefficients in $\textbf{M}$ from Hall algebra identities.\bigbreak
We will give an example of this in the case at hand, but first we address a little technicality.  We would really like to be able to consider symbols such as $[\id:\mathcal{M}_{\tilde{A}}\rightarrow \mathcal{M}_{\tilde{A}}]$ as objects in our Hall algebra.  However the stack $\mathcal{M}_{\tilde{A}}$ is not of finite type, and so this symbol is a priori excluded from our treatment.  We work around this by considering it as an object of a completion of our Hall algebra.  Recall that in general we have in the background some morphism $\theta:\Kth(\mathcal{S})\rightarrow \Gamma$ to a lattice $\Gamma$.  Assuming we know how to construct it, the stack of objects in the category $\mathcal{S}$ decomposes as an infinite union of closed and open substacks $\mathcal{M}_{\gamma}$, each representing the moduli functor for families of objects $M$ with $\theta([M])=\gamma$.  In general $\theta([\mathcal{S}])$, the image of the effective cone, is a subsemigroup of $\Gamma$, and $\theta([\mathcal{S}])\otimes_{\mathbb{Z}} \mathbb{R}_{\geq 0}$ is a cone in $\Gamma\otimes_{\mathbb{Z}}\mathbb{R}$.  Let $T$ be a convex subset of this real cone, which is an ideal, when considered as a subsemigroup of the semigroup $\theta([\mathcal{S}])\otimes_{\mathbb{Z}} \mathbb{R}_{\geq 0}$.  We will deal with the case where $\mathcal{S}=\tilde{A}\lmod$.
\smallbreak 
We define $\st(\mathcal{M}_{\tilde{A}})/T$ to be generated as a $\Kth(\st_k)$-module by $\st(\mathcal{M}_{\tilde{A}})$ with the extra relation
\begin{equation}
\label{extrarel}
[f:S\rightarrow \mathcal{M}_{\tilde{A}}]=0
\end{equation}
if $f$ factors through an inclusion $\mathcal{M}_{\tilde{A},\gamma}\rightarrow \mathcal{M}_{\tilde{A}}$ for $\gamma\in T$.  The Hall algebra product descends to a product on this quotient.  There are inclusions $nT\subset (n-1)T$ inducing maps $\st(\mathcal{M}_{\tilde{A}})/nT\rightarrow \st(\mathcal{M}_{\tilde{A}})/(n-1)T$, and we let $\hat{\st}_{T}(\mathcal{M}_{\tilde{A}})$ be the inverse limit, which inherits a natural product structure from the products on the quotients.  In our example we picked $\Gamma=\mathbb{Z}^2$, with $\theta$ being the map given by taking dimension vector, and so we may pick $T$ to be the set of elements $(a,b)\in \mathbb{R}_{\geq 0}^2=\Image([\mathcal{M}_{\tilde{A}}])\otimes_{\mathbb{Z}}\mathbb{R}_{\geq 0}$ with $a+b\geq 1$.  Then we have a natural element
\[
[\id:\mathcal{M}_{\tilde{A}}\rightarrow \mathcal{M}_{\tilde{A}}]\in \hat{\st}_{T}(\mathcal{M}_{\tilde{A}}).
\]
From now on we'll work in this completion.  An arbitrary $\tilde{A}$-module $M$ fits into a unique short exact sequence
\[
\xymatrix{
0\ar[r]&M'\ar[r]&M\ar[r]& V\ar[r]&0
}
\]
where $M'$ is an $A$-module, considered as an $\tilde{A}$-module in the natural way, and $V$ is an $\tilde{A}$-module whose dimension vector is zero when restricted to the vertex $v_0$.  So if $V_{\infty}$ is the stack of such $\tilde{A}$-modules, and $\mathcal{M}_A$ is the stack of $A$-modules, then there is an identity in the Hall algebra
\[
[\mathcal{M}_{\tilde{A}}]=[\mathcal{M}_A]\star [V_{\infty}],
\]
where here we are a little lax and we identify each stack with the natural morphism from it to $\mathcal{M}_{\tilde{A}}$.  We define \textit{multicyclic} $\tilde{A}$-modules to be those $\tilde{A}$-modules $M$ such that the sub-vector space $e_{\infty}M$ generates $M$ as an $\tilde{A}$-module.  Every $\tilde{A}$-module sits in a unique short exact sequence
\[
\xymatrix{
0\ar[r]&M''\ar[r]&M\ar[r]&M'\ar[r]&0
}
\]
where $M'$ is an $A$-module as before, while $M''$ is a multicyclic $\tilde{A}$-module (this short exact sequence comes from a coarsening of a Harder-Narasimhan filtration, (see \cite{King94}).  So, letting $\mathcal{M}_{\tilde{A},\cyc}$ be the stack of multicyclic $\tilde{A}$-modules, there is an identity in the Hall algebra
\[
[\mathcal{M}_{\tilde{A}}]=[\mathcal{M}_{\tilde{A},\cyc}]\star [\mathcal{M}_{A}],
\]
where we again identify these stacks with their natural morphisms to $\mathcal{M}_{\tilde{A}}$.  There is a substack $(\mathcal{M}_{\tilde{A},\cyc})_{(1,\mathbb{Z})}$ of $\mathcal{M}_{\tilde{A},\cyc}$ consisting of multicyclic modules that have dimension vector $(v_{\infty},v_0)=(1,n)$, where $n\in \mathbb{Z}$.  This is the stack of cyclic $\tilde{A}$-modules, which is a $k^*$-torsor over the noncommutative Hilbert scheme for $A$, meaning that there is an identity in the Hall algebra $[\Hilb(A)]=[k^*]\cdot [(\mathcal{M}_{\tilde{A},\cyc})_{(1,\mathbb{Z})}]$, where the multiplication here is coming from the $\Kth(\st_k)$-action on $\st(\mathcal{M}_{\tilde{A}})$.\bigbreak
So let us assume that we are interested in $\Phi([\Hilb(A)])$.  We have shown that up to a $[k^*]$ factor, it is given by conjugating $\Phi([V_{\infty}])$ by $\Phi([\mathcal{M}_{A}])$ and then truncating.  In practice these two quantities are often easier to work out, and so if we can just define $\Phi$ we are in good shape.  Actually defining $\Phi$ is where the orientation data comes in.\bigbreak
As an aside we remark that this example provides an interesting advertisement for the study of motivic Donaldson--Thomas theory, in that it shows a way that we can obtain statements about the ordinary Donaldson--Thomas partition function from the motivic setup.  Note that unlike at least some other motivic realizations of the ring $\textbf{M}$, the Euler characteristic is not defined on the whole ring $\textbf{M}$.  This is because there are invertible elements (e.g. $[\GL_n(k)]$ for all $n$) that have Euler characteristic zero.  Now the motivic generating series for $\Phi([\Hilb(A)])$ can be shown to have schemes as coefficients, so we are certainly able to apply Euler characteristic to it, and this is exactly what we do to recover the ordinary Donaldson--Thomas partition function.  However we cannot take the identity 
\begin{equation}
\label{neatid}
\Phi([\Hilb(A)])=[k^*]\cdot\Phi([\mathcal{M}_{A}])\cdot\Phi([(V_{\infty})_{(1,0)}])\cdot\Phi([\mathcal{M}_{A}])^{-1}
\end{equation}
and just evaluate the Euler characteristic of each side, since the stack functions appearing in the right hand side are genuine stacks, with stabilizer groups having Euler characteristic zero.  We can, however, evaluate Serre polynomials of both sides, and we get something reasonable.  Then we specialize the Serre polynomial to get the Euler characteristic of the left hand side.  Note that this way of getting a handle on the ordinary partition function for Donaldson--Thomas counts of noncommutative Hilbert schemes relies on being able to see the identity (\ref{neatid}) at the level of motivic generating series -- working at the level of ordinary counting invariants all along and making use of (\ref{neatid}) is a rather tricky proposition (though see \cite{JS08}).\bigbreak
Finally we get on to the integration map $\Phi$.  In short, the change we make when we refine from ordinary Donaldson--Thomas theory to motivic Donaldson--Thomas theory is that instead of integrating schemes with respect to a constructible function, with measure given by Euler characteristic, we take the motivic integral of a motivic weight.  In cases where the ordinary Donaldson--Thomas count is defined, it is then recovered as the Euler characteristic of the resulting element in the ring of motives $\overline{\Mot}(\Spec(k))$ (we forget the $\hat{\mu}$-action -- this corresponds to the statement that in order to calculate the Milnor number we don't need to remember anything about monodromy).  The remaining data, then, is this motivic weight, which we will denote by $w$.  It turns out that there is a very reasonable candidate for this $w$, defined purely in terms of the category we are building an integration map for, in this case $\tilde{A}\lmod$, as long as we place some assumptions on $A$, i.e. we need some kind of 3-Calabi-Yau property, to be discussed at length in Chapter \ref{quiveralgebras}.  We can briefly explain what happens in this case.  Let $M$ be some $\tilde{A}$-module.  Then a local neighborhood of $M$, considered as a point in the moduli stack $\mathcal{M}_{\tilde{A}}$, is given by the scheme-theoretic degeneracy locus of a function $f$ on some ambient smooth space $X$ (see \cite{Conifold}), and in this case we have that $\nu_{\mathcal{M}_{\tilde{A}}}(M)=(-1)^{\dim(X)}(1-\chi(\mf(p^*f,\tilde{M})))$ where $\mf(p^*f,\tilde{M})$ is the Milnor fibre of the pullback of $f$ to a smooth atlas $p:X'\rightarrow X$, at a lift $\tilde{M}$ of the point $M$.  There is then a ready-made motivic refinement of the function $\nu_{\mathcal{M}_{\tilde{A}}}$ given by the motivic Milnor fibre $\MF(p^*f,\tilde{M})$ defined by Denef and Loeser.  Precisely, we give the point $M$ the weight $\mathbb{L}^{-\dim(X)/2}(1-\MF(p^*f,\tilde{M}))$.
\bigbreak
There is, however, a little ambiguity in the refinement as it stands.  If there is some other smooth $Y$, with a smooth atlas $q:Y'\rightarrow Y$, and a function $g$ on $Y$ such that $\mathcal{M}_{\tilde{A}}$ is described locally around $M$ as the scheme-theoretic degeneracy locus of $g$, it turns out that $(-1)^{\dim(Y)}(1-\chi(\MF(q^*g,\tilde{M})))$ is the same number (it is just the microlocal function $\nu_{\mathcal{M}_{\tilde{A}}}$ evaluated at $M$ again).  However, there needn't be an equality $\mathbb{L}^{-\dim(X)/2}(1-\MF(p^*f,\tilde{M}))=\mathbb{L}^{-\dim(Y)/2}(1-\MF(q^*g,\tilde{M}))$.  To remedy this we take a canonical choice of $X$ and $f$: we let $X'=\Ext^1(M,M)$, a smooth atlas for the formal deformation stack $X$ of $M$, and let $f=W_{\min,M}$ be the minimal potential for $M$ (see Definition \ref{CYcatdef} and Theorem \ref{cycmm}).  This is a function on $\Ext^1(M,M)$ defined purely in terms of the category $\tilde{A}\lmod$.
\bigbreak
It turns out, however, that even though we have a canonical choice of $X$ and $f$, it isn't quite right for defining the motivic weight: if we use this weight, the map $\Phi$ need not preserve the product.  Call this weight $w_{\min}$.  The remedy involves quite a mild change to $w_{\min}$.  Consider again the function $W_{\min,M}$ on $\Ext^1(M,M)$.  If we replace $\Ext^1(M,M)$ by $\Ext^1(M,M)\oplus V$ for some finite-dimensional vector space $V$, and let $Q$ be a nondegenerate quadratic form on $V$, then one can verify that if we treat $Q$ as a function on $V$ by setting $Q(x)=Q(x,x)$, the scheme-theoretic degeneracy locus of $W_{\min,M}+Q$ on $\Ext^1(M,M)\oplus V$ is isomorphic to the scheme-theoretic degeneracy locus of $W_{\min,M}$ on $\Ext^1(M,M)$.  This statement remains true at the level of quotient stacks.  So the pair $(\Ext^1(M,M)\oplus V, W_{\min,M}+Q)$ is another candidate for the function we take the motivic Milnor fibre of.  By the motivic Thom-Sebastiani theorem (\cite{DL99}), this replacement has the effect of multiplying our original motivic weight by $(1-\MF(Q))\mathbb{L}^{-\frac{\dim(V)}{2}}$.  It turns out that there is a fix given by such a replacement, and this is what orientation data is, at a rough approximation.  At a slightly closer approximation, since we only care about the motivic weight $w$, and this is determined by the motivic weight $w_Q=(1-\MF(Q))\mathbb{L}^{-\frac{\dim(V)}{2}}$, we identify pairs $(V,Q)$ that lead to the same motivic weight, and orientation data is an equivalence class under these identifications.  In order to be classed as orientation data, this equivalence class needs to make the integration map $\Phi$ that integrates with respect to the weight $w_{\min}\circ w_{Q}$ preserve ring structure (here the product is in the ring of relative motives over an atlas for our moduli space).  This condition ultimately becomes what is called the \textit{cocycle condition} -- see Condition \ref{cocycle}. This thesis is dedicated to describing such a choice in cases that don't differ too much from our motivating example of modules over an algebra, and describing the range of choices in such cases too.  
\section{Structure of the thesis}
Once the role of orientation data is clear, and its construction is demystified, it turns out to be a rather manageable entity, and as mentioned above, the rewards, such as the ability to actually calculate motivic Donaldson--Thomas invariants, are great.  So in many ways the main problem is in seeing what exactly orientation data does for us.  To this end we start the thesis with a worked example, in which almost all the features of the theory emerge, and we motivate the twisted definition of the integration map, involving orientation data.  The example itself is rather simple, and we are able to calculate everything in sight by hand, and see what kind of trouble the untwisted version of the integration map gets us into.
\bigbreak
We then progress towards the construction of orientation data in a wide class of examples.  This part of the story rests quite heavily on the theory of $A_{\infty}$-algebras and categories.  As a result, before we can get on with the definition of the orientation data associated to the presentation of a category as the category of modules over a certain type of ($A_{\infty}$) algebra, we recall some of this theory.  This occupies the whole of Chapter \ref{inftystuff}.
\bigbreak
We concentrate, for the most part, on categories of perfect modules over compact Calabi-Yau algebras.  It turns out that these categories lend themselves to Algebraic Geometry, in the sense that their objects (up to quasi-isomorphism) can be arranged (with repetitions) into an infinite ascending union of algebraic varieties, with compositions of morphisms given by operations on finite-dimensional bundles lying over these varieties.  This is achieved by restricting to a smaller category, the category of twisted objects of Definition \ref{twl}, which is the smallest subcategory of the category of perfect modules to contain the image of the Yoneda embedding and be closed under shifts and triangles.  In many examples, this category contains a representative of every quasi-isomorphism class of objects in the category of perfect modules, and this smaller category lends itself very well to the whole theory of motivic Donaldson--Thomas theory.  Our exposition of the necessary $A_{\infty}$ algebra for understanding the construction of orientation data will concentrate heavily on the nice geometric nature of this category.  In particular we work towards a proof of the technical lemma that this smaller category is a cyclic Calabi-Yau category, which is essentially the statement that it fits into the motivic Donaldson--Thomas machine.  This stage-setting is again essentially technical in nature, and occupies us in Chapter \ref{quiveralgebras}.
\bigbreak
Chapter \ref{stackofobjects} amounts to showing that, at least in the quiver cases we care about, up to quasi-isomorphism all flat families of modules can be pulled back from the nice category of twisted objects.  This is important, since the motivic weight that forms the crucial ingredient of the integration map $\Phi$, applied to a family of objects in a category, is built from the structure of that category, and so we need to be able to arrange that structure (i.e. homomorphisms and their compositions) into families.  This chapter deals with the problem of finding cyclic Calabi-Yau models for endomorphism algebras in (constructible) families, where the previous chapter merely deals with this problem for points.
\bigbreak
Once all of the ingredients are in place, we actually get on with constructing orientation data.  It turns out that the notion of orientation data makes sense even when the theory of motivic Donaldson--Thomas invariants is not really defined (i.e. in higher dimensions, and when our category of modules is not quasi-equivalent to a nice subcategory like the category of twisted objects).  In Chapter \ref{ODchapter} we give a very general construction of orientation data, the original version of which can be found in \cite{KS}.
\bigbreak
In Chapter \ref{exchap} we deal with some examples.  The main technical difficulties here are in dealing with noncompact spaces, for which the categories of compactly supported sheaves, which are the natural candidates for Calabi-Yau categories, diverge a little from the categories we consider up until this point, in that they are not given by categories of perfect modules over finite-dimensional Calabi-Yau algebras.  We deal with this problem in the smooth case, at the same time making clear the link between this work and \cite{BBS}.  
\bigbreak
It often happens that our big ($A_{\infty}$) categories have Abelian subcategories $\mathcal{A}$, and that we have a natural handle on the objects in $\mathcal{A}$ as living in a space cut out from a smooth ambient space as a scheme-theoretic degeneracy locus of a function, which we will denote $\tra(W)$.  In such a situation we have an alternative integration map given by pulling back the motivic vanishing cycle of $\tra(W)$.  We address the issue of how this compares with the Kontsevich--Soibelman integration map.
\bigbreak
It turns out that our `natural handle' on the space of objects in this Abelian subcategory as the scheme-theoretic degeneracy locus of a function comes from much the same information that gives us a choice of orientation data.  The comparison result that one would like to be true, then, is that the two integration maps (the Kontsevich--Soibelman integration map and the integration against vanishing cycles) are the same, and this is Proposition \ref{BBScompare}.  The vanishing cycles in the examples we consider in Chapter \ref{exchap} are easy to define, and actually possible to calculate in many examples (e.g. \cite{BBS}, \cite{DM11}, \cite{MMNS11}, \cite{moz11}...).  So one may wonder why we should bother with the seemingly very heavy machine of \cite{KS}, given this result.  A partial answer is provided in Section \ref{conif} -- the reduction of the theory to considering minimal potentials and orientation data in fact makes the example of the noncommutative conifold almost trivial, while calculating the relevant motivic vanishing cycles cannot be called straightforward (though it is done, from this point of view, in \cite{MMNS11}).
\bigbreak
A deeper answer to the question of why we should be concerned with the orientation data picture is that we are not always interested solely in comparing moduli of objects in a single Abelian subcategory of our categories.  In studying phenomena such as the PT/DT correspondence, cluster transformations and noncommutative crepant resolutions, we would like to be able to determine the value of moduli of objects in different hearts of derived categories under the integration map.  The Kontsevich--Soibelman machine is designed to let us do this.  This leads us to a natural question: say we have our two Abelian categories, and the space of objects of both of them occur naturally as the critical locus of functions on smooth spaces, then are the two integration maps coming from motivic vanishing cycles the same?  The correct language in which to consider this problem is that of orientation data, where it tends to be rather manageable.  Chapter \ref{gocompare} is devoted to answering a number of variations of this question.  We show that natural choices of orientation data are unchanged by cluster transformations, flops, and a large class of tilts in general.  In particular, we settle Conjecture 12 of \cite{KS} in the affirmative.  We finish by considering the conifold once more -- the way in which the orientation data glues across the two categories of coherent sheaves obtained by flopping the resolved conifold singularity is a nice demonstration of how well the seemingly difficult orientation data behaves.
\section{Notation and conventions}
\label{notation}
We work throughout over a field $k$, of characteristic zero.  We will work for much of the time with graded $k$-linear categories, that is, categories $\mathcal{C}$ such that for each $X,Y\in\mathcal{C}$, $\Hom_{\mathcal{C}}(X,Y)$ has the additional structure of a $\mathbb{Z}$-graded $k$-linear vector space, and for $X,Y,Z\in\mathcal{C}$ the map
\[
\Hom_{\mathcal{C}}(Y,Z)\otimes \Hom_{\mathcal{C}}(X,Y)\rightarrow \Hom_{\mathcal{C}}(X,Z)
\]
is a degree zero homogeneous map of $\mathbb{Z}$-graded $k$-vector spaces.  Given an arbitrary $k$-linear category $\mathcal{C}$ one constructs a graded category $\mathcal{C}_{\mathbb{Z}}$ whose objects are sequences of objects in $\mathcal{C}$, indexed by $\mathbb{Z}$, and whose morphisms are given by
\[
\Hom_{\mathcal{C}_{\mathbb{Z}}}((\ldots,X^{-1},X^0,\ldots),(\ldots,Y^{-1},Y^0,\ldots)):=\prod_{i\in\mathbb{Z}} \Hom_{\mathcal{C}}(X^i,Y^i).
\]
If $\mathcal{A}$ is an Abelian category with enough projectives or enough injectives, we denote by $\mathcal{A}_{\ext}$ the graded category obtained by setting $\ob(\mathcal{A}_{\ext}):=\ob(\mathcal{A})$, and setting 
\[
\Hom_{\mathcal{A}_{\ext}}^n(X,Y):=\Ext_{\mathcal{A}}^n(X,Y)
\]
for $X,Y\in \ob(\mathcal{A})$.
\bigbreak
If, on the other hand, $\mathcal{A}$ is a category with an understood shift functor, and an understood notion of homotopy equivalence of morphisms, preserved by the shift functor, we denote by $\mathcal{A}_{\ext}$ the graded category, again satisfying $\ob(\mathcal{A}_{\ext}):=\ob(\mathcal{A})$, and with
\[
\Hom_{\mathcal{A}_{\ext}}^n(X,Y):=\Hom_{\mathcal{A}}(X,Y[n])/\sim_{\mathrm{homotopy}}.
\]
In general we give objects the cohomological grading, denoted by superscripts, and the shift functor is given by $Z[n]^m:=Z^{n+m}$.  In swapping with the homological grading we set $Z^n:=Z_{-n}$.
\bigbreak
We will be dealing for much of the time with differential graded categories, by which we mean a category $\mathcal{C}$ such that each $\Hom_{\mathcal{C}}(X,Y)$ has the structure of a differential graded complex of $k$-vector spaces, and each map
\[
\Hom_{\mathcal{C}}(Y,Z)\otimes \Hom_{\mathcal{C}}(X,Y)\rightarrow \Hom_{\mathcal{C}}(X,Z)
\]
has the structure of a graded map of complexes.  Finally we will use also the notion of symmetric monoidal category.  A monoidal category is a category $\mathcal{C}$ equipped with a bifunctor 
\[
-\otimes -:\mathcal{C}\times\mathcal{C}\rightarrow \mathcal{C}
\]
along with a natural isomorphism
\[
(-\otimes -)\otimes -\rightarrow -\otimes (-\otimes -)
\]
expressing associativity, and an object $\idmon_{\mathcal{C}}$ that is a monoidal unit.  These structures are required to satisfy the usual coherence conditions (see \cite{cfwm}).  The symmetric structure comes from another natural isomorphism
\[
\nu: -\otimes -\rightarrow (-\otimes -)\circ \sw
\]
where $\sw:\mathcal{C}\times\mathcal{C}\rightarrow \mathcal{C}\times\mathcal{C}$ is the isomorphism swapping arguments.  This natural isomorphism $\nu$ is required to satisfy $(\nu\circ \sw)\nu=\id_{-\otimes -}$, and another set of obvious coherence conditions (again see \cite{cfwm}).
\smallbreak
Generally we consider twisted symmetric monoidal structures.  In the case of categories of $\mathbb{Z}$-graded objects for which there is a natural symmetric monoidal structure (for example $(\vect)_{\mathbb{Z}}$, the category of graded vector spaces over $k$), the twisted structure is obtained by multiplying the isomorphism $\nu$ by -1 each time it swaps two objects which are both oddly graded.  This is a fancy way of encoding the Koszul sign convention.  
\medbreak
We denote by $\dgvect$ the symmetric monoidal differential graded category of differential graded vector spaces over $k$, with the twisted monoidal structure.  If $(V,d_V)$ is a differential graded vector space, i.e. a pair of a graded vector space $V$ and a $k$-linear map $d_V:V\rightarrow V$ of (cohomological) degree 1 such that $d_V^2=0$, and $(W,d_W)$ is another such differential graded vector space, then $\Hom_{\dgvect}^n((V,d_V),(W,d_W))$ is the subspace of the vector space of morphisms between the underlying, ungraded, vector spaces $V$ and $W$, such that homogeneous elements $v\in V^m$ are sent to homogeneous elements $w\in W^{m+n}$.  So we obtain a graded vector space $\Hom_{\dgvect}((V,d_V),(W,d_W))$, which comes with a differential defined on homogeneous $f$ given by
\[
d(f)=d_W\circ f-(-1)^{|f|} f\circ d_V,
\]
where here and elsewhere $|\bullet|$ is the  function taking homogeneous elements to their degree.
\smallbreak
For much of the time we are interested in the category of $A_{\infty}$-modules over some $A_{\infty}$-algebra $A$ (these notions are recalled in Chapter \ref{inftystuff}).  We adopt the following conventions:
\begin{enumerate}
\label{basicnot}
\item
$A\lMod$, $A\lmod$ will denote the ordinary category of left $A$-modules, or the ordinary category of left $A$-modules with finite-dimensional total homology, respectively.
\item
$\rMod A$, $\rmod A$ will denote the ordinary category of right $A$-modules, or the ordinary category of right $A$-modules with finite-dimensional total homology, respectively.
\item
$A\biMod B$, $A\bimod B$ will denote the ordinary category of $(A,B)$-bimodules, $(A,B)$-bimodules with finite-dimensional total homology, respectively.
\end{enumerate}
All of these categories have $A_{\infty}$-enriched versions, which again are recalled in Chapter \ref{inftystuff}.  These are constructed as differential graded categories, and they are denoted by adding the subscript $\infty$ to the notation, so for example the enriched category of left $A$-modules is denoted $A\lModi$.  In each of these $A_{\infty}$-categories the morphism spaces are differential graded vector spaces, and we obtain a triangulated category by taking the zeroth homology of these differential graded vector spaces.  In each case this new category is denoted by the same notation as the old one, but placed inside a $\Di(-)$, so for example $\Di(A\lmod)$ denotes the category of left $A$-modules with finite-dimensional total homology, with morphism spaces obtained by setting
\[
\Hom_{\Di(A\lmod)}(M,N):=\Ho^0(\Hom_{A\lmodi}(M,N)).
\] 
\smallbreak
Many of the categories we encounter will come with a natural triangulated structure, or a natural triangulated structure will be understood on their homotopy categories.  We adopt the convention that if a background (triangulated) category $\mathcal{C}$ is understood, and $S$ is a subset $S\subset \ob(\mathcal{C})$, then $\langle S \rangle_{\triang}$ is the smallest sub-triangulated category of $\mathcal{C}$ containing all the objects of $S$, and which is also closed under isomorphisms, shifts, and triangles.  If the triangulated structure is understood on the homotopy category, the condition that $\langle S \rangle_{\triang}$ is closed under triangles is replaced by the condition that it is closed under diagrams that become triangles upon descent to the homotopy category.  We also adopt the convention that $\langle S \rangle_{\thick}$ is the smallest full subcategory of $\mathcal{C}$ to contain $\langle S \rangle_{\triang}$ and also be closed under diagrams
\begin{equation}
\label{retractp}
M\rightarrow N\rightarrow M
\end{equation}
that compose to the identity in the original category (in which case (\ref{retractp}) is a retract), or in the homotopy category (these diagrams will be called \textit{homotopy retracts}).  If we wish to make the category $\mathcal{C}$ clear we write $\triang_{\mathcal{C}}(S)$ for $\langle S \rangle_{\triang}$ and $\thick_{\mathcal{C}}(S)$ for $\langle S \rangle_{\thick}$.\smallbreak
Let $A$ be an $A_{\infty}$-algebra.  Then $A$ forms naturally a left $A_{\infty}$-module over itself.  We will generally write
\[
\Perf(A\lMod):=\thick_{A\lMod}(A)
\]
and
\[
\Perf(A\lModi):=\thick_{A\lModi}(A)
\]
and we will use the following proposition:
\begin{prop}\cite{derivingDG}
\label{cpctprop}
The sets $\ob(\Perf(A\lMod))=\ob(\Perf(A\lModi))$ are exactly the sets of compact objects of $\Di(A\lMod)$, i.e. those objects $M$ for which the functor $\Hom_{\Di(A\lMod)}(M,-)$ commutes with the taking of arbitrary direct sums.
\end{prop}
We define $\Perf(\rModi A)$ and $\Perf(\rMod A)$ similarly, and the analogous result is true for these categories.  There is a version of the above proposition for modules over an $A_{\infty}$-category $\mathcal{C}$, where we replace the free left or right $A$-module with the image of the contravariant or covariant Yoneda embedding -- these notions will be recalled in Chapter \ref{inftystuff}.
\smallbreak
Finally, we adopt the following convention: $A_{\infty}$ algebra is a subject, while an $A_{\infty}$-algebra is one of the objects studied in it.

\chapter{Motivic Donaldson--Thomas invariants and the role of orientation data}
\label{example}
The purpose of this chapter is to introduce the salient features of the theory of motivic Donaldson--Thomas invariants, and explain the role of orientation data.  The aim is to explain, roughly, what orientation data is, why it is needed, and what goes wrong when we ignore it.   
\section{An introduction to the Donaldson--Thomas invariant}
Let $X$ be a smooth projective 3-fold.  Then for a given Hilbert polynomial $p$ we may consider $\mathcal{M}_p$, the moduli space of semistable torsion-free coherent sheaves $\mathcal{F}$ on $X$ that have Hilbert polynomial $p$.  In order to get a reasonable space we impose some kind of stability condition (Gieseker stability or slope stability), and under suitable conditions this space will be a finite type scheme, which we will denote by $\mathcal{M}$, for example if we fix the rank and the degree of $\mathcal{F}$ to be given coprime numbers, so that semistability implies stability (see for example \cite{HL}).  It is an important feature of the space $\mathcal{M}$ that it is compact: the Donaldson--Thomas count for $\mathcal{M}$ is the degree of some cycle class of zero-dimensional subschemes of $\mathcal{M}$, and in the compact case this is just given by the count of the points in this class, with multiplicity.  In the noncompact case this breaks down somewhat.
\bigbreak
We arrive at this zero-dimensional class by next assuming that our 3-fold $X$ was, all along, a Calabi-Yau 3-fold.  This implies, in particular, that the expected dimension of $\mathcal{M}$ is zero.  More precisely, $\mathcal{M}$ comes equipped with a perfect obstruction theory $L^{\bullet}:=[E_1\rightarrow E_0]$ which satisfies the condition $\dim(E_1)=\dim(E_0)$.  From such data one can construct (see \cite{NormalCone}) a virtual fundamental class of the correct dimension in $A^{*}(\mathcal{M})$, i.e. a class $[\mathcal{M}]_{\vir}\in A^{0}(\mathcal{M})$.  Finally, the Donaldson--Thomas invariant is given by $\deg[\mathcal{M}]_{\vir}$.
\bigbreak
The justification for taking this virtual fundamental class is the fact that, since our moduli scheme $\mathcal{M}$ `should' be zero-dimensional, there is an intuition that the correct number (the Donaldson--Thomas invariant) should be obtained by perturbation.  So if $\mathcal{M}$ has a component that is smooth, with an obstruction bundle over it that is just a vector bundle, the contribution from that component should be the Euler class of that vector bundle.  Similarly, if $\mathcal{M}$ has a component that has underlying topological space a point, but structure sheaf of length $n$, then its contribution to the DT invariant should be $n$, as this component `should' generically deform to give $n$ points (the inverted commas here are on account of the fact that we remain vague as to where these deformations are taking place).  The taking of a virtual fundamental class is a way of using excess intersection theory to make all of this precise.
\bigbreak
The perfect obstruction theory $L^{\bullet}$ constructed by Thomas in \cite{RTthesis} has the extra property that it is symmetric, in the sense of \cite{ObstructionsHilbert}, that is there is an isomorphism $\theta:L^{\bullet}\rightarrow (L^{\bullet})^{\vee}[1]$ in the derived category of coherent sheaves on $\mathcal{M}$ satisfying $\theta^{\vee}[1]=\theta$.  So, returning to the situation in which a component $\mathcal{M}_1$ of $\mathcal{M}$ is smooth, the obstruction bundle is automatically a vector bundle, and isomorphic to the cotangent bundle, and in this case we can say exactly what we think the contribution to the Donaldson--Thomas count of the component should be: $(-1)^{\dim(\mathcal{M}_1)}\chi(\mathcal{M}_1)$.  This is the first indication that in fact the contribution of every component should be (and actually is, in the case in which the perfect obstruction theory with which we calculate our Donladson--Thomas count is symmetric) a weighted Euler characteristic, with the weighting of smooth points given by the parity of the dimension, and the weight of isolated points given by the length of their structure sheaves.
\bigbreak
The goal, then, is to associate to an arbitrary finite-type scheme $Y$ a constructible function $\nu_Y$, with image lying in the integers, such that, in the event that $Y$ is compact and is equipped with a symmetric perfect obstruction theory, there is an equality
\begin{equation}
\label{Behrends}
\deg[Y]_{\vir}=\sum_{n\in \mathbb{Z}} n \cdot \chi(\nu_Y^{-1}(n)),
\end{equation}
where the class on the left hand side is the virtual fundamental class constructed from the symmetric perfect obstruction theory.  For schemes defined over $\Cp$, this is achieved in \cite{Behr09}, and this function $\nu_Y$ is called the microlocal function for $Y$.  Note that in the case in which $Y$ is a noncompact scheme with a symmetric perfect obstruction theory, the machinery of \cite{NormalCone} still gives us a virtual fundamental class $[Y]_{\vir}$, for which (\ref{Behrends}) does not make sense, since $\deg[Y]_{\vir}$ will be undefined.  In this case, however, we can take the right hand side as our definition of the Donaldson--Thomas count.  
\bigbreak
Recall the moduli space $\mathcal{M}$ we started with.  For gauge-theoretic reasons (see \cite{JS08}), a formal neighborhood of an arbitrary sheaf $\mathcal{F}$, considered as a point in $\mathcal{M}$, is given by the following setup.  Let $\Cp^t$ be some affine space, and let $f$ be the germ of an analytic function defined and equal to zero at the origin.  Then a formal neighborhood of $\mathcal{F}$ is isomorphic to a formal neighborhood of the origin in the critical locus of $f$.  This becomes an important observation given the following fact regarding the microlocal function $\nu_{\mathcal{M}}$: if a scheme $Y$ is given by the critical locus of some function $f$ on some smooth $d$-dimensional scheme, at least formally around some point $y\in Y$, then $\nu_Y(y)$ is given by $(-1)^d(1-\chi(\mf(f,y)))$, where $\mf(f,y)$ is the Milnor fibre of $f$ at the point $y$ (see \cite{Behr09}).
\section{Motivic vanishing cycles and Milnor fibres}
\label{mvc}
We first recall just the group structure of the ring of motives $\Mot(X)$, for $X$ a finite type scheme.  This is generated, as an Abelian group, by symbols $[f:Y\rightarrow X]$, where $Y$ is a finite type scheme, subject to the relations
\[
[f_1:Y_1\rightarrow X]+[f_2:Y_2\rightarrow X]=[f_1\amalg f_2: Y_1\amalg Y_2\rightarrow X]
\]
and 
\[
[g_1:Z_1\rightarrow X]=[g_2:Z_2\rightarrow X]
\]
if there is a morphism $h:Z_1\rightarrow Z_2$, which induces an isomorphism of geometric points, making the obvious diagram commute.  Let $\mu_n$ be the group of the $n$th roots of unity in $\Cp$, with group operation given by multiplication.  We give $X$ the trivial $\mu_n$-action, and define $\Mot^{\mu_n}(X)$ in the natural way, as being generated by $\mu_n$-equivariant maps $[f:Y\rightarrow X]$ with relations as above, but with all maps assumed to be $\mu_n$-equivariant.  We consider only $\mu_n$-equivariant schemes $Y$ satisfying the condition that every orbit lies in a $\mu_n$-invariant open affine subscheme of $Y$.  We impose also the extra relation that if $Z$ is a $\mu_n$-equivariant $d$-dimensional vector bundle on a $\mu_n$-equivariant variety $Y$, and $[f:Y\rightarrow X]$ is a $\mu_n$-equivariant map, then 
\[
[Z\rightarrow X]=[Y\times \mathbb{A}^d\rightarrow X]
\]
where in both symbols the map is given by the projection to $Y$, followed by $f$.
There is a natural morphism 
\[
\mu_{mn}\rightarrow \mu_{n}
\]
given by sending $z$ to $z^m$.  This embeds $\Mot^{\mu_n}(X)$ in $\Mot^{\mu_{mn}}(X)$, and we let $\Mot^{\hat{\mu}}(X)$ be the colimit of these embeddings, as we vary $m$ and $n$.  
\medbreak
Let $h:X_1\rightarrow X_2$ be a morphism of finite type schemes.  Then we obtain a morphism $h_*:\Mot^{\hat{\mu}}(X_1)\rightarrow \Mot^{\hat{\mu}}(X_2)$ by sending $[f:Y\rightarrow X_1]$ to $[h\circ f:Y\rightarrow X_2]$.  The pullback morphism $h^*:\Mot^{\hat{\mu}}(X_2)\rightarrow \Mot^{\hat{\mu}}(X_1)$ is defined by sending $[f:Y\rightarrow X_1]$ to $[Y\times_{X_1} X_2\rightarrow X_2]$.
\medbreak
The motive
\begin{equation}
\label{baddef1}
\sum_{n\in \mathbb{Z}} n [\nu_{X}^{-1}(n)]\in \Mot(\Spec(k))
\end{equation}
is in some sense a motivic refinement of the Donaldson--Thomas invariant, but it is a somewhat unnatural halfway point.  For we have replaced the measure $\chi$ with a motivic measure, without replacing the weight by a motivic weight.  The natural refinement of our weight, from a number to a motive, is given by taking the motivic Milnor fibre, instead of its Euler characteristic.  So we next recall some of the definitions and formulae regarding motivic vanishing cycles and nearby fibres.  
\bigbreak
Let $f$ be some function from a smooth complex finite type scheme $X$ to $\Cp$.  Let 
\[
\xymatrix{
Y\ar[d]^h\\
X
}
\]
be an embedded resolution of the function $f$.  Then the motivic nearby cycle $[\psi_f]$, as defined by Denef and Loeser in \cite{DL01}, in terms of arc spaces, has an explicit formula in terms of this embedded resolution, which we will describe presently.  The level set $(fh)^{-1}(0)$ consists of a set of divisors, indexed by a set forever denoted $J$, with each divisor $D_i$ meeting every other one transversally.  Given $I\subset J$, a nonempty subset, let $D^0_{I}$ be the complement in the intersection of all the divisors in $I$ of the union of the divisors that are not in $I$.  So the $D^0_{I}$ form a stratification of $(fh)^{-1}(0)$, with deeper strata coming from larger subsets $I\subset J$.  
\medbreak
Let $I\subset J$ be a subset.  Let $U\subset Y$ be an open patch (in the analytic topology), intersecting only those $D_i$ for $i$ that are in $I$.  We pick $U$ so that $fh$ has defining equation
\[
\prod_{D_i\in I}x_i^{a_i}u
\]
where $x_i$ is a local (analytic) coordinate for $D_i$, $a_i$ is the order of vanishing of $fh$ on $D_i$, and $u$ is a unit.  Let $a_I$ be the greatest common divisor of the $a_i$ appearing above.  Then we form an \'{e}tale cover of $D^0_I\cap U$ by taking the natural projection to $D^0_I\cap U$ from the scheme
\begin{align}
U'\subset (D^0_I\cap U)\times \Cp\\
U'=\{(x,y)|y^{a_I}=u(x)\}.
\end{align}
These \'{e}tale covers patch to form an \'{e}tale cover
\[
\xymatrix{
\tilde{D}_I\ar[d]\\
D^0_I.
}
\]
The scheme $\tilde{D}_I$ over $D_I$ carries the obvious action under the group of $a_I$th roots of unity, and so we obtain an element of $\Mot^{\hat{\mu}}(X)$ by pushforward from $D_I$ to $X$ along $h$.  Finally, the formula is
\begin{equation}
\label{MFformula}
[\psi_f]=\sum_{\emptyset\neq I\subset J}(1-\mathbb{L})^{|I|-1}[\tilde{D}_I]\in \Mot^{\hat{\mu}}(X).
\end{equation}
Let $T$ be a constructible subset of $X$.  Restriction to $T$ defines a map from $\hat{\mu}$-equivariant motives over $X$ to $\hat{\mu}$-equivariant motives over $T$.  Pushforward from $T$ to a point gives us an absolute $\hat{\mu}$-equivariant motive.  We let $[\psi_f]_T$ denote the image of $[\psi_f]$ under the composition of these two maps.
\smallbreak
Let $f$ be as above, and let $p\in X$ be a point in $f^{-1}(0)$.  Then the \textit{motivic Milnor fibre} of $f$ at $p$ is defined to be
\[
\MF(f,p):=[\psi_f]_p \in \Mot^{\hat{\mu}}(\Spec(k)).
\]
If $X$ is affine space, and $f$ is a function vanishing at the origin, then we define
\[
\MF(f):=\MF(f,0).
\]
Finally, define the \textit{motivic vanishing cycle}:
\[
[\phi_f]:=[\psi_f]-[f^{-1}(0)]\in \Mot^{\hat{\mu}}(X).
\]
In the above equation, $[f^{-1}(0)]$ carries the trivial $\hat{\mu}$-action.
\smallbreak
We close this section with a fundamental theorem regarding motivic vanishing cycles.
\begin{thm}(Motivic Thom-Sebastiani)\cite{DL99}
\label{TStheorem}
Let $V$ and $V'$ be vector bundles on smooth schemes $X$ and $X'$ respectively.  Let $\pi$ and $\pi'$ be the projections from $X\times X'$ to $X$ and $X'$ respectively.  Let $f$ and $f'$ be algebraic functions on the vector bundles $V$ and $V'$ respectively.  Denote by $f\oplus f'$ the sum of the pullbacks of $f$ and $f'$ to the vector bundle $\pi^*(V)\oplus \pi'^*(V')$.  Then there is an equality
\begin{equation}
[-\phi_{f\oplus f'}]=\pi^*([-\phi_f])\cdot \pi'^*([-\phi_{f'}])\in \Mot^{\hat{\mu}}(X\times X').
\end{equation}
\end{thm}
Of course this theorem doesn't yet make much sense for us: we have not defined the multiplication in the ring of motives appearing on the right hand side.  This is given by Looijenga's exotic product, as introduced in \cite{Looi00} -- its salient feature for us is that it makes the above theorem true, and also, common realizations of motives become ring homomorphisms, once the Abelian group of motives is given this ring structure.  We must also add that what one really needs is a slightly stronger theorem, with the assumption that $f$ and $f'$ are algebraic replaced by the assumption that they are formal functions defined on the zero sections of $V$ and $V'$ respectively.  In this section we will stick to algebraic functions, for which the theorem is indeed a theorem.
\smallbreak
Ultimately the ring $\Mot^{\hat{\mu}}(X)$ is not ideal for our purposes.  Orientation data is most easily managed in the ring $\overline{\Mot}^{\hat{\mu}}(X)$, which is the target of a ring homomorphism $\Mot^{\hat{\mu}}(X)\rightarrow \overline{\Mot}^{\hat{\mu}}(X)$ through which the realizations of motives factor (so the Thom-Sebastiani theorem remains true in it).  It is also one of the key features of the motivic Donaldson--Thomas count that it adequately handles stacks with affine stabilizers.  So in the case of absolute motives the ring that we end up working in is the localized ring $\textbf{M}$, as in (\ref{Mdef}).

\section{A basic example}
\label{basicexample}
We haven't even got to the definition of the motivic Donaldson--Thomas count, but we are already launching into an example!  The idea is to get a grip of what we would like a motivic Donaldson--Thomas count to do, and what problems we expect to find when we ask it to do these things.  We start with an apology -- while the following example is in some ways the most basic possible example exhibiting all the relevant phenomena, it can hardly be said to be basic in an absolute sense.  There is some fiddling about with motives to be done; in the interests of presentation some of these calculations are postponed to the Appendix \ref{motivesapp}
\bigbreak
We start on the `geometric' side of Koszul duality, since this seems a better place to motivate things, and is perhaps more familiar.  Koszul duality lies in the background of much of the work in this thesis -- excellent references are \cite{derivingDG} and \cite{Yoneda}.  Let $B=\Cp[x]/\langle x^3\rangle$.  We will be looking at the moduli space of finite-dimensional modules over $B$.
\medbreak
In fact $B$ is a special example of a `Jacobi algebra', or a `superpotential' algebra.  Let $Q$ be the quiver with one vertex and one loop.  Then $\Cp Q\cong \Cp\langle a \rangle$, where $\Cp Q$ denotes the free path algebra of the quiver $Q$.  Let $W$ be the cyclic word in this quiver given by $W=a^4$.  Then we are meant to form the `noncommutative differentials' of $W$ by differentiating it with respect to each of the arrows in $Q$ (see \cite{ginz} or \cite{CBEG} for an explanation of what this means).  Here, this noncommutative generalization of differential calculus reduces to familiar calculus, since $\Cp Q$ is commutative.  So the only noncommutative differential we need to think about is
\[
\frac{\partial}{\partial a}W=4a^3.
\]
The statement that $B$ is a Jacobi algebra amounts to saying that 
\[
B\cong \Cp Q/\langle \frac{\partial}{\partial a}W \rangle.
\]
This puts us in a special situation, noted in \cite{ginz}, \cite{Conifold}, \cite{Seg08}, in which we have a way of coherently embedding the representation spaces of $B$-modules as subschemes of smooth schemes.  The word `coherently' doesn't yet have a precise meaning here, but has to do with the problem of comparing the motivic weight associated to extensions of modules to the motivic weights of those modules themselves, which in turn will be the central difficulty when it comes to checking that putative integration maps from families of $B$-modules to motives preserve associative products.  This in turn is the central problem motivating the introduction of orientation data.
\bigbreak
How this works out in our case is as follows.  Define
\[
\Rep_n(B)=\Hom_{\alg}(B,\Mat_{n\times n}(\Cp)),
\]
the set of homomorphisms of unital algebras.  This is a scheme, the points of which correspond to representations of $B$.  In general the more natural object to study is perhaps the stack formed under the conjugation action of $\GL_n(\Cp)$, but for the time being we will really just be looking at the above scheme.  Similarly, we define
\[
\Rep_n(\Cp Q)=\Hom_{\alg}(\Cp Q, \Mat_{n\times n}(\Cp)),
\]
then since a representation of $B$ is just a representation of $\Cp Q$ satisfying some relations, $\Rep_n(B)$ is defined as a Zariski closed subscheme of this \textit{smooth} scheme.  There is a map 
\[
\ev_a:\Rep_n(\Cp Q)\rightarrow \Mat_{n\times n}(\Cp)
\]
that sends 
\[
\theta\mapsto\theta(a).
\]
In fact this is clearly an isomorphism.  It turns out (and this is a general fact about Jacobi algebras) that 
\[
\Rep_n(B)=\crit(\tra((\ev_a)^4)).
\]
For a general Jacobi algebra we replace $(\ev_a)^4$ with a function of evaluation maps built from $W$, and the corresponding statement remains true.
\medbreak
The goal of this subject is to define motivic Donaldson--Thomas counts, that soup up the old one, which was just the Euler characteristic weighted by a microlocal function $\nu$.  Recall that the microlocal function of a scheme at a point $x$, at which the scheme is locally described as $\crit(f)$ for some $f$ on a $d$-dimensional ambient smooth scheme, is just $(-1)^d(1-\chi(\MF(f,x))$.  Consider just 
\[
\Rep_1(B)\cong\Spec(B).
\]
The fact that we have an explicit presentation of our space as a critical locus enables us to go ahead and refine the microlocal function $\nu_{\Rep_1(B)}$ to a motive, which is given by minus the (absolute) motivic vanishing cycle of the function $x^4$.  Here and elsewhere we will adopt the shorthand that where a function $f(x_1,\ldots,x_n)$ appears without reference to a space that it is a function on, that space will always be assumed to be affine $n$-space, and the motivic vanishing/nearby cycle of it is the motivic vanishing/nearby cycle of the function on affine $n$-space.  We define
\[
\DT(\Rep_1(B)):=[-\phi_{x^4}]_{\Rep_1(B)}.
\]
(There is a sign discrepancy here, since we are going to want to pull back the motivic weight from an ambient stack that takes into account stabilizer groups, which in this case are 1-dimensional).  In order to establish uniform notation with what follows we rewrite this as
\begin{equation}
\label{DTtry}
\DT(\Rep_1(B)):=[-\phi_{\tra(T^4)}]_{\Rep_1(B)},
\end{equation}
where $\tra(T^4)$ is considered as a function on $\Cp$ by identifying $\Cp$ with the ring of $1\times 1$ matrices.  Since $\Rep_1(B)$ is just a point, in this case we have
\[
\DT(\Rep_1(B))=(1-\MF(x^4)).
\]
The unique closed point of the space $\Rep_1(B)$ is given by a $1\times 1$ matrix, the zero matrix.  Call this representation $M$.  Considered as a module for the quiver algebra $\Cp Q/\langle a^3\rangle$, this is the one-dimensional simple module killed by all the arrows of $Q$.  In this example it is easy enough to explain what we mean by `preservation of the ring structure'.  Define
\[
\Rep_1(B)\star \Rep_1(B)
\]
to be the stack of flags $M\subset N$ with $N/M\cong M$.  The stabilizer at any point is given by $\Hom(M,M)\cong \Cp$, and in fact this stack can be described explicitly as a group quotient of the space $\Mat_{\sut, 2\times 2}(\Cp)$ of strictly upper-triangular $2$ by $2$ matrices by the trivial action of the additive group $\Cp\cong \Hom(M,M)$.  So we write the motive of this stack as
\[
[\Rep_1(B)\star \Rep_1(B)]=[\Mat_{\sut,2\times 2}(\Cp)]/\mathbb{L}.
\]
Now what we really want is the identity, in the ring of motives:
\begin{equation}
\label{yesplease}
\DT[\Rep_1(B)\star \Rep_1(B)]=\DT[\Rep_1(B)]\cdot \DT[\Rep_1(B)],
\end{equation}
where on the right hand side we use Looijenga's product on the ring of motives.  From the motivic Thom-Sebastiani Theorem \ref{TStheorem} , we deduce that
\[
\DT[\Rep_1(B)]\cdot \DT[\Rep_1(B)]=(1-\MF(x^4+y^4)).
\]
\begin{prop}
\label{x4y4}
Denote the representation ring of $\mathbb{Z}_4$ by $\mathbb{Z}[\alpha]/\alpha^4$, where $\alpha$ is the 1-dimensional representation sending $1\in \mathbb{Z}_4$ to multiplication by $i$.  There is an equality of motives
\[
\MF(x^4+y^4)=[C_1]-4\mathbb{L},
\]
where $C_1$ is a genus 3 curve with the representation $2(\alpha+\alpha^2+\alpha^3)$ on its middle cohomology.
\end{prop}
We defer the proof of this proposition to the start of Appendix \ref{motivesapp}.
\medbreak

By using Proposition \ref{x4y4} and the motivic Thom-Sebastiani theorem \ref{TStheorem} we can calculate the right hand side of equation (\ref{yesplease}).  What, then, of the left hand side?  Well, first we should define it!  This we do as follows: the coarse moduli space $\Mat_{\sut,2\times 2}(\Cp)$ of our stack $\Mat_{\sut, 2\times 2}(\Cp)/\mathbb{A}^1$ is a subscheme of $\Rep_{2}(B)$.  Let 
\[
\iota:\Mat_{\sut, 2\times 2}(\Cp)\hookrightarrow \Rep_{2}(B)
\]
be the inclusion.  Then recall that we want a motivic refinement of the weighted Euler characteristic
\[
\sum_{n\in \mathbb{Z}}n\cdot \chi(\iota^*(\nu_{\Rep_{2}}(B))^{-1}(n)).
\]
It's clear enough what this should be.  The space $\Rep_{2}(B)$ occurs again as a critical locus of a function on a smooth space, the function $\tra(T^4)$ on the space of $2\times 2$ matrices, and so a refinement of the pullback of the microlocal function is already at hand, we can just pull back the motivic vanishing cycle of the function $\tra(T^4)$ along the inclusion of the space of strictly upper-triangular matrices into the space of all matrices, i.e. take $[-\phi_{\tra(T^4)}]_{\Mat_{\sut,2\times 2}(\Cp)}$.  The content of the word `coherently' in the statement that a Jacobi algebra presentation enables us to coherently express different representation spaces as critical loci will amount to the claim that this naive pulling back actually gives a good answer, one that gives the equality (\ref{yesplease}).  So let us unpick this particular case.\medbreak
We follow the natural suggestion for defining the left hand side of (\ref{yesplease}), that is we write
\begin{equation}
\label{DTstar}
\DT[\Rep_1(B)\star \Rep_1(B)]:=[-\phi_{\tra(T^4)}]_{\Mat_{\sut, 2\times 2}(\Cp)}\mathbb{L}^{-1}.
\end{equation}
Working this quantity out will occupy the next section.
\section{Verifying preservation of ring structure: an example}
To start with, we should work out an embedded resolution of 
\[
\tra(T^4):\Mat_{2\times 2}(\Cp)\rightarrow \Cp.
\]
The function $\tra(T^4)$ has its worst singularity at $0$, and is homogeneous, so a good start would be to blow up at the zero matrix.  Write $X=\Mat_{2\times 2}(\Cp)$ and let
\[
\xymatrix{
\tilde{X}\ar[d]^h
\\
X
}
\]
be the blowup at the zero matrix.  The strict transform of $(\tra(T^4))^{-1}(0)$ in $\tilde{X}$, intersected with the exceptional $\mathbb{P}^3$, is the projective surface cut out by the homogeneous equation $\tra(T^4)$.  Call this projective variety $V(\tra(T^4))$.\smallbreak
Let
\[
\xymatrix{
Y\ar[d]^{h_p}\\
\mathbb{P}^3
}
\]
be an embedded resolution of the singular projective variety $V(\tra(T^4))$.  Then we have a diagram
\[
\xymatrix{
\tilde{X}_{1}\ar[r]^{h_1}\ar[d]^{\pi_1}&\tilde{X}\ar[d]^{\pi} \ar[r]^h &X\\
Y\ar[r]^{h_p} &\mathbb{P}^3
}
\]
with the leftmost square a pullback (in fact this is a pullback of a vector bundle, since $\tilde{X}$ is the total space of the tautological bundle for $\mathbb{P}^3$, and $\pi$ is the projection).  It is not hard to see that $h':=h\circ h_1$ is an embedded resolution for $\tra(T^4)$.  It follows from the fact that $\tra(T^4)\circ h$ vanishes to order 4 on $\mathbb{P}^3$ that there is an equality of divisors 
\begin{equation}
\label{diveasy}
(\tra(T^4)\circ h')^*(0)=(h_p\circ \pi_1)^*(V(\tra(T^4)))+4Y
\end{equation}
where $Y$ is considered as a divisor on $\tilde{X}_1$, the zero section of the vector bundle $\tilde{X_1}\rightarrow Y$.  
\medbreak
So we just need to work out an embedded resolution of $V(\tra(T^4))$.  Note that $\PSL(2,\Cp)$ acts on $X$ by conjugation, $\tra(T^4)$ is invariant under this action, the action lifts to $\tilde{X}$, and $V(\tra(T^4))$ is also invariant under the action.  There are exactly three orbits of the $\PSL(2,\Cp)$-action in $V(\tra(T^4))$.  Define
\begin{enumerate}
\item $S_1$ to be the orbit consisting of matrices whose eigenvalues differ by a factor of $e^{i\pi/4}$,
\item $S_2$ to be the orbit consisting of matrices whose eigenvalues differ by a factor of $e^{3i\pi/4}$,
\item $S_3$ to be the orbit of nilpotent matrices.
\end{enumerate}
\begin{prop}
\label{nonsingloc}
In the ring $\Mot^{\hat{\mu}}(\Spec(k))$ there are equalities 
\begin{equation}
[S_1]=[S_2]=[\mathbb{P}^1\times \Cp],
\end{equation}
where all of these motives carry the trivial $\hat{\mu}$-action.
\end{prop}
\begin{proof}
Fix two nonzero numbers $a$ and $b$ differing by a factor of $e^{i\pi/4}$.  Then to pick a matrix with these two numbers as eigenvalues is the same as to pick two distinct vectors (up to rescaling) to be the respective eigenvalues.  So pick the eigenvector for $a$ first, this gives us a $\mathbb{P}^1$ of choice, then pick the eigenvector for $b$, giving a $\Cp$ of choices, one can in fact see that $S_1$ is a line bundle over $\mathbb{P}^1$.  The motive of any line bundle is the same as the motive of the trivial line bundle -- any ordered open cover underlying a trivialization induces a stratification on which each restriction of the line bundle is trivial.
\end{proof}
\begin{prop}
There is an isomorphism $S_3\cong \mathbb{P}^1$.
\end{prop}
\begin{proof}
Give $\mathbb{P}^3$ coordinates $(X:Y:Z:W)$ by writing matrices as
\[
\left(\begin{array}{cc} X&Z\\ W&Y\end{array}\right).
\]
Then the nilpotent matrices are precisely those satisfying $\mathbf{trace}=\mathbf{det}=0$.  So they are the ones satisfying 
\begin{align*}
&X=-Y,\\
&XY=WZ,
\end{align*}
giving a $\mathbb{P}^1$ inside $\mathbb{P}^3$.
\end{proof}
The singular locus of $V(\tra(T^4))$ is precisely $S_3$.  Since $S_3$ is a $\PSL(2,\Cp)$-orbit, the singularity is going to end up being the same all along this $\mathbb{P}^1$.  We restrict to an affine patch $U$ by setting $W\neq 0$.  On this patch we use the coordinates
\[
(x,y,z)\mapsto \left(\begin{array}{cc}x&z\\1&y\end{array}\right).
\]
There is an isomorphism $U\cap S_3\cong \Cp$, and $U\cap S_3$ can be parameterised as follows
\begin{align}
&\Cp\rightarrow U\cap S_3\\
&t\mapsto \left(\begin{array}{cc}t&-t^2\\1&-t\end{array}\right).
\end{align}
We can extend this to a coordinate system $(t,a,b)$ for $U$, given by
\begin{align}
&(t,a,b)\mapsto \left(\begin{array}{cc}t+a &b-t^2\\1&-t\end{array}\right).
\end{align}
In these coordinates the local defining equation for $\tra(T^4)$ becomes 
\[
\tra(T^4)=a^4+4a^3t+4a^2b+2a^2t^2+4abt+2b^2,
\]
or, after rearranging,
\[
\tra(T^4)=-a^4+2(at+b+a^2)^2.
\]
After replacing $b$ with $b'=b+at+a^2$ we get that the local defining equation for $\tra(T^4)$ is
\[
\tra(T^4)=-a^4+2b'^2,
\]
and so we have a $\mathbb{P}^1$ of $A_3$ singularities along $S_3$.  If we blow up $S_3$ we replace this with an exceptional divisor (the projectivization of the normal bundle of $S_3$), on which there is another $\mathbb{P}^1$ of singularities, this time of type $D_4$.  Blowing up this new $\mathbb{P}^1$ gives our embedded resolution
\[
\xymatrix{
Y\ar[d]^{h_p}\\
\mathbb{P}^3.
}
\]
\bigbreak
Let $J$ be the set of divisors in $(\tra(T^4)\circ h')^{-1}(0)$.  We wish to calculate the absolute equivariant motive
\begin{equation}
\label{pok}
[\psi_{\tra(T^4)}]_{\Mat_{\sut,2\times 2}(\Cp)}=\sum_{\emptyset\neq I\subset J}(1-\mathbb{L})^{|I|-1}[\tilde{D}_I]_{h'^{-1}(\Mat_{\sut,2 \times 2}(\Cp))}.
\end{equation}
Consider the decomposition 
\[
\Mat_{\sut, 2\times 2}(\Cp)=\{0\}\coprod H,
\]
where $\{0\}$ is the zero matrix, and $H\cong \Cp$ is the complement.  This decomposition induces a decomposition of the sum (\ref{pok}): if we define
\begin{align}
M_{\nt}=[\psi_{\tra(T^4)}]_H=&\sum_{\emptyset\neq I\subset J}(1-\mathbb{L})^{|I|-1}[\tilde{D_I}]_ {h'^{-1}(H)}
\\
\label{tricky}
M_t=[\psi_{\tra(T^4)}]_{\{0\}}=&\sum_{\emptyset\neq I\subset J}(1-\mathbb{L})^{|I|-1}[\tilde{D_I}]_{h'^{-1}(\{0\})},
\end{align}
then
\[
[\psi_{\tra(T^4)}]_{\Mat_{\sut,2\times 2}(\Cp)}=M_{\nt}+M_t.
\]
Since $H$ is just the complement to the zero section in the fibre $\pi^{-1}\left(\left(\begin{array}{cc}0&1\\0&0\end{array}\right)\right)$, and $V(\tra(T^4))$ has an $A_3$ singularity at this matrix, i.e. the singularity defined by the singular curve $x^4+y^2$, the following proposition follows from equation (\ref{diveasy}).
\begin{prop}
There are equalities of absolute motives 
\begin{align}
M_{\nt}&=(\mathbb{L}-1)\MF(x^4+y^2)\\
&=(\mathbb{L}-1)([C_2]-2\mathbb{L})
\end{align}
where $C_2$ is a torus with the representation $\alpha+\alpha^3$ on its middle cohomology.
\end{prop}
\begin{proof}
Only the second equality needs proving.  This is implied by Proposition \ref{x4y2} in the appendix to this chapter. 
\end{proof}
\begin{prop}
\label{trivpart}
There is an equality of absolute motives 
\[
M_{t}=(1-\mathbb{L})\MF(x^4+y^2)+\mathbb{L}\MF(x^4+y^4).
\]
\end{prop}
\begin{proof}
One of the terms in the sum (\ref{tricky}) comes from setting $I=\{Y\}$.  Now $fh'$ vanishes to order 4 on $Y$, and so $\tilde{D}_I$ is a 4-sheeted \'{e}tale cover over the complement of $V(\tra(T^4))$ in $\mathbb{P}^3$.  It follows from Proposition \ref{4sheetcover} in the Appendix \ref{motivesapp} that
\[
[\tilde{D}_{\{Y\}}]_{h'^{-1}(\{0\})}=[\tilde{D}_{\{Y\}}]_X=\mathbb{L}\MF(x^4+y^4)+(\mathbb{L}-1)\mathbb{L}\MF(x^4+y^2)+2\mathbb{L}(\mathbb{L}^2-1).
\]
The subvariety of $\Mat_{2\times 2}(\Cp)$ cut out by $\tra(T^4)$ has two components, the cones over the divisors $S_1\cup S_3$ and $S_2\cup S_3$, and we denote the strict transform of these divisors in the embedded resolution $\tilde{X}_1$ by $F_1$ and $F_2$, respectively.  These divisors occur with multiplicity 1.  Since we only blow up along $S_3$, there is an isomorphism $\tilde{D}_{\{F_i,Y\}}\cong S_i$ for $i=1,2$.  So these two subsets of $J$ each contribute
\[
(1-\mathbb{L})[D_{\{Y,F_i\}}]_{h'^{-1}(\{0\})}=(1-\mathbb{L})[\mathbb{P}^1\times \Cp]
\]
to $M_t$, by Proposition \ref{nonsingloc}.  All the other contributions to (\ref{tricky}) come from the modifications made to the singular locus of $V(\tra(T^4))$, i.e. from subsets $I\subset J$ that contain $Y$ and at least one divisor occurring as the cone over an exceptional divisor of 
\[
\xymatrix{
Y\ar[d]^{h_P}\\
\mathbb{P}^3.
}
\]
\medbreak
At the first blowup, along the $\mathbb{P}^1$ of $A_3$-singularities $S_3$, we introduce a $\mathbb{P}^1$-bundle, along with a $\mathbb{P}^1$ of new singularities.  Since we are working in the motivic ring, we can assume that the bundle in question is trivial.  The same is true for the second blowup.  The result is the equation
\[
\sum_{\emptyset\neq J\subset D | J\nsubseteq \{Y, F_1,F_2\}}(1-\mathbb{L})^{|J|-1}[\tilde{D_J}]_{h'^{-1}(\{0\})}=(1-\mathbb{L})[\mathbb{P}^1]\MF(x^2+y^4).
\]
Putting all this together gives the result.
\end{proof}
It turns out, then, that we have exactly what we want:
\begin{prop}
There is an equality of $\hat{\mu}$-equivariant motives
\begin{equation}
[-\phi_{\tra(T^4)}]_{\Mat_{\sut,2\times 2}(\Cp)}\mathbb{L}^{-1}=1-\MF(x^4+y^4)
\end{equation}
and so there is an equality of $\hat{\mu}$-equivariant motives 
\begin{equation}
\DT[\Rep_1(B)\star \Rep_1(B)]=\DT[\Rep_1(B)] \cdot \DT[\Rep_1(B)]
\end{equation}
where these `DT counts' are as defined in (\ref{DTtry}) and (\ref{DTstar}).
\end{prop}
\begin{rem}
We have shown this equality directly, but also it turns out to be a comparatively simple application of the (partly conjectural) Kontsevich--Soibelman integral identity (see Section 4.4 of \cite{KS}).  This motivic identity implies that this motivic refinement of the Donaldson--Thomas invariant preserves ring structure for more general moduli spaces of objects in the Abelian category of $B$-modules, and more general Jacobi algebras.
\end{rem}
\section{Towards motivic Donaldson--Thomas invariants}
The above calculations show that a `naive' motivic refinement of the Donaldson--Thomas count preserves ring structure, at least in our basic example.  It will turn out that the key ingredient for achieving this was the \textit{extra} data provided by a realization of our algebra as a superpotential algebra, which in turn enables us to realise the representation spaces we were interested in as critical loci in such a way that the integration map defined via the associated motivic vanishing cycles preserves products.  The key question is: can we do without this extra data?  
\begin{que}
\label{bigq}
\textbf{If we are handed a `Calabi-Yau 3-dimensional category', whatever that may turn out to be, can we construct a motivic integration map from the ring of stack functions, preserving the product?}
\end{que}
There is a notion of quasi-equivalence of Calabi-Yau categories, that in particular induces quasi-isomorphisms of homomorphism spaces and quasi-isomorphisms of endomorphism spaces as $A_{\infty}$-algebras.  Again, we needn't worry at the moment about what that means precisely, but already an implication for a satisfactory theory of motivic Donaldson--Thomas invariants follows from the fact that quasi-equivalences of Calabi-Yau categories induce isomorphisms of derived categories:
\begin{req}
\label{req2}
\textbf{
The motivic Donaldson--Thomas count associated to a stack function should be invariant under pullback along quasi-equivalences of Calabi-Yau 3-dimensional categories.}
\end{req}

Consider again our archetypal Donaldson--Thomas setup: associating numbers `counting' sheaves $\mathcal{F}$ in fine moduli spaces $\mathcal{M}$.  Recall that if $\mathcal{F}$ is a coherent sheaf on our Calabi-Yau 3-fold $X$, the constructible function $\nu_{\mathcal{M}}(\mathcal{F})$ depends solely on the scheme structure of the moduli space $\mathcal{M}$, where we use the common abuse of notation whereby $\mathcal{F}$ also denotes the point of $\mathcal{M}$ representing it.  The fact that the scheme structure of $\mathcal{M}$ tells us what kind of contribution $\mathcal{F}$ should make to the Donaldson--Thomas invariant is explained by the fact that $\mathcal{M}$ is a fine moduli space, and so carries information about infinitesemal deformations of $\mathcal{F}$.
\bigbreak
The idea is that the contribution of an object $\mathcal{F}$ need not be calculated from the local structure around $\mathcal{F}$ in some moduli scheme $\mathcal{M}$.  In the example above we used a particular way of realising our moduli spaces as critical loci in order to give a motivic refinement of the Donaldson--Thomas count, but of course this application of extra data means that we have not provided an affirmative answer to Question \ref{bigq}.  The crucial observation is that some version of a critical locus description around $\mathcal{F}$ can be read off straight from the formal deformation theory of that object.  This is important since we want our invariant to be motivic, that is, to behave well under cutting and pasting of families of objects in our category.  Say we have a constructible decomposition
\begin{equation}
\label{decolo}
X=\coprod X_i
\end{equation}
of schemes parameterising objects of some Calabi-Yau 3-dimensional category $\mathcal{C}$ (i.e. \textit{stack functions}).  Since the motivic contribution that $\mathcal{F}$ makes is just a function of $\mathcal{F}$, and wholly independent of how we view $\mathcal{F}$ as a point in any family of objects of our category, it follows trivially that when we integrate these motivic contributions over the families in the decomposition on the right hand side of (\ref{decolo}) we will get the same contribution as we get when we integrate over the left hand family.  This happy fact can be observed already in the setup sketched in the example of Section \ref{basicexample}: it corresponds to the identity
\[
[\phi_f]_X=\sum [\phi_f]_{X_i}
\]
for $f$ any function $f:X\rightarrow \Cp$, and $\coprod X_i$ any constructible decomposition of $X$.
\smallbreak
The contribution of an object $\mathcal{F}$, sitting inside a fine moduli space $\mathcal{M}$, to the ordinary Donaldson--Thomas count is a function of the Euler characteristic of the Milnor fibre of a function 
\[f:\Cp^t\rightarrow \Cp,\]
for some $t$, satisfying the condition that $\crit(f)$ looks (locally) like a formal neighborhood of the point $x$ representing $\mathcal{F}$ in $\mathcal{M}$.  To refine this number to a motive we would like to find some way of building such an $f$ directly from the category, and as a preliminary step we should find somewhere for $f$ to \textit{live}, e.g. as a function on a vector bundle on a stack of objects.  It turns out that a reasonable candidate for $f$, at $x$, is a function defined on $\Ext^1(\mathcal{F},\mathcal{F})$.  Now our aim was to write down motivic Donaldson--Thomas invariants for arbitrary families, at which point we are confronted by the fact that the dimension of $\Ext^1(\mathcal{F},\mathcal{F})$ is liable to jump as we vary $\mathcal{F}$, so we cannot hope that our $f$ will be a function on a vector bundle.  The appropriate sheaf (which we will call $\mathcal{EXT}^1$ here) will, rather, be a \textit{constructible} vector bundle.

\section{Some remarks on constructible vector bundles}
\label{cvect}
Let $X$ be a locally Noetherian scheme.  By a constructible decomposition of $X$ we will hereafter mean a decomposition of $X$ into locally closed subschemes such that there is a cover of $X$ by open affine schemes $U_i$ for which the restriction of the decomposition of $X$ to each $U_i$ is a finite constructible decomposition.  A constructible vector bundle $V$ on $X$ is given by a constructible decomposition of $X$, and a vector bundle on each component of the decomposition.  There is, in principle, no reason why one must impose any kind of finite-dimensionality of $V$ in the definition, but we will see shortly that doing so makes the category of such constructible vector bundles much better behaved.  We identify a constructible vector bundle with the one obtained, by restrictions, on a subordinate constructible decomposition.  A morphism between two constructible vector bundles $V_1$ and $V_2$ is given by taking a constructible decomposition subordinate to the two decompositions defining $V_1$ and $V_2$, and giving a morphism, for each $X_i$ in the decomposition, from $V_1|_{X_i}$ to $V_2|_{X_i}$.  We identify a morphism $f$ with the morphism obtained by restricting $f$ to a constructible decomposition subordinate to the one defined by $f$.\smallbreak
The following basic observation underlies much of what follows.
\begin{prop}
\label{SSab}
Let $X$ be a locally Noetherian scheme, and let $\mathcal{V}_{\fin}$ be the full subcategory of the category of constructible vector bundles on $X$ consisting of locally finite-dimensional vector bundles, by which we mean constructible vector bundles such that every fibre is finite-dimensional.  Then $\mathcal{V}_{\fin}$ is a semisimple Abelian category.
\end{prop}
\begin{proof}
The homomorphism sets of $\mathcal{V}_{\fin}$ are clearly Abelian groups.  Given a finite set of constructible vector bundles $V_i$ on $X$ one forms the direct sum and product by passing to a constructible decomposition of $X$ that is subordinate to the constructible decomposition defined by each of the $V_i$, and then taking the direct sum of vector bundles.  So $\mathcal{V}_{\fin}$ is an additive category.\smallbreak
By our finiteness assumption, the entire category is self-dual under the operation of taking duals of vector spaces.  This follows from the fact that there is a natural isomorphism of $A$-modules, for $A$ a commutative ring, $L$ a free $A$-module, and $I$ an ideal of $A$:
\begin{align*}
\Hom_{A\lMod}(L,A)\otimes_A A/I\rightarrow& \Hom_{A/I \lMod}(L\otimes_A A/I,A/I)\\
(\phi,a)\mapsto & (\phi\otimes \id_{A/I})\cdot a.
\end{align*}
This implies that taking vector duals is compatible with taking subordinate constructible decompositions.  So if we can find cokernels, kernels will follow for free.  Let
\[
\xymatrix{
\phi:L_1\ar[r]& L_2
}
\]
be a morphism in $\mathcal{V}_{\fin}$.  Let $U\subset X$ be an open affine subscheme of $X$, with $U=\Spec(R)$ for $R$ a Noetherian ring.  Then the constructible decomposition defining the morphism $\phi$ induces a finite constructible decomposition of $U$.  Let $V$ be one of the schemes in this decomposition.  Considering $\phi|_{V}$ instead as a morphism in the category of coherent sheaves on $V$, the cokernel $\Coker(\phi|_{V})$ exists.  By generic freeness (see \cite{CAwav}) and Noetherian induction there is a finite stratification of $V$ into subschemes $V_i$ such that each $\Coker(\phi|_{V})|_{V_i}$ is free, and by the right exactness of tensor products, the restriction of $\Coker(\phi|_{V})$ to each stratum $V_i$ remains a cokernel (in the category of coherent sheaves on $V_i$) to the morphism $\phi$ restricted to $V_i$.  After further stratification of the $V_i$ into $V'_i$ we may assume that the cokernel is a free module on an affine scheme, and so the exact sequence
\begin{equation}
\label{splitter}
\xymatrix{
L_1|_{V'_i}\ar[r]^{\phi|_{V'_i}}&L_2|_{V'_i}\ar[r]&\Coker(\phi|_{V'_i})\ar[r]&0
}
\end{equation}
splits.  In addition, $\Coker(\phi|_{V})|_{V'_i}$ is the cokernel for $\phi|_{V'_i}$ in the category of constructible vector bundles on $V'_i$, again by the right exactness of tensor products.  It follows that we have constructed a cokernel, in the category of constructible vector bundles on $U$, for $\phi|_U$.\smallbreak
Let $\{U_i,i\in I\}$ be an ordered cover of $X$ such that each $U_i$ is equal to $\Spec(R_i)$ for $R_i$ a Noetherian ring.  Then each $U_i$ carries a finite decomposition 
\[
U_i=\coprod_{j\leq i}(U_j\cap U_i-(\bigcup_{k<j}U_k)).
\]
Since each $U_i$ is the spectrum of a Noetherian ring, all but finitely many terms in the decomposition are empty.  Since we have shown that each $U_j$ admits a finite constructible decomposition into affine schemes $V_t$ such that on each $V_t$ the coherent sheaf $\Coker(\phi|_{V_t})$ is free, and is defined by a universal property in the category of constructible vector bundles on $V_t$, it follows that $\Coker(\phi)$ is a well-defined constructible vector bundle, which is clearly locally finite.\smallbreak
One forms kernels as dual to cokernels.  Let $V_0$ be, as above, one of the affine subschemes of $X$ in the constructible decomposition of $X$ we have defined for $\Coker(\phi)$.  Note that since
\[
\xymatrix{
L_1|_{V_0}\ar[r]& L_2|_{V_0}\ar[r] &\Coker(\phi)|_{V_0}\ar[r]&0
}
\]
is an exact sequence of free modules, its dual is too, as is the restriction of its dual to subschemes of $V_0$.  It follows that after restricting to a subscheme $X'$ in the constructible decomposition defined by the kernel of $\phi$, the kernel we have constructed is just the kernel in the category of coherent sheaves on $X'$.  Since the category of coherent sheaves on a Noetherian scheme is Abelian, the natural morphism from the coimage to the image of $\phi$ is an isomorphism of coherent sheaves, once we restrict to each subscheme in a constructible decomposition of $X$ subordinate to the constructible decompositions defined by the kernel, cokernel, image and coimage of $\phi$.  It follows that $\coimage(\phi)\rightarrow \image(\phi)$ is an isomorphism of constructible vector bundles, and so $\mathcal{V}_{\fin}$ is an Abelian category.  Since we can pick a constructible decomposition such that (\ref{splitter}) splits, it follows also that $\mathcal{V}_{\fin}$ is semisimple.
\end{proof}

We define the (ordinary) category of constructible differential graded vector bundles on $X$ as the category with objects given by pairs of a constructible decomposition of $X$, and on each subscheme of the decomposition a differential graded vector bundle.  Morphisms are given by morphisms of such objects that preserve degree and commute with the differential, and we make the obvious identifications of objects and morphisms under subordinate decompositions.
\begin{cor}(Formality)
\label{formcor}
Let $V^{\bullet}$ be a constructible differential graded vector bundle on a locally Noetherian scheme $X$ such that each fibre of $V^{\bullet}$ is finite-dimensional in each degree, and on each of the subschemes $X_i$ of $X$ defined by the constructible decomposition associated to $V^{\bullet}$, the homology $\Ho^i(V^{\bullet})$, considered as a constructible vector bundle on $X_i$, is nonzero for only finitely many $i$.  Then there is a quasi-isomorphism from a constructible differential graded vector bundle with zero differential to $V^{\bullet}$.
\end{cor}
\begin{proof}
We can define the $i$th homology of $V^{\bullet}$, in the category of constructible vector bundles, since the category $\mathcal{V}_{\fin}$ of Proposition \ref{SSab} is Abelian.  Then the formality follows from the fact that $\mathcal{V}_{\fin}$ is semisimple, and our local finiteness assumption on the homology.
\end{proof}
\begin{rem}
Kernels in the category $\mathcal{V}_{\fin}$ above are maybe a little surprising.  For instance, the homomorphism of $\Cp[x]$-modules
\[
\xymatrix{
\Cp[x]\ar[r]^{\cdot x}&\Cp[x]
}
\]
is of course an injection of coherent sheaves on the scheme $\Cp$.  Considered as a morphism of constructible vector bundles, however, one readily verifies that the kernel consists of a rank 1 vector bundle over the origin.  The same example shows that the homology of a differential graded vector bundle, considered as a constructible differential graded vector bundle, can be very different from the homology of the vector bundle considered as a complex of coherent sheaves.
\end{rem}
\section{Formal deformation theory}
Since we are working in the ring of motives, we may treat the constructible vector bundle $\mathcal{EXT}^1$ (once it is properly defined) as though it were a vector bundle.  In the original setup, in which we were working out Donaldson--Thomas counts associated to fine moduli spaces, this constructible vector bundle played an important role: it is naturally identified with the Zariski tangent space of our scheme $\mathcal{M}$ (see \cite{HL} for example).\smallbreak
Given an object $\mathcal{F}$ in a Calabi-Yau 3-dimensional category $\mathcal{C}$, we obtain an $A_{\infty}$-algebra $A=\Hom^{\bullet}(\mathcal{F},\mathcal{F})$.  Such an algebra is like a differential graded algebra, in that it has two operations $m_1:A\rightarrow A[1]$ and $m_2:A\otimes A\rightarrow A$, but it also has countably many higher operations 
\[
m_n:A^{\otimes n}\rightarrow A[2-n]
\]
which are required to satisfy some compatibility conditions (see (\ref{struc-compatibility})).  Given such a set of $m_n$ we get a set of $b_n$ making the following diagram commute
\[
\xymatrix{
A^{\otimes n}\ar[d]^{S^{\otimes n}}\ar[r]^(.45){m_n}& A[2-n]\ar[d]^{S^n}\\
A[1]^{\otimes n}\ar^{b_n}[r]& A[2],
}
\]
where $S$ is the degree -1 map sending $a\in A$ to $a$ in $A[1]$.  Clearly these $b_n$ contain the same information as the $m_n$, so we may just as well describe an $A_{\infty}$-algebra using them.  Here begins the constant tension in this subject between the $m_n$, which naturally extend our notions of ordinary algebras and differential graded algebras, but have increasingly awkward sign rules, and the $b_n$, which do not.\medbreak
We can describe the formal deformation theory of $\mathcal{F}$, using the functor
\begin{align*}
Def_{\mathcal{F}}:\hbox{Artinian nonunital algebras}\rightarrow &Sets\\
\textbf{m}\mapsto&\{\gamma\in \textbf{m}\otimes \Hom^1(\mathcal{F},\mathcal{F})| MC(\gamma)=0\}
\end{align*}
where $MC:\Hom^1(\mathcal{F},\mathcal{F})\rightarrow \Hom^2(\mathcal{F},\mathcal{F})$ is given by the formal sum of the degree $n$ functions
\[
MC_n(a)=b_n(\gamma,\ldots,\gamma),
\]
and $b_n$ are the higher multiplications of $\Hom^1(\mathcal{F},\mathcal{F})$.  We have shifted from the usual maps $m_n:A^{\otimes n}\rightarrow A[2-n]$ to maps $b_n:A[1]^{\otimes n}\rightarrow A[2]$ just to make the signs trivial here.
\medbreak
The fact that $\mathcal{C}$ is supposed to be a Calabi-Yau 3-dimensional category over some ground field $k$ enables us to make some extra assumptions on our $A_{\infty}$-Yoneda algebra $\End^{\bullet}(\mathcal{F})$, namely we assume that it has a cyclic structure.  We will come to what exactly this means in Chapter \ref{quiveralgebras}, but for the time being it is sufficient to note that this extra structure implies that we have a nondegenerate antisymmetric pairing
\[
\langle\bullet,\bullet\rangle:\Hom^1\otimes \Hom^2\rightarrow k
\]
and that if we define
\[
W_n(x)=\frac{1}{n}\langle b_{n-1}(x,\ldots,x),x\rangle
\]
and let $W$ be the formal sum of these degree $n$ functions, we have that $dW=MC$.  This makes sense once one views $\End^2(\mathcal{F})$ as the vector dual of $\End^1(\mathcal{F})$ via the pairing $\langle\bullet,\bullet\rangle$ and identifies each fibre of the cotangent space of $\End^1(\mathcal{F})$ with the vector dual of $\End^1(\mathcal{F})$ in the natural way. 
\smallbreak
It follows, then, that we are given a formal critical locus description for $\mathcal{F}$ without any reference to a moduli space, directly from the structure of a 3-dimensional Calabi-Yau category.
\medbreak
The $W$ we have here, though, is in some sense not yet intrinsic to the category -- it changes as we vary the representative we take of the quasi-equivalence class of the category $\mathcal{C}$, varying by quasi-isomorphisms the representative we take of the $A_{\infty}$-algebra $\End^{\bullet}(\mathcal{F})$.  Help is at hand though: it turns out (see Theorem 5 and Corollary 2 of \cite{KS}, as well as \cite{KSnotes}, \cite{kaj03} and Theorem \ref{cycmm} below) that we can always find a (noncanonical!) minimal model for our category, at least around a neighbourhood of our object $\mathcal{F}$, after constructible decomposition of the space of objects in the category.  So there is a quasi-isomorphism (at least after we replace $\mathcal{C}$ by the full subcategory whose objects are a constructible neighborhood of $\mathcal{F}$):
\begin{equation}
\label{roughmin}
\xymatrix{
\mathcal{C}\ar[r]^{\sim}& \mathcal{C}'
}
\end{equation}
to a Calabi-Yau $A_{\infty}$-category $\mathcal{C}'$ where the morphism spaces have zero differential, and so we have the identification $\End_{\mathcal{C}'}^1(\mathcal{F})\cong \Ext_{\mathcal{C}}^1(\mathcal{F})$.  This is good, since the graded vector space of $\Ext$s between two objects, as opposed to the differential graded vector space of $\Hom$s, is a true invariant under quasi-isomorphisms of $A_{\infty}$-categories.  What's more, since we have taken this minimal model in the category of \textit{cyclic} $A_{\infty}$-categories, this new $\End_{\mathcal{C}'}^1(\mathcal{F})$ comes also with its potential function, denoted $W_{\min}$.  Finally, the really good news is that this $W_{\min}$ doesn't depend on the choice of minimal model (up to some changes that have no effect on motivic Milnor fibres).  So $W_{\min}$, considered as a formal function on the constructible vector bundle $\mathcal{EXT}^1$, presents itself as a likely candidate for our intrinsic critical locus description of the category.
\section{An example in the general framework}
Before providing some of the gory details of how the above theory works, let us see how some of it works in a specific example.  First we fix some data.  We will start by defining $A$, an $A_{\infty}$-algebra.  Such an algebra has an underlying graded vector space, which in our case is just going to be
\[
A=\Cp\oplus \Cp[-1] \oplus \Cp[-2]\oplus \Cp[-3].
\]
Such an algebra comes also with a countable collection of operations
\[
m_n:A^{\otimes n}\rightarrow A[2-n],
\]
for $n\geq 1$, satisfying some compatibilities.  For example, in the case where $m_n=0$ for all $n\geq 3$ the algebra can be thought of (and indeed really is) just a differential graded algebra, with $m_2$ equal to the multiplication and $m_1$ giving the differential; in this case, the compatibility conditions say exactly that our algebra satisfies the conditions required of a differential graded algebra.  The $A$ we are going to consider is slightly different.  We first set $m_1=0$, i.e. the differential is zero -- this puts us in the `minimal' situation of (\ref{roughmin}).  Next, we set the thing to be unital.  So there is some $1\in A^0=\Cp$ which functions just like the identity under $m_2$, and such that $m_i(\ldots,1,\ldots)=0$ for all $i\geq 3$.  Let us extend this unit to a basis 
\[
\{1\in A^0, a\in A^1, a^*\in A^2, w\in A^3\}
\]
so that we have a graded basis for the whole of $A$.  Next, set
\begin{align*}
m_2(a,a^*)=m_2(a^*,a)=w\\
m_2(a,a)=0.
\end{align*}
For degree reasons, this and the unital property determine $m_2$ entirely.  We define $m_i=0$ unless $i\in \{2,3\}$.  We let $m_3(a,a,a,)=a^*$, and set $m_3$ to be zero on all other 3-tuples of basis elements.\medbreak
In fact this algebra hasn't been plucked from nowhere: it is the $A_{\infty}$ Koszul dual (as in \cite{Yoneda}) of the Ginzburg differential graded algebra $\Gamma(Q,W)$ (see Section \ref{NCCY}) associated to the quiver with potential we considered in Section \ref{basicexample}.  This is a differential graded algebra with cohomology concentrated in negative degrees, with zeroeth cohomology isomorphic to our algebra $B$ as defined in Section \ref{basicexample}.  So the Abelian category of $B$-modules sits inside the derived category of $\Gamma(Q,W)$-modules as the heart of a t-structure, and $A$ is the Yoneda algebra $\Ext^{\bullet}_{\Gamma(Q,W)\lmodi}(M,M)$ of the 1-dimensional simple module $M$ of Section \ref{basicexample}.  Note that this algebra is very different from the Yoneda algebra $\Ext^{\bullet}_{B\lmodi}(M,M)$, which is concentrated in infinitely many degrees.
\smallbreak
Under Koszul duality, the $B$-module $M$ gets sent to the free (right) $A$-module.  But it is maybe worth forgetting that for now, and just taking some category of modules over $A$ to be our Calabi-Yau category, and seeing what the programme sketched above, involving $W_{\min}$, does in this case.\smallbreak
As in Section \ref{basicexample} we will be interested in some very simple spaces of modules over $A$ (indeed the same spaces, under Koszul duality).  First we need to write down our version of the superpotential coming from the structure of our category.  To this end we introduce the symmetric pairing
\[
\langle\bullet, \bullet \rangle:A\otimes A\rightarrow \Cp[-3]
\]
given by letting $\langle a, a^*\rangle=\langle 1 , w\rangle=1$.  This gives us our $W$: if we let $x$ be a coordinate on $\Ext^1(M,M)\cong A^1$, then 
\[
W=x^4.
\]
(Recall that $W$ is actually defined in terms of the $b_n$, maps from $A[1]^{\otimes n}$ to $A[2]$, but up to sign this makes no difference to our $W$.)  The only modules we will be interested in are $A$ and extensions of $A$ by itself. Denote by $N$ the free left $A$-module. We denote by $N_{\alpha}$ the cone of a morphism $\alpha:N[-1]\rightarrow N$.  Such a module is really just the extension determined by $\alpha\in \Ext^1(N,N)$, but souped up to an object in an $A_{\infty}$-category.  Such an extension has, as underlying $A$-module, $N_1\oplus N_2$, where we have labelled the two copies of $N$ merely for convenience.  $N_{\alpha}$ has a differential determined by $\alpha$:
\begin{equation}
\label{dif}
d \left(\begin{array}{c}a_1\\a_2\end{array}\right)=m_2\left(\left(\begin{array}{cc}0&\alpha\\ 0 &0\end{array}\right),\left(\begin{array}{c}a_1\\a_2\end{array}\right)\right)=\left(\begin{array}{c}m_2({\alpha},a_2)\\0\end{array}\right).
\end{equation}
By a slight abuse of notation we denote the matrix appearing in (\ref{dif}) simply by $\alpha$.  By a slightly larger abuse of notation we have used the same $m_i$ as appear in the definition of $A$ to denote the natural extension to matrix calculus.  What we are really interested in is $\End^{\bullet}(N_{\alpha})$.
\begin{prop}
The $A_{\infty}$-algebra $\End^{\bullet}(N_{\alpha})$ has a model whose underlying graded vector space is
\begin{align}
\label{underlying}
H:=&\End^{\bullet}(N_1,N_1)\oplus \End^{\bullet}(N_1,N_2)\oplus \End^{\bullet}(N_2,N_1)\oplus \End^{\bullet}(N_2,N_2)\nonumber\\
=&A_{11}\oplus A_{12}\oplus A_{21}\oplus A_{22}\nonumber\\
=&M_{2\times 2}(A)
\end{align}
where the subscripts do not change the mathematical object denoted by the terms they are subscripts to, and are just added for notational convenience.  This algebra carries natural higher products coming from $A$, which we denote by $m_{2\times 2,n}$, or the shifted version by $b_{2\times 2,n}$, and twist by setting
\begin{equation}
\label{conemult}
b_{\alpha,i}(A_1,\ldots,A_i)=\sum_{n\geq i} b_{2\times 2,n}(\alpha,\ldots,\alpha,A_1,\alpha\ldots\alpha,A_2,\alpha,\ldots,\alpha,A_i,\alpha,\ldots,\alpha).
\end{equation}
\end{prop}
See Section \ref{twistcat} for an explanation of where this model is coming from, and an actual definition of $b_{2\times 2,n}$.  Note that the sum in (\ref{conemult}) is actually finite: any term in which $\alpha$ appears in consecutive places is automatically zero, from the definition of $b_{2\times 2,n}$.  So, for example
\begin{equation}
b_{\alpha,1}(A)=b_{2\times 2,2}(A,\alpha)+b_{2\times 2,2}(\alpha,A)+b_{2\times 2,3}(\alpha,A,\alpha).
\end{equation}
\smallbreak
Let $\nu_N$ be the scheme consisting of a single closed point, which we make into a parameter space of $A$-modules by decreeing that the module over the point is just $N$.  It is somewhat unfortunate that the same letter $\nu$ is used to define Behrend's microlocal function; from now on it will always mean a parameter space over a point, as here.  In the language of stack functions, this is just the map $\Spec(k)\rightarrow \ob(\mathcal{C})$ sending the point to $N$.  The stack function/parameter space $\nu_N\star \nu_N$ is, as in Section \ref{basicexample}, just $\Ext^{1}(N_2,N_1)/\mathbb{A}^1$, where the point $\alpha\in \Ext^1(N_2,N_1)$ parameterises the module $N_{\alpha}$.
\begin{defn}
\label{homdef}
We define a graded vector bundle $\mathcal{END}$ over the vector space $\Ext^1(N_2,N_1)$, given by the trivial bundle with fibre $H$ as defined in (\ref{underlying}).  This differential graded vector bundle has operations
\[
m_{\mathcal{END},i}:\mathcal{END}^{\otimes i}\rightarrow \mathcal{END}
\]
as defined fibrewise in (\ref{conemult}).
\end{defn}
\medbreak
While $\mathcal{END}$ is a useful object, it isn't quite right for our purposes, since it isn't minimal.  In particular, if we build the function $W$ using it, as it is, it has quadratic terms, since $m_{\mathcal{END}^1,1}\neq 0$.  Consider the decomposition
\[
\Ext^1(N_2,N_1)=E_{t}\coprod E_{\nt}
\]
where $E_t=0$ and $E_{\nt}\cong \Cp^*$ is the complement of $E_t$.  Consider first the part $E_t$.  Here $\alpha=0$, and so $\mathcal{END}^{\bullet}|_{E_{t}}$ \textit{is} minimal, and there is nothing for us to do.\smallbreak
Now take the part $E_{\nt}$.  The vector bundle $\mathcal{END}^0|_{E_{\nt}}$ is spanned by sections
\[
1_{ij}\in \Ext^0(N_i,N_j)\cong A_{ij},
\]
where as before the subscripts are being used to distinguish the two copies of $N$, not to pick out degrees, and our differential acts on these as follows:
\begin{align*}
\label{d0}
d(1_{11})=&a_{21}\alpha,\\
d(1_{12})=&-a_{11}\alpha+a_{22}\alpha,\\
d(1_{21})=&0,\\
d(1_{22})=&-a_{21}\alpha,
\end{align*}
where $\alpha$ denotes a coordinate on $\Ext^1(N_2,N_1)$ and the vector bundle $\mathcal{END}^1|_{E_{\nt}}$ is spanned by sections
\[
a_{ij}\in \Ext^1(N_i,N_j),
\]
which in turn are acted on as follows
\begin{align}
\label{d1}
d(a_{11})=&0,\\
d(a_{12})=&\alpha^2 a_{21}^*,\\
d(a_{21})=&0,\\
d(a_{22})=&0.
\end{align}
So the section $a_{11}$ gives us an embedding of $\mathcal{EXT}^1|_{E_{\nt}}$ into $\mathcal{END}^1|_{E_{\nt}}$.  In fact we can almost realise $\mathcal{EXT}^{\bullet}|_{E_{\nt}}$ as a sub $A_{\infty}$-vector bundle of $\mathcal{END}|_{E_{\nt}}$, by writing
\begin{align}
\label{almost2}
\mathcal{EXT}^{\bullet}|_{E_{\nt}}=\{1_{11}+1_{22},1_{21},a_{11}+a_{22}, a^*_{11}+a^*_{22},w_{11}+w_{22},w_{12}\}.
\end{align}
The identity \ref{almost2} isn't quite right though, since this sub-bundle isn't closed under the operations $m_{\mathcal{END}^{\bullet},i}$.  The fix involves tweaking the inclusion $i:\mathcal{EXT}^{\bullet}|_{E_{\nt}}\rightarrow \mathcal{END}^{\bullet}|_{E_{\nt}}$ -- we are working with $A_{\infty}$-morphisms, with `higher' parts that can be modified to counteract the failure of our sub-bundle to be closed under the $A_{\infty}$-operations $m_{\mathcal{END},i}$ -- this is the process of taking a minimal model.  None of this technicality matters to us at the moment, since the thing we really care about, $m_{\mathcal{EXT}^{\bullet},i}$, is unchanged by these modifications, and so we can read off our function $W_{E_{\nt},\min}$ -- it is just the function $x^4$ (after rescaling) on the 1-dimensional vector bundle $\mathcal{EXT}^1|_{E_{\nt}}$.
\medbreak
We are working with the idea that our motivic refinement, which we will denote $``DT"$ for now, looks something like
\begin{align*}
``\DT":\hbox{stack functions for } A\lmod\rightarrow &\Mot^{\hat{\mu}}(\Spec(k))\\
S\mapsto  &\int_{S}(1-\MF(W_{\min})).
\end{align*}
There will in general be some twists by powers of $\mathbb{L}^{1/2}$, a formal square root of the motive of the affine line, but we have conveniently picked our example so that these powers are all trivial, in the end.  Let us work out what this map does in our example.  It turns out we have already done most of the work.  Firstly one can easily check that
\[
``\DT"([E_t]\mathbb{L}^{-1})=(1-\MF(\tra(T^4)))\mathbb{L}^{-1}
\]
and so 
\[
``\DT"([E_t]\cdot \mathbb{L}^{-1})=\LL^{-1}-(1-\mathbb{L}) \mathbb{L}^{-1} \MF(x^4+y^2)-\MF(x^4+y^4)
\]
by Proposition \ref{trivpart}.  Secondly, we have that 
\[
``\DT"([E_{\nt}]\mathbb{L}^{-1})=(\LL-1) \LL^{-1}-(\mathbb{L}-1) \MF(x^4)\mathbb{L}^{-1}.
\]
In order for the map $``\DT"$ to preserve the ring structure, then, we need
\begin{equation}
(\MF(x^4)-\MF(x^4+y^2))=0.
\end{equation}
While equalities in the ring of motives can perhaps be a little elusive, there are certain realizations from the ring of motives to more manageable rings that make inequalities easier to identify.  For example, from the functoriality of the weight filtration of the mixed Hodge structure of a scheme $X$, it follows that if a finite group $G$ acts on $X$ we may form an equivariant version $\chi_{eq,q}$ of the Serre polynomial for $X$.  Using Propositions \ref{x4y4} and \ref{x4y2} one can show
\[
\chi_{eq,q}(\MF(x^4)-\MF(x^4+y^2))=(\alpha+\alpha^2+\alpha^3-2(\alpha+\alpha^2+\alpha^3)\sqrt{q}-q)
\]
from which we deduce that our map $``\DT"$ does \textit{not} preserve the ring structure, as it stands.
\begin{rem}
The expert reader will have spotted that we are being somewhat disingenuous here.  We are working under the assumption that the powers of $\mathbb{L}^{1/2}$ alluded to above can be ignored.  In fact this is because there is a natural choice of orientation data in our situation for which the power of $\mathbb{L}^{1/2}$ appearing in the correct motivic weight (see (\ref{intmapeq})) is zero.  But we are meant to be presenting a malfunctioning attempt at the integration map, for which we have ignored the orientation data.  So we should throw in some power of $\mathbb{L}^{1/2}$, so that our example of the problems one runs into when ignoring orientation data is true to its word.  This correction to the motivic weight turns out to be multiplication by $\mathbb{L}^{1/2}$, over $E_{\nt}$, and multiplication by $1$ over $E_{t}$.  What's more, since $\sqrt{-1}\in \Cp$, there is a square root for $\mathbb{L}$ given by $(1-\MF(x^2))$.  By motivic Thom-Sebastiani, if we multiply the weight by this correction term, the motivic weight over $E_{\nt}$ transforms from $(1-\MF(x^4))$ to $(1-\MF(x^4+y^2))$, and our integration map appears to preserve the product.  However, even laying aside the fact that we have used the condition that $\sqrt{-1}$ is in our field, which won't be true in general, note that even in this most trivial of situations there is an alternative square root to $\mathbb{L}$ provided by $(\MF(x^2)-1)$.  So this putative fix is rather ad hoc.  The role of orientation data in this example is to pick the correct square root, i.e. to replace the formal square root of $\mathbb{L}$ with the correct one for our situation.  In this case that means multiplying by $[Q]\cdot\mathbb{L}^{-1/2}$, where $[Q]=(1-\MF(x^2))$ is the correct square root.
\end{rem}
\medbreak
Let us compare the case where things looked better, Section \ref{basicexample}, with what has happened here.  The following basic observation makes this easier.
\begin{prop}(see Proposition \ref{crux})
Let $\alpha\in \Ext^1(M_2,M_1)$.  Define
\[
W_{\alpha}(a)=\sum_{n\geq 2}\frac{1}{n}W_{\alpha, n}(a),
\]
a function on $2\times 2$ matrices with entries in $Ext^1(M,M)$, by
\[
W_{\alpha,n}(a)=\langle b_{\alpha,n-1}(a,\ldots,a),a\rangle.
\]
Write $W:=W_0$.  Then $W_{\alpha}(a)=W(\alpha+a)$.
\end{prop}
There is a smooth function $+:\Mat_{\sut,2\times 2}(\Cp)\times \Mat_{2\times 2}(\Cp)\rightarrow \Mat_{2\times 2}(\Cp)$ given by matrix addition, and the proposition states that $+^*(W)=W_{-}$, the function on $\Mat_{\sut,2\times 2}(\Cp)\times M_{2\times 2}(\Cp)$ that restricts to $W_{\alpha}$ over $\alpha\in \Mat_{\sut,2\times 2}$.  It follows by the properties of the transformation of the motivic vanishing cycle under pullback that $[-\phi_{W_{-}}]|_{\Mat_{\sut,2\times 2}(\Cp)}=[-\phi_{W}]|_{\Mat_{\sut,2\times 2}(\Cp)}$.  So as well as integrating motivic weights across the same 1-dimensional subspace of $\Mat_{2\times 2}(\Cp)$ both times, we have actually been integrating against the same motivic weight $[-\phi_{\tra(T^4)}]$ both times as well, \textit{almost}.  The almost here comes from the fact that along $E_{\nt}$ we have modified the function $W_{\alpha}$, breaking it into a quadratic part and a part with cubic and higher terms -- this is what we do when we restrict to the minimal superpotential $W_{\min}$.  What is this quadratic part?  As noted in \cite{KS}, to a first approximation it is just $W_{\alpha,2}$ on the constructible vector space 
\begin{equation}
\label{Vdef}
V=\mathcal{HOM}^{1}/\Ker(b_{\mathcal{HOM}^{\bullet},1}).  
\end{equation}
On $E_t$ this is trivial, so we concentrate on $E_{\nt}$.  Here, $V$ is spanned by $a_{12}$ (see (\ref{d1})), and the quadratic function induced by $W_{-}$ equals $\alpha^2 y^2$, where $y$ is the coordinate on the vector space \mbox{$\langle a_{12}\rangle\subset H$}, as defined in (\ref{underlying}).  After rescaling, this is just the function $y^2$.  If we had modified $``\DT"$ so that instead of integrating $(1-\MF(x^4))$ along $E_{\nt}$ we integrated $(1-\MF(x^4))\cdot (1-\MF(y^2))=(1-\MF(x^4+y^2))$ we would have arrived at the right answer.
\section{The role of orientation data in fixing preservation of ring structure}
So let us recall the situation we have arrived at.  Firstly, our goal was to associate motivic Donaldson--Thomas counts to arbitrary stack functions of a Calabi-Yau 3-dimensional category $\mathcal{C}$.  In the slightly different example of the Abelian category of modules over a superpotential algebra (in our case, $\Cp[x]/\langle x^3\rangle$), we have a good idea of how to do this, that seems to work, with the product preserved on account of an application of the Kontsevich--Soibelman integral identity, followed by the motivic Thom-Sebastiani Theorem.  If we just start from the data of a 3-Calabi-Yau category $\mathcal{C}$, we have some proxy for the critical locus description, the minimal superpotential $W_{\min}$ considered as a function on the constructible vector bundle $\mathcal{EXT}^1$, the problem is that we don't know how to apply the integral identity.  More precisely, in the case of two stack functions from single points both parameterising the object $N$, we \textit{do} have something to apply the integral identity to -- the induced potential on the differential graded vector bundle $\mathcal{END}^{\bullet}$ over $\Ext^1(N,N)$, defined as in Definition \ref{homdef} -- but away from the origin, quadratic terms show up, that are removed when we only consider the minimal superpotential $W_{\min}$.
\smallbreak
The same story occurs if we replace the two stack functions we were multiplying before, which were both $\nu_N$, with arbitrary $\nu_{E_i}$, for $E_1,E_2\in \mathcal{C}$.  Let us denote the version of the vector bundle $V$ from (\ref{Vdef}) that we get after making these replacements by $V_{E_1,E_2}$, so $V_{E_1,E_2}$ is a vector bundle on $\Ext^1(E_2,E_1)$.  The key, then, is to get some control over the constructible vector bundle $V_{E_1,E_2}$, and its associated quadratic form, which we will denote $Q_{E_1,E_2}$, so that we know how to correct our map $``\DT"$ in order to get something that preserves products.  It turns out that (up to a notion of equivalence that induces isomorphisms of motivic Milnor fibers in the quotient ring $\overline{\Mot}^{\muhat}(\Ext^1(E_2,E_1))$) the pair of the vector bundle $(V_{E_1,E_2},Q_{E_1,E_2})$ is intrinsic to the category $\mathcal{C}$, i.e. if we had picked a different minimal model for the category consisting just of the two objects $E_1,E_2$, and so obtained a new pair of a vector bundle with nondegenerate quadratic form, $(V'_{E_1,E_2},Q'_{E_1,E_2})$, the modification to the motivic Milnor fibre obtained by replacing the motivic weight 
\[
(1-\MF(W_{\min}))
\]
by 
\[
(1-\MF(W_{\min}))(1-\MF(Q'_{E_1,E_2}))
\]
would be the same as before.  This shouldn't come as a great shock: the failure of our naive $``\DT"$ map to preserve the product is again intrinsic to $\mathcal{C}$, by construction.  So the dream is not dead at this point: if we can come up with a way to coherently counteract the error term introduced by ignoring the contribution from $(V_{E_1,E_2},Q_{E_1,E_2})$ we will have come up with a fix that is invariant under quasi-equivalences of Calabi-Yau categories.
\smallbreak
This then, defines the role of orientation data in the theory of motivic Donaldson--Thomas theory:
\begin{cond}
\label{maincon}
\textbf{Orientation data provides a way of replacing $(\mathcal{EXT}^1,W_{\min})$ with a pair $(\mathcal{EXT}^1\oplus \mathcal{V}, W_{\min}\oplus Q)$ in such a way that the map} $\Phi$ \textbf{defined by integrating with respect to the weight which, over an element $M\in \mathcal{C}$ is $(1-\MF(W_{\min}\oplus Q))\mathbb{L}^{-\dim(V)/2+\sum_{i\leq 1}(-1)^i \dim(\Ext^{i}(M,M))/2})$ provides a map preserving associative product}.
\end{cond}
\begin{rem}
We have finally applied some diligence and added in the powers of $\mathbb{L}^{1/2}$.  As we will see later, there is a natural choice of orientation data such that this power is zero for all modules in the Abelian category of $B$-modules.  This is already evident at least for the 2-dimensional module we have been studying in this chapter.
\end{rem}

\section{Appendix: deferred motivic calculations}
\label{motivesapp}
Recall Proposition \ref{x4y4}, which stated the equality of $\hat{\mu}$-equivariant motives
\begin{equation}
\MF(x^4+y^4)=[C_1]-4\mathbb{L},
\end{equation}
where $C_1$ is a genus 3 complex curve, with the action $2(\alpha+\alpha^2+\alpha^3)$ on its middle cohomology.
\begin{proof}
One can show this as follows: first, note that if $X=\Cp^2$, the blowup at the origin
\[
\xymatrix{
\tilde{X}\ar[d]^h\\ X
}
\]
provides an embedded resolution of $f=x^4+y^4$.  As ever, let $J$ denote the set of divisors in $(fh)^{-1}(0)$, as in the formula (\ref{MFformula}).  There are then 5 elements in $J$, which we denote $E,D_1,D_2,D_3,D_4$, where $E$ is the exceptional $\mathbb{P}^1$.  The preimage $h^{-1}(0)$ is $E$, which intersects all of the divisors of $J$ nontrivially.  So there are 5 terms in the sum 
\begin{equation}
\label{formu}
\sum_{\emptyset\neq I\subset J}(1-\mathbb{L})^{|I|-1}[\tilde{D_J}]_{\{0\}}
\end{equation}
coming from the 4 sets $\{E,D_i\}$ as well as from the singleton set $\{E\}$.  All divisors of $(fh)^{-1}(0)$ apart from the exceptional $\mathbb{P}^1$ have multiplicity 1, so it follows that the \'{e}tale cover corresponding to each of the points $E\cap D_i$ is just the 1-sheeted cover.  So each of these points contributes $(1-\mathbb{L})$ to (\ref{formu}) .  There remains the \'{e}tale cover over the complement to the projective variety $V(x^4+y^4)$ in $E$, which is denoted, as in the formula (\ref{MFformula}) by $\tilde{D}_{\{E\}}$.  This cover is 4-sheeted, since $fh$ vanishes to order 4 along $E$.  One can complete in the obvious way the resulting 4-sheeted \'{e}tale cover to a branched cover 
\[
\xymatrix{
C_1\ar[d]\\
\mathbb{P}^1
}
\]
of $\mathbb{P}^1$.  Since this branched cover is simply ramified at each branch point of $\mathbb{P}^1$, i.e. there is only one point in the fibre of each branch point, it follows that the cover is connected, and $C_1$ is a genus 3 curve.  One can work out the equivariant Euler characteristic of $C_1$ by taking a good cover, in the analytic topology, of $\mathbb{P}^1$, such that any open set in the cover contains at most one of the branchpoints.  This calculation yields
\begin{align*}
\chi_{eq}(C_1)=&(1+\alpha+\alpha^2+\alpha^3)\chi(\mathbb{P}^1-\{\hbox{4 points}\})+4\\
=&2-2(\alpha+\alpha^2+\alpha^3).
\end{align*}

Since we know that $\mathbb{Z}_4$ acts trivially on the top and bottom cohomology, we deduce that $C_1$ has the cohomology stated in the proposition.  Putting everything together we have
\begin{align*}
\MF(x^4+y^4)=&([C_1]-4)+4(1-\mathbb{L})\\
=&[C_1]-4\mathbb{L}.
\end{align*}

\end{proof}

In similar fashion we can explicitly describe $\MF(x^4+y^2)$:
\begin{prop}
\label{x4y2}
There is an equality of $\hat{\mu}$-equivariant motives
\begin{equation}
\MF(x^4+y^2)=[C_2]-2\mathbb{L}
\end{equation}
where $C_2$ is a genus 1 curve with the action $\alpha+\alpha^3$ on its middle cohomology.
\end{prop}
\begin{proof}
The motivic Milnor fibre of $x^4+y^2$ is obtained by performing a couple of blowups as in our resolution of $S_3$, the $\mathbb{P}^1$ of $A_3$ singularities in the projective variety $V(\tra(T^4))$.  After the first blowup we introduce an exceptional $\mathbb{P}^1$, which the two components of the strict transform of the divisor given by the original vanishing locus of $x^4+y^2$ meet in a single point, as in the leftmost part of Figure \ref{bu1}.  Blowing up this point gives us the rightmost arrangement of divisors of Figure \ref{bu1}.  The new exceptional $\mathbb{P}^1$ we label $E_2$, and the strict transform of the first exceptional $\mathbb{P}^1$ we label $E_1$.   Let
\[
\xymatrix{
\tilde{Z}\ar[d]^s\\
\Cp^2
}
\]
be the map of schemes obtained by performing these two blowups.  Then the numbers next to the exceptional divisors in Figure \ref{bu1} indicate the order of vanishing of the function $(x^4+y^2) s$ on those divisors.\smallbreak
\begin{figure}
\centering
\input{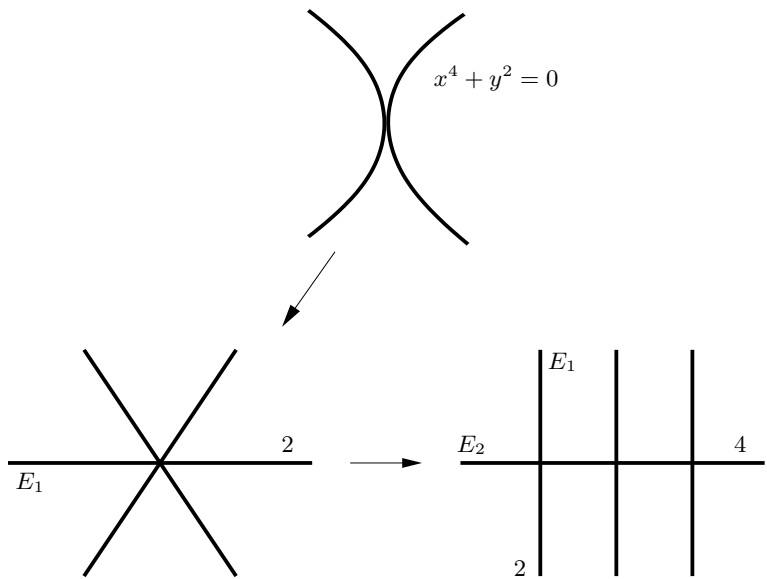}
\caption{Resolved $x^4+y^2$}
\label{bu1}
\end{figure}

The preimage $s^{-1}(0)$ is equal to the union $E_1\cup E_2$.  The complement to $E_2$ in $E_1$ is a copy of $\Cp$, from which it follows that our 2-sheeted \'{e}tale cover of it, $\tilde{D}_{\{E_2\}}$, must be the trivial $\mathbb{Z}_2$-torsor.  The (resolved) completion of the 4-sheeted \'{e}tale cover of $E_1$, which we denote $C_2$, is again connected, since two of its branching points are simply ramified.  So we can use the same trick as for Proposition \ref{x4y4} to work out its cohomology using equivariant Euler characteristics.  This gives that
\[
\chi(C_2)=(1+\alpha+\alpha^2+\alpha^3)\chi(\mathbb{P}^1-\{\hbox{3 points}\})+2+(1+\alpha^2)=2-(\alpha+\alpha^3),
\]
implying that $C_2$ is a torus with the action of $\mathbb{Z}_4$ on its middle cohomology given by the sum $\alpha+\alpha^3$.  Putting all the pieces together,
\begin{align*}
\MF(x^4+y^2)=&[C_2]-(2+(1+\alpha^2))+(1+\alpha^2)\mathbb{L}+(1-\mathbb{L})(2+(1+\alpha^2))\\
=&[C_2]-2\mathbb{L}.
\end{align*}
\end{proof}
Next we tidy up the unfinished business of calculating $[\tilde{D}_{\{Y\}}]$ from Proposition \ref{trivpart}.
\begin{prop}
\label{4sheetcover}
There is an equality of absolute equivariant motives
\begin{align}
\label{6part}
[\tilde{D}_{\{Y\}}]=&\mathbb{L}[C_1]+\mathbb{L}(\mathbb{L}-1)[C_2]-2\mathbb{L}(\mathbb{L}+1)\\
\label{6parta}
=&\mathbb{L}\MF(x^4+y^4)+(\mathbb{L}-1)\mathbb{L}\MF(x^4+y^2)+2\mathbb{L}(\mathbb{L}^2-1).
\end{align}
\end{prop}
\begin{proof}
We stratify the cover $\tilde{D}_{\{Y\}}$ by stratifying the base $D(\tra(T^4))$, the complement in $\mathbb{P}^3$ to $V(\tra(T^4))$.  Denote matrices of $D(\tra(T^4))$ by
\[
\left(\begin{array}{cc}a&b\\c&d\end{array}\right).
\]
Note that there is a $\Cp^*$-action on $D(\tra(T^4))$ given by
\[
t\cdot\left(\begin{array}{cc}a&b\\ c&d\end{array}\right)\mapsto\left(\begin{array}{cc}a&tb\\t^{-1}c&d\end{array}\right).
\]
\begin{enumerate}
\item
First consider the subscheme $P_1\subset D(\tra(T^4))$ of matrices with nonzero trace, and $c\neq 0$.  $P_1$ is acted on freely by $\Cp^*$ with the above action.  So we may take the quotient, and multiply the motive we get by $(\mathbb{L}-1)$.  So we fix the trace to be equal to 1, thereby fixing an element in the line of matrices determined by an arbitrary matrix with nonzero trace, and set $c=1$, thereby passing to the quotient by the $\Cp^*$-action.  Once we have fixed the trace, the complement $D(\tra(T^4))$ is determined entirely by the determinant, it is given by those matrices with determinant not equal to $\theta_1=1+\sqrt{1/2}$ or $\theta_2=1-\sqrt{1/2}$.  There is an isomorphism
\begin{align*}
\Cp\times (\Cp-\{\theta_1,\theta_2\})\rightarrow &P_1/\Cp^*\\
(x,y)\mapsto &\left(\begin{array}{cc}x&x(1-x)-y\\1&1-x\end{array}\right).
\end{align*}
Now
\[
p^4+q^4=(p+q)^4-4pq(p+q)+2(pq)^2
\]
from which it follows that the local defining function for $\tra(T^4)$ on $P_1/\Cp^*$ is $1-4y+2y^2$.  The function $2y^2-4y+1$ defines a 4-sheeted \'{e}tale cover in the usual way, and this is just the \'{e}tale cover occurring in the calculation of the motivic Milnor fibre of $\MF(x^4+y^2)$, since we form a homogeneous quartic from $2y^2-4y+1$ by introducing the variable $z$ and taking $2y^2z^2-4yz^3+z^4$, which vanishes to order 2 at infinity.  This is just the cover obtained by removing the branchpoints from the equivariant curve $C_2$ of Proposition \ref{x4y2}.  We conclude that there is an equality of absolute equivariant motives 
\begin{equation}
[D_{\{Y\}}]_{P_1}=\mathbb{L}(\mathbb{L}-1)([C_2]-(3+\alpha^2)).
\end{equation}
\item
Next let $P_2\subset D(\tra(T^4))$ be the subscheme of matrices with nonzero trace, $c=0$, and $b\neq 0$.  Again we take representatives with trace equal to 1, and again we use the free $\Cp^*$-action to assume that $b=1$.  Then there is an isomorphism
\begin{align*}
\Cp-\{\hbox{roots of }p(z)=z^4+(1-z)^4\}\rightarrow &P_2/\Cp^*\\
x\mapsto &\left(\begin{array}{cc}x&1\\0&1-x\end{array}\right).
\end{align*}
The local defining function for $\tra(T^4)$ becomes $x^4+(1-x)^4$.  This polynomial has 4 separate roots, so the 4-sheeted \'{e}tale cover it defines over $\Cp$ is the curve $C_1$, minus the branchpoints, and also minus the 4 points lying over infinity.  So
\begin{equation}
[D_{\{Y\}}]_{P_2}=(\mathbb{L}-1)([C_1]-4-(1+\alpha+\alpha^2+\alpha^3)).
\end{equation}
\item
Let $P_3\subset \mathbb{P}^3$ be the subscheme consisting of matrices with trace equal to zero, $a\neq 0$, and $c\neq 0$.  Then we can assume $a=1$, after taking an appropriate scalar multiple.  Furthermore we again have a free $\Cp^*$-action, and so we take the quotient again, and assume $c=1$.  There is an isomorphism
\begin{align*}
\Cp^*\rightarrow &P_3/\Cp^*\\
x\mapsto &\left(\begin{array}{cc} 1&x-1\\1&-1\end{array}\right).
\end{align*}
The local defining equation for $\tra(T^4)$ becomes $2x^2$.  The resulting 4-sheeted cover of $\Cp^*$ has 2 components, each a torus, and we conclude that
\begin{equation}
[D_{\{Y\}}]_{P_3}=(\mathbb{L}-1)(1+\alpha^2)(\mathbb{L}-1).
\end{equation}
\item
Let $P_4\subset \mathbb{P}^3$ be the subscheme consisting of matrices with zero trace, $a\neq 0$, $c=0$, $b\neq 0$.  We again may assume $a=1$.  $P_4$ is just a single free $\Cp^*$-orbit, and so we conclude that
\begin{equation}
[D_{\{Y\}}]_{P_4}=(\mathbb{L}-1)(1+\alpha+\alpha^2+\alpha^3).
\end{equation}
\item
Let $P_5\subset \mathbb{P}^3$ be the subscheme of diagonal matrices.  Then $P_5\cong \mathbb{P}^1$, and $V(\tra(T^4))\cap P_5$ consists of four points.  It follows that the \'{e}tale cover, restricted to $P_5$ is just the \'{e}tale cover occurring in the calculation of the motivic Milnor fibre of $x^4+y^4$, and so
\begin{equation}
[D_{\{Y\}}]_{P_5}=[C_1]-4.
\end{equation}
\item
Let $P_6\subset \mathbb{P}^3$ be the subscheme consisting of off-diagonal matrices.  Both entries $b$ and $c$ must be nonzero for the matrix to be in $D(\tra(T^4))$.  So we may assume $c=1$.  On this orbit $\Cp^*$ again doesn't act freely, so we will ignore it.  There is an isomorphism
\begin{align*}
\Cp^*\rightarrow &P_6\\
x\mapsto &\left(\begin{array}{cc}0&x\\1&0\end{array}\right).
\end{align*}
The local defining equation for $\tra(T^4)$ is $2x^2$.  So the resulting 4-sheeted \'{e}tale cover of $\Cp^*$ is given by a cover by 2 tori, and we have the equality
\begin{equation}
[D_{\{Y\}}]_{P_6}=(\mathbb{L}-1)(1+\alpha^2).
\end{equation}
\end{enumerate}
Putting all this together gives equation (\ref{6part}).  In light of Propositions \ref{x4y4} and \ref{x4y2} we also deduce equation (\ref{6parta}).

\end{proof}

\chapter{Background on $A_{\infty}$ algebra}
\label{inftystuff}
\section{$A_{\infty}$-algebra objects}
\label{inftybackground}
We start with some well-rehearsed definitions and facts regarding $A_{\infty}$-algebras, modules, and bimodules. The definitions and many of the facts contained in this chapter can be found in the excellent introduction \cite{keller-intro} -- our treatment of the sign issue is taken from here.  Further details can be found in \cite{Keller-A2}.  A comprehensive account of much of the needed material is to be found in Kenji Lef\`{e}vre-Hasegawa's thesis \cite{KLH} -- we will often refer to this for complete proofs of the background we need.  We mention also Kontsevich's paper \cite{Kontsevich95} and Kontsevich and Soibelman's notes \cite{KSnotes}, although these follow a slightly different approach to the one here (Section \ref{enrichments} is closer to these in spirit than the preceeding sections).  Finally we mention Paul Seidel's book \cite{Seidel08}, the treatment of which, modulo some signs, is basically the one given here.  In short, no claim is made for any originality here, except perhaps for some easy details -- in particular the statement that the category of twisted objects over a Calabi-Yau category is again a Calabi-Yau category.  \smallbreak
Throughout this section we discuss two constructions.  The first is a natural, abstract extension of the definitions of the categories of $A_{\infty}$-algebras, modules, and bimodules, the underlying objects of which lie in $(\vect)_{\mathbb{Z}}$, to the setting of an arbitrary monoidal $k$-linear $\mathbb{Z}$-graded category $\mathcal{C}$.  To the reader already familiar with these notions, these extensions will be as expected, and straightforward.  The second type restricts attention to $A_{\infty}$-algebras, modules, and bimodules in their more well-known setting (this corresponds to setting $\mathcal{C}=(\vect)_{\mathbb{Z}}$).  In this setting, the definitions are even more standard, but we introduce a relative notion of unitality, in order to accommodate the type of algebra we are heading towards in Chapter \ref{quiveralgebras}, which is a natural notion of $A_{\infty}$-quiver algebra, as well as to state a large class of algebras for which the construction of orientation data suggested in \cite{KS} works.
\begin{defn}
Let $\mathcal{C}$ be a monoidal $\mathbb{Z}$-graded category.  An $A_{\infty}$-algebra object in $\mathcal{C}$ is a pair $(A,(m_{A,i}))$, where $A\in\mathcal{C}$ and $m_{A,i}$ is a sequence of morphisms,
\[
m_{A,i}:A^{\otimes i}\rightarrow A,
\]
of degree $(2-i)$, for $i\in\mathbb{Z}$, $i\geq 1$, satisfying the compatibility relations
\begin{equation}
\label{struc-compatibility}
\sum_{a+b+c=n}(-1)^{a+bc}m_{A,a+c+1}(\id^{\otimes a}\otimes m_{A,b}\otimes \id^{\otimes c})=0
\end{equation}
for each $n\geq 1$.
\end{defn}
We make the assumptions on $\mathcal{C}$ in the definition above a standing assumption throughout this section.
\begin{defn}
An $A_{\infty}$-algebra object $(A,(m_{A,i}))$ in a category $\mathcal{C}$ is called \textit{minimal} if $m_{A,1}=0$.  It is called strict if $m_{A,i}=0$ for all $i\geq 3$.  An ordinary algebra object in $\mathcal{C}$ is an $A_{\infty}$-algebra object $(A,m_i)$ such that $m_i=0$ for all $i\neq 2$.
\end{defn}

\begin{rem}
If $(A,m_i)$ is an $A_{\infty}$-algebra object in $\mathcal{C}$, then the compatibility condition (\ref{struc-compatibility}), for $n=1$, ensures that $(A,m_1)$ is a differential graded object of $\mathcal{C}$.
\end{rem}
Here we follow the sign convention of \cite{keller-intro}, which in turn follows \cite{signsstart}.  As in that case, extra signs appear in the symmetric monoidal categories we consider, due to the Koszul sign rule/twisted symmetric monoidal structure.
\begin{defn}
An $A_{\infty}$-algebra object $(A,m_{A,n})$ in a category $\mathcal{C}$ is given a \textit{unital structure} by the data of a morphism $f:\idmon_{\mathcal{C}}\rightarrow A$ from the monoidal unit to $A$, such that the following diagrams commute
\[
\xymatrix{
A\ar[d]^{\id}\ar[r]&\idmon_{\mathcal{C}}\otimes A\ar[d]^{f\otimes \id}&&A\ar[d]^{\id}\ar[r]&A\ar[d]^{\id\otimes f}\otimes \idmon_{\mathcal{C}}\\
A&A\otimes A\ar[l]^(.55){m_{A,2}}&&A&A\otimes A\ar[l]^(.55){m_{A,2}},
}
\]
and such that for all $n\geq 3$, and for all $i\geq 0$, $m_{A,n}(\id^{\otimes i}\otimes f\otimes \id^{\otimes (n-i-1)})=0$.
\end{defn}
\begin{rem}
In the case $\mathcal{C}=(\vect)_{\mathbb{Z}}$, $A_{\infty}$-algebra objects are $A_{\infty}$-algebras, as presented in $\cite{keller-intro}$.  Due to the symmetric monoidal structure stipulated in Section \ref{notation}, when we evaluate the compatibility relations on actual elements of underlying graded vector spaces, the Koszul sign rule appears, as in Section 3.1 of \cite{keller-intro}.  Finally, a unital structure on an $A_{\infty}$-algebra object $(A,(m_{A,n}))$ in $(\vect)_{\mathbb{Z}}$ is a \textit{strictly} unital structure on the $A_{\infty}$-algebra $A$, in the usual ($A_{\infty}$) sense.
\end{rem}
\begin{defn}
An $A_{\infty}$-algebra object in $(\vect)_{\mathbb{Z}}$ will just be called an $A_{\infty}$-algebra.
\end{defn}

\begin{defn}
\label{algmorph}
A morphism of $A_{\infty}$-algebra objects in a category $\mathcal{C}$, from $(A,(m_{A,i}))$ to $(B,(m_{B,i}))$ is defined as a series of morphisms
\[
f_n:A^{\otimes n}\rightarrow B,
\]
of degree $(1-n)$, satisfying the following compatibility conditions, for $n\geq 1$
\begin{equation}
\label{morph-compatibility}
\sum_{a+b+c=n} (-1)^{a+bc}f_{a+1+c}(\id^{\otimes a} \otimes m_{A,b}\otimes \id^{\otimes c})=\sum_{\substack{r\leq n\\ i_1+\ldots+i_r=n}}(-1)^{\bigstar}m_{B,r}(f_{i_1}\otimes\ldots\otimes f_{i_r})
\end{equation}
where here $\bigstar=(r-1)(i_1-1)+(r-2)(i_2-1)+\ldots+(i_{r-1}-1)$.  A morphism is called \textit{strict} if $f_m=0$ for $m\geq 2$.  A morphism of \textit{unital} $A_{\infty}$-algebra objects $f:(A,g)\rightarrow (B,h)$ in a category $\mathcal{C}$ is a series of maps, as above, such that the following diagram commutes
\[
\xymatrix{
\id_{\mathcal{C}}\ar[r]^{g}\ar[rd]^{h}&A\ar[d]^{f_1}\\
&B,
}
\]
and such that, for all $n\geq 2$, and all $i\geq 0$, $f_{n}(\id^{\otimes i}\otimes g\otimes \id^{\otimes (n-i-1)})=0$.
\end{defn}
In the case $\mathcal{C}=(\vect)_{\mathbb{Z}}$ this definition of morphism is the usual definition of an $A_{\infty}$-algebra morphism, as found in, for example, \cite{keller-intro}, \cite{KLH}, \cite{Seidel08}, and the definition of a unital morphism is the definition of a strictly unital $A_{\infty}$-algebra morphism.
\begin{examp}
\label{idmorphism}
Let $A$ be an $A_{\infty}$-algebra object in $\mathcal{C}$.  Then we define the \textit{identity morphism} $\id_A:A\rightarrow A$, as a morphism of $A_{\infty}$-algebra objects, by the following series of morphisms in $\mathcal{C}$ (we use the same symbol for the usual identity morphism of $A$ as an element of $\mathcal{C}$):
\begin{align*}
\id_{A,1}=&\id_{A}\\
\id_{A,i}=& 0\hbox{ for }i\geq 2.
\end{align*}
\end{examp}
\begin{defn}
Let $(A,(m_{A,i})),(B,(m_{B,j})),(C,(m_{C,k}))$ be $A_{\infty}$-objects in $\mathcal{C}$, and let $f:(A,(m_{A,i}))\rightarrow (B,(m_{B,j}))$, $g:(B,(m_{B,j}))\rightarrow (C,(m_{C,k}))$ be morphisms between them.  Then we define the composition $g\circ f$ by setting 
\[
(g\circ f)_n=\sum_{\substack{r\leq n\\ i_1+\ldots+i_r=n}}(-1)^{\bigstar}g_{r}(f_{i_1}\otimes\ldots\otimes f_{i_r})
\]
where $\bigstar$ is as in Definition \ref{algmorph}.
\end{defn}
\begin{rem}
The above composition rule is associative, and so we obtain a $k$-linear category, with objects the $A_{\infty}$-objects of $\mathcal{C}$, and with morphisms as defined above.  The identity morphisms are as defined in Example \ref{idmorphism}.
\end{rem}
\smallbreak
In all the cases we consider, the graded category $\mathcal{C}$ is the graded category associated to an Abelian category $\mathcal{A}$.  As such there is a functor $\Ho^{\bullet}(-)$ from the category of objects equipped with differential in $\mathcal{C}$, to $\mathcal{A}_{\mathbb{Z}}=\mathcal{C}$.  One can check that the $A_{\infty}$-structure on $A$ induces a graded algebra structure on $\Ho^{\bullet}(A)$, and that a morphism of $A_{\infty}$-objects $f:A\rightarrow B$ induces a morphism of graded algebra objects $\Ho^{\bullet}(f):\Ho^{\bullet}(A)\rightarrow \Ho^{\bullet}(B)$.  The first part of this verification is to see that $f_1$ is a morphism of objects with differential.
\begin{defn}
A morphism $(f_n,n\geq 1):A\rightarrow B$ of $A_{\infty}$-algebra objects in $\mathcal{C}$ is a \textit{quasi-isomorphism} if $\Ho^{\bullet}(f)$ is an isomorphism of algebra objects in $\mathcal{C}$.  
\end{defn}
\begin{rem}
This is equivalent to $\Ho^{\bullet}(f_1)$ being an isomorphism of objects in $\mathcal{C}$.
\end{rem}

\begin{defn}
\label{unitalalg}
A triple $(A,S,f)$ gives $A$ the structure of an $A_{\infty}$-algebra \textit{over} the $A_{\infty}$-algebra $S$ if $f$ is a strict morphism of $A_{\infty}$-algebras $f:S\rightarrow A$ such that $f_1$ is an injection.
\end{defn}
Normally from the context the morphism $f$ is clear, in which case we call $A$ simply an $A_{\infty}$-algebra over $S$.
\begin{defn}
\label{unitalover}
A triple $(A,S,l)$ gives the $A_{\infty}$-algebra $A$ the structure of an $S$-unital algebra, where $S$ is an $A_{\infty}$-algebra, if it gives $A$ the structure of an $A_{\infty}$-algebra over $S$, and for all $n\geq 3$, $i\geq 0$, $m_{A,n}(\id^{\otimes i}\otimes l\otimes \id^{\otimes (n-i-1)})=0$.  A triple $((A,g),(S,h),l)$ gives the unital $A_{\infty}$-algebra $(A,g)$ the structure of an $(S,h)$-unital algebra, where $(S,h)$ is a unital $A_{\infty}$-algebra, if it gives $A$ the structure of an $S$-unital algebra, and $l$ is a morphism of unital $A_{\infty}$-algebras.
\end{defn}
\begin{defn}
\label{Saug}
Let $(A,S,l)$ give $A$ the structure of an $S$-unital algebra.  Then a strict morphism $p:A\rightarrow S$ gives $A$ the structure of an \textit{augmented} algebra over $S$ if $p\circ l=\id_S$.
\end{defn}

\begin{defn}
A morphism (respectively, unital morphism) $f:(A,S,g)\rightarrow (B,S,h)$ of $A_{\infty}$-algebras over $S$ is a morphism (respectively, unital morphism) of $A_{\infty}$-algebras $f:A\rightarrow B$ such that the following diagram commutes
\[
\xymatrix{
S\ar[r]^{g}\ar[rd]^{h}&A\ar[d]^{f}\\
&B.
}
\]
\end{defn}
Let $A$ be a unital $A_{\infty}$-algebra object in a category $\mathcal{C}$.  The monoidal unit $\idmon_{\mathcal{C}}$ of $\mathcal{C}$ possesses the structure of a strict $A_{\infty}$-algebra object in $\mathcal{C}$, and the unital morphism for $A$ becomes a strict morphism $\textit{u}$ of $A_{\infty}$-algebra objects in $\mathcal{C}$.  
\begin{defn}
\label{Caug}
Let $A$ be a unital $A_{\infty}$-algebra object in $\mathcal{C}$.  An augmentation of $A$ is a strict morphism of $A_{\infty}$-algebra objects $p:A\rightarrow \idmon_{\mathcal{C}}$ satisfying $p\circ u=\id_{\idmon_{\mathcal{C}}}$.
\end{defn}
Let $S$ be an arbitrary ordinary unital graded algebra.  Let $S\biModn S$ be the category of ordinary $S$-bimodules.  Then $S$ is the monoidal unit for $S\biModn S$, and the algebra structure on the monoidal unit is just the algebra structure on $S$.
\begin{prop}
\label{yeti}
There is an equivalence of categories between unital $A_{\infty}$-algebra objects in $S\biModn S$ and $S$-unital $A_{\infty}$-algebras.  There is also an equivalence between augmented $A_{\infty}$-algebra objects in $S\biModn S$ and augmented $A_{\infty}$-algebras over $S$.
\end{prop}
\begin{proof}
Let $A$ be a $S$-unital $A_{\infty}$-algebra.  We define a unital $A_{\infty}$-object $A_S$ in $S\biModn S$ as follows.  First let the underlying bimodule of $A_S$ be $A$.  Let
\[
(a_1,\ldots,a_n)\in A_S^{\otimes_S n},
\]
where we write $\otimes_S$ to remind the reader that this tensor product is in the category of $S$-bimodules.  Then we can lift $(a_1,\ldots,a_n)$ to some $(\tilde{a}_1,\ldots,\tilde{a}_n)\in A^{\otimes n}$.  We evaluate $m_A$ on this $n$-tuple, and let this be our element $m_{A_S,n}(a_1,\ldots,a_n)$.  To show that this is well defined, apply the compatibility equation (\ref{struc-compatibility}) to the $(n+1)$-tuple
\[
(\tilde{a}_1,\ldots,\tilde{a}_i,s,\tilde{a}_{i+1},\ldots,\tilde{a}_n).
\]
By strict unitality this becomes 
\[
(-1)^{i-1}m_{A,n}(\tilde{a}_1,\ldots,\tilde{a}_i s,\tilde{a}_{i+1},\ldots,\tilde{a}_n)+(-1)^i m_{A,n}(\tilde{a}_1,\ldots,\tilde{a}_i,s\tilde{a}_{i+1},\ldots,\tilde{a}_n)=0.
\]
Given a $S$-unital morphism $f:A\rightarrow B$ of strictly $S$-unital $A_{\infty}$-algebras, we define $f_S:A_S\rightarrow B_S$ similarly, by applying $f$ to $(\tilde{a}_1,\ldots,\tilde{a}_n)$.  Then the compatability equation (\ref{morph-compatibility}), applied to 
\[
(\tilde{a}_1,\ldots,\tilde{a}_{i},s,\tilde{a}_{i+1},\ldots,\tilde{a}_n)
\]
yields
\[
\begin{array}{c}
\pm\left(f_n(\tilde{a}_1,\ldots,\tilde{a}_{i}s,\tilde{a}_{i+1},\ldots,\tilde{a}_n)-f_n(\tilde{a}_1,\ldots,\tilde{a}_i,s\tilde{a}_{i+1},\ldots,\tilde{a}_n)\right)=\\ \sum_{\substack{p_1+\ldots+p_i=i\\q_1+\ldots+q_j=n-i}}\limits\pm m_{i+j+1}(f_{p_1}(\tilde{a}_1,\ldots,\tilde{a}_{p_1}),\ldots,f_{p_i}(\tilde{a}_{i-p_i+1},\ldots,\tilde{a}_{i}),f_1(s),\\f_{q_1}(\tilde{a}_{i+1},\ldots,\tilde{a}_{i+q_1}),\ldots,f_{q_j}(\tilde{a}_{n-q_j+1},\ldots,\tilde{a}_n))=0.
\end{array}
\]
In the opposite direction, let $A_S$ be a unital $A_{\infty}$-object of $S\biModn S$.  Then we define the maps $m_{A,n}$ of an $A_{\infty}$-algebra $A$ by setting 
\[
m_{A,n}(a_1,\ldots,a_n)=m_{A_S,n}(a_1,\ldots,a_n),
\]
and we deal with morphisms similarly.  The statement regarding augmentations is clear.
\end{proof}

If $M$ is an object of a monoidal category, $\mathcal{C}$, then there is a natural notion of the free (nonunital) $A_{\infty}$-algebra object generated by $M$ (we assume that $\mathcal{C}$ admits infinite direct sums).  This is the analogue of the tensor algebra (without unit) construction for a bimodule over an algebra.  It is constructed by taking direct sums of copies of $M^{\otimes i}$, for various $i$, labelled by rooted planar trees with $i$ branches, with conditions imposed by the compatibility conditions (see e.g. \cite{KSnotes}).  We denote this $A_{\infty}$-algebra object
\[
\Free_{\infty,\nonu}(M).
\]
\begin{defn}
\label{fingen}
Say a forgetful functor 
\[
\forget:\mathcal{C}\rightarrow (\vect)_{\mathbb{Z}}
\]
is understood.  We say that an $A_{\infty}$-algebra object $A$ is \textit{finitely generated} if there is a $M\in \mathcal{C}$ and a strict morphism $f:\Free_{\infty,\nonu}(M)\rightarrow A$ such that $\forget(M)$ is finite-dimensional, and $\forget(f)$ is surjective.
\end{defn}
\begin{rem}
Formally, $\Free_{\infty,\nonu}(M)$ is the free module over the $A_{\infty}$-operad in $\mathcal{C}$ generated by $M$, as defined in e.g. \cite{GKop}.
\end{rem}
Similarly, for $M$, as above, an object of some monoidal category $\mathcal{C}$ which admits infinite direct \textit{products}, we define the $A_{\infty}$-algebra object of $\mathcal{C}$ given by considering direct products of $M^{\otimes i}$ labelled by rooted planar trees
\[
\Free_{\infty,\nonu,\mathrm{formal}}(M).
\]
\section{$A_{\infty}$-modules and bimodules}
Assume as before that $\mathcal{C}$ is a graded monoidal category.  Where it is needed in the following section, let $\mathcal{D}_l$ be a graded category, and assume that we are given a bifunctor
\[
\otimes: \mathcal{C}\times \mathcal{D}_l\rightarrow \mathcal{D}_l,
\]
such that the associativity diagram
\[
\xymatrix{
\mathcal{C}\times\mathcal{C}\times\mathcal{D}_l\ar[d]_{\otimes \times \id}\ar[r]^(.57){\id\times \otimes}& \mathcal{C}\times \mathcal{D}_l\ar[d]\\
\mathcal{C}\times\mathcal{D}_l\ar[r] &\mathcal{C}
}
\]
commutes up to natural isomorphism, the obvious unitality axiom is satisfied, and the usual pentagon axiom is satisfied.  We assume we are given a natural isomorphism $\epsilon:\idmon_{\mathcal{C}}\otimes -\rightarrow -$.  Where it is required, we let $\mathcal{D}_r$ be another graded category, with a bifunctor
\[
\otimes: \mathcal{D}_r\times\mathcal{C}\rightarrow \mathcal{D}_r,
\]
subject to analagous conditions.
\begin{defn}
\label{moddef}
Let $A$ be an $A_{\infty}$-algebra object in a graded monoidal category $\mathcal{C}$.  A left $A_{\infty}$-module over $A$ is a pair $(M,(m_{M,i}))$ consisting of $M\in \ob(\mathcal{D}_l)$, and a series of maps
\[
m_{M,i}:A^{\otimes (i-1)}\otimes M\rightarrow M
\]
of degree $(2-i)$ satisfying the compatibility relations
\begin{equation}
\label{modcompatability}
\sum_{\substack{a+b+c=n\\c\geq 1}} (-1)^{a+bc}m_{M,a+1+c}(\id^{\otimes a}\otimes m_{A,b}\otimes \id^{\otimes c})+\sum_{a+b=n}(-1)^{a} m_{M,a+1}(\id^{\otimes a}\otimes m_{M,b})=0
\end{equation}
for every $n\geq 1$.  A \textit{right} $A$-module is defined similarly.
\end{defn}
As in the case of $A_{\infty}$-algebra objects, $M$ acquires the structure of a differential graded object of $\mathcal{D}_l$ under $m_{M,1}$.
\begin{defn}
\label{yot}
Let $M$ be a module over the unital $A_{\infty}$-object $(A,g)$.  Then we say $M$ is unital if the following diagram commutes 
\[
\xymatrix{
M\ar[d]_{\id}\ar[r]^(.40){\epsilon_M^{-1}}&\idmon_{\mathcal{C}}\otimes M\ar[d]^{g\otimes \id}\\
M&A\otimes M.\ar[l]_(.59){m_{M,2}}
}
\]
\end{defn}
If $\mathcal{C}=(\vect)_{\mathbb{Z}}$, then by a left module over $A$ we will always mean an $A_{\infty}$-module object of $\mathcal{D}_l=(\vect)_{\mathbb{Z}}$, with $\mathcal{C}\times\mathcal{D}_l\rightarrow \mathcal{D}_l$ provided by the tensor product on $(\vect)_{\mathbb{Z}}$.
\begin{defn}
\label{Sunital}
Let $(A,S,g)$ be an $A_{\infty}$-algebra over $S$.  Then an $S$-unital $A$-module $(M,(m_{M,i}))$ is defined to be an $A$-module satisfying the condition that for all $n\geq 3$ and $n-1 \geq i \geq 1$, $m_{M,n}(\id^{\otimes (n-i-1)}\otimes g\otimes \id^{\otimes i})=0$.  If $((A,g),(S,h),l)$ is a unital $A_{\infty}$-algebra over $(S,h)$, then we define $S$-unital modules to be those that are also unital in the sense of Definition \ref{yot}.
\end{defn}

\begin{defn}
Let $A$ be an $A_{\infty}$-algebra object of $\mathcal{C}$.  A morphism $f:(M,(m_{M,i}))\rightarrow (N,(m_{N,j}))$ of left $A$-modules is defined by a series of linear maps, for $n\geq 1$,
\[
f_n:A^{\otimes (n-1)}\otimes M\rightarrow N
\]
of degree $(1-n)$, satisfying the compatibility relations
\begin{equation}
\label{modmorphcompatibility}
\sum_{a+b+c=n}(-1)^{a+bc}f_{a+1+c}(\id^{\otimes a}\otimes m_b\otimes \id^{\otimes c})=\sum_{a+b=n,b\geq 1}(-1)^{a}m_{N,a+1}(\id^{\otimes a}\otimes f_{b}).
\end{equation}
Here we adopt the convention that $m_i$, where $i\in\mathbb{N}$, is to be taken to mean $m_{K,i}$ if applied to a tensor product of the form $A^{\otimes (i-1)}\otimes K$ for $A$ an $A_{\infty}$-algebra object, $K$ a module over it, and is to be taken to mean $m_{A,i}$ if applied to a tensor product of the form $A^{\otimes i}$.
\end{defn}
\begin{rem}
\label{lazyconvention}
In order to reduce clutter we will adopt this convention wherever appropriate, i.e. wherever only one possible multiplication on a given tensor product has been defined.
\end{rem}
\begin{defn}
\label{interunit}
A morphism $f:M\rightarrow N$ of $S$-unital modules over a $S$-unital $A_{\infty}$-algebra $(A,S,l)$ is a morphism of $A$-modules satisfying the condition that for all $n\geq 2$ and all $n-1 \geq i \geq 1$, $f_n(\id^{\otimes(n-i-1)} \otimes l\otimes \id^{\otimes i})=0$.
\end{defn}
Let $(A,S,l)$ be a $S$-unital algebra, for $S$ an ordinary algebra, and let $M$ be an $S$-unital left $A$-module.  Then $M$ is an ordinary left $S$-module.  There is a natural bifunctor
\[
\otimes_S:S\biModn S\times S\lModn \rightarrow S\lModn,
\]
where $S\lModn$ is the category of ordinary left $S$-modules, and $M$ is a unital left module over $A_S$, the unital $A_{\infty}$-algebra object of $S\biModn S$ formed by considering $A$ as an $S$-bimodule.  The following is proved in the same way as Proposition \ref{yeti}.
\begin{prop}
There is an equivalence of categories between the category of unital left/right modules over the unital $A_{\infty}$-algebra object $A_S$ of $S\biModn S$, and $S$-unital left/right modules over the $S$-unital $A_{\infty}$-algebra $(A,S,l)$.
\end{prop}
\begin{defn}
Let $f:(M,(m_{M,i}))\rightarrow (N,(m_{N,i}))$ and $g:(N,(m_{N,i}))\rightarrow (K,(m_{K,i}))$ be a pair of $A$-module morphisms.  Then we define the composition $g\circ f:(M,(m_{M,i}))\rightarrow (K,(m_{K,l}))$ by defining the following series of morphisms, for $n\geq 1$:
\[
(g\circ f)_n=\sum_{a+b=n, n\geq b \geq 1}(-1)^{(b+1)a}g_{a+1}(\id^{\otimes a}\otimes f_{b}).
\]
\end{defn}
If $A$ is a unital $A_{\infty}$-algebra object, then we define the category $A\lMod$ to be the category of unital left $A$-modules with morphisms as in Definition \ref{interunit}.  We denote by $A\lModnu$ the category of nonunital left $A$-modules.  If $A$ has the structure of an $S$-unital $A_{\infty}$-algebra for some $A_{\infty}$-algebra $S$ then we denote by $A\lMod_{S}$ the category of $S$-unital left $A$-modules, with morphisms given by $S$-unital $A$-module morphisms.  We define the categories of unital, nonunital and $S$-unital right $A$-modules similarly, and denote them by $\rMod A$, $\rModnu A$ and $\mathsf{Mod}_S\mathrm{-}A$ respectively.  The full subcategories of all these categories with objects those modules satisfying the condition that the total homology of the object $M$ under $m_{M,1}$ is finite-dimensional (assuming we have some forgetful functor $\forget$ from $\mathcal{D}_l$ or $\mathcal{D}_r$ to $(\vect)_{\mathbb{Z}}$) will be denoted with the lowercase: $\amod$.  We will mostly be interested in these subcategories. \smallbreak

\begin{defn}
Let $A$ be an $A_{\infty}$-object in $\mathcal{C}$.  Let $M,N\in A\lModbs$, where $\blacksquare=nu,S,-$.  Let $f,g:M\rightarrow N$ be morphisms between them.  A \textit{homotopy} between them is a series of morphisms
\[
h_i:A^{\otimes (i-1)}\otimes M\rightarrow N
\]
of degree $-i$, such that 
\begin{eqnarray*}
&\sum_{a+b+c=n}(-1)^{a+bc}h_{a+1+c}(\id^{\otimes a}\otimes m_b\otimes \id^{\otimes c})-\\&\sum_{a+b=n,b\geq 1}(-1)^{a}m_{N,a+1}(\id^{\otimes a}\otimes h_{b})=f_n-g_n.
\end{eqnarray*}
We define homotopies of right modules similarly.
\end{defn}
The most useful facts about modules over an $A_{\infty}$-algebra are the following:
\begin{lem}[Homological perturbation theory]
Let $\mathcal{C}$ and $\mathcal{D}_l$ be the graded categories associated to Abelian categories, with $D_l$ semisimple.  Let $A$ be an $A_{\infty}$-object in $\CC$, and let left $A_{\infty}$-modules over it live in $\mathcal{D}_l$.  Let $f:M\rightarrow N$ be a morphism in $A\lModbs$, where again $\blacksquare=S,nu,-$.  Then if $f$ is a quasi-isomorphism, there is a $g:N\rightarrow M$ with $f\circ g$ and $g\circ f$ homotopic to the identity morphisms on $N$ and $M$ respectively.
\end{lem}
\begin{lem}[Minimal model lemma]
Let $A$ be as above, and let $M$ be a module over $A$.  Then there is a minimal module $N$ (i.e. $m_1(N)=0$), and a quasi-isomorphism
\[
f:N\rightarrow M.
\]
\end{lem}
We will consider a similar result in Theorem \ref{cycmm}, where we consider a version of the minimal model theorem for (cyclic) Calabi-Yau algebras.  A neat exposition of the method of proof for the above facts can be found in \cite{KLH} -- in general, by considering the underlying differential graded object of an object in $A_{\infty}$ algebra as the direct sum of a contractible part $V$ and a part with trivial differential $N$, the morphism $N\oplus V\rightarrow V$ is upgraded to a morphism of the relevant kind of $A_{\infty}$-object, and using the equivalences discussed in Section \ref{enrichments}, we identify the kernel of the projection in the category of appropriate co-differential graded co-objects as having underlying co-object the co-object associated to $N$, giving $N$ the appropriate $A_{\infty}$-structure.  As is somewhat standard, there is a degree of choice in the construction of quasi-inverses to quasi-isomorphisms, and to the construction of minimal models.  The following partially mitigates the overabundance of choices here:
\begin{lem}
Let $A$ and $M$ be as above.  Say we are given a splitting
\[
M^i\cong \Image(m_{M,1}:M^{i-1}\rightarrow M^i)\oplus \Ho^i(M)\oplus \frac{M^i}{\Ker(m_{M,1}:M^i\rightarrow M^{i+1})},
\]
for every $i\in \mathbb{N}$.  Then we have a \textit{canonical} minimal model, associated to this splitting.
\end{lem}
\begin{defn}
We denote by $\Di(A\lModbs)$ the category obtained by quotienting out the morphism spaces of $A\lModbs$ by the homotopies.  We define $\Di(\rModbs A)$ similarly.
\end{defn}
\begin{rem}
In the introduction we claimed that $\Di(A\lModbs)$ was going to be defined as $\Ho^0(A\lModibs)$, where $A\lModibs$ is an $A_{\infty}$-category that we are yet to define.  We will see in due course that this definition agrees with the one above.
\end{rem}
\medbreak
The definitions for bimodules are similar to those of left or right modules (see section 2.5 of \cite{KLH}, where both of the approaches we consider in this chapter are spelt out, or \cite{TradBimodules}).  We recall them here since they play such an important part in the sequel.  Let $\mathcal{C}$ and $\mathcal{E}$ be symmetric monoidal categories.  In what follows, $A$ will be an $A_{\infty}$-algebra object in $\mathcal{C}$, and $B$ will be an $A_{\infty}$-algebra object of $\mathcal{E}$.  $A_{\infty}$-bimodule objects will live in a graded category $\mathcal{D}$, where we have two functors
\[
\otimes :\mathcal{C}\times \mathcal{D}\rightarrow \mathcal{D}
\]
and
\[
\otimes:\mathcal{D}\times\mathcal{E}\rightarrow \mathcal{D}
\]
satisfying the obvious associativity and unit axioms (we end up with four copies of the pentagon axiom, corresponding to the four categories $\mathcal{D}\times \mathcal{E}\times \mathcal{E}\times \mathcal{E}$, $\mathcal{C}\times \mathcal{D}\times \mathcal{E}\times \mathcal{E}$, $\mathcal{C}\times\mathcal{C}\times\mathcal{D}\times \mathcal{E}$ and $\mathcal{C}\times \mathcal{C}\times \mathcal{C}\times \mathcal{D}$).
\begin{defn}
\label{bimoddef}
If $A$ and $B$ are two $A_{\infty}$-algebra objects of $\mathcal{C}$ and $\mathcal{E}$ respectively, an $(A,B)$ \textit{bimodule} object $(M,(m_{M,i,j}))$ is given by a $M\in \mathcal{D}$ and a series of maps
\[
m_{M,i,j}:A^{\otimes i}\otimes M\otimes B^{\otimes j}\rightarrow M
\]
of degree $(1-i-j)$, defined for all $i,j\geq 0$, satisfying the following compatibility condition on $A^{\otimes i}\otimes M\otimes B^{\otimes j}$, for each $i,j\geq 0$,
\[
\sum_{a+b+c=i+j+1}(-1)^{a+bc}m_{a+1+c}(\id^{\otimes a}\otimes m_{b}\otimes \id^{\otimes c})=0,
\]
where here we adopt the convention mentioned in Remark \ref{lazyconvention}.  A morphism $f$ of $(A,B)$-bimodules is given by a series of maps
\[
f_{i,j}:A^{\otimes i}\otimes M\otimes B^{\otimes j}\rightarrow N
\]
of degree $(-i-j)$, satisfying the compatibility relation on each $A^{\otimes i}\otimes M\otimes B^{\otimes j}$, for $i,j\geq 0$
\begin{equation}
\sum_{\substack{a+b+c=i+1+j\\b\geq 1}}\limits(-1)^{ a+bc}f_{a+1+c}(\id^{\otimes a}\otimes m_b\otimes \id^{\otimes c})=\sum_{\substack{a+b+c=i+1+j\\a\leq i,c\leq j,b\geq 1}}\limits(-1)^{bc}m_{a+1+c}(\id^{\otimes a}\otimes f_{b} \otimes \id^{\otimes c}),
\end{equation}
where here we extend the abbreviation of Remark \ref{lazyconvention}, and write simply $f_e$ for any of the maps $f_{t,s}$, where $t+s=e-1$.
\end{defn}
As usual, if $\mathcal{C}=\mathcal{E}=(\vect)_{\mathbb{Z}}$, then we set $\mathcal{D}=(\vect)_{\mathbb{Z}}$ and recover the usual notion of an $A_{\infty}$-bimodule.
\begin{defn}
Let $f:(M,(m_{M,i,j}))\rightarrow (N,(m_{N,i,j}))$, $g:(N,(m_{N,i,j}))\rightarrow (K,(m_{K,i,j}))$ be morphisms of $(A,B)$-bimodules.  then we define their composition $g\circ f$ by setting
\[
(g\circ f)_{i,j}=\sum_{\substack{0\leq a\leq i\\0\leq b\leq j}}\limits(-1)^{(b+a)(j-b)}(g_{i-a,j-b}(\id^{\otimes (i-a)}\otimes f_{a,b}\otimes\id^{\otimes (j-b)})).
\]
\end{defn}
\bigbreak

\begin{defn}
If $(A,S,l)$ is an $A_{\infty}$-algebra over $S$, and $(B,T,l')$ is an $A_{\infty}$-algebra over $T$, then an $(S,T)$-unital bimodule $(M,(m_{i,j}))$ over $(A,B)$ is an $A_{\infty}$-bimodule, as defined above, satisfying the condition that for all $i,j\geq 0$ such that $i+j\geq 2$, and all $t<i$, $v<j$ 
\[
m_{M,i,j}(\id^{\otimes t}\otimes l\otimes \id^{\otimes (i+j-t)})=0=m_{M,i,j}(\id^{\otimes (i+j-v)}\otimes l' \otimes \id^{\otimes v}).
\]
If $((A,g),(S,h),l)$ is a unital $A_{\infty}$-algebra over $(S,h)$ and $((B,g'),(T,h'),l')$ is a unital $A_{\infty}$-algebra over $(T,h')$, then for $M$ to be a $(S,T)$-unital bimodule over $A$ we ask also that the following diagram commutes
\[
\xymatrix{
M\otimes k\ar[d]_{\id\otimes g'} & M\ar[l]\ar[d]^{\id}\ar[r] & k\otimes M\ar[d]^{g\otimes \id}\\
M\otimes A\ar[r]^<<<<<{m_{M,0,1}} & M & A\otimes M.\ar[l]_{m_{M,1,0}}
}
\]
\end{defn}
We leave the definition of a morphism of $(S,T)$-unital bimodules to the reader.  As is the case throughout our treatment of $A_{\infty}$ algebra, the notion of unitality above corresponds to \textit{strict} unitality as commonly defined (e.g. \cite{KLH}, \cite{keller-intro}).  Let $(A,S,l)$ be an $S$-unital $A_{\infty}$-algebra, and let $(B,T,g)$ be a $T$-unital $A_{\infty}$-algebra, for $S$ and $T$ ordinary algebras.  Let $\mathcal{C}$ be the category of ordinary $S$-bimodules, and let $\mathcal{E}$ be the category of ordinary $T$-bimodules.  Then, following Proposition \ref{yeti} we may define the $A_{\infty}$-algebra objects $A_S$ and $B_T$ in the categories $\mathcal{C}$ and $\mathcal{E}$ respectively.  An $S$-unital and $T$-unital $(A,B)$-bimodule $M$ has the structure of an ordinary $(S,T)$-bimodule.  The following is proved in the same way as Proposition \ref{yeti}
\begin{prop}
There is an equivalence of categories between $(S,T)$-unital $(A,B)$-bimodules, and unital $(A_S,B_T)$-bimodule objects.
\end{prop}
If $f,g$ are two morphisms $M\rightarrow N$ for $M,N\in A\biModbs B$, then a homotopy between them is given by a series of morphisms
\[
h_{i,j}:A^{\otimes i}\otimes M\otimes B^{\otimes j}\rightarrow M
\]
of degree $-1-i-j$ such that the following equality holds on morphisms $A^{\otimes i}\otimes M\otimes B^{\otimes j}\rightarrow M$:
\begin{eqnarray*}
&f-g=\sum_{\substack{a+b+c=i+1+j\\b\geq 1}}\limits(-1)^{a+bc}h_{a+1+c}(\id^{\otimes a}\otimes m_b\otimes \id^{\otimes c})-\\&-\sum_{\substack{a+b+c=i+1+j\\a\leq i,c\leq j,b\geq 1}}\limits(-1)^{a+bc}m_{a+1+c}(\id^{\otimes a}\otimes h_{b} \otimes \id^{\otimes c}).
\end{eqnarray*}
Just as in the case of left or right $A$-modules, we define the derived category of $(A,B)$-bimodules $\Di(A\biModbs B)$ by identifying homotopic morphisms.
\smallbreak
This ends our roll-call of definitions regarding the \textit{ordinary} category of $A_{\infty}$-algebras, modules and bimodules.  We will need some basic facts about enriched categories in the main text, and in particular the enriched category $\Perf(A\lModi)$, and the full subcategory of this $A_{\infty}$-category obtained by taking shifts and cones of objects in the image of the Yoneda embedding, which is discussed in Section \ref{twistcomp}.
\section{$A_{\infty}$-categories}
\begin{defn}
A (small) $k$-linear $A_{\infty}$-category $\mathcal{C}$ is given by a set of objects $\ob(\mathcal{C})$, a $\mathbb{Z}$-graded $k$-vector space $\Hom_{\mathcal{C}}(x,y)$ for each $x,y\in \ob(\mathcal{C})$, and, for each $(n+1)$-tuple of objects $(x_0,\ldots,x_n)\in \ob(\mathcal{C})^{n+1}$, for $n\geq 1$, a $k$-linear composition
\[
m_{\mathcal{C},n}:\Hom_{\mathcal{C}}(x_{n-1},x_n)\otimes\ldots\otimes \Hom_{\mathcal{C}}(x_0,x_1)\rightarrow \Hom_{\mathcal{C}}(x_0,x_n)
\]
of degree $(2-n)$, satisfying compatibility relations as in (\ref{struc-compatibility}).  A \textit{unital structure} on $\mathcal{C}$ is given by a graded morphism
\[
\xymatrix{
k\ar[r]^<<<<<{u_x} &\Hom_{\mathcal{C}}(x,x)
}
\]
for each $x\in \ob(\mathcal{C})$ (where $k$ is placed in degree zero), such that for all $x,y \in \ob(\mathcal{C})$ the following diagrams commute
\[
\xymatrix{
\Hom_{\mathcal{C}}(x,y)\otimes k\ar[d]^{\id\otimes u_x} & \Hom_{\mathcal{C}}(x,y)\ar[d]^{\id}\ar[l] \\
\Hom_{\mathcal{C}}(x,y)\otimes \Hom_{\mathcal{C}}(x,x) \ar[r]^<<<<<{m_{\mathcal{C},2}} & \Hom_{\mathcal{C}}(x,y)
}
\]
\[
\xymatrix{
k\otimes \Hom_{\mathcal{C}}(y,x)\ar[d]^{u_x\otimes \id} & \Hom_{\mathcal{C}}(y,x)\ar[d]^{\id}\ar[l] \\
\Hom_{\mathcal{C}}(x,x)\otimes \Hom_{\mathcal{C}}(y,x) \ar[r]^<<<<<{m_{\mathcal{C},2}} & \Hom_{\mathcal{C}}(y,x)
}
\]
and such that, for all $n\geq 3$, $n-1 \geq i \geq 0$, $m_{\mathcal{C},n}(\id^{\otimes i}\otimes u_x \otimes \id^{\otimes (n-i-1)})=0$.
\end{defn}
\begin{rem}
Strictly speaking, some of the $A_{\infty}$-categories that will follow, such as categories of modules over an $A_{\infty}$-algebra, will not be small.  We extend the notion by the usual move of considering Grothendieck universes, and ignore it from now on.  Note that the categories that we are \textit{most} interested in, for example categories of modules with finite-dimensional total homology over a fixed $A_{\infty}$-algebra, and categories of perfect modules over a fixed $A_{\infty}$-algebra \textit{will} be quasi-essentially small.
\end{rem}
\begin{defn}
Let $\mathcal{C}$ be an $A_{\infty}$-category.  If $m_{\mathcal{C},1}=0$ then we say that $\mathcal{C}$ is \textit{minimal}.
\end{defn}

\begin{examp}
An ordinary category is a special case of an $A_{\infty}$-category satisfying the property that $m_{\mathcal{C},i}=0$ for all $i\neq 2$.  A differential graded category is a special case satisfying the property that $m_{\mathcal{C},i}=0$ for all $i\geq 3$.
\end{examp}
 
\begin{examp}
An $A_{\infty}$-algebra can be thought of as an $A_{\infty}$-category with one object.  In this case, the two notions of unitality coincide.
\end{examp}
\begin{rem}
The definition of unitality that we have given is commonly referred to as \textit{strict} unitality.  Ordinarily one would only ask for a unit at the level of homology, since this is a notion that is stable under homotopy equivalences of $A_{\infty}$-categories.
\end{rem}
\begin{defn}
Let $\mathcal{C}$ be an $A_{\infty}$-category.  Then we define a new $A_{\infty}$-category $\mathcal{C}^{\op}$ as follows:
\begin{enumerate}
\item
The objects of $\mathcal{C}^{\op}$ are just the objects of $\mathcal{C}$.
\item
$\Hom_{\mathcal{C}^{\op}}(x,y):=\Hom_{\mathcal{C}}(y,x)$.
\item
$m_{\mathcal{C}^{\op},i}=m_{\mathcal{C}}\circ \mathsf{ref}_i$ where $\mathsf{ref}_i$ is any combination of applications of the morphism expressing the (untwisted) symmetric monoidal structure of $\dgvect$ that gives a map
\[
\Hom_{\mathcal{C}}(x_{i-1},x_i)\otimes\ldots\otimes \Hom_{\mathcal{C}}(x_0,x_1)\rightarrow \Hom_{\mathcal{C}}(x_0,x_1)\otimes\ldots\otimes \Hom_{\mathcal{C}}(x_{i-1},x_i).
\]
\end{enumerate}
It is  straightforward to check that this gives a new $A_{\infty}$-category $\mathcal{C}^{\op}$.
\end{defn}
\begin{defn}
Let $\mathcal{C},\mathcal{D}$ be two (small) $A_{\infty}$-categories.  Then an $A_{\infty}$-functor $F$ between them is given by a map of sets (again denoted $F$) $F:\ob(\mathcal{C})\rightarrow \ob(\mathcal{D})$ and, for all $(n+1)$-tuples of objects $(x_0,\ldots,x_n)\in \ob(\mathcal{C})^{n+1}$, a $k$-linear map
\[
F_n:\Hom_{\mathcal{C}}(x_{n-1},x_n)\otimes\ldots\otimes \Hom_{\mathcal{C}}(x_0,x_1)\rightarrow \Hom_{\mathcal{D}}(\phi_F(x_0),\phi_F(x_n))
\]
of degree $(1-n)$, satisfying compatibility relations analogous to (\ref{morph-compatibility}).  A morphism of \textit{unital} $A_{\infty}$-categories is a morphism, as above, satisfying the added condition that the following diagram commutes for all $x\in \ob(\mathcal{C})$:
\begin{equation*}
\xymatrix{
k\ar[dr]_-{u_{F(x)}}\ar[r]^(.35){u_x}&\Hom_{\mathcal{C}}(x,x)\ar[d]^{F_1} \\
&\Hom_{\mathcal{D}}(F(x),F(x)),
}
\end{equation*}
and also the condition that for all $n\geq 2$, and all $n-1\geq i \geq 0$, and all $x\in \ob(\mathcal{C})$, we have $F_{n}(\id^{\otimes i}\otimes u_x\otimes \id^{\otimes (n-i-1)})=0$.
\end{defn}
\begin{defn}
\label{cataug}
Given a unital $A_{\infty}$-category $\mathcal{C}$, we define $\mathcal{C}_{\id}$ to be the category with the same objects as $\mathcal{C}$, and with morphisms given by scalar multiples of the identity morphisms.  There is a strict functor $i:\mathcal{C}_{\id}\rightarrow \mathcal{C}$, and we define an augmentation of $\mathcal{C}$ to be another strict functor 
\[
\nu:\mathcal{C}\rightarrow \mathcal{C}_{\id}
\]
such that $\nu\circ i=\id_{\mathcal{C}_{\id}}$.
\end{defn}In analogy with the case of $A_{\infty}$-algebras, if $\mathcal{C}$ is an $A_{\infty}$-category then there is an induced structure of an ordinary $\mathbb{Z}$-graded category with set of objects the same as those of $\mathcal{C}$, and with morphism spaces between $x$ and $y$ replaced by $\Ho^{\bullet}(\Hom_{\mathcal{C}}(x,y))$.  This graded category is denoted $\Ho^{\bullet}(\mathcal{C})$.  To be precise, we have only defined the structure of a graded category if the original category $\mathcal{C}$ is a unital $A_{\infty}$-category, otherwise it just has the structure of a nonunital graded category.  Similarly we define the ordinary category $\Ho^0(\mathcal{C})$, and the category $\Zo^0(\mathcal{C})$, which has the same objects as $\mathcal{C}$, and 
\[
\Hom_{\Zo^0(\mathcal{C})}(x,y):=\Ker(d:\Hom_{\mathcal{C}}^0(x,y)\rightarrow \Hom^1_{\mathcal{C}}(x,y)).
\]
We say that two objects in $\mathcal{C}$ are quasi-isomorphic if they are isomorphic in $\Ho^0(\mathcal{C})$.  A functor $F:\mathcal{C}\rightarrow\mathcal{D}$ between $A_{\infty}$-categories induces an ordinary $\mathbb{Z}$-graded functor $\Ho^{\bullet}(F)$ between the (possibly nonunital) ordinary $\mathbb{Z}$-graded categories $\Ho^{\bullet}(\mathcal{C})$ and $\Ho^{\bullet}(\mathcal{D})$.  We say that $F$ is a quasi-isomorphism if $\Ho^{\bullet}(F)$ is an isomorphism.\smallbreak
\begin{defn}
Let $\mathcal{C},\mathcal{D}$ be $A_{\infty}$-categories.  Then a functor $F:\mathcal{C}\rightarrow \mathcal{D}$ is called an \textit{equivalence} if it induces a bijection $\ob(\mathcal{C})\rightarrow \ob(\mathcal{D})$, and induces isomorphisms
\[
F_1:\Hom_{\mathcal{C}}(x,y)\rightarrow \Hom_{\mathcal{D}}(F(x),F(y)).
\]
It is a quasi-equivalence if it induces a bijection between quasi-isomorphism classes of objects in $\mathcal{C}$ and quasi-isomorphism classes of objects in $\mathcal{D}$, and for each $x,y\in \ob(\mathcal{C})$ the morphism
\begin{equation}
\label{F1}
F_{1}:\Hom_{\mathcal{C}}(x,y)\rightarrow \Hom_{\mathcal{D}}(F(x),F(y))
\end{equation}
is a quasi-isomorphism.  $F$ is called \textit{quasi-full} if the morphism (\ref{F1}) induces a surjection on homology, for every $x,y\in \ob(\mathcal{C})$, and is called \textit{quasi-faithful} if it induces an injection on homology.  Functors satisfying both of these conditions will be called \textit{quasi-fully faithful}.  The quasi-essential image of a functor $F$ is the set of objects in $\ob(\mathcal{D})$ that are quasi-isomorphic to objects $F(x)$, for $x\in \ob(\mathcal{C})$.
\end{defn}
\begin{defn}
Let $\mathcal{D}\subset \mathcal{C}$ be an inclusion of $A_{\infty}$-categories.  We say that $\mathcal{D}$ is a quasi-strictly full subcategory if, for all $x,y\in \ob(\mathcal{D})$, $\Hom_{\mathcal{D}}(x,y)\rightarrow \Hom_{\mathcal{C}}(x,y)$ is a quasi-isomorphism, and $\mathcal{D}$ is closed under quasi-isomorphisms in $\mathcal{C}$.
\end{defn}
We next recall the definition of a natural transformation between $A_{\infty}$-functors.  This is the definition given in \cite{Seidel08}, though with different signs.
\begin{defn}
\label{nattrans}
Let $F$, $G$ be $A_{\infty}$-functors between $A_{\infty}$-categories $\mathcal{C}$ and $\mathcal{D}$.  Then a natural transformation $\nu:F\rightarrow G$ is given by
\begin{enumerate}
\item
For each $x\in \ob(\mathcal{C})$ a morphism $\nu_{0,x}:F(x)\rightarrow G(x)$.
\item
For each $(n+1)$-tuple of objects in $\mathcal{C}$ $(x_0,\ldots,x_n)\in \ob(\mathcal{C})^{n+1}$, for $n\geq 0$, a $k$-linear map
\[
\nu_{n}:\Hom_{\mathcal{C}}(x_{n-1},x_n)\otimes\ldots\otimes \Hom_{\mathcal{C}}(x_0,x_1)\rightarrow \Hom_{\mathcal{D}}(F(x_0),G(x_n))
\]
of degree $-n$.
\item
This data is required to satisfy the compatibility conditions
\begin{align*}
\sum_{r\leq n,i_1+\ldots+i_r=n,j\leq r}(-1)^{\blacklozenge}m_{\mathcal{D},r}(G_{i_1}\otimes\ldots\otimes G_{i_{j-1}}\otimes \nu_{i_j}\otimes F_{i_{j+1}}\otimes\ldots\otimes F_{i_r})=\\=\sum_{a+b+c=n}\nu_{a+1+c=n}(-1)^{a+bc}(\id^{\otimes a}\otimes m_{\mathcal{C},b}\otimes \id^{\otimes c})
\end{align*}
where $\blacklozenge=[(r-1)(i_1-1)+(r-r+1)(i_{r-1}-1)]+(r-j)+[i_1+\ldots+i_{j-1}]$.
\end{enumerate}
A natural transformation $\nu$ is called a natural quasi-equivalence if $\nu_{0,x}$ is a quasi-isomorphism for every $x\in \mathcal{C}$.  Let $F$ and $G$ now be \textit{unital} functors.  Then a natural transformation between them is required to satisfy the condition that 
\[
\nu_{i+j+1}(\id^{\otimes i} \otimes u_x\otimes \id^{\otimes j})=0
\]
for all $x\in \ob(\mathcal{C})$, and for all $i+j\geq 1$.
\end{defn}
\begin{rem}
\label{nadir}
The definition of the sign $\blacklozenge$ marks a nadir as far as sign conventions go.  We direct the exasperated reader to Section \ref{enrichments}, where a different approach is explained that does away with these signs -- see Remark \ref{annoyingsigns}.  This is the approach of \cite{KSnotes} and \cite{KS}, and one of the approaches of \cite{KLH}.
\end{rem}
\begin{rem}
\label{sketchyhpt}
There is a natural extension of the notion of a homotopy between morphisms of $A_{\infty}$-algebras to the notion of a homotopy between $A_{\infty}$-functors.  There is also a version of the homological perturbation lemma for $A_{\infty}$-categories, which states that if $F:\mathcal{C}\rightarrow \mathcal{D}$ is a quasi-equivalence between two unital $A_{\infty}$-categories, there is a functor $G$ and a natural quasi-equivalence $\nu:GF\rightarrow L$, such that $L$ is homotopic to the identity functor.
\end{rem}
\smallbreak
There is a standard notion of a \textit{module} over an $A_{\infty}$-category.
\begin{defn}
\label{catmods}
A left module over an $A_{\infty}$-category $\mathcal{C}$ is an $A_{\infty}$-functor $F$ from $\mathcal{C}$ to $\dgvect$.  If $\mathcal{C}$ is a unital $A_{\infty}$-category, we say the module is unital if $F$ is a unital functor.  A morphism of $\mathcal{C}$-modules is a natural transformation of functors.  If $F$ and $G$ are unital modules over $\mathcal{C}$, a morphism between them is a unital natural transformation between them.
\end{defn}
\begin{rem}
We define \textit{right} modules in the natural way, as $A_{\infty}$-functors from $\mathcal{C}^{\op}$ to $\dgvect$.
\end{rem}
\begin{examp}
If $A$ is an $A_{\infty}$-algebra, considered as an $A_{\infty}$-category with one object, then this is just the usual definition of an $A_{\infty}$-module over $A$.  The definition of morphism can also be checked to be the usual definition.
\end{examp}
\begin{rem}
\label{shiftdef}
We define direct sums and products of modules in the natural way, by taking pointwise direct sums and products, and direct sums and products of multiplication morphisms.  For $M\in\mathcal{C}\lMod$ we define the \textit{shift} of $M$, denoted $M[1]$, as the module taking $x\in\mathcal{C}$ to $M(x)[1]$.  The multiplication morphisms are given by
\[
m_{M[1],n}=-m_{M,n}.
\]
\end{rem}

\section{Bifunctors, and bimodules over $A_{\infty}$-categories}
We give a slightly odd looking definition for a bifunctor between $A_{\infty}$-categories.  It is motivated by the fact that an $A_{\infty}$-bifunctor to the category $\dgvect$ should be thought of as an $A_{\infty}$-bimodule for a pair of `$A_{\infty}$-algebras with many objects'.  Compare in this regard section 5.3 of \cite{KLH}, where the promotion to the $A_{\infty}$ world of the link between $(A,B)$-bimodule structures on a graded vector space $M$ and homomorphisms from $B$ to the endomorphism ring of $M$ as a left $A$-module is discussed and made precise.
\begin{defn}
Let $\mathcal{C},\mathcal{D},\mathcal{E}$ be $A_{\infty}$-categories.  A bifunctor $F$ from the ordered pair $(\mathcal{C},\mathcal{D}^{\op})$ to $\mathcal{E}$ is given by the following data
\begin{enumerate}
\item
For each pair $(x,y)\in(\mathcal{C},\mathcal{D})$ an object $F(y,x)\in\mathcal{E}$.
\item
For each $(n+1)$-tuple of objects $(x_0,\ldots,x_n)\in \ob(\mathcal{C})^{n+1}$ and each $(m+1)$-tuple of objects $(y_0,\ldots,y_{m})\in \ob(\mathcal{D})^{m+1}$ a map
\begin{eqnarray*}
m_{F,n,m}:\Hom_{\mathcal{C}}(x_{n-1},x_n)\otimes\ldots\otimes \Hom_{\mathcal{C}}(x_0,x_1)\otimes \Hom_{\mathcal{D}}(y_{m-1},y_m)\otimes\ldots\\
\ldots\otimes \Hom_{\mathcal{D}}(y_0,y_1)\rightarrow \Hom_{\mathcal{E}}(F(y_m,x_0),F(y_0,x_n))
\end{eqnarray*}
of degree $(1-m-n)$.
\item
These maps are to satisfy the compatibility relations, for each $i,j\geq 0$
\begin{eqnarray*}
&\sum_{\substack{1\leq a\leq i,\\b\leq i-a}}\limits(-1)^{(i+j-b-a)a+b}m_{F,i-a+1,j}(\id^{\otimes b}\otimes m_{\mathcal{C},a}\otimes \id^{\otimes (i+j-b-a)})+\\
&+\sum_{\substack{1\leq a\leq j,\\b\leq j-a}}\limits(-1)^{ba+i+j-b-a}m_{F,i,j-a+1}(\id^{\otimes(i+j-b-a)}\otimes m_{\mathcal{D},a}\otimes \id^{\otimes b})=\\
&= \sum_{\substack{r\leq i+j,\\a_1+\ldots+a_r=i,\\b_1+\ldots+b_r=j}}\limits(-1)^{\blacktriangledown}m_{\mathcal{E},r}(m_{F,a_r,b_r}, m_{F,a_{r-1},b_{r-1}},\ldots, m_{F,a_1,b_1})\sw_{a_1,\ldots,a_r,b_1,\ldots,b_r}
\end{eqnarray*}
where the final sum is over all $a_1,\ldots,a_r,b_1,\ldots,b_r$ such that $a_1+\ldots+a_r=i,b_1+\ldots+b_r=j$ and $a_t+b_t\geq 1$ for all $t\leq r$, $\sw$ is the isomorphism
\[
\begin{array}{c}
V_1\otimes\ldots\otimes V_i\otimes W_1\otimes\ldots\otimes W_j\rightarrow \\V_1\otimes\ldots\otimes V_{a_1}\otimes W_{b_1+\ldots+b_{r-1}+1}\otimes\ldots\otimes W_{j}\otimes\ldots\otimes V_{a_1+\ldots+a_{r-1}+1}\otimes\ldots\otimes V_i\otimes W_1\otimes\ldots\otimes W_{b_1}
\end{array}
\]

coming from the untwisted monoidal structure on $(\vect)_{\mathbb{Z}}$, and 
\[
\blacktriangledown=(a_r+b_r-1)(r-1)+(a_{r-1}+b_{r-1}-1)(r-1)+\ldots+(a_2+b_2-1).
\]
\end{enumerate}
We say the bifunctor $F$ is unital if the following diagram commutes
\[
\xymatrix{
\Hom_{\mathcal{D}}(y,y)\ar[dr]_{m_{F,0,1,x,y}} & k \ar[r]^(.41){u_x} \ar[l]_(.37){u_y} \ar[d]^{u_{F(y,x)}} & \Hom_{\mathcal{C}}(x,x)\ar[dl]^(.45){\hbox{ }\hbox{ }m_{F,1,0,x,y}}\\
& \Hom_{\mathcal{E}}(F(y,x),F(y,x)),
}
\]
and for all $n,m\geq 0$ such that $n+m\geq 2$ and all $0\leq i\leq n+m-1$, $i\neq n$, and all $x\in \mathcal{C}$, $y\in\mathcal{D}$ and $z\in \mathcal{C},\mathcal{D}$,
\[
m_{F,n,m,x,y}(\id^{\otimes i}\otimes u_z \otimes \id^{\otimes (n+m-i-1)})=0.
\]

\end{defn}
\begin{defn}
Let $\mathcal{C},\mathcal{D}$ be $A_{\infty}$-categories.  Then a $(\mathcal{C},\mathcal{D})$-bimodule is a bifunctor from $(\mathcal{C},\mathcal{D}^{\op})$ to $\dgvect$.  A bimodule is \textit{unital} if it is unital as a bifunctor.
\end{defn}
\begin{rem}
In the case of bimodules, since we are considering bifunctors to the $A_{\infty}$-category $\dgvect$, which has vanishing higher compositions, the bifunctor compatibility relations reduce to much simpler ones, which are essentially just the conditions of Definition \ref{bimoddef} for an algebra with many objects.  We do not give here the complicated definition of a morphism of bifunctors, since we do not need it.  A morphism of bimodules is given by the natural extension of the definition given in Definition \ref{bimoddef} to the $A_{\infty}$-category case.
\end{rem}
\begin{examp}
\label{diagcat}
Let $\mathcal{C}$ be an $A_{\infty}$-category.  Then
\[
\Hom_{\mathcal{C}}(-,-)(X,Y):=\Hom_{\mathcal{C}}(Y,X)
\]
is a bimodule for $\mathcal{C}$.  The bimodule multiplications are just the category compositions.
\end{examp}
\begin{examp}
Let $M$ be a $(\mathcal{C},\mathcal{D})$-bimodule.  Let $x\in\mathcal{D}$.  Then there is a forgetful functor
\[
-\cdot \id_x:\mathcal{C}\biMod \mathcal{D}\rightarrow \mathcal{C}\lMod
\]
satisfying $(M \cdot \id_x)(z):=M(x,z)$.  The module multiplications are given by
\[
m_{M\cdot \id_x,i}=m_{M,i,0}.
\]
\end{examp}
\section{$A_{\infty}$-enrichments of categories}
\label{enrichments}
Let $\mathcal{C}$ be a $\mathbb{Z}$-graded category.  Let $\mathcal{D}$ be an $A_{\infty}$-category satisfying $\Ho^{\bullet}(\mathcal{D})\cong \mathcal{C}$, then we say that $\mathcal{D}$ is an $A_{\infty}$\textit{-enrichment} of $\mathcal{C}$.  The classic case of an enrichment occurs when $\mathcal{A}$ is an Abelian category with enough projectives or enough injectives and $\mathcal{C}=\mathcal{A}_{\ext}$, the category with $\ob(\mathcal{A}_{\ext})=\ob(\mathcal{A})$, and $\Hom_{\mathcal{A}_{\ext}}(x,y)=\oplus_{n\in\mathbb{Z}}(\Ext_{\mathcal{A}}^n(x,y))$.
\smallbreak
There are two common types of enrichment of this category: firstly, there are differential graded categories, in which we replace the morphism spaces $\Hom_{\mathcal{C}}(M,N)$ with differential graded vector spaces $\RHom^{\bullet}(M,N)$ calculating $\Ext$ groups.  Secondly, there are minimal models of such categories (recall that a differential graded category is just a special kind of $A_{\infty}$-category).
\begin{rem}
There is a possibility of confusion regarding the phrase `enrichment of $\mathcal{A}$', where $\mathcal{A}$ is an Abelian category.  The phrase could, on the one hand, mean an enrichment, in the sense above, of the graded category $(\mathcal{A})_{\ext}$, obtained from $\mathcal{A}$ by replacing the vector spaces $\Hom_{\mathcal{A}}(x,y)$, for $x,y\in \ob(\mathcal{A})$, with the graded vector spaces $\Ext_{\mathcal{A}}^{\bullet}(x,y)$.  On the other hand the phrase could mean an $A_{\infty}$-category $\mathcal{C}$ satisfying the condition that the category $\Ho^{\bullet}(\mathcal{C})$ is triangulated, and $\mathcal{A}$ occurs as a heart.  Note that this phenomenon has already arisen in Section \ref{basicexample}, where the Yoneda algebra of the augmentation module $M$ of $\Cp[x]/\langle x^3\rangle$ was said to have finite-dimensional total homology, despite the fact that the Yoneda algebra of $M$, considered as an object of the category of $\Cp[x]/\langle x^3 \rangle$-modules, is infinite-dimensional.  This is because it was being calculated in an $A_{\infty}$-category that was not an enrichment of the graded category $(\mathcal{A})_{\ext}$, for $\mathcal{A}$ the Abelian category of $\Cp[x]/\langle x^3\rangle$-modules, but rather it was an $A_{\infty}$-category $\mathcal{D}$ such that $\Ho^{0}(\mathcal{D})$ happens to be a triangulated category with $\mathcal{A}$ as a heart, without itself being equivalent to $\mathrm{D}^b(\Cp[x]/\langle x^3\rangle)$, the bounded derived category of $\Cp[x]/\langle x^3 \rangle$-modules.
\end{rem}
\bigbreak
We next consider an enrichment of $\mathcal{C}$ for the more general case in which $A$ is an $A_{\infty}$-algebra and we put 
\[
\mathcal{C}=(A\lMod)_{\ext},
\]
where $(A\lMod)_{\ext}$ is the graded category with objects the left modules over $A$, and with
\[
\Hom_{(A\lMod)_{\ext}}^n(M,N):=\Hom_{A\lMod}(M,N[n])/\sim_{\mathrm{homotopy}}.
\]
\medbreak
Let $M$ be a module over the $A_{\infty}$-algebra $A$.  Firstly, there is a differential graded coalgebra, denoted $\TC(A)$, associated to $A$.  Its underlying coalgebra is $\Free_{\coalg}(\forget(A)[1])$, the free non-counital coalgebra generated by the underlying vector space of $A$, shifted by 1.  For the time being we'll use $\forget(A)$ to make it clear that we are forgetting the $A_{\infty}$-structure of $A$ at this stage in the construction.  Explicitly, this has the underlying vector space
\[
\Free_{\coalg}(\forget(A)[1])=\bigoplus_{i\geq 1} \forget(A)^{\otimes i}[i],
\]
and the comultiplication is given by
\[
\Delta(a_1,\ldots,a_n)=\sum_{1\leq i\leq n-1}(a_1,\ldots,a_i)\otimes (a_{i+1},\ldots,a_n).
\]
Now any coderivation $b_A:\TC(A)\rightarrow \TC(A)$ is determined by its postcomposition with the natural projection 
\[
p:\Free_{\coalg}(\forget(A)[1])\rightarrow \forget(A[1]),
\]
i.e. $b_A$ is determined by the composition
\[
\xymatrix{
\TC(A)\ar[r]^-{b_A}& \TC(A)\ar[r]^-p &A[1].
}
\]
We denote by $b_{A,n}$ the restriction of $p\circ b_A$ to $\forget(A)^{\otimes i}[i]$.  We assume that $b_A$ is a degree 1 map.  There is a unique sequence of maps $m_{A,n}$ such that the following diagram commutes, for all $n$
\[
\xymatrix{
(A[1])^{\otimes n}\ar[r]^(.6){b_{A,n}} & A[1]\\
A^{\otimes n}\ar[u]^{S^{\otimes n}}\ar[r]^(.54){m_{A,n}} & A\ar[u]^{S},
}
\]
where $S:A\rightarrow A[1]$ is the degree $-1$ map sending $a\in A$ to $a\in A[1]$.  It is a standard fact (e.g. see \cite{Sta63}) that the condition $b_A^2=0$ is equivalent to the condition that the $m_{A,i}$ determine the structure of an $A_{\infty}$-algebra, i.e. $b_A^2=0$ is equivalent to the equations of (\ref{struc-compatibility}) holding.  Since $A$ is an $A_{\infty}$-algebra, it comes with operations $m_{A,i}$, and so we obtain a derivation $b_A$.  We denote by $(\TC(A),b_A)$ the differential graded coalgebra associated to $A$.  
\begin{rem}
\label{annoyingsigns}
Note that, converted into a statement regarding the $b_{A,n}$, the equations (\ref{struc-compatibility}) become
\begin{equation}
\sum_{a+b+c=n} b_{A,a+c+1}(\id^{\otimes a}\otimes b_{A,b}\otimes \id^{\otimes c})=0
\end{equation}
and we have lost the annoying signs.  One can systematically remove all of the signs from the treatment of $A_{\infty}$ algebra given so far, by recasting everything in terms of the shifts $A[1]$ for algebras, $M[1]$ for modules, and so on.  This is the answer to Remark \ref{nadir}.  Unless there is good reason not to do so, we will avoid the sign-heavy compositions $m_i$, and work with the compositions $b_i$, from now on.
\end{rem}
\bigbreak
Now let $M$ be a left $A$-module.  One can construct a left $(\TC(A),b_A)$-differential graded comodule associated to $M$ as follows: first one takes the cofree left $\TC(A)$-comodule generated by $\forget(M)[1]$
\[
\Free_{\TC(A)\lcomod}(\forget(M)[1])
\]
which has the underlying vector space $(k\oplus \Free_{\coalg}(\forget(A)[1]))\otimes \forget(M)[1]$.  The comodule structure is given by
\[
\Delta_M(a_1,\ldots,a_n,m)=\sum_{n\geq i\geq 1}(a_1,\ldots,a_i)\otimes(a_{i+1},\ldots,a_n,m)
\]
for $n\geq 1$ and $\Delta_M(m)=0$.
Let $b_M$ be the coderivation on $\Free_{\TC(A)\lcomod}(\forget(M)[1])$ defined by the commutativity of the following diagram
\begin{equation}
\label{ala}
\xymatrix{
A[1]^{\otimes (i-1)}\otimes M[1]\ar[r]^(.69){b_{M,i}}& M[1]\\
A^{\otimes (i-1)}\otimes M\ar[u]^{S^{\otimes i}}\ar[r]^(.67){m_{M,i}}& M\ar[u]^{S}
}
\end{equation}
where $b_{M,i}$ is the restriction to $(A[1])^{\otimes (i-1)}\otimes M[1]$ of the composition
\[
\xymatrix{
\Free_{\TC(A)\lcomod}(\forget(M)[1])\ar[r]^{b_M}& \Free_{\TC(A)\lcomod}(\forget(M)[1])\ar[r]^(.75)p & M[1].
}
\]
One checks again that the condition $b_M^2=0$ is equivalent to the equations (\ref{modcompatability}), i.e. $b_M^2=0$ if and only if the $m_{M,i}$ define the structure of a left $A$-module on $M$.  
\medbreak
In a similar vein we have the following construction:  let $M$ now be an $(A,B)$-bimodule, where $A$ and $B$ are $A_{\infty}$-algebras.  Then define the differential graded cobimodule \\$\Free_{\TC(A)\coBimod \TC(B)}(\forget(M[1]))$ by stipulating that its underlying graded vector space is
\[
\bigoplus_{i,j\geq 0} \forget(A)[1]^{\otimes i}\otimes \forget(M)[1] \otimes \forget(B)[1]^{\otimes j}.
\]
This has commuting left $\TC(A)$-coaction and right $\TC(B)$-coaction, and so it has the structure of a $(\TC(A),\TC(B))$-cobimodule.  In the usual way one defines a coderivation $b_M$ on this structure, and the usual result applies: the condition that the maps 
\[
m_{M,i,j}: A^{\otimes i}\otimes M\otimes B^{\otimes j}\rightarrow M,
\]
defined \`{a} la (\ref{ala}), define an $(A,B)$-bimodule structure on $M$, is equivalent to the condition that $b_M^2=0$.\smallbreak
In each case we have a correspondence between coderivations $b$ satisfying $b^2=0$, and $A_{\infty}$-structures.  This correspondence extends to an equivalence of categories.  We state it in the comodule/cobimodule case.
\begin{prop}\cite{KLH}
The above constructions underlie an equivalence of categories between
\begin{enumerate}
\item
$A$-modules, for an $A_{\infty}$-algebra $A$, and cofree differential graded comodules for $(\TC(A),b_A)$, and
\item
$(A,B)$-bimodules, for $A,B$ $A_{\infty}$-algebras, and left cofree and right cofree \\$((\TC(A),b_A),(\TC(B),b_B))$-cobimodules.
\end{enumerate}
In the category of cofree differential graded $\TC(A)$-comodules the morphisms are the degree zero morphisms commuting with the differential.  The case is similar for cobimodules.
\end{prop}
This gives a handle on the enriched structure for $A\lModinu$ and $A\biModinu B$.
\begin{defn}
We define the category $A\lModinu$ as follows:
\begin{enumerate}
\item
The objects of $A\lModinu$ are just the left $A$-modules.
\item
Let $M,N\in \ob(A\lModinu)$.  Then 
\[
\Hom_{A\lModinu}^{\bullet}(M,N):=\Hom^{\bullet}_{\TC(A)\lcomod}(\overline{M}^l,\overline{N}^l)
\]
where here 
\[\overline{M}^l:=\Free_{\TC(A)\lcomod}(\forget(M)[1]).
\]
This morphism space is graded, since it denotes \textit{ungraded} comodule morphisms.
\item
The composition of morphisms $m_{A\lModinu,2}$ is just the natural composition of comodule morphisms.  All higher compositions are zero.
\item
The differential $d=m_{A\lmodinu,1}$ is given on homogeneous elements $f$ by $d(f)=f\circ b_{\overline{M}^l}+(-)^{|f|}b_{\overline{N}^l}\circ f$, where $b_{\overline{M}^l}$ and $b_{\overline{N}^l}$ are determined by the $A$-module structures on $M$ and $N$.
\end{enumerate}
\end{defn}
One verifies straight from the definitions that 
\[
\Zo^0(A\lModinu)\cong A\lModnu.
\]
It is also straightforward to show that 
\[
\Ho^0(A\lModinu)\cong \Di(A\lModnu)
\]
and so we are justified in calling $A\lModinu$ an \textit{enrichment} of $\Di (A\lModnu)$ -- it is an enrichment of the category $(A\lModnu)_{\ext}$.  Note that this $A_{\infty}$-enrichment is merely a differential graded category.
\bigbreak
The definition of $A\biModinu B$ is totally analogous, and is left to the reader, who may, as ever, consult \cite{KLH} for the precise definitions.
\bigbreak
In this context it is easy to define a tensor product for $A_{\infty}$-modules.  Note first that there is a naturally defined functorial tensor product on cofree differential graded comodules: as a first example let $\overline{M}^r$ be a cofree right $\TC(A)$-comodule, and let $\overline{N}^l$ be a cofree left $\TC(A)$-comodule.  Then, dually to the situation for free differential graded algebras, we define the tensor product $\overline{M}^r\otimes_{\TC(A)} \overline{N}^l$ as the shifted (we'll explain this shift in due course) \textit{limit}
\begin{equation}
\label{tp}
\xymatrix{
\lim(\overline{M}^r\otimes\overline{N}^l\ar@<.5ex>[r]^(.39){\id\otimes \Delta_N}\ar@<-.5ex>[r]_(.39){\Delta_M\otimes\id} & \overline{M}^r\otimes \TC(A) \otimes \overline{N}^l)[-1],
}
\end{equation}
where $\Delta_-$ are the comultiplications.  This limit inherits a natural differential, given by the fact that both maps commute with the differentials.  Now let $M$ be a $(B,A)$-bimodule, where $B$ is another $A_{\infty}$-algebra, and replace $\overline{M}^l$ with 
\[
\overline{M}^{rl}:=\Free_{B\coBimod A}(\forget(M[1])).
\]
Then $\overline{M}^{rl}$ carries two commuting coactions
\begin{align*}
\Delta_M^r:\overline{M}^{rl}\rightarrow \overline{M}^{rl}\otimes \TC(A)\\
\Delta_M^l:\overline{M}^{rl}\rightarrow \TC(B)\otimes \overline{M}^{rl},
\end{align*}
and the two terms of the diagram
\begin{equation}
\label{tp2}
\xymatrix{
\overline{M}^{rl}\otimes\overline{N}^l\ar@<.5ex>[r]^>>>>{\id\otimes \Delta_N}\ar@<-.5ex>[r]_>>>>{\Delta^r_M\otimes\id} & \overline{M}^{rl}\otimes \TC(A) \otimes \overline{N}^l
}
\end{equation}
are cofree $\TC(B)$-comodules.  One may check directly that the limit is also a cofree left $\TC(B)$-comodule, and that the differential again commutes with the morphisms of (\ref{tp2}), so that under the correspondence between free differential graded $\TC(B)$-comodules and $B$-modules, we have defined a left $B$-module.  Similarly, if $N$ is a $(A,C)$-bimodule, then now $\overline{M}^{rl}\otimes_{\TC(A)}\overline{N}^{rl}$ inherits the structure of a $(\TC(B),\TC(C))$-cobimodule.  Furthermore, by definition of a comodule map (not just the more restrictive differential graded comodule map) a comodule morphism
\[
f:M\rightarrow M'
\]
induces a morphism
\[
\xymatrix{
f\otimes N:\lim( \overline{M}^r\otimes\overline{N}^l\ar@<.5ex>[r]^>>>>{\id\otimes \Delta_N}\ar@<-.5ex>[r]_>>>>{\Delta_M\otimes\id} & \overline{M}^r\otimes \TC(A) \otimes \overline{N}^l)\ar[r]& \lim( \overline{M'}^r\otimes\overline{N}^l\ar@<.5ex>[r]^>>>>{\id\otimes \Delta_N}\ar@<-.5ex>[r]_>>>>{\Delta_M'\otimes\id} & \overline{M'}^r\otimes \TC(A) \otimes \overline{N}^l).
}
\]
Putting all this together one obtains
\begin{defn}
\label{enrichedtensor}
Let $N$ be a $(A,B)$-bimodule, and let $C$ be an $A_{\infty}$-algebra.  Then there is a strict $A_{\infty}$-functor 
\begin{equation*}
\xymatrix{
-\otimes_A N :C\biModinu A\hbox{ }\ar[r]& \hbox{ }C\biModinu B
}
\end{equation*}
as defined above.
\end{defn}
\begin{rem}
\label{funnyshift}
The reader may be wondering why we have this shift in (\ref{tp}).  One answer is that it is this shift that makes $\overline{A}^{rl}$ a monoidal unit in the category of $\TC(A)$-cobimodules, and it is this answer that recreates derived tensor products for usual algebras, as computed by the bar resolution, as we are about to see.
\end{rem}
We need to convert this description of $A_{\infty}$-tensor products back into the world of $A_{\infty}$-modules as defined in Definition \ref{moddef}.  So, concretely, we have that the underlying graded vector space of $M\otimes_A N$ is given by
\[
M\otimes_A N:=\bigoplus_{i\geq 0} M\otimes A[1]^{i}\otimes N
\]
with a differential $d$.  Write $d_i$ for the restriction of $d$ to the summand $M\otimes A[1]^i\otimes N$.  Then it is given by
\[
d_i=\sum_{\substack{l+m+n=i+2\\m<i+2}}(S^{-1}\otimes \id^{\otimes(i-m+1)}\otimes S^{-1})\circ(\id^{\otimes l}\otimes b_{m} \otimes \id^{\otimes (n-1)})\circ(S\otimes \id^{\otimes i}\otimes S).
\]
\begin{rem}
This is just the bar resolution.  As mentioned in \cite{LOT}, in the case where $N$ has a right $B$-action, this vector space has a right $B$-action too.  This again can be read off more easily from the coalgebra/comodule viewpoint, where it is part of Definition \ref{enrichedtensor}.
\end{rem}
\begin{rem}
For future reference we write down the differential on the right $\TC(B)$-comodule $\overline{M\otimes_{A}N}^r$ too.  The underlying comodule is given by
\[
\overline{M\otimes_{A}N}^r=\bigoplus_{i,j\geq 0} G_{i,j}[-1],
\]
where
\[
G_{i,j}:=M[1]\otimes A[1]^{\otimes i} \otimes N[1]\otimes  B[1]^{\otimes j}.
\]
The differential, restricted to $G_{i,j}$, is
\[
b_{\overline{M\otimes_{A}N}^r}|_{G_{i,j}}=\sum_{\substack{l+m+n=i+j+2\\l\geq 1}}\id^{\otimes l}\otimes b_{m}\otimes \id^{\otimes n}+\sum_{\substack{l+m=i+j+2\\l\leq i+1}}b_{l}\otimes \id^{\otimes m}.
\]
\end{rem}

So far we have considered enrichments only for categories of nonunital modules and bimodules.  Ultimately we will only be interested in unital modules and bimodules (or $S$-unital modules and bimodules), so we should adapt our treatment to this situation.  Furthermore, we must generalise to the case of the category of modules over a small $A_{\infty}$-category, to this end we briefly recall the treatment of Section 5 of \cite{Keller-A2}.
\bigbreak
So let $\mathcal{C}$ be a small $A_{\infty}$-category.  We denote by $\TCC(\mathcal{C})$ the (non-counital) cocategory that has the same objects as $\mathcal{C}$, and for $x,y\in \mathcal{C}$ has
\begin{equation}
\Hom_{\TCC(\mathcal{C})}(x,y):=\bigoplus_{i\geq 1} T_i(x,y),
\end{equation}
where
\begin{equation}
\label{Tdef}
T_i(x,y):=\bigoplus_{x_1,\ldots,x_{i-1}\in\mathcal{C}} \Hom_{\mathcal{C}}(x_{i-1},y)[1]\otimes \Hom_{\mathcal{C}}(x_{i-2},x_{i-1})[1]\otimes\ldots\otimes \Hom_{\mathcal{C}}(x,x_1)[1]
\end{equation}
(we take the above to mean that $T_1(x,y):=\Hom_{\mathcal{C}}(x,y)[1]$).  The cocomposition is defined via the obvious maps
\[
T_i(x,y)\rightarrow \bigoplus_{\substack{1\leq j\leq i-1\\z\in \ob(\mathcal{C})}}T_j(z,y)\otimes T_{i-j}(x,z).
\]
A degree 1 coderivation $b_{\mathcal{C}}$ on $\TCC(\mathcal{C})$ is a set of maps
\[
\Hom_{\TCC(\mathcal{C})}(x,y)[1]\rightarrow \Hom_{\TCC(\mathcal{C})}(x,y)[1]
\]
of degree 1 such that the following diagram commutes, for all $x$, $y$ and $z$:
\[
\xymatrix{
\Hom_{\TCC(\mathcal{C})}(x,z)[1]\ar[r]^(.35){\Delta_{x,y,z}} & \Hom_{\TCC(\mathcal{C})}(y,z)[1]\otimes \Hom_{\TCC(\mathcal{C})}(x,y)[1]\\
\Hom_{\TCC(\mathcal{C})}(x,z)[1]\ar[u]^{b_{\mathcal{C}}}\ar[r]^(.35){\Delta_{x,y,z}}& \Hom_{\TCC(\mathcal{C})}(y,z)[1]\otimes \Hom_{\TCC(\mathcal{C})}(x,y)[1]\ar[u]_{b_{\mathcal{C}}\otimes \id+\id\otimes b_{\mathcal{C}}},
}
\]
where $\Delta_{x,y,z}:\Hom_{\TCC(\mathcal{C})}(x,z)\rightarrow \Hom_{\TCC(\mathcal{C})}(y,z)\otimes \Hom_{\TCC(\mathcal{C})}(x,y)$ is the cocomposition map.  Then a coderivation $b_{\mathcal{C}}$ is determined by the compositions 
\[
\xymatrix{
T_n(x,y)\ar[r]^(.4){b_{\mathcal{C}}|_{T_n}}& \bigoplus_{i\geq 1} T_i(x,y)\ar[r]&T_1(x,y)=\Hom_{\mathcal{C}}(x,y)[1],
}
\]
and we have the same story as with $A_{\infty}$-algebras.  Let $M$ be a $\mathcal{C}$-module.  Then following this lead we define the $\TCC(\mathcal{C})$-comodule
\[
\overline{M}^l(x):=\bigoplus_{\substack{i\geq 0\\x,y\in \ob(\mathcal{C})}}T_i(x,y)\otimes M[1](x),
\]
where we set $T_0(x,x)=k$, and the comodule structure is given by the maps
\[
T_i(x,y)\otimes M[1](x)\rightarrow \bigoplus_{\substack{1\leq j\leq i\\z\in \ob(\mathcal{C})}}T_j(z,y)\otimes (T_{i-j}(x,z)\otimes M[1](x)).
\]
As before, there is a 1-1 correspondence between coderivations $b$ on $\overline{M}^l$ satisfying $b^2=0$ and $\mathcal{C}$-module structures on $M$.  In the same way as for modules over an $A_{\infty}$-algebra we may define the enriched category $\mathcal{C}\lModinu$ as the category having as objects the nonunital modules $M$ over $\mathcal{C}$, and as morphism complexes the (ungraded) morphisms of comodules:
\begin{equation*}
\Hom_{\mathcal{C}\lModinu}(M,N)=\Hom_{\TCC(\mathcal{C})\lcomod} (\overline{M}^l,\overline{N}^l),
\end{equation*}
the grading coming from the grading on $\overline{M}^l$ and $\overline{N}^l$.
\bigbreak
It is again easy to verify that 
\[
\Ho^n(\Hom_{\mathcal{C}\lModinu}(M,N))\cong \Hom_{\mathcal{C}\lModnu}(M,N[n])/\sim_{\mathrm{homotopy}},
\]
and so we have an enrichment of the graded category $(\mathcal{C}\lModnu)_{\ext}$ with objects the $\mathcal{C}$-modules, and 
\[
\Hom_{(\mathcal{C}\lModnu)_{\ext}}^n(M,N)=\Hom_{\mathcal{C}\lModnu}(M,N[n])/\sim_{\mathrm{homotopy}}.
\]  
We define $(\mathcal{C}\lMod)_{\ext}$, in the natural way, to be the subcategory of the graded category $(\mathcal{C}\lModnu)_{\ext}$ with objects the strictly unital modules, and
\[
\Hom_{(\mathcal{C}\lMod)_{\ext}}^n(M,N)=\Hom_{\mathcal{C}\lMod}(M,N[n])/\sim_{\mathrm{homotopy}},
\]
the unital morphisms up to homotopy.  The key proposition is the following:
\begin{prop}
\label{nicething}
Let $\mathcal{C}$ be a strictly unital $A_{\infty}$-category.  The inclusion $(\mathcal{C}\lMod)_{\ext}\rightarrow (\mathcal{C}\lModnu)_{\ext}$ is fully faithful.
\end{prop}
This is Proposition 3.3.1.8 of \cite{KLH}.
Now we expect of any reasonable enrichment $\mathcal{E}$ of the category $(\mathcal{C}\lMod)_{\ext}$ that we should be able to find a lift of the inclusion 
\[
(\mathcal{C}\lMod)_{\ext}\rightarrow (\mathcal{C}\lModnu)_{\ext}
\]
to an inclusion of $A_{\infty}$-categories
\[
\mathcal{E}\rightarrow \mathcal{C}\lModinu
\]
and so from the homological perturbation theorem for categories (see Remark \ref{sketchyhpt}), we deduce that we can just take the full subcategory of $\mathcal{C}\lModinu$ whose objects are strictly unital modules, to be our category $\mathcal{C}\lModi$.
\bigbreak
There is a suitably modified version of Proposition \ref{nicething} for bimodules (again, this is just Proposition 3.3.2.2 of \cite{KLH}), and so, similarly, we define the category $\mathcal{C}\biModi\mathcal{D}$ to be the full subcategory of $\mathcal{C}\biModinu \mathcal{D}$ with objects the strictly unital $(\mathcal{C},\mathcal{D})$-bimodules.
\bigbreak
We finally come to deal with enrichments of the category $A\lMod_S$, for $(A,S,l)$ an $S$-unital $A_{\infty}$-algebra.  
\begin{assum}
\label{assume}
We assume that $l$ is an injection of graded vector spaces, that $S$ is an ordinary algebra, and that $m_{A,1}\circ l=0$.  Under these circumstances we will denote by $\otimes_S$ the \textit{ordinary} tensor product of $S$-modules.  $A$ is a graded bimodule over $S$, and its $A_{\infty}$-structure as a strictly unital algebra over $S$ corresponds to the structure of a unital $A_{\infty}$-algebra object in the category of graded $S$-bimodules.  We finally assume that the inclusion of ordinary graded $S$-bimodules $S\rightarrow A$ splits.
\end{assum}
Now let $M$ be an ordinary left $S$-module.  Then $\End_k(M)$ acquires the structure of an ordinary algebra object in the category of graded $S$-bimodules, and so also the structure of an $A_{\infty}$-algebra object in this category, and a strictly $S$-unital $A$-module structure on $M$ corresponds to a morphism of $A_{\infty}$-algebra objects $A_S\rightarrow \End_k(M)$, where $A_S$ is the unital $A_{\infty}$-algebra object in the category of graded $S$-bimodules induced by $A$, as in Proposition \ref{yeti}.
\bigbreak
We have a completely analogous story to the one before, involving free co-objects associated to $A_{\infty}$-objects.  Briefly, strictly $S$-unital $A_{\infty}$-algebra structures on $A$ give coderivations $b$ satisfying $b^2=0$ on the graded coalgebra object 
\[
\TC_S(A)=\oplus_{i\geq 1}(A[1])^{\otimes_S^i}
\]
satisfying $b^2=0$.  Furthermore, if we let
\[
\overline{M}^{l_S}=\Free_{\TC_S(A)\lcomod}(\forget_S(M)[1])=\oplus_{i\geq0} ((A[1])^{\otimes_S^i}\otimes_S M[1])
\]
then strictly $S$-unital $A$-module structures on $M$ extending the $S$-module structure on $M$ give rise to coderivations $b$ on this comodule object satisfying $b^2=0$.
\bigbreak
We define $A\lModi_{,S}$ for $A$ a strictly $S$-unital algebra by letting the objects be the strictly $S$-unital modules, and letting the homomorphism complexes be given by the spaces of (ungraded) free comodule morphisms
\[
\Hom_{A\lModi_{,S}}(M,N)=\Hom_{\TC_S(A)\lcomod}(\overline{M}^{l_S},\overline{N}^{l_S}).
\]
Proposition \ref{nicething} works in this setting too, and so these morphism spaces actually compute strictly $S$-unital morphisms, after taking homology, as before.  The story for bimodules is similar.
\begin{rem}
The notion of $S$-unital algebra, even for $S$ satisfying the restrictive conditions of Assumption \ref{assume}, extends the notion of a (small) unital category, since any (small) category $\mathcal{C}$ can be encoded in its quiver algebra $Q\mathcal{C}$, with each unit of the category then becoming an idempotent.  Letting $S$ be the algebra spanned as a vector space by these idempotents, modules over the category $\mathcal{C}$ become modules over the quiver algebra $Q\mathcal{C}$, with unital $\mathcal{C}$-modules corresponding to $S$-unital modules.  By construction, this is actually an equivalence at the level of $A_{\infty}$-categories.
\end{rem}
\section{The $A_{\infty}$ Yoneda embedding}
\label{twistcomp}
So far we have defined enrichments of the ordinary categories of $A_{\infty}$-modules over an $A_{\infty}$-category, and tensor products for appropriate right/left/bimodules.  Both of these constructions produce collections of differential graded vector spaces.  In the main body of the thesis these complexes will form differential graded vector \textit{bundles} over some proxy of the `stack of objects' in the category of perfect modules for an $A_{\infty}$-category $\mathcal{C}$.  We would like, in some sense, for these differential graded vector bundles to be finite-dimensional.  Ideally what we want is the rather strong statement that the underlying graded vector bundle should be finite-dimensional, not just the total homology of the differential graded vector bundle (see Remark \ref{tomp}).  Consider just homomorphism spaces for now.  It is clear enough that the spaces we have defined, which are basically variations on the bar complex, will never be finite-dimensional.  However, this (strong) notion of finiteness for differential graded homomorphism spaces clearly is not invariant under quasi-equivalences, so we might hope, at least when we restrict to the category of perfect modules, to be able to find a quasi-equivalence to a category with finite-dimensional homomorphism spaces, and maybe also find finite-dimensional models for tensor products.  We next introduce the Yoneda embedding/Lemma, which is the crucial ingredient for dealing with this problem.  The properties of this embedding are studied extensively in \cite{KLH}.
\bigbreak
Let $\mathcal{C}$ be an $A_{\infty}$-category.  We assume throughout that $\mathcal{C}$ is strictly unital.  Associated to an arbitrary $x\in \ob(\mathcal{C})$, there is a right $\mathcal{C}$-module $h_x$ satisfying
\[
h_x(z):=\Hom_{\mathcal{C}}(z,x),
\]
and a left $\mathcal{C}$-module $h^x$ satisfying
\[
h^x(z):=\Hom_{\mathcal{C}}(x,z).
\]
The right $\mathcal{C}$-module structure on $h_x$ is given by the compositions
\[
b_{n+1}:\Hom_{\mathcal{C}}(z_n,x)[1]\otimes \Hom_{\mathcal{C}}(z_{n},z_{n-1})[1]\otimes\cdots\otimes \Hom_{\mathcal{C}}(z_0,z_1)[1]\rightarrow \Hom_{\mathcal{C}}(z_0,x)[1]
\]
that define $\mathcal{C}$ as an $A_{\infty}$-category, and the left $\mathcal{C}$-module structure on $h^x$ is defined similarly.  We recall (see \cite{Seidel08} and \cite{KLH}) that $x\mapsto h^x$ can be upgraded to an $A_{\infty}$-functor of $A_{\infty}$-categories: 
\begin{align}
h^{-}:\mathcal{C}^{\op}\rightarrow &\mathcal{C}\lModi\\
h^{-}:x\rightarrow & \Hom_{\mathcal{C}}(x,-),
\end{align}
and the analogous statement holds for $h_-$.  In the language of left $\TCC(\mathcal{C})$-cocategories, the module $h^x$ is sent to the left $\TCC(\mathcal{C})$-comodule
\[
z\mapsto\bigoplus_{x_1,\ldots,x_{i-1}\in \ob(\mathcal{C})}\Hom_{\mathcal{C}}(x_{i-1},z)[1]\otimes\ldots\otimes \Hom_{\mathcal{C}}(x,x_1)[1].
\]
As ever, define 
\[
\overline{h^x}^l=\Free_{\TCC(\mathcal{C})\lcomod}(h^x[1])
\]
and 
\[
\overline{h^y}^l=\Free_{\TCC(\mathcal{C})\lcomod}(h^y[1]).
\]
A map $\phi:\overline{h^y}^l\rightarrow \overline{h^x}^l$ is determined by the induced maps
\[
\phi_i:\bigoplus_{y_1,\ldots,y_{i-1}\in \ob(\mathcal{C})}\Hom_{\mathcal{C}}(y_{i-1},z)[1]\otimes\ldots\otimes \Hom_{\mathcal{C}}(y,y_1)[1]\rightarrow \Hom_{\mathcal{C}}(x,z)[1].
\]
To define the contravariant functor $h^-$ we must give maps
\[
(h^-)_n:\Hom_{\mathcal{C}}(x_{n-1},y)[1]\otimes\ldots\otimes \Hom_{\mathcal{C}}(x,x_1)[1]\rightarrow \Hom_{\TCC(\mathcal{C})}(\overline{h^y}^l,\overline{h^x}^l).
\]
There is an obvious choice, given by setting $((h^-)_n)_m=b_{n+m}$, and this satisfies the required compatibility conditions by definition.
\begin{prop}[$A_{\infty}$ Yoneda embedding \cite{KLH}]

The $A_{\infty}$-functor
\begin{align*}
h^-:\mathcal{C}^{\op}\rightarrow &\mathcal{C}\lModi\\
x\mapsto &h^x:=\Hom_{\mathcal{C}}(x,-)
\end{align*}
is a quasi-fully faithful functor.
\end{prop}
Similarly, we have 
\begin{prop}[$A_{\infty}$ Yoneda embedding, covariant version]
The $A_{\infty}$-functor
\begin{align*}
h_-:\mathcal{C}\rightarrow &\rModi \mathcal{C}\\
x\mapsto &h_x:=\Hom_{\mathcal{C}}(-,x)
\end{align*}
is a quasi-fully faithful functor.
\end{prop}

\begin{defn}
\label{embe}
Given an $A_{\infty}$-category $\mathcal{C}$, we denote by 
\[
h^{\mathcal{C}}\subset \mathcal{C}\lModi
\]
the subcategory of left $\mathcal{C}$-modules with objects the $h^x$, for $x\in \ob(\mathcal{C})$, and with morphisms the $h^{\phi}$, for $\phi$ morphisms of $\mathcal{C}$  We denote by
\[
h^{\mathcal{C}}[\mathbb{Z}]\subset \mathcal{C}\lModi
\]
the subcategory of left $\mathcal{C}$-modules with objects shifts of $h^x$, and with morphisms from $h^y[m]$ to $h^x[n]$ given by $S^n h^{\phi} S^{-m}$, for $\phi\in \Hom_{\mathcal{C}}(x,y)$.  We define $h_{\mathcal{C}}$ and $h_{\mathcal{C}}[\mathbb{Z}]$ similarly.
\end{defn}
There is a natural isomorphism of $A_{\infty}$-categories
\[
h^{\mathcal{C}}[\mathbb{Z}]\cong h_{\mathcal{C}}[\mathbb{Z}]^{\op}.
\]

The $A_{\infty}$-Yoneda embedding provides the start of an answer to the problem of finding finite models for $\Hom_{\mathcal{C}\lmodi}(M,N)$, where $M$ and $N$ are unital left $\mathcal{C}$-modules, since if $x,y\in \ob(\mathcal{C})$ then we have that $\Hom_{\mathcal{C}\lmodi}(h^x,h^y)$ is functorially quasi-isomorphic to $\Hom_{\mathcal{C}}(y,x)$.  This motivates the following definition.
\begin{defn}
\label{findimcat}
Let $\mathcal{C}$ be an $A_{\infty}$-category.  We say $\mathcal{C}$ is finite-dimensional if $|\ob(\mathcal{C})|<\infty$ and for each $x,y\in \ob(\mathcal{C})$, $\dim(\Hom_{\mathcal{C}}(x,y))<\infty$.
\end{defn}
Note that if $\mathcal{C}$ is a finite-dimensional $A_{\infty}$-category then $h^{\mathcal{C}}[\mathbb{Z}]$ is a subcategory of $\mathcal{C}\lmodi$, and $h_{\mathcal{C}}[\mathbb{Z}]$ is a subcategory of $\rmodi\mathcal{C}$.
\section{Triangulated structure}
\label{triangstruc}
We recall here the triangulated structure on $\mathcal{C}\lModi$, discussed elsewhere in \cite{keller-intro}, \cite{KLH} and \cite{Seidel08}.  Given an arbitrary $A_{\infty}$-category $\mathcal{D}$ we have a quasi-fully faithful functor 
\[
h_-:\mathcal{D}\rightarrow \rModi \mathcal{D}
\]
into a differential graded category (recall that our model for $\rModi\mathcal{D}$ is actually a differential graded category).  Now $\Zo^0(\rModi\mathcal{D})$ is an ordinary category, and we can describe mapping cones in this category in the same way that we describe mapping cones of differential graded modules over a differential graded algebra.  Namely, let 
\[
\phi:M\rightarrow N
\]
be a morphism of differential graded $\TCC(\mathcal{D})$-comodules from $(M,d_M)$ to $(N,d_N)$, then we define the underlying graded comodule of $\cone(\phi)$ to be $M[1]\oplus N$.  We form the map $d_{\phi}$ as the composition
\[
\xymatrix{
M[1]\oplus N\ar[r]^(.57){\pi} &M[1]\ar[r]^(.55){\phi}& N\ar[r]^(.33)i&M[1]\oplus N
}
\]
where $\pi$ is the natural projection, $i$ the natural inclusion, and we are considering $\phi$ as a degree 1 map.  We give $M[1]\oplus N$ the differential $d_{M[1]}+d_{N}+d_{\phi}$, where $d_{M[1]}^n=-d_{M}^{n+1}$.  There is a natural series of maps
\[
\xymatrix{
M\ar[rr]^{\phi}&&N\ar[dl]\\
&\cone(\phi)\ar@{.>}[ul]
}
\]
where the dotted arrow is a degree 1 map.  Allowing diagrams that are quasi-isomorphic to these to be our distinguished triangles, and passing to the homotopy category, we obtain the structure of a triangulated category on $\Ho^{0}(\TCC(\mathcal{D})\lcomod)$, inducing the structure of a triangulated category on $\Di(\mathcal{D}\lModi)$.  One says $\mathcal{D}$ is a triangulated $A_{\infty}$-category if its quasi-essential image under $h_-$ is closed under taking triangles in $\rModi\mathcal{D}$.\bigbreak
One can use the Yoneda embedding to understand the compositions of morphisms into and out of cones in $A_{\infty}$-categories.  So we begin by fixing $\mathcal{D}$, a triangulated $A_{\infty}$-category.  Let $\phi\in \Zo^n(\Hom_{\mathcal{D}}(x,y))$ be a morphism.  Then we construct, as above, the cone of the induced map
\[
h_-(\phi):h_x\rightarrow h_y[n].
\]
This corresponds to the underlying $\TCC(\mathcal{D})$-comodule
\begin{equation*}
\begin{array}{c}
\cone(h_-(\phi))(z)=\left(\bigoplus_{\substack{i\geq 1\\x_1,\ldots,x_{i-1}\in \ob(\mathcal{C})}}\limits \Hom_{\mathcal{D}}(x_{i-1},x)[1]\otimes\ldots\otimes \Hom_{\mathcal{D}}(z,x_1)[1]\right)[1] \oplus\\ \left(\bigoplus_{\substack{i\geq 1\\x_1,\ldots,x_{i-1}\in \ob(\mathcal{C})}}\limits \Hom_{\mathcal{D}}(x_{i-1},y)[1]\otimes\ldots\otimes \Hom_{\mathcal{D}}(z,x_1)[1]\right)[n]
\end{array}
\end{equation*}
with differential twisted by $\sum_{i\geq 1}(h_-(\phi))_i$, where the composition with the natural projection
\[
\xymatrix{
\left(\bigoplus_{x_1,\ldots,x_{i-1}\in \ob(\mathcal{C})}\limits \Hom_{\mathcal{D}}(x_{i-1},x)[1]\otimes\ldots\otimes \Hom_{\mathcal{D}}(z,x_1)[1]\ar[d]^(.44){(h_-(\phi))_i}\right)[1] \\ \left(\bigoplus_{\substack{j\geq 1\\y_1,\ldots,y_{j-1}\in \ob(\mathcal{C})}}\limits \Hom_{\mathcal{D}}(y_{j-1},y)[1]\otimes\ldots\otimes \Hom_{\mathcal{D}}(z,y_1)[1]\ar[d]^(.57){\pi}\right)[n] \\ \Hom_{\mathcal{D}}(z,y)[n+1]
}
\]
is just $b_{i+1}(1_{\phi}\otimes \id^{\otimes i})$ (we let $1_{\phi}$ be the map $k\rightarrow \Hom_{\mathcal{D}}(x,y)$ sending $1$ to $\phi$).  From this we can already read off the structure of the right $\mathcal{D}$-module $h_{\cone(\phi)}$.  We have that $h_{\cone(\phi)}$ is the module $M$, where
\begin{equation}
\label{cork}
M(z)=\Hom_{\mathcal{D}}(z,x)[1]\oplus \Hom_{\mathcal{D}}(z,y)[n],
\end{equation}
and if we decompose, for arbitrary $z\in \ob(\mathcal{D})$, an arbitrary morphism $\psi\in M(z)$ as $\psi_x\oplus \psi_y$, according to the decomposition (\ref{cork}), then we have that
\begin{equation}
\label{cot}
\begin{array}{c}
b_{M,i}(\psi_x\oplus\psi_y,a_1,\ldots,a_{i-1})=\\=b_{\mathcal{D},i}(\psi_x,a_1,\ldots,a_{i-1})\oplus(b_{\mathcal{D},i}(\psi_y,a_1,\ldots,a_{i-1})+b_{\mathcal{D},i+1}(\phi,\psi_x,a_1,\ldots,a_{i-1})).
\end{array}
\end{equation}
We would like to get a handle also on the higher compositions of composable series of morphisms going through $\cone(\phi)$ that do not begin or terminate there.  There's a trick for doing this, starting from the obvious guess as to what the answer is.
\bigbreak
Let $\phi:x\rightarrow y$ be a morphism in $\Zo^{n}(\mathcal{D})$.  We enlarge the category $\mathcal{D}$, forming a category $\mathcal{D}'$, by adding an object $\widetilde{\cone}(\phi)$.  We write
\begin{equation}
\Hom_{\mathcal{D}'}(z,\widetilde{\cone}(\phi)):=\left(\begin{array}{c}\Hom_{\mathcal{D}}(z,y)[n] \\ \Hom_{\mathcal{D}}(z,x)[1]\end{array}\right),
\end{equation}
\begin{equation}
\Hom_{\mathcal{D}'}(\widetilde{\cone}(\phi),z):=\left(\begin{array}{cc}\Hom_{\mathcal{D}}(y,z)[-n] & \Hom_{\mathcal{D}}(x,z)[-1]\end{array}\right)
\end{equation}
and 
\begin{equation}
\Hom_{\mathcal{D}'}(\widetilde{\cone}(\phi),\widetilde{\cone}(\phi)):=\left(\begin{array}{cc}\Hom_{\mathcal{D}}(y,y)& \Hom_{\mathcal{D}}(x,y)[n-1]\\ \Hom_{\mathcal{D}}(y,x)[1-n]& \Hom_{\mathcal{D}}(x,x)\end{array}\right).
\end{equation}
In other words, we can think of $\widetilde{\cone}(\phi)$ as being a twisted direct sum of two objects $x[1]$ and $y$, with morphisms into $\widetilde{\cone}(\phi)$ being given by $1$-by-$2$ matrices, and morphisms out given by $2$-by-$1$ matrices, and endomorphisms given by $2$-by-$2$ matrices.  The morphism $\phi:x\rightarrow y$ becomes a strictly upper-triangular matrix, which we denote $M_{\phi}$.  Given
\[
(\psi_1,\ldots,\psi_n)\in \Hom_{\mathcal{D}'}(x_{n-1},x_n)[1]\otimes\ldots\otimes \Hom_{\mathcal{D}'}(x_0,x_1)[1],
\]
where $x_{i_1},\ldots,x_{i_{\tau}}=\widetilde{\cone}(\phi)$, we define
\begin{align*}
b_{\mathcal{D}'}(\psi_1,\ldots,\psi_n)=&b_{\mathcal{D},n}(\psi_1,\ldots,\psi_n)+\sum_{a\in\{1,\ldots,\tau\}}b_{\mathcal{D},n+1}(\psi_1,\ldots,\psi_{i_{a}},M_{\phi},\psi_{i_a+1},\ldots,\psi_n)+\\+&\sum_{\substack{a,b\in\{1,\ldots,\tau\}\\a<b}}b_{\mathcal{D},n+2}(\psi_1,\ldots,\psi_{i_a},M_{\phi},\psi_{i_a+1},\ldots,\psi_{i_b},M_{\phi},\psi_{i_b+1},\ldots,\psi_n)+\ldots
\end{align*}
It's a simple check to see that this indeed defines the structure of an $A_{\infty}$-category on the objects of $\mathcal{D}'$ (this is essentially the check that the structure of an $A_{\infty}$-category on the set of twisted objects generated by the objects in the above expression, recalled in the next section, is indeed an $A_{\infty}$-category, for which one may refer to Chapter 7 of \cite{KLH}), and by construction it is an extension of the $A_{\infty}$-structure on $\mathcal{D}$.
\begin{prop}
There is a quasi-isomorphism of right $\mathcal{D}'$-modules 
\[
\xymatrix{
h_{\widetilde{\cone}(\phi)}\ar[r]^{\sim} &h_{\cone(\phi)}.
}
\]
\end{prop}
The proof follows straight from our earlier description (\ref{cot}) of the cone.  
\begin{prop}
There is a quasi-isomorphism in $\mathcal{D}'$
\[
\xymatrix{
\widetilde{\cone}(\phi)\ar[r]^{\sim} & \cone(\phi).
}
\]
\end{prop}
\begin{proof}
This follows from the fact that the $A_{\infty}$-Yoneda embedding is quasi-fully faithful.
\end{proof}
\begin{rem}
This is a useful fact, since it gives us a good grip on the $A_{\infty}$-algebra $\Hom_{\mathcal{D}}(\cone(\phi),\cone(\phi))$.  The proposition above tells us that there is a quasi-isomorphism of $A_{\infty}$-algebras
\[
\xymatrix{
\Hom_{\mathcal{D}}(\cone(\phi),\cone(\phi))\ar[r]^(.49){\sim} & \Hom_{\mathcal{D}'}(\widetilde{\cone}(\phi),\widetilde{\cone}(\phi)).
}
\]
\end{rem}
As an instance of this usefulness, we mention the following statement.
\begin{prop}
Let $f:M\rightarrow N$ be a closed morphism in an $A_{\infty}$-category, and say the category consisting of the two objects $M$ and $N$ is given a Calabi-Yau structure of dimension $n$.  Then $\End_{\CC}(\cone(f))$ is quasi-isomorphic to a $n$-dimensional Calabi-Yau algebra.
\end{prop}
These notions will be explained in the next chapter, as well as the proof.
\section{The category of twisted objects}
\label{twistcat}
We recall some details of the category of twisted objects over a $A_{\infty}$-category, these can be found (with proofs) in \cite{KLH}.  First we'll make precise the matrix notation alluded to above.
\begin{defn}
\label{matrixmult}
A matrix $M\in \Mat_{m\times n}(\mathcal{C})$ is given by choosing an ordered $m$-tuple $(y_{M,1},\ldots,y_{M,m})\in \ob(\mathcal{C})^{m}$, and an ordered $n$-tuple $(x_{M,1},\ldots,x_{M,n})\in \ob(\mathcal{C})^n$, and for each $i\in \{1,\ldots,m\}$ and each $j\in \{1,\ldots,n\}$ a morphism $M_{i,j}\in \Hom_{\mathcal{C}}(x_{M,j},y_{M,i})$.  We say a sequence of matrices $M_1,\ldots,M_t$ is composable if for each $s\in \{1,\ldots,t\}$ we have that $m_s=n_{s+1}$ and $(y_{M_s,1},\ldots,y_{M_s,m_{s}})=(x_{M_{s+1},1},\ldots,x_{M_{s+1},n_{s+1}})$.  In that case we define $b_{\mathcal{C},t}(M_t,\ldots,M_1)$ to be the $m_{t}\times n_1$ matrix with
\[
b_{\mathcal{C},t}(M_t,\ldots,M_1)_{i,j}=\sum_{\substack{q_1\in \{1,\ldots,m_1\}\\ \ldots \\q_{t-1}\in \{1,\ldots,m_{t-1}\}}} b_{\mathcal{C},t} ((M_t)_{i,q_{t-1}},(M_{t-1})_{q_{t-1},q_{t-2}},\ldots,(M_1)_{q_1,j}).
\]
\end{defn}
\begin{defn}
Given a matrix $M$ in $\mathcal{C}$, associated to the ordered $m$-tuple 
\[
(y_{M,1},\ldots,y_{M,m})\in \ob(\mathcal{C})^m
\]
and the ordered $n$-tuple 
\[
(x_{M,1},\ldots,x_{M,n})\in \ob(\mathcal{C})^n,
\]
we define the transpose matrix $M^{T}$ in $\mathcal{C}^{\op}$ by setting
\[
(y_{M^T,1},\ldots,y_{M^T,m'})=(x_{M,1},\ldots,x_{M,n})
\]
and
\[
(x_{M^T,1},\ldots,x_{M^T,n'})=(y_{M,1},\ldots,y_{M,m}),
\]
and letting $(M^T)_{i,j}=M_{j,i}$.
\end{defn}
\begin{examp}
Let $S$ be a collection of objects in an $A_{\infty}$-category $\mathcal{C}$ admitting finite direct sums.  Then we may form a new $A_{\infty}$-category $\mathcal{C}^{\oplus}$, the full subcategory of $\mathcal{C}$ consisting of direct sums of the elements of $S$.  Morphisms between objects in $\mathcal{C}^{\oplus}$ are given by matrices in $\mathcal{C}$, and composition of morphisms is given as in Definition \ref{matrixmult}.  The transposition map gives a strict equivalence
\[
(\mathcal{C}^{\oplus})^{\op}\rightarrow (\mathcal{C}^{\op})^{\oplus}.
\]
\end{examp}
\begin{defn}
\label{twl}
The $A_{\infty}$-category $\tw_l(\mathcal{C})$ has objects given by pairs $((z_1,\ldots,z_n),A)$, for $n\geq 1$, where $(z_1,\ldots,z_n)$ is an ordered $n$-tuple of objects in $h^{\mathcal{C}}[\mathbb{Z}]$, as defined in Definition \ref{embe}, and $A$ is a strictly lower-triangular $n\times n$ degree 1 matrix in $h^{\mathcal{C}}[\mathbb{Z}]$ with 
\[
(x_{A,1},\ldots,x_{A,n})=(y_{A,1},\ldots,y_{A,n})=(z_1,\ldots,z_n),
\]
such that $A$ satisfies $\sum_{i\geq 1}b_{h^{\mathcal{C}}[\mathbb{Z}],i}(A,\ldots,A)=0$.  For 
\[
\begin{array}{c}
T_1=((z_{1,1},\ldots,z_{1,p_1}),A_1)\\
T_2=((z_{2,1},\ldots,z_{2,p_2}),A_2)
\end{array}
\] 
two objects of  $\tw_l(\mathcal{C})$, we define $\Hom_{\tw_l(\mathcal{C})}(T_1,T_2)$ to be the space of matrices $M$ in $h^{\mathcal{C}}[\mathbb{Z}]$ with $(x_{M,1},\ldots,x_{M,n})=(z_{1,1},\ldots,z_{1,p_1})$ and $(y_{M,1},\ldots,y_{M,m})=(z_{2,1},\ldots,z_{2,p_2})$.\smallbreak
For $i\in\{1,\ldots,t\}$ let $M_i$ be a morphism from $((z_{i-1,1},\ldots,z_{i-1,n_{i-1}}),A_{i-1})$ to $((z_{i,1},\ldots,z_{i,n_{i}}),A_{i})$, objects of $\tw_l(\mathcal{C})$.  We define composition in $\tw_l(\mathcal{C})$ by setting
\begin{equation}
\label{twmult}
b_{\tw_l(\mathcal{C}),t}(M_{t},\ldots,M_1)=\sum_{p_0,\ldots,p_t\geq 0} b_{h^{\mathcal{C}}[\mathbb{Z}],t+\sum p_i}(A_t^{\otimes p_t},M_t,A_{t-1}^{\otimes p_{t-1}},\ldots,M_1,A_0^{\otimes p_0}).
\end{equation}
Note that this sum is finite, since the matrices $A_i$ are strictly lower-triangular.
\end{defn}
A proof that these definitions actually define an $A_{\infty}$-category can be found in \cite{KLH}.
\begin{defn}
We define the $A_{\infty}$-category $\tw_r(\mathcal{C})$ in exactly the same way, instead considering ordered $n$-tuples of shifts of objects of $h_{\mathcal{C}}[\mathbb{Z}]$, and strictly upper-triangular matrices.
\end{defn}
\begin{rem}
Let us motivate these conventions before moving on with them.  The category of twisted objects will turn out to be quasi-equivalent to the category one gets from taking repeated mapping cones of objects in $h_{\mathcal{C}}[\mathbb{Z}]$ or $h^{\mathcal{C}}[\mathbb{Z}]$.  One can think of these (as in the Abelian case), as recursively defined extensions
\[
M_{i-1}\rightarrow M_{i}\rightarrow h_{x_i}[n_i],
\]
so when converted into matrices, the morphisms will go from $h_{x_{b}}[n_b]$ to $h_{x_{a}}[n_a][1]$, for $a<b$.  These maps come from maps in $\Hom_{\mathcal{C}}(x_{b},x_{a})$.  If one thinks of a matrix as a morphism from column vectors to column vectors, this corresponds to taking a strictly upper-triangular matrix.  Indeed it is sensible to consider a matrix this way, since it conserves the convention that we read morphisms from right to left.  The fact that we consider lower-triangular matrices when dealing with left modules is due to the fact that $\mathcal{C}$ maps contravariantly to this category, so we consider maps from $h^{x_{a}}[n_a]$ to $h^{x_{b}}[n_b][1]$, for $a<b$.
\end{rem}
\begin{prop}
\label{diamonddual}
There is a strict isomorphism of categories
\[
-^{\diamond_{\tw}}:\tw_l(\mathcal{C})^{\op}\rightarrow \tw_r(\mathcal{C}),
\]
sending objects $((x_1,\ldots,x_n),A)$ to $((x_1,\ldots,x_n),A^T)$ and sending morphisms $M$ of $\tw_l(\mathcal{C})^{\op}$ to morphisms $M^T$ of $\tw_r(\mathcal{C})$.
\end{prop}
In the sequel we will be a little lax and denote by $-^{\diamond_{\tw}}$ any of the four functors as defined above with preimage $\tw_l(\mathcal{C})^{\op}$, $\tw_r(\mathcal{C})^{\op}$, $\tw_l(\mathcal{C})$ or $\tw_r(\mathcal{C})$.
\begin{prop}
\label{coney}
There are functorial mapping cones in $\tw_l(\mathcal{C})$ and $\tw_r(\mathcal{C})$.
\end{prop}
\begin{proof}
We recall the proof here since it gives us a chance to write down what mapping cones in the category of twisted objects look like.  We only deal with $\tw_r(\mathcal{C})$, the construction for $\tw_l(\mathcal{C})$ is identical.  The construction is exactly as one would expect.  Given $T_1$ and $T_2$ as in the statement of Definition \ref{twl}, and given $M$ a closed morphism from $T_2$ to $T_1$, i.e. a degree zero matrix in $h_{\mathcal{C}}[\mathbb{Z}]$ with $(x_{M,1},\ldots,x_{M,n})=(z_{2,1},\ldots,z_{2,p_1})$ and $(y_{M,1},\ldots,y_{M,m})=(z_{1,1},\ldots,z_{1,p_2})$, satisfying $b_{\tw_r(\mathcal{C}),1}(M)=0$, we make the new twisted object
\[
\cone(M):=((z_{1,1},\ldots,z_{1,p_1},z_{2,1}[1],\ldots,z_{2,p_2}[1]),\left(\begin{array}{cc} A_1 & M\\0 & A_2 \end{array}\right)).
\]
From the easy identity
\begin{equation}
\label{muchbetter}
\begin{array}{c}
\sum_{i\geq 1}b_{\mathcal{C},i}\left(\left(\begin{array}{cc} A_1 & M\\0 & A_2 \end{array}\right),\ldots,\left(\begin{array}{cc} A_1 & M\\0 & A_2 \end{array}\right)\right)=\\=\left( \begin{array}{cc}\sum_{i\geq 1} b_{h_{\mathcal{C}}[\mathbb{Z}],i}(A_1,\ldots,A_1)&b_{\tw_r(\mathcal{C}),1}(M) \\0& \sum_{i\geq 1} b_{h^{\mathcal{C}}[\mathbb{Z}],i}(A_2,\ldots,A_2)\end{array}\right)
\end{array}
\end{equation}
we deduce that this gives us a new element of $\tw_r(\mathcal{C})$.  
\end{proof}
By construction, there are full and faithful embeddings of $A_{\infty}$-categories given by the (strict) functors
\[
i_r:\mathcal{C}\rightarrow \tw_r(\mathcal{C})
\]
\[
i_l:\mathcal{C}^{\op}\rightarrow \tw_l(\mathcal{C})
\]
with $i_l(x)=((h^x),0)$ and $i_r(x)=((h_x),0)$, and so one builds in a natural way, out of an object $\alpha\in \tw_l(\mathcal{C})$, a left $\mathcal{C}$-module $\Hom_{\tw_l(\mathcal{C})}(i_l(-),\alpha)$, giving an $A_{\infty}$-functor
\[
j_l:\tw_l(\mathcal{C})\rightarrow \mathcal{C}\lModi.
\]
We obtain a commuting diagram of $A_{\infty}$-functors, interpreting the strict functor $i_l$ as an $A_{\infty}$-functor with $(i_l)_n=0$ for all $n\geq 2$ (see \cite{KLH})
\[
\xymatrix{
\tw_l(\mathcal{C})\ar[r]^(.44){j_l} & \mathcal{C}\lModi\\
\mathcal{C}^{\op}\ar[u]^{i_l}\ar[ur]_{h^-}.
}
\]
Note that under $j_l$, $h^{\mathcal{C}}[\mathbb{Z}]$ is identified as a full subcategory of the category of twisted objects with associated matrices equal to zero.  Recall that, for a triangulated category $\mathcal{E}$, and $S$ a subset of objects of $\mathcal{E}$, we defined $\langle S \rangle_{\triang}$ to be the smallest strictly full (i.e. closed under isomorphisms in $\mathcal{E}$) subcategory of $\mathcal{E}$ which contains all of the objects of $S$ and is closed under shifts and triangles.  Since $\mathcal{C}\lModi$ has a triangulated homotopy category, i.e. $\Di(\mathcal{C}\lModi)$ is triangulated, if $S$ is a set of the objects of $\mathcal{C}\lModi$, then we set $\langle S \rangle_{\triang}$ to be the smallest strictly quasi-full subcategory of $\mathcal{C}\lModi$ containing $\ob(\triang_{\Di(\mathcal{C}\lModi)}(S))$.
\begin{prop}[\cite{KLH}]
\label{trac}
The $A_{\infty}$-functor $j_l$ maps to $\langle h^{\mathcal{C}}\rangle_{\triang}$, and is a quasi-equivalence of categories.
\end{prop}
The above proposition is important, for it tells us that, up to quasi-isomorphism, the objects of $\tw_l(\mathcal{C})$ give us the whole of $\langle h^{\mathcal{C}}\rangle_{\triang}$, and up to quasi-isomorphism, the morphism spaces of $\tw_l(\mathcal{C})$ give us the morphism spaces of $\langle h^{\mathcal{C}}\rangle_{\triang}$ as well.  Of course the equivalent statements hold for $\tw_r(\mathcal{C})$ too.  This is good news for doing Geometry, because these categories are much leaner, a statement we make precise with the following propositions.  As mentioned in \cite{KS}, it is these propositions that make the categories $\tw_l(\mathcal{C})$ and $\tw_r(\mathcal{C})$ suitably geometric for motivic Donaldson--Thomas theory.
\begin{prop}[\cite{KS}]
\label{kl1}
Let $\mathcal{C}$ be a finite-dimensional category.  Then for a fixed $n$-tuple $\tau=(x_1,\ldots,x_n)$ of objects of $h^{\mathcal{C}}[\mathbb{Z}]$, there is a finite type affine scheme $\VV_{l,\tau}$ parameterising objects $\alpha=((y_1,\ldots,y_m),A)$ of $\tw_l(\mathcal{C})$ such that $(y_1,\ldots,y_m)=\tau$.
\end{prop}
\begin{proof}
By `parameterise', here, we mean that there is a family of objects of $\tw_l(\mathcal{C})$ over $\VV_{l,\tau}$, such that for any object $\alpha$ as in the statement of the proposition, $\alpha$ corresponds to a (probably non-unique) $k$-point of this family.  This is discussed in \cite{KS}.  We need to show that there is a finite type affine scheme parameterising the possible strictly lower-triangular matrices $A$, in the above description of $\alpha$.  These matrices form a subset of
\[
\bigoplus_{i<j\in \{1,\ldots,n\}}\Hom_{h^{\mathcal{C}}[\mathbb{Z}]}^1(x_i,x_j)
\]
which is a finite-dimensional affine space, by assumption.  The equation 
\[
\sum_{i\geq 1}b_{h^{\mathcal{C}}[\mathbb{Z}],i}(A,\ldots,A)=0
\]
reduces to
\[
\sum_{1\leq i\leq n-1}b_{h^{\mathcal{C}}[\mathbb{Z}],i}(A,\ldots,A)=0,
\]
which gives $n^2$ algebraic equations on this space.
\end{proof}
\begin{defn}
\label{objects}
Fix a finite-dimensional $A_{\infty}$-category $\mathcal{C}$.  Then we denote by $\VV_l$ the ind-variety $\coprod \VV_{l,\tau}$, the disjoint union of the $\VV_{l,\tau}$, for all $n$-tuples $\tau$ of objects in $h^{\mathcal{C}}[\mathbb{Z}]$, for all $n$.  We define $\VV_r$ similarly.
\end{defn}
The following proposition comes directly from how we have set up $\tw_r(\mathcal{C})$.
\begin{prop}
\label{indvariety}
Let $\mathcal{C}$ be a finite-dimensional $A_{\infty}$-category.  For each pair $\tau_1$ and $\tau_2$ of of ordered sets of objects of $h_{\mathcal{C}}[\mathbb{Z}]$, there is a finite-dimensional graded vector bundle $\mathcal{HOM}_{\tau_1,\tau_2}$ over $\VV_{r,\tau_1}\times \VV_{r,\tau_2}$, and these satisfy the following condition:  If, for all $j\in\{0,\ldots,t\}$, $\tau_j$ is an ordered set of objects of $h_{\mathcal{C}}[\mathbb{Z}]$, there are maps of algebraic vector bundles 
\[
\xymatrix{
b_{\mathcal{HOM},j}:\pi_{0,1}^*(\mathcal{HOM}_{\tau_0,\tau_1})\times\ldots\times \pi_{t-1,t}^*(\mathcal{HOM}_{\tau_{t-1},\tau_t})\rightarrow \pi_{0,t}^*(\mathcal{HOM}_{\tau_0,\tau_t})
}
\]
where $\pi_{a,b}:\VV_{r,\tau_t}\times\ldots\times \VV_{r,\tau_0}\rightarrow \VV_{r,\tau_a}\times \VV_{r,\tau_b}$ is the natural projection.  These vector bundles, and maps between them, satisfy the property that for $x_0$ a point of $\VV_{r,\tau_0}$ parameterising an object $\alpha_0$ of $\tw_r(\mathcal{C})$, and $x_t$ a point of $\VV_{r,\tau_t}$ parameterising an object $\alpha_t$ of $\tw_r(\mathcal{C})$, the fibre of the vector bundle $\mathcal{HOM}_{\tau_0,\tau_t}$ over the point $(x_0,x_t)$ is naturally identified with $\Hom_{\tw_r(\mathcal{C})}(\alpha_t,\alpha_0)$, and the vector bundle maps $b_{\mathcal{HOM},j}$ agree with the composition maps in the category $\tw_r(\mathcal{C})$.
\end{prop}
\begin{defn}
\label{lambda1}
We denote by $\lambda:\VV_r\rightarrow\VV_l$ the isomorphism of ind-varieties inducing the same map of underlying sets of $k$-points as $-^{\diamond_{\tw}}$, as defined in Proposition \ref{diamonddual}.
\end{defn}

\begin{thm}
\label{triangparam}
Let $\mathcal{C}$ be as above, then there is an ind-variety $\VV_{l,3}$ parameterising triangles in $\tw_r(\mathcal{C})$.
\end{thm}
By functoriality of the mapping cone, this ind-scheme can be taken as the total space of the constructible bundle $\Zo^0(\mathcal{HOM})$ on $\VV_r\times\VV_r$.  Note that this comes with three natural projections $p_1, p_2$ and $p_3$.  The first two are just the projections onto the first and second $\VV_r$ factors, while the third maps the point represented by a matrix $M$ over a point $((\tau_1,A_1),(\tau_2,A_2))\in \VV_r\times \VV_r$ to the cone $((\tau_1,\tau_2[1]),\left(\begin{array}{cc}A_1 & M\\0 & A_2\end{array}\right))$ -- see Proposition \ref{coney}.

\section{Finite models for tensor products}
\label{finmodtens}
Let $M$ be a $(\mathcal{C},\mathcal{D})$-bimodule, for $\mathcal{C}$ and $\mathcal{D}$ finite-dimensional $A_{\infty}$-categories, and assume $M$ is finite-dimensional too, in the sense that for all pairs $(x,y)\in (\mathcal{C},\mathcal{D})$, $M(y,x)$ is a finite-dimensional differential graded vector space.  We have defined a tensor product between $M$ and left $\mathcal{D}$-modules, that is functorial, i.e. we have a functor (see the discussion following Remark \ref{funnyshift})
\begin{align*}
M\otimes_{\mathcal{D}}-:\mathcal{D}\lModi\rightarrow &\mathcal{C}\lModi\\
N\mapsto & M\otimes_{\mathcal{D}}N.
\end{align*}
This functor uses the bar construction, and so it generally churns out infinite-dimensional $\mathcal{C}$-modules, no matter how nice $N$ or $M$ might be.  It turns out that if we restrict the functor to $\tw_l(\mathcal{C})$, considered as a subcategory of $\mathcal{C}\lModi$, we can do better than this.
\begin{prop}
\label{bas}
For all $y\in \ob(\mathcal{D})$, there is a quasi-equivalence of left $\mathcal{C}$-modules
\[
M\otimes_{\mathcal{D}} h^{y}\rightarrow M(y,-).
\]
\end{prop}
\begin{proof}
This kind of result is pretty standard, related as it is to the fact that the bar resolution actually is a resolution.  We will recount the proof anyway.  Consider the associated left $\TCC(\mathcal{C})$-comodules:
\[
\overline{M\otimes_{\mathcal{D}} h^{y}}^l(w)=\bigoplus_{i\geq 0, j\geq 1} H_{i,j}(w)[-1]
\]
where
\[
H_{i,j}(w)=\bigoplus_{z_1,z_2\in \mathcal{C}} T_{\mathcal{C},i}(z_2,w)\otimes M(z_1,z_2)[1]\otimes T_{\mathcal{D},j}(y,z_1),
\]
and
\[
\overline{M(y,-)}^l(w)=\big(\bigoplus_{\substack{i\geq 0\\z\in \mathcal{C}}}T_{\mathcal{C},i}(z,w)\otimes M(y,z)[1]\big).
\]
Here we have extended a little the notation of (\ref{Tdef}), setting, for an arbitrary $A_{\infty}$-category $\mathcal{E}$,
\[
T_{\mathcal{E},i}(x,y):=\bigoplus_{x_1,\ldots,x_{i-1}\in\mathcal{E}} \Hom_{\mathcal{E}}(x_{i-1},y)[1]\otimes \Hom_{\mathcal{E}}(x_{i-2},x_{i-1})[1]\otimes\ldots\otimes \Hom_{\mathcal{E}}(x,x_1)[1]
\]
for $i\geq 1$ and
\[
T_{\mathcal{E},0}(x,y):=\begin{cases} 0 \mbox{ for } x\neq y\\k\mbox{ otherwise.}\end{cases}
\]
We define 
\begin{align*}
\phi:\overline{M\otimes_{\mathcal{D}} h^{y}}^l\rightarrow &\overline{M(y,-)}^l
\end{align*}
by defining $\phi_{i,j}[1]$, the restriction of $\phi$ to $H_{i,j}$.  Set
\begin{align*}
\phi_{i,j}[1]:=\sum_{k>j} \id^{\otimes (i+j+1-k)}\otimes b_{k}.
\end{align*}
Next we define a map
\begin{align*}
\psi:\overline{M(y,-)}^l\rightarrow &\overline{M\otimes_{\mathcal{D}} h^{y}}^l
\end{align*}
by sending $\kappa\in T_i(z,w)\otimes M(y,z)[1]$ to $(\kappa, \id_{y})$.\smallbreak
Using strict unitality, we see that $\phi\circ \psi$ is the identity on $\overline{M(y,-)}^l$.  Finally, the degree -1 map
\begin{align*}
h:\overline{M\otimes_{\mathcal{D}} h^{y}}^l\rightarrow \overline{M\otimes_{\mathcal{D}} h^{y}}^l
\end{align*}
sending $\eta\in T_{\mathcal{C},i}(z_2,w)\otimes M(z_1,z_2)[1]\otimes T_{\mathcal{D},j}(y,z_1)$ to $(\eta,\id_y)$ is a homotopy between the identity on $\overline{M\otimes_{\mathcal{D}} h^{y}}^l$ and the map $\psi \circ \phi$.
\end{proof}
Let $M$ be a $(\mathcal{C},\mathcal{D})$-bimodule.  Then we build a $(\tw_r(\mathcal{C}),\tw_l(\mathcal{D})^{\op})$-bimodule $M_{\tw}$ as follows.  First, we set $M(h^y[b],h_x[a])=M(y,x)[a+b]$.  Next, let $\alpha=((x_1,\ldots,x_m),A)$ be an element of $\tw_r(\mathcal{C})$, and $\beta=((y_1,\ldots,y_n),B)$ be an element of $\tw_l(\mathcal{D})$.  Then 
\[
M_{\tw}(\beta,\alpha):=\bigoplus_{\substack{i\in\{1,\ldots,n\}\\j\in \{1,\ldots,m\}}}M(y_i,x_j).
\]
We interpret elements of $M_{\tw}(\beta,\alpha)$, then, as $m$ by $n$ matrices, with $(i,j)$th entry taking values in $M(y_j,x_i)$.  Then the module multiplication is given by matrix multiplications: let $\beta_0,\ldots,\beta_t\in \tw_l(\mathcal{D})$ and $\alpha_0,\ldots,\alpha_s\in \tw_r(\mathcal{C})$ and let $\phi_i\in \Hom_{\tw_r(\mathcal{C})}(\alpha_{i-1},\alpha_i)[1]$ for $i\in\{1,\ldots,s\}$ and $\psi_j\in \Hom_{\tw_l(\mathcal{D})}(\beta_{j-1},\beta_j)[1]$ for $j\in \{1,\ldots,t\}$, and let $T\in M_{\tw}(\beta_t,\alpha_0)$, then we set
\begin{equation}
\begin{array}{c}
b_{M_{\tw},s,t}(\phi_s,\ldots\phi_1,T,\psi_t,\ldots,\psi_1)=\\=\sum_{\substack{p_0,\ldots,p_s\geq 0\\q_0,\ldots,q_t\geq 0}}\limits b_{M,s+\sum p_i, t+\sum q_j}(A_s^{\otimes p_s},\phi_s,A_{s-1}^{\otimes p_{s-1}},\ldots,A_{0}^{\otimes p_0},T,(B_t^T)^{\otimes q_t},\psi_{t}^T,(B_{t-1}^T)^{\otimes q_{t-1}},\ldots,\psi_1^T,(B_1^T)^{\otimes q_0})
\end{array}
\end{equation}
where $\alpha_i=((\ldots),A_i)$ for $i\in\{1,\ldots,s\}$ and $\beta_j=((\ldots),B_j)$ for $j\in \{1,\ldots,t\}$.  One may easily check that this makes $M_{\tw}$ a $(\tw_r(\mathcal{C}),\tw_l(\mathcal{D})^{\op})$-bimodule, and that the $(\mathcal{C},\mathcal{D})$-module induced via the embeddings $i_r$ and $i_l$ is just the original bimodule $M$.\bigbreak
\begin{prop}
\label{finmods}
Let $\mathcal{C}$ and $\mathcal{D}$ be $A_{\infty}$-categories, and let $M$ be a $(\mathcal{C},\mathcal{D})$-bimodule.  Then there is a quasi-isomorphism of $(\tw_r(\mathcal{C}),\tw_l(\mathcal{D})^{\op})$-bimodules:
\[
\xymatrix{
M_{\tw}(-,-)\ar[r]^(.45){\sim} & -\otimes_{\mathcal{C}} M\otimes_{\mathcal{D}} -.
}
\]
\end{prop}
The proof is completely formal, using the fact that mapping cones are functorial in differential graded categories, functors between differential graded categories preserve mapping cones, and Proposition \ref{bas}.  The point is that functors between differential graded categories and natural transformations between such functors determine their extensions to triangulated hulls.  
\begin{cor}
\label{diagfact}
Let $\mathcal{C}$ be an $A_{\infty}$-category, then there is a quasi-isomorphism of $(\tw_{r}(\mathcal{C}),\tw_{r}(\mathcal{C}))$-bimodules
\[
\xymatrix{
-\otimes_{\mathcal{C}} \mathcal{C} \otimes_{\mathcal{C}} -^{\diamond_{\tw}}\ar[r]^(.465){\sim} & \Hom_{\tw_r(\mathcal{C})}(-,-).
}
\]
\end{cor}
Note the confusing swapping of arguments here!
\section{The category $\Perf(\mathcal{C}\lModi)$}
\label{perfconstr}
We now recall some details regarding the category of perfect $\mathcal{C}$-modules.  This section is not essential to the understanding of the thesis, since in all the cases in which we are interested this category is the quasi-essential image of the inclusion from $\tw_l(\mathcal{C})$.  However, there are categories $\mathcal{C}$ for which the two categories do not coincide, even up to quasi-equivalence, and for which we can still say something positive about orientation data.  This is because orientation data is a concept that exists independently of the story of motivic Donaldson--Thomas invariants (as noted in \cite{KS}, we do not even need the category for which we provide orientation data to be 3-dimensional).  It is anticipated that a more satisfactory treatment of these complications will be provided by working directly with the derived stack of objects in $\Perf(\mathcal{C}\lModi)$.  
\begin{defn}
The category $\Perf(\mathcal{C}\lModi)$ is the smallest quasi-strictly full sub-$A_{\infty}$-category of the category of left $\mathcal{C}$-modules that is closed under shifts, cones, and homotopy retracts, containing the image of the Yoneda embedding.  A homotopy retract is a diagram in $\mathcal{C}\lMod$ (i.e. we ask that $\alpha,\beta\in \Zo^0(\mathcal{C}\lModi)$):
\begin{equation}
\label{retdi}
\xymatrix{
M\ar[r]^{\alpha}& N\ar[r]^{\beta} &M
}
\end{equation}
such that the two maps compose to the identity in $\Di(\mathcal{C}\lModi)$.  
\end{defn}
The category $\Perf(\CC\lMod)$ is defined as $\mathrm{Z}^0(\Perf(\CC\lModi))$, and $\Perf(\rMod\CC)$ and\\ $\Perf(\rModi\CC)$ are defined likewise.
\begin{lem}
Every homotopy retract in $\CC\lMod$ is quasi-isomorphic to a retract.
\end{lem}
\begin{proof}
By the minimal model theorem and the homological perturbation lemma (see \cite{KLH}), we may assume that $M$ is minimal in diagram (\ref{retdi}).  Then since $\beta\circ \alpha$ is a quasi-isomorphism on the minimal module $M$, it is in fact an $A_{\infty}$-isomorphism (again, see \cite{KLH}), and the lemma follows.
\end{proof}
\begin{lem}
\label{permtake}
Let 
\[
\xymatrix{
M\ar[r]^{\alpha} & N\ar[r]^{\beta} & M
}
\]
be a homotopy retract, and let $\phi:M\rightarrow S$ be a morphism.  Then $\cone(\phi)$ is quasi-isomorphic to a homotopy retract of a cone of a morphism from $N$ to $S$.  Alternatively, let $\psi:S\rightarrow M$ be a morphism.  Then, dually, $\cone(\psi)$ is quasi-isomorphic to a homotopy retract of a cone of a morphism from $S$ to $N$.
\end{lem}
\begin{proof}
After taking a minimal model for $M$ we may assume that the diagram in the statement of the lemma is an actual retract diagram, by the previous lemma.  Consider the commutative diagram
\[
\xymatrix{
M\ar[d]^{\phi}\ar[r]^{\alpha} & N\ar[d]^{\phi\circ \beta} \ar[r]^{\beta} &M\ar[d]^{\phi}\\
S\ar[d]\ar[r]^{\id_S} & S\ar[d]\ar[r]^{\id_S} &S\ar[d]\\
\cone(\phi)\ar[r] & \cone(\phi\circ \beta)\ar[r]& \cone(\phi),
}
\]
where the maps in the bottom row are supplied by the functoriality of the mapping cone in differential graded categories (see e.g \cite{ToenDGnotes}).  Since the composition of the morphisms in the top row is a quasi-isomorphism, and so is the composition of the morphisms in the second, it follows that the third row is a homotopy retract, by the five lemma.
\smallbreak 
The argument for $\cone(\psi)$ is given by the dual.
\end{proof}
The point of Lemma \ref{permtake} is that we can permute the taking of homotopy retracts and the taking of cones.  Combined with Proposition \ref{trac}, this tells us that all elements of the category $\Perf(\mathcal{C}\lModi)$ can be obtained, up to quasi-isomorphism, by taking homotopy retracts of objects of $\tw_l(\mathcal{C})$.  
\bigbreak
Before introducing an ind-variety parameterising perfect modules, we give a short digression on $A_{\infty}$-limits.\bigbreak
We first discuss the classical case.  Let $D$ be a small unital $k$-linear category, and let $\mathcal{E}$ be a category such as the category of $A$-modules for an associative algebra $A$.  Then we consider the functor category
\[
\mathcal{E}^D:=\Fun(D,\mathcal{E})
\]
whose objects are $k$-linear functors, and whose morphisms are natural transformations.  Given a functor $D\rightarrow k$, where by abuse of notation $k$ denotes the category with one object, the endomorphism algebra of which is just $k$, one obtains a functor
\[
c:\mathcal{E}\rightarrow \mathcal{E}^D,
\]
obtained from the natural isomorphism $\mathcal{E}^k\cong \mathcal{E}$.
\smallbreak
A \textit{limit} functor for the diagram $D$ is a right adjoint to $c$ (obviously it depends also on the functor $D\rightarrow k$).  By the uniqueness of adjoints it is unique up to unique natural isomorphism.  A \textit{colimit} functor for the diagram $D$ is a left adjoint to $c$.
\begin{examp}
\label{retlim}
Let $D'$ be the non-unital $k$-linear category with one object $\omega$, and one-dimensional morphism space spanned by the morphism $f$ satisfying $f^2=f$.  One obtains in the natural way, from $D'$, an (ordinary) augmented $k$-linear category $D$, with one object, such that the induced ring structure on the augmentation ideal of the endomorphism algebra of the unique object is just the non-unital algebra spanned by $f$.  We define $D\rightarrow k$ by sending $f$ to $1$.  Let $\mathcal{E}=A\lMod_{\mathrm{norm}}$ denote the ordinary category of (ungraded) $A$-modules for some (ungraded) associative algebra $A$.  Then an object of $\mathcal{E}^D$ is given by an $A$-module $M$ and a morphism $\phi\in \End(M)$ satisfying $\phi^2=\phi$.  Given such a functor, we obtain a retract diagram
\[
\xymatrix{
\Image(\phi)\ar[r]^<<<<<{\iota} & M \ar[r]^<<<<<{\phi'} & \Image(\phi),
}
\]
where $\phi'$ is the map induced by $\phi$ and $\iota$ is the inclusion.  In this situation $\Image(-)$ is the limit functor and the colimit functor.
\medbreak
In fact it is possible to see, purely categorically, that one obtains a retract diagram from the limit for $D$.  The situation is summed up by the following commuting diagram:
\begin{equation}
\label{clownhat}
\xymatrix{
&M\ar@(ul,ur)^-{\phi}\\
N\ar[ur]^-i & M\ar[l]^-{\alpha}\ar[u]^-{\phi} & N\ar[l]^-i \ar[ul]_-{\phi\circ i}\ar@(dl,dr)[ll]^-{\beta}.
}
\end{equation}
The morphism $i$ comes from the fact that $N$ is the limit of the diagram 
\[
\xymatrix{
&M.\ar@(ul,ur)^{\phi}
}
\]
The morphisms $\alpha$ and $\beta$ come from the definition of the limit.  Note that if we take $\beta=\id_N$ then everything commutes.  By one last application of the definition of limit, it follows that $\beta=\id_N=\alpha\circ i$.
\end{examp}
\smallbreak
The theory of $A_{\infty}$-limits is really just the same thing but for $A_{\infty}$-functors.  Let $\mathcal{E}$ now denote a strictly unital $A_{\infty}$-category, and let $D$ be a small $k$-linear category as before.  We now let
\[
\mathcal{E}^D:=\Fun_{\infty}(D,\mathcal{E})
\]
be the set of $A_{\infty}$-functors from $D$ to $\mathcal{E}$.  This has the structure of an $A_{\infty}$-category (see \cite{Keller-A2}).  As before there is a (strict) functor 
\[
c:\mathcal{E}\rightarrow \mathcal{E}^D.
\]
We use this functor to construct a $(\mathcal{E},\mathcal{E}^D)$-bimodule $\Hom_{\mathcal{E}^D}(c-,-)$.  Given an arbitrary $A_{\infty}$-functor 
\[
G:\mathcal{E}^D\rightarrow \mathcal{E}
\]
we similarly obtain an $(\mathcal{E},\mathcal{E}^D)$-bimodule
\[
\Hom_{\mathcal{E}}(-,G-)
\]
and we define an adjunction between $c$ and $G$ to be a quasi-isomorphism of bimodules
\[
\xymatrix{
\Hom_{\mathcal{E}^D}(c-,-)\ar[r]^{\sim} & \Hom_{\mathcal{E}}(-,G-),
}
\]
and in this case we say the functor $G$ is the $A_{\infty}$-limit functor.
\medbreak
Now let $D$ be as before, and let $\mathcal{E}$ be the $A_{\infty}$-category $\mathcal{C}\lModi$.  Then a limit and a colimit functor exist for $D$ (see \cite{ToenDGnotes}).  Let $M$ be a $\mathcal{C}$-module, let $F\in \mathcal{E}^D$ be a functor sending the object of $D$ to $M$, and let $N$ be the limit.  Then we obtain a diagram as in \ref{clownhat}, commuting up to homotopy.  After taking $\Ho^{\bullet}(-)$ of the diagram, we get that the bottom row is again an isomorphism, and so the limit is a homotopy retract of $M$.
\smallbreak
Note that an $A_{\infty}$-functor from $D$ to $\mathcal{E}$ is given by the following data (which is the data mentioned in the introduction to \cite{KS}): a module $M$ and a morphism of non-unital $A_{\infty}$-algebras from $k$ to $\End(M)$.  For each $n$, $k^{\otimes n}\cong k$, and so such a morphism is given by a series of elements $f_n\in \End^{1-n}(M)$, satisfying the compatibility relations of an $A_{\infty}$-algebra morphism.  We assume, now, that we are working with the explicit (differential graded) model for the $A_{\infty}$-category $A\lModi$ that we recalled in Section \ref{enrichments}.  Say, conversely, that we are given a homotopy retract diagram
\[
\xymatrix{
N\ar[r]^i & M\ar[r]^q & N.
}
\]
Then there is some morphism $h\in \End(N,N)$ such that $dh=\id_N-q\circ i$.  Let $f_1\in \End(M,M)$ be given by $f_1=i\circ q$.  Let $f_2=i\circ h\circ q$. For $n\geq 3$, we set $f_n=i\circ h^{n-1} \circ q$.  Then one can easily check that these $f_n$ satisfy the required compatibility relations.  We claim, finally, that $N$ is the limit and the colimit of the associated element of $\mathcal{E}^D$.  This leads to the following proposition:
\begin{prop}\cite{KS}
There is an ind-variety $\VV_l^+$ parameterising objects of $\Perf(\mathcal{C}\lMod)$, and up to quasi-isomorphism, every element of $\ob(\Perf(\mathcal{C}\lMod))=\ob(\Perf(\mathcal{C}\lModi))$ occurs in this family.
\end{prop}
\begin{proof}
We take as our starting point $\VV_l$, the ind-variety of Definition \ref{objects} parameterising objects of $\tw_l(\mathcal{C})$.  Let $\VV_{l,\tau}$ be one of the components of this variety, parameterising twisted objects of the form $(\tau,-)$.  This variety has a graded vector bundle $\mathcal{END}^{\leq 0}$ on it, the graded vector bundle of endomorphisms of degree less than or equal to zero.  Over a point $x$ of $\VV_{l,\tau}$, corresponding to a twisted object $M\in \tw_l(\mathcal{C})$, points in the fibre correspond to potential morphisms of nonunital $A_{\infty}$-algebras from $k$ to $\End_{\tw_l(\mathcal{C})}(M)$.  They are actual morphisms if they satisfy the compatibility equations (\ref{morph-compatibility}).  Since $k$ and $\End_{\tw_l(\mathcal{C})}(M)$ are both concentrated in finitely many degrees, it follows that only finitely many of these equations are nonzero.  It follows that there is a subvariety of $\Tot(\mathcal{END}^{\leq 0})$, the total space of the underlying ungraded vector bundle, parameterising strictly unital functors from $D$, such that the target of the object of $D$ is parameterised by $\VV_{l,\tau}$.  We construct the required bundle of $\mathcal{C}$-modules over $\VV_l^+$, which we denote by $L_{\mathfrak{V}_l}$, by taking the limit of this functor.
\end{proof}
\begin{prop}
\label{dimdef}
There is a quasi-equivalence of categories 
\[
-^{\diamond}:\Perf(\mathcal{C}\lModi)\rightarrow \Perf(\rModi\mathcal{C})^{\op}.
\]
This quasi-equivalence is induced by an isomorphism of ind-varieties $\lambda:\VV^+_r\rightarrow\VV^+_l$ (where we use the same symbol $\lambda$ here as we did to denote the isomorphism $\lambda:\VV_r\rightarrow\VV_l$).
\end{prop}
\begin{proof}
The equivalence 
\[
-^{\diamond_{\tw}}:\tw_r(\mathcal{C})\rightarrow \tw_l(\mathcal{C})^{\op}
\]
induces a quasi-equivalence $\Fun_{\infty}(D,\tw_r(\mathcal{C}))\rightarrow \Fun_{\infty}(D,\tw_l(\mathcal{C})^{\op})$, and the equivalence $D\rightarrow D^{\op}$ induces equivalences
\[
\Fun_{\infty}(D,\tw_l(\mathcal{C})^{\op})\rightarrow \Fun_{\infty}(D^{\op},\tw_l(\mathcal{C})^{\op})\rightarrow \Fun_{\infty}(D,\tw_l(\mathcal{C}))^{\op}.
\]
The second statement is then clear, by construction.
\end{proof}
\begin{prop}
There is a constructible vector bundle $\mathcal{HOM}$ lying over \mbox{$\VV_l^+\times \VV_l^+$}, with homology over $(x,y)$ calculating the homology of the space of homomorphisms between the two perfect modules parameterised by $x$ and $y$.
\end{prop}
\begin{proof}
This is given by the bar resolution and the bundle of $\mathcal{C}$-modules considered above.
\end{proof}
\begin{cor}
\label{perflag}
There is an ind-variety $\VV^+_{l,3}$ parameterising triangles of objects in $\Perf(\mathcal{C}\lMod)$, and up to quasi-isomorphism, every triangle in $\Perf(\mathcal{C}\lMod)$ occurs in this family.
\end{cor}
\begin{proof}
Define the ind-variety $\VV^+_l$ as above.  Consider two components $\VV^+_{l,\tau_1}$ and $\VV^+_{l,\tau_2}$, with $\VV_{l,\tau_i}^+$ parameterising functors whose target is an object of $\VV_{l,\tau_i}$.  There is an ind-constructible vector bundle $\mathcal{NAT}$ lying over $\VV^+_l\times \VV_l^+$ such that the homology of the fibre over a point $x,y$ computes the homotopy classes of natural transformations from the functor $y$ to the functor $x$.  From Definition \ref{nattrans} we deduce that $\mathcal{NAT}^0|_{\VV^+_{l,\tau_1}\times\VV^+_{l,\tau_2}}$ is finite-dimensional, and so the total space of $\Zo^0(\mathcal{NAT})$ is an ind-variety, as required.  The analogue of the projection to the mapping cone ($p_3$ in Theorem \ref{triangparam}) is given by functoriality of the limit and functoriality of the mapping cone.
\end{proof}

\chapter{Calabi-Yau $A_{\infty}$-algebras}
\label{quiveralgebras}
\section{Preliminaries on Calabi-Yau quiver algebras}
We introduce here a class of $A_{\infty}$-algebras that includes our two motivating examples -- the non-commutative crepant resolutions of \cite{NonComCrep}, \cite{FlopsAndNoncommRings}, \cite{Conifold} and the cluster collections studied in \cite{Kellerdg}, \cite{KellerMutations}, \cite{BridgelandM}, \cite{KS}.  Throughout this chapter, we define
\begin{equation}
\label{Sringdef}
S:=\bigoplus_{v\in I} k
\end{equation}
for some finite set $I$, which we will, rather suggestively, call our set of vertices.  We give this the algebra structure given by the direct sum of $|I|$ copies of the standard unital $k$-algebra structure on $k$.  This, then, is a unital algebra, and so trivially a unital $A_{\infty}$-algebra, after we decree that the whole of $S$ lives in degree zero.  We consider $S$ as a $S$-unital $A_{\infty}$-algebra, as in Definition \ref{unitalalg}.
\bigbreak
Consider the category $S\biMods S$ of $S$-unital $S$-bimodules (see Definition \ref{Sunital}).  This is identified with the category of ordinary differential graded $S$-bimodules, in the non-$A_{\infty}$ sense.  Let $(M,m_{M,n})$ be a non-unital $A_{\infty}$-algebra object in this category.  Consider the $S$-bimodule $M\oplus S$.  This has the structure of a strictly $S$-unital $A_{\infty}$-algebra:  we define $m_{M\oplus S,1}$ via the differentials on the two objects, and we define $m_{M\oplus S,2}$ via the pre-existing multiplication on $M$ and $S$, and via the ordinary left and right $S$-module structures on $M$.  All higher multiplications $m_{M\oplus S,i}$, for $i\geq 3$, are equal to $m_{M,i}$ when evaluated on $M^{\otimes i}$, and are equal to zero otherwise.  By construction, this algebra has a natural $S$-augmentation given by the projection $M\oplus S\rightarrow S$.
\begin{defn}
\label{quidef}
An $A_{\infty}$-quiver algebra with (finite) vertex set $I$ is given by the augmented $A_{\infty}$-algebra over $S$ with the underlying vector space $M\oplus S$ constructed as above, where $M$ is a finitely generated $A_{\infty}$-algebra object in the category of $S$-unital $S$-bimodules $S\biMods S$.
\end{defn}
\begin{rem}
There are three different ways to think about these objects: as augmented $A_{\infty}$-algebras over $S$ (see Definition \ref{Saug}), as augmented $A_{\infty}$-algebra objects in the category of ordinary $S$-bimodules (see Definition \ref{Caug}), and as augmented $A_{\infty}$-categories with set of objects $I$ (see Definition \ref{cataug}).
\end{rem}
\begin{examp}
Let $E\in S\bimods S$ be a finite-dimensional $S$-bimodule.  Consider the non-unital free tensor algebra on $E$, defined by
\begin{equation*}
\Tnu(E):=E\oplus (E\otimes_S E)\oplus (E\otimes_S E\otimes_S E)\oplus\ldots
\end{equation*}
Now let $I\subset \Tnu(E)$ be a two-sided ideal, and consider the (ordinary) finitely generated (in the ordinary sense) algebra object $\Tnu(E)/I$ in $S\biMods S$.  This is also a finitely generated $A_{\infty}$-algebra object (see Definition \ref{fingen}) in $S\biMods S$, since there is a strict morphism from the free non-unital $A_{\infty}$-algebra object generated by $E$ in $S\biMods S$ to $\Tnu(E)$ killing higher multiplications, that is surjective on underlying vector spaces, and clearly also a strict morphism from $\Tnu(E)$ to $\Tnu(E)/I$.  As such the $S$-bimodule 
\[
(\Tnu(E)/I)\oplus S
\]
acquires the structure of an $A_{\infty}$-quiver algebra with vertex set $I$.  So an ordinary quiver algebra with relations, with `edges' $E$, is an example of an $A_{\infty}$-quiver algebra, which is finitely generated when the number of arrows is finite, where the arrows are in the usual way just a basis for $E$ living in $\bigcup_{i,j\in I}e_i\cdot E\cdot e_j$.
\end{examp}
In the case of ordinary algebras there is a well-known correspondence between quiver algebras and small categories, with the vertices $I$ forming the objects of a category, and the arrows generating the morphisms.  We first start with the data of a quiver $\Gamma$ and some relations $R$, which are linear sums of paths in the graph.  One then forms the $S$-bimodule $E$, where we give $e_j\cdot E\cdot e_i$ a basis given by the set of arrows from $i$ to $j$.  Associated to this bimodule is a \textit{free} $k$-linear category $\mathcal{C}$ -- morphisms are given by (linear sums of) chains of arrows from the quiver.  Finally we set some of the morphisms to be equal to zero -- these prescriptions come from the relations $R$.  This gives us a category we denote by $\mathcal{C}_{\Gamma, R}$, and the first thing one shows is that left modules over the algebra $A$ associated to the data $(\Gamma,R)$ are essentially the same thing as functors $\mathcal{C}_{\Gamma,R}\rightarrow \vect$.
\bigbreak
The situation in the $A_{\infty}$ case is exactly analogous.  Let $M\oplus S$ be given the structure of an $A_{\infty}$-quiver algebra.  Assume we have some finite-dimensional bimodule $E\in S\bimods S$ such that there is a strict morphism of $S\bimods S$ $A_{\infty}$-algebra objects 
\[
f:\Free_{\infty,\nonu}(E)\rightarrow M,
\]
which is surjective on underlying $S$-bimodules, from the free unital $A_{\infty}$-algebra object $\Free_{\infty,\nonu}(E)$ to $M$.  Then as before we form a quiver, with vertices given by $I$ and with arrows from $i$ to $j$ given by a basis of $m_2(m_2(e_j,E),e_i)$.  These generate a free $A_{\infty}$-category, with relations, again labelled $R$, given by the kernel of $f$.  We think of this quiver with relations as a category with objects the vertices $I$, and we label this category $\mathcal{C}_{\Gamma,R}$.
\smallbreak
In the case of ordinary, ungraded quiver algebras, we define a representation of a quiver as a functor from $\mathcal{C}_{\Gamma,R}$ to $\vect$.  The first thing one shows here is that the category of modules for the algebra $A$ associated to $\Gamma$ and $R$ is equivalent to the category of representations of $(\Gamma,R)$.  In the $A_{\infty}$ case this situation is replicated.
\begin{prop}
\label{equicatd}
Let $(\Gamma$, $R)$ be a pair of a directed graph and higher relations on $\Gamma$.  Let $\mathcal{C}_{\Gamma,R}$ be the associated $A_{\infty}$-category, and let $A$ be the associated $A_{\infty}$-algebra.  The category of unital $A_{\infty}$-functors from $\mathcal{C}_{\Gamma,R}$ to the category $\dgvect$ is isomorphic to the category of finite-dimensional $S$-unital $A$-modules.
\end{prop}

\begin{defn}
From now on, where $A$ is understood as an $A_{\infty}$-quiver algebra over a fixed vertex set $I$, we \textit{define} $A\lMod$ to be the category of $S$-unital $A$-modules $A\lMod_S$.  Similarly we define $A\lModis$ to be the full subcategory of $A\lModinu$ with objects the $S$-unital modules (see the end of Section \ref{enrichments}).
\end{defn}
\begin{rem}
\label{augfun}
If $\mathcal{C}$ is the $A_{\infty}$-category associated to a quiver with relations, then $\mathcal{C}$ has a natural augmentation: by construction, for each $x\in \ob(\mathcal{C})$ we have that 
\begin{equation}
\label{pqw}
\End_{\mathcal{C}}(x)\cong e_xMe_x\oplus k,
\end{equation}
the factor of $k$ coming from the formal addition of an identity.  This enables us to define a natural augmentation of this $A_{\infty}$-category, and putting these together we obtain our augmentation of $\mathcal{C}$.  In the quiver language the copy of $k$ in (\ref{pqw}) is the usual canonical 1-dimensional module associated to the vertex $x$, which we denote $s_x$. 
\end{rem}
Let $A$ be an $A_{\infty}$-quiver algebra.  Then, since $A$ is an $A_{\infty}$-algebra, we can take its homology algebra $\Ho^{\bullet}(A)$.  This inherits, naturally, the structure of an ordinary quiver algebra, and therefore also the structure of an $A_{\infty}$-quiver algebra, with vanishing higher multiplications.
\medbreak
Let $T$ be an $A_{\infty}$-object in the category $S\biMods S$, where $S$ is considered as an $S$-unital $A_{\infty}$-algebra.  As before we obtain an $S$-unital $A_{\infty}$-algebra structure on $T\oplus S$.  We say that this is a formal $A_{\infty}$-quiver algebra if there is some $E\in S\biMods S$ with associated non-unital $A_{\infty}$-algebra object $\Free_{\infty,\nonu,\mathrm{formal}}(E)$, and a strict morphism
\[
f:\Free_{\infty,\nonu,\mathrm{formal}}(E) \rightarrow T
\]
which is surjective on underlying vector spaces, and such that $\Ho^{\bullet}(T\oplus S)$ is complete with respect to the ideal generated by $\Ho^{\bullet}(f)(\Free_{\infty,\nonu,\mathrm{formal}}(E))=\Ho^{\bullet}(T)$.  We say that $T\oplus S$ is a finitely generated formal $A_{\infty}$-quiver algebra if the same condition holds, but for finite-dimensional $E\in S\bimods S$.
\begin{prop}
\label{retnice}
Let $A=M\oplus S$ be a finitely generated formal $A_{\infty}$-quiver algebra such that $\Ho^{\bullet}(A)$ is a nonpositively graded quiver algebra.  Then there is a quasi-equivalence of categories
\[
\xymatrix{
\langle s_i|i\in I\rangle_{\triang}\ar[r]^(.55){\sim}& A\lmodi,
}
\]
where the $s_i$ are as defined in Remark \ref{augfun}, and an equivalence of categories
\[
A\lmodinilp \rightarrow A\lmodi,
\]
where $A\lmodinilp$ is the full $A_{\infty}$-subcategory of $A\lmodi$ containing those $N$ such that 
\[
(\Ho^{\bullet}(M))^{t}\Ho^{\bullet}(N)=0
\] for some $t$.
\end{prop}
\begin{proof}
Since in the first statement we are working up to quasi-equivalence, we may take a minimal model for $A$, and we may restrict attention to the full subcategory of $A$-modules containing only minimal modules.  Let $M$ be such a minimal module.  Then since $A$ is concentrated in nonpositive degrees, $M$ admits a filtration of $A$-modules 
\[
\ldots\subset M^{\leq i}\subset M^{\leq i+1}\subset\ldots
\]
Each factor of this filtration is an $A^0$-module (note that under our minimality assumption, $A^0$ is an ordinary quiver algebra completed according to path length).  The first result follows, and then so does the second, by induction.
\end{proof}

\begin{defn}{\cite{KSnotes}}
An $A_{\infty}$-quiver algebra is compact if its homology algebra is finite-dimensional.
\end{defn}
This is equivalent to the $A_{\infty}$-category $\CC_{\Gamma,R}$ (see the discussion before Proposition \ref{equicatd}) being finite-dimensional, in the sense of Definition \ref{findimcat}.
\begin{examp}
Let $L_1,\ldots,L_n$ be bounded complexes of coherent sheaves on a smooth projective scheme $X$, with $\Hom_{\mathrm{D}^b(X)}(L_i,L_i)\cong k$ for every $i$, where $\mathrm{D}^b(X)$ is the bounded derived category of complexes of coherent sheaves on $X$.  We assume that $\Ext^n(L_i,L_j)=0$ for $n<0$, and that no automorphism of any of the $L_i$  factors through a morphism $f:L_i\rightarrow L_j$ for $j\neq i$.  Then the full subcategory of $\mathrm{D}^b(X)$ with objects consisting of the $L_1,\ldots,L_n$ determines a (ordinary) graded quiver algebra $R$.  The $a$th graded piece of the vector space $e_i\cdot R\cdot e_j$ is given by $\Ext^{a}_{\Coh(X)}(L_j,L_i)$.  This algebra is, by assumption, naturally an augmented $S$-algebra, for $I$ the set $\{L_1,\ldots,L_n\}$.  There is a well defined notion of the (generally non-unique) enriched derived category $\mathrm{D}^b_{\infty}(X)$ (e.g. as in \cite{Pol04}), which incorporates the data of Massey products, and we can consider also the full $A_{\infty}$-subcategory of $\mathrm{D}^b_{\infty}(X)$ with objects given by the $L_1,\ldots,L_n$.  After taking a minimal model, we obtain an $A_{\infty}$-quiver algebra $R'$.  By definition, $\Ho^{\bullet}(R')=R$, and since morphism spaces in $\mathrm{D}^b(X)$ are finite-dimensional graded vector spaces, both these algebras are compact.
\end{examp}
\begin{examp}
Consider $\mathcal{O}_{\mathbb{P}^1}$ and $\mathcal{O}_{\mathbb{P}^1}(1)$ in $\Coh(\mathbb{P}^1)\subset \mathrm{D}^b(\mathbb{P}^1)$.   Again, denote the associated quiver algebra $R$.  Then $\Ho^0(R)$ is the free path algebra of the Kronecker quiver, with two vertices labelled by $\mathcal{O}_{\mathbb{P}}$ and $\mathcal{O}_{\mathbb{P}}(1)$, and with two arrows going from the first vertex to the second.
The full Yoneda algebra $R'$ is just $R$ again, since there are no $\Ext^1$s.  It follows that the enriched algebra $R'_{\infty}$ is just the same as $R$.  There is an equivalence of derived categories
\[
\xymatrix{
\mathrm{D}^b(R\lmod)\ar[r]^>>>>>{\sim} &\mathrm{D}^b(\mathbb{P}^1)
}
\]
which, since $R\cong R'_{\infty}$, lifts to an equivalence of $A_{\infty}$-categories
\[
\xymatrix{
\mathrm{D}^b_{\infty}(\mathbb{P}^1)\ar[r]^>>>>>{\sim}& R\lmodi.
}
\]
\end{examp}
\begin{defn}{\cite{CostTFT}}
\label{CYcatdef}
A Calabi-Yau category of dimension $n$ is an $A_{\infty}$-category $\mathcal{C}$, such that for each pair $A,B\in\mathcal{C}$ there exists a nondegenerate pairing
\[
\langle\bullet,\bullet\rangle_{B,A}:\Hom_{\mathcal{C}}(A,B)\otimes \Hom_{\mathcal{C}}(B,A)\rightarrow k[-n]
\]
such that $\langle\bullet,\bullet\rangle_{A,B}$ is obtained from $\langle\bullet,\bullet\rangle_{B,A}$ by precomposing with the symmetry isomorphism in $\dgvect$
\[
\Hom_{\mathcal{C}}(A,B)\otimes \Hom_{\mathcal{C}}(B,A)\cong \Hom_{\mathcal{C}}(B,A)\otimes \Hom_{\mathcal{C}}(A,B).
\]
For $A_0,\ldots,A_{n-1}\in \ob(\mathcal{C})$ define the linear functional $W_n$ on
\[
\Hom^{\bullet}(A_{n-1},A_0)[1]\otimes \Hom^{\bullet}(A_{n-2},A_{n-1})[1]\otimes\ldots\otimes \Hom^{\bullet}(A_{0},A_{1})[1]
\]
by $W_{n}=\frac{1}{n}\langle\bullet,\bullet \rangle_{A_0,A_1}(b_{n-1}\otimes\id)$.  The symmetric monoidal structure on $\dgvect$ gives an isomorphism $\phi$ between
\[
\Hom^{\bullet}(A_{n-1},A_0)[1]\otimes \Hom^{\bullet}(A_{n-2},A_{n-1})[1]\otimes\ldots\otimes \Hom^{\bullet}(A_{0},A_{1})[1]
\]
and 
\[
\Hom^{\bullet}(A_{n-2},A_{n-1})[1]\otimes\ldots\otimes \Hom^{\bullet}(A_0,A_1)[1]\otimes \Hom^{\bullet}(A_{n-1},A_0)[1].
\]
We ask, in addition, that $W_{n}=W_{n}\circ \phi$.
\end{defn}
\begin{rem}
\label{tomp}
Note that if $\mathcal{C}$ is an $A_{\infty}$-category which can be given the structure of a Calabi-Yau category, then in particular, the morphism space between any two objects of $\mathcal{C}$ must be finite-dimensional in each degree, a property that is not stable under quasi-equivalence of $A_{\infty}$-categories.
\end{rem}
\begin{defn}
If $A$ is an $A_{\infty}$-algebra object in a symmetric monoidal graded category $\mathcal{D}$ admitting internal $\Hom$s, then we say that it is a $n$-dimensional Calabi-Yau $A_{\infty}$-algebra object if there is a symmetric degree $n$ map
\[
\langle\bullet,\bullet\rangle: A\otimes A\rightarrow \idmon_{\mathcal{D}}
\]
such that the $W_n$ defined as above are cyclically symmetric, and such that the map $\phi(\brackets)$ is an isomorphism, where $\phi:\Hom_{\mathcal{D}}(A\otimes A,\idmon_{\mathcal{D}})\rightarrow \Hom_{\mathcal{D}}(A,\underline{\Hom}_{\mathcal{D}}(A,\idmon_{\mathcal{D}}))$ is given by the presence of internal $\Hom$s. 
\end{defn}
\begin{rem}
Let $\mathcal{C}$ be a Calabi-Yau category.  Then it is easy to see that $\mathcal{C}^{\op}$ possesses a Calabi-Yau structure too.
\end{rem}
\begin{defn}
\label{compactCyalg}
An $n$-dimensional \textit{compact Calabi-Yau} $A_{\infty}$\textit{-quiver algebra} is an $A_{\infty}$-quiver algebra on some vertex set $V$, with the structure of a Calabi-Yau category given on the associated category $\mathcal{C}$ with objects $\ob(\mathcal{C})=V$ and morphisms $\Hom_{\mathcal{C}}(v_i,v_j)=m_2(m_2(e_j,A),e_i)$.
\end{defn}
Many of the categories we work with are categories of modules \textit{over} Calabi-Yau categories.  So the following theroem is essential.
\begin{thm}
\label{goup}
Let $\mathcal{C}$ be a Calabi-Yau category of dimension $n$.  Then $\tw_l(\mathcal{C})$ and $\tw_r(\mathcal{C})$ (as in Definition \ref{twl}) are Calabi-Yau categories of dimension $n$.
\end{thm}
\begin{proof}
We work only with $\tw_r(\mathcal{C})$, the proof for $\tw_l(\mathcal{C})$ is identical (recall also the isomorphism $\tw_r(\mathcal{C})\cong \tw_l(\mathcal{C})^{\op}$ of Proposition \ref{diamonddual}).  Objects of $\tw_r(\mathcal{C})$ are given by $n$-tuples of objects $\tau=(h_{x_1}[m_1],\ldots,h_{x_n}[m_n])$, and matrices $A$ in $h_\mathcal{C}[\mathbb{Z}]$, as defined in Definition \ref{embe}, satisfying the equation
\[
\sum_{i\geq 1} b_i(A,\ldots,A)=0,
\]
while morphisms from $(\tau_1,A)$ to $(\tau_2,A')$ are given by matrices from $\tau_1$ to $\tau_2$ in $h_{\mathcal{C}}[\mathbb{Z}]$.  Given a matrix $M$ from 
\[
\tau_1=(h_{x_1}[m_1],\ldots,h_{x_n}[m_n])
\]
to
\[
\tau_2=(h_{y_1}[p_1],\ldots,h_{y_{n'}}[p_{n'}])
\]
and a matrix $M'$ from $\tau_2$ to $\tau_1$ we define 
\[
\langle M',M\rangle_{\tw_r(\mathcal{C})}=\sum_{\substack{1\leq a \leq n\\ 1\leq b \leq n'}}\langle M'_{a,b},M_{b,a}\rangle.
\]
Nondegeneracy of $\langle \bullet,\bullet \rangle_{\tw_r(\mathcal{C})}$ follows straight from nondegeneracy of $\langle\bullet,\bullet\rangle$.  Cyclic symmetry of $\langle \bullet,\bullet \rangle_{\tw_r(\mathcal{C})}$ follows from the definition of composition in $\tw_r(\mathcal{C})$ (see (\ref{twmult})) and cyclic symmetry of $\langle\bullet,\bullet\rangle$.
\end{proof}
\begin{defn}[\cite{KS}]
An ind-constructible Calabi-Yau category of dimension $n$ is given by an ind-variety $\mathcal{M}$, together with an ind-constructible finite-dimensional graded vector bundle $\mathcal{HOM}$ on $\mathcal{M}\times \mathcal{M}$, together with morphisms $\pi_{0,1}^*(\mathcal{HOM})\otimes\ldots \otimes \pi_{i-1,i}^*(\mathcal{HOM})\rightarrow \pi_{0,i}^*(\mathcal{HOM})$ satisfying the obvious $A_{\infty}$ compatibility conditions, and a morphism $\langle \bullet,\bullet \rangle:\mathcal{HOM}\otimes \sw^*(\mathcal{HOM})\rightarrow \id_{\mathcal{M}\times\mathcal{M}}[-n]$, satisfying the obvious cyclic invariance property (as in Definition \ref{CYcatdef}), inducing an isomorphism $\mathcal{HOM}\cong \sw^*(\mathcal{HOM})^*[-n]$, where $\sw:\mathcal{M}\times\mathcal{M}\rightarrow \mathcal{M}\times\mathcal{M}$ is the isomorphism swapping the two copies of $\mathcal{M}$, and $\sw^*(\mathcal{HOM})^*$ is the vector dual of $\sw^*(\mathcal{HOM})$.
\end{defn}
\begin{thm}
Let $\mathcal{C}$ be a finite-dimensional Calabi-Yau category of dimension $n$.  Then $\triang_{\mathcal{C}\lmodi}(\mathcal{C})$ and $\triang_{\rmodi\mathcal{C}}(\mathcal{C})$ are quasi-equivalent to ind-constructible Calabi-Yau categories of dimension $n$.
\end{thm}
\begin{proof}
This follows from Propositions \ref{trac}, \ref{indvariety}, and the structure defined in the proof of Theorem \ref{goup}.
\end{proof}

\begin{defn}
Let $\mathcal{C}$ and $\mathcal{D}$ be ind-constructible Calabi-Yau categories of dimension $n$, with the same underlying ind-varieties of objects, $\mathcal{M}$.  Then a morphism $f$ of ind-constructible Calabi-Yau categories $\mathcal{C}\cong \mathcal{D}$ is given by a morphism $f_1:\mathcal{HOM}_{\mathcal{C}}\rightarrow \mathcal{HOM}_{\mathcal{D}}$, compatible with the inner products on the two constructible vector bundles, along with higher morphisms $f_i:\pi_{0,1}^*(\mathcal{HOM}_{\mathcal{C}})\otimes\ldots \otimes \pi_{i-1,i}^*(\mathcal{HOM}_{\mathcal{C}})\rightarrow \pi_{0,i}^*(\mathcal{HOM}_{\mathcal{D}})$ satisfying the compatibility conditions of an $A_{\infty}$-morphism (see Definition \ref{algmorph}).  We require also that 
\begin{equation}
\label{kajdef}
\sum_{a+b=n}\langle f_a,f_b\rangle_{\mathcal{D}}=0
\end{equation}
for all $n\geq 3$.  The morphism $f$ is an isomorphism if $f_1$ is.
\end{defn}
\begin{rem}
The condition \ref{kajdef} is considered by Kajiura in \cite{kaj03}.  Converted into the language of symplectic forms on $\KK$, introduced below, the condition that $f_1$ preserves the inner product corresponds to the condition that the constant term of the pulled back noncommutative symplectic form is unchanged.  The condition (\ref{kajdef}) corresponds to the condition that the pulled back symplectic form is still constant.
\end{rem}
\begin{thm}[Minimal model theorem for ind-constructible Calabi-Yau categories]
\label{cycmm}
Let $\mathcal{C}$ be an ind-constructible Calabi-Yau category of dimension $n$.  Then $\mathcal{C}$ is isomorphic to an ind-constructible Calabi-Yau category $\mathcal{D}$, with compositions $m_i$ respecting the decomposition $\mathcal{HOM}_{\mathcal{D}}\cong \mathcal{HOM}_{\min}\oplus V$, satisfying
\begin{enumerate}
\item $m_1|_{\mathcal{HOM}_{\min}}=0,$
\item $m_{i}|_{\pi_{0,1}^*(V)\otimes\ldots\otimes \pi_{i-1,i}^*(V)}=0$ for all $i\geq 2$, and
\item $\Ho^{\bullet}(V)=0$.
\end{enumerate}
\end{thm}
\begin{proof}
The proof for this Theorem is essentially contained in \cite{kaj03}, an excellent reference for cyclic $A_{\infty}$-algebras.  There is, though, a possibility of mishap when transposing the proof of Theorem 5.15 of that paper, in that we have to introduce some analogue of the free tensor algebra associated to the underlying vector space of a cyclic $A_{\infty}$-algebra.  So we will recount the ideas of \cite{kaj03}, translating them into an appropriate ind-constructible framework.
\smallbreak
Firstly, fix the underlying ind-variety $\mathcal{M}=\ob(\mathcal{C})=\ob(\mathcal{D})$.  We denote by $\mathcal{HOM}$ the constructible vector bundle of homomorphisms on the ind-variety $\mathcal{M}\times\mathcal{M}$.  Consider the ind-constructible vector bundle $\homdual1:=\sw^*(\mathcal{HOM}[1])^*$, the vector dual of the pullback along the isomorphism that swaps arguments (every constructible vector bundle will be locally finite-dimensional).  The bracket $\langle \bullet,\bullet\rangle$ provides an isomorphism
\[
\homdual1\cong \HOM [2].
\]
Consider now the ind-variety $\mathcal{M}^n\cong (\mathcal{M}\times\mathcal{M})\times_{\mathcal{M}}\ldots\times_{\mathcal{M}} (\mathcal{M}\times\mathcal{M})$, where here and elsewhere we adopt the convention that $\MM^n\times_{\MM}\MM^m$ is the pullback in the following diagram
\[
\xymatrix{
\MM^n\times_{\MM}\MM^m\ar[d]\ar[r]&\MM^m\ar[d]^{\pi_1}\\
\MM^n\ar[r]^{\pi_n}&\MM.
}
\]
We define the constructible vector bundle $\KK$ on $\ch(\MM):=\coprod_{i\geq 2} \MM^i$, which is given by
\begin{equation}
\label{SI}
\bigoplus_{\substack{a_1<\ldots <a_i\in\{1,\ldots,n\}\\a_1=1,a_i=n}}\pi_{a_1,a_2}^*(\homdual1)\otimes\ldots\otimes \pi_{a_{i-1},a_i}^*(\homdual1)
\end{equation}
when restricted to $\MM^n$.  Consider the natural morphism given by concatenation of chains of terms in $\ob(\mathcal{C})$
\[
\mm:\ch(\mathcal{M})\times_{\mathcal{M}}\ch(\mathcal{M})\rightarrow \ch(\mathcal{M})
\]
where the fibre product is the pullback of the diagram
\[
\xymatrix{
\ch(\mathcal{M})\times_{\mathcal{M}} \ch(\mathcal{M})\ar[d]^-{\pi_1}\ar[r]^-{\pi_2} &\ch(\mathcal{M})\ar[d]^{\pi_{\mathrm{first}}}\\
\ch(\mathcal{M})\ar[r]^{\pi_{\mathrm{last}}}\ar[r] &\mathcal{M}
}
\]
and the projection $\pi_{\mathrm{first}}$ is the morphism projecting each component $\mathcal{M}^i$ onto its first copy of $\mathcal{M}$, and $\pi_{\mathrm{last}}$ is defined similarly.  We denote by $\Delta_{\KK}$ the natural morphism
\[
\Delta_{\KK}:\mm_*(\pi_{1}^*(\KK)\otimes \pi_2^*(\KK))\rightarrow \KK.
\]
Note that this is a morphism of locally finite-dimensional constructible vector bundles.  We say $b':\KK\rightarrow \KK$ is a derivation with respect to $\Delta_{\KK}$ if it satisfies the condition
\[
\Delta_{\KK}\circ(\mm_*(\pi_1^*(b')\otimes \id)+\mm_*(\id\otimes \pi_2^*(b')))-b'\circ \Delta_{\KK}=0.
\]
Let $b_1<\ldots<b_j\in\{1,\ldots,n\}$, and let $a_1<\ldots<a_i\in\{1,\ldots,j\}$, with $a_1=1, a_i=j, b_1=1, b_j=n$.  Then $b'$ induces a map of constructible vector bundles on $\mathcal{M}^n$:
\[
b_{\mathrm{part}}':\pi^*_{b_{a_1},b_{a_2}}(\homdual1)\otimes\ldots\otimes \pi^*_{b_{a_{i-1}},b_{a_i}}(\homdual1)\rightarrow\pi^*_{b_1,b_2}(\homdual1)\otimes\ldots\otimes\pi^*_{b_{j-1},b_j}(\homdual1).
\]
Now let 
\begin{equation}
\label{iotadef}
\iota:\{1,\ldots,m\}\rightarrow \{1,\ldots,n\}
\end{equation}
 be an order preserving injection, whose image contains $\{b_1,\ldots,b_j\}$.  We define $\tilde{b}_l=\iota^{-1}(b_l)$.  Then $\iota$ induces a projection $\MM^n\rightarrow \MM^m$.  We consider also the morphism
\[
\xymatrix{
b'_{\mathrm{part}}:\pi^*_{\tilde{b}_{a_1},\tilde{b}_{a_2}}(\homdual1 )\otimes\ldots\otimes\pi^*_{\tilde{b}_{a_{i-1}},\tilde{b}_{a_i}}(\homdual1)\ar[r]&\pi^*_{\tilde{b}_{1},\tilde{b}_{2}}(\homdual1)\otimes\ldots\otimes\pi^*_{\tilde{b}_{j-1},\tilde{b}_{j}}(\homdual1),
}
\]
and the following diagram:
\[
\xymatrix{
\pi^*_{b_{a_1},b_{a_2}}(\homdual1)\otimes\ldots\otimes \pi^*_{b_{a_{i-1}},b_{a_i}}(\homdual1)\ar[d]^{b'_{\mathrm{part}}} \ar[r]^-{\cong}& \pi_{\iota}^*(\pi^*_{\tilde{b}_{a_1},\tilde{b}_{a_2}}(\homdual1)\otimes\ldots\otimes\pi^*_{\tilde{b}_{a_{i-1}},\tilde{b}_{a_i}}(\homdual1))\ar[d]^{\pi_{\iota}^*(b'_{\mathrm{part}})}\\
\pi^*_{b_{1},b_{2}}(\homdual1)\otimes\ldots\otimes \pi^*_{b_{j-1},b_{j}}(\homdual1)\ar[r]^-{\cong}& \pi_{\iota}^*(\pi^*_{\tilde{b}_{1},\tilde{b}_{2}}(\homdual1)\otimes\ldots\otimes\pi^*_{\tilde{b}_{j-1},\tilde{b}_{j}}(\homdual1)).
}
\]
If all such diagrams commute, we say that $b'$ is a coherent morphism $\KK\rightarrow\KK$, and in the present case, where $b'$ is also a derivation, we say that it is a coherent derivation (even though the mangling of English given by `herent derivation' gives a more consistent notation).\smallbreak
Since we have set things up to involve only locally finite-dimensional constructible vector bundles, we may dualize everything in sight, and define the notion of a cocoherent coderivation.  Then coherent derivations $b'$ on $\mathcal{\KK}$ satisfying $b'^2=0$ are in 1-1 correspondence with $A_{\infty}$-structures on $\HOM$.
Consider the ind-constructible variety $\MM^n$.  Since $\MM$ is ind-constructible, $\MM^n$ has a locally finite constructible decomposition with components given by terms of the form $X_1\times\ldots X_{n}$ with $X_i$ a term in the constructible decomposition of $\MM$.  It follows that $\MM$ carries an action of $\ZZ_{n}$, given by cyclic permutation.  There is a further locally finite constructible decomposition of $\MM^{n}$ such that the $\ZZ_{n}$-action factors through a free $\ZZ_t$-action on each piece, for $t|n$, and we can assume each piece to be affine and smooth, from which it follows that the quotient scheme of each $\mathbb{Z}_{n}$-equivariant piece exists and is smooth.  We define the ind-constructible vector bundle $\cc$ by setting its restriction to $\MM^n$ to be
\[
\bigoplus_{a_1<\ldots<a_i\in\{1,\ldots,n\}}\pi_{a_1,a_2}^*(\homdual1)\otimes\ldots \otimes \pi_{a_{i-1},a_{i}}^*(\homdual1)\otimes \pi_{a_i,a_1}^*(\homdual1),
\]
and we define the restriction of $\Kcyc$ to $\MM^{n}_{\cyc}:=\MM^{n}/\ZZ_{n}$ to be the constructible vector bundle of $\ZZ_{n}$-invariant sections.  Let $\iota$ be as in (\ref{iotadef}), but without the restriction that $1$ is mapped to $1$ and $m$ to $n$.  There is an obvious cyclic symmetrization map from sections of $\Kcyc$ to sections of $\cc$, and we say a global section $S_{\cyc}\in\Gamma(\Kcyc)$ is coherent if its cyclic symmetrization is.  Let $S_{\cyc}$ be a coherent section of $\Kcyc$, with cyclic symmetrization $S$.  Then we next consider the Hamiltonian vector field associated to $S_{\cyc}$.  First, we define the coherent derivation $(\bullet,S_{\cyc})$ via its action on each $\pi_{1,n}^*(\homdual1)$, the constructible vector bundle on $\MM^n$.  For $a_1<\ldots<a_i\in\{1,\ldots,n\}$, with $a_1=1$ and $a_i=n$, we define a morphism
\[
\tau:\pi_{a_1,a_2}^*(\homdual1)\otimes...\otimes \pi_{a_{i-1},a_i}^*(\homdual1)\otimes \pi_{a_i,a_1}^*(\homdual1) \otimes \pi_{1,n}^*(\homdual1)\rightarrow \pi_{a_1,a_2}^*(\homdual1)\otimes...\otimes \pi_{a_{i-1},a_i}^*(\homdual1)
\]
via $\tau=\id^{\otimes (i-1)}\otimes\langle\bullet,\bullet\rangle$.  Extending linearly we get our desired morphism on each $\pi_{1,n}^*(\homdual1)$, and one can check that these morphisms define a coherent graded morphism $\KK\rightarrow \KK$ precisely if $S_{\cyc}$ is homogeneous of degree $-1$ -- we assume this to be the case for $S_{\cyc}$ from now on.\smallbreak
Next, define the Hamiltonian vector flow to be given by 
\[
e^{(\bullet,S_{\cyc})}:=\id+\frac{1}{k!}(\bullet,S_{\cyc})^k.
\]
Note that as long as the restriction of $S$ to $\MM^2$ is zero, this is a finite sum of morphisms of constructible vector bundles, upon restriction to each $\MM^n$, and so it is well-defined.  One can show (see \cite{kaj03} Section 4.3) that $e^{(\bullet,S_{\cyc})}$ preserves the noncommutative symplectic structure on $\Kcyc$ given by the Calabi-Yau pairing $\brackets$.  We now have everything in place to perform homological perturbation on our ind-constructible category.\smallbreak
The bracket $\brackets$ corresponds to a constant symplectic form (in the language of \cite{kaj03}, and \cite{KSdef}) on $\KK$, which in turn gives rise to the Poisson bracket $(\bullet,\bullet)$ used above.  Given this symplectic form, the data of an ind-constructible 3-dimensional Calabi-Yau $A_{\infty}$-category is given precisely by a coherent section $S_{\cyc}\in\Gamma(\Kcyc)$, called the \textit{action} (in the constant case, i.e. the one in which the symplectic form comes from a nondegenerate bracket on the underlying algebra, this is especially familiar, corresponding to the fact that the bracket establishes an isomorphism between the underlying vector space and its dual).  An isomorphism of ind-constructible Calabi-Yau $A_{\infty}$-categories, given by actions $S_{\cyc,1}$ and $S_{\cyc,2}$, is given by a coherent isomorphism $\mathcal{F}:\KK\rightarrow\KK$, commuting with $\Delta_{\KK}$, such that $\mathcal{F}^*(\omega)=\omega$ and $\mathcal{F}^*(S_{\cyc,2})=S_{\cyc,1}$. 
By the results of Section \ref{cvect}, we may split the constructible vector bundle $\HOM$ into $\HOM\cong\HOM_{\min}\oplus V_1\oplus V_2$, where $m_1$ acts trivially on $\HOM_{\min}$, and maps $V_1$ isomorphically onto $V_2$.  The strategy of the proof then is as in \cite{kaj03}.  One can assume, further, that these constructible vector bundles are trivialized, giving coordinates $x_1,\ldots,x_a$ on $\HOM_{\min}$, $y_1,\ldots,y_b$ on $V_1$ and $z_1,\ldots,z_c$ on $V_2$.  One writes the action $S_{\cyc}$ as $S_{\cyc}=\sum_{i\geq 2}S_{\cyc,i}$, where on each $\MM^n$, $S_{\cyc,i}$ is the sub bundle of $\KK$ given by the expression (\ref{SI}).  Then by assumption, $S_{\cyc,2}$ contains no $x$ instances.  We assume in addition that for $3\leq i\leq t$, $S_{\cyc,i}$ contains no $y$ or $z$ instances.  Then there is a $H_{\cyc}\in\Kcyc$, such that if we define $S'_{\cyc}=(e^{(\bullet,H_{\cyc})})^*(S_{\cyc})$, then $S'_{\cyc,i}$ contains no $y$ or $z$ instances for $3\leq i\leq t+1$.  Furthermore, $H_{\cyc,i}$ is zero for $i<t$, from which it follows that we can compose an infinite chain of these Hamiltonian flows together, recursively getting $S_{\cyc,i}$ into the right form as $i$ goes to infinity, since after restricting to each $\MM^n$ this is a finite composition of morphisms of constructible vector bundles, which is again a morphism of constructible vector bundles.
\end{proof}

\begin{rem}
By construction, the inclusion and the projection between $\HOM_{\min}$ to $\HOM_{\mathcal{D}}$ can be upgraded to morphisms of ind-constructible $A_{\infty}$-categories.  If $\HOM_{\min}'$ is some alternative summand appearing as in the theorem, we obtain a morphism $\HOM_{\min}\rightarrow \HOM_{\min}'$ which is furthermore an isomorphism of ind-constructible Calabi-Yau $A_{\infty}$-categories.
\end{rem}
\begin{rem}
The decomposition of the potential $W$ has a nice geometric interpretation.  Firstly, it is zero on the image of the differential.  The way to interpret this is that deformations in this direction do not move one along the moduli space of objects (recall that Zariski tangent spaces of fine moduli spaces are generally identified with $\Ext^1$).  Also, the component of $\mathcal{HOM}^1$ that is mapped injectively to $\mathcal{HOM}^2$ has a quadratic potential lying over it, and after differentiating it gives a linear term -- so passing to the scheme-theoretic degeneracy locus, this part gets ignored.  All that is left, then, is $\Ext^1$, and its minimal potential $W_{\min}$.
\end{rem}
Calabi-Yau categories are the objects that motivic Donaldson--Thomas theory needs to get going.  Since we are interested in categories up to quasi-equivalence -- by Requirement \ref{req2}, we should be able to push stack functions through quasi-equivalences of Calabi-Yau categories and get the same motivic Donaldson--Thomas invariant -- if we have a category $\mathcal{D}$ and we would like to start assigning motivic Donaldson--Thomas invariants to stack functions for $\mathcal{D}$, it is good enough to have chosen a quasi-equivalence
\[
\mathcal{D}'\rightarrow \mathcal{D}
\]
from a Calabi-Yau category.  So if $\mathcal{C}$ is a 3-dimensional Calabi-Yau category, we can use the Theorem \ref{goup} and Proposition \ref{trac} to deal with $\langle h^{\mathcal{C}}\rangle_{\triang}$ and $\langle h_{\mathcal{C}}\rangle_{\triang}$.
\begin{rem}
For $\mathcal{C}$ a $n$-dimensional Calabi-Yau category, we have a manageable subcategory of $\langle h^{\mathcal{C}} \rangle_{\triang}$, that is quasi-equivalent to the whole category, namely $\tw_l(\mathcal{C})$.  It is an interesting question whether $\Perf(\mathcal{C}\lModi)$ contains an ind-constructible subcategory that is Calabi-Yau, such that the inclusion is a quasi-equivalence.
\end{rem}
\section{Noncompact Calabi-Yau algebras}
\label{NCCY}
Above we follow the natural suggestion for the definition of a `Calabi-Yau $A_{\infty}$-quiver algebra', which is to use the correspondence between $A_{\infty}$-quiver algebras and small $A_{\infty}$-categories, and then use Definition \ref{CYcatdef}.  This gives us our notion of a \textit{compact} Calabi-Yau $A_{\infty}$-quiver algebra.  There is however a different notion, called, in contrast, a noncompact Calabi-Yau $A_{\infty}$-algebra, that is perhaps equally commonly used, which we recall from \cite{ginz}.  First, assume $A$ is a differential graded algebra (all $A_{\infty}$-algebras are quasi-isomorphic to differential graded algebras, so this entails no loss of generality).  Then $A$ naturally has the structure of a left $A\otimes A^{\op}$-module.  Furthermore, $A\otimes A$ has also the structure of a $A\otimes A^{\op}$-\textit{bimodule} -- by setting 
\[
(a,b)\cdot (c,d)\cdot (e,f)=(a\cdot c\cdot e,f\cdot d\cdot b).
\]
We assume that $A$ is \textit{homologically finite}, that is, it belongs to $\Perf(A\otimes A^{\op}\lMod)$.  Next consider
\[
\RHom_{A\otimes A^{\op}\lMod}(M,A\otimes A)
\]
for an arbitrary left $A\otimes A^{\op}$-module $M$ (i.e. an $A$-bimodule).  This object inherits the structure of a \textit{right} $A\otimes A^{\op}$-module.  Using the natural isomorphism between $A\otimes A^{\op}$ and $(A\otimes A^{\op})^{\op}$, the above object can be considered instead as a \textit{left} $A\otimes A^{\op}$-module.  Using functoriality of $\RHom$, this defines a functor 
\[
-^!:A\biMod A\rightarrow A\biMod A.
\]
\begin{defn}[\cite{ginz}]
We say that $A$ is a \textit{noncompact Calabi-Yau algebra} of dimension $d$ if there is a quasi-isomorphism
\[
\xymatrix{
\phi:A\ar[r]^{\sim}&A^![d]
}
\]
such that $\phi= \phi^![d]$ in the derived category.
\end{defn}
By Lemma 4.1 of \cite{CYTC}, if $A$ is a (noncompact) Calabi-Yau algebra of dimension $d$, then there is a nondegenerate functorial pairing, for all $M_1,M_2\in A\lmod$
\begin{equation}
\label{CYTCpairing}
\Hom_{(A\lmod)_{\ext}}(M_1,M_2)\otimes \Hom_{(A\lmod)_{\ext}}(M_2,M_1)\rightarrow k[-d].
\end{equation}
In the case in which $A$ is a differential graded quiver algebra associated to some quiver $Q$, we would like a lift of this pairing to a Calabi-Yau structure on the full subcategory $\mathcal{C}$ of $A\lmodi$ consisting of just the simple objects $s_i$, for $i$ the vertices of $Q$.  This would make $\mathcal{C}$ the category arising from a compact Calabi-Yau $A_{\infty}$-quiver algebra.  We next discuss a class of differential graded quiver algebras for which we do indeed have such a lift, and for which the Calabi-Yau structure on $\mathcal{C}$ is easily read off.
\bigbreak
For now we will only be interested in Calabi-Yau algebras and Calabi-Yau categories of dimension 3.  Let $Q$ be an (ungraded) quiver, i.e. an $S$-bimodule $E$, where $S$ is the algebra (\ref{Sringdef}), and let $W$ be a formal linear combination of cyclic paths, of length at least 2, of $Q$.  We form a graded quiver $Q'$, as follows (this construction is found in \cite{ginz} originally, though with a slightly different completion, and also in \cite{Kellerdg}):
\begin{enumerate}
\item
The vertices of $Q'$ are the vertices of $Q$.
\item
The degree 0 arrows of $Q'$ from $i$ to $j$ are just the arrows of $Q$ from $i$ to $j$.
\item
For each arrow $\alpha$ from $i$ to $j$ of $Q$, $Q'$ has an arrow $\alpha^*$ of degree -1 from $j$ to $i$.
\item
For each vertex $i$ of $Q$, $Q'$ has a loop $t_i$ at $i$ of degree -2.
\end{enumerate}
We let $E$ be the $S$-bimodule with the above basis of arrows, i.e. $e_j\cdot E\cdot e_i$ has a graded basis given by the arrows from $i$ to $j$.  We take the (ordinary) quiver algebra $\Free_{\nonu,\mathrm{formal}}(E)\oplus S$.  This is the completed (ordinary) quiver algebra freely generated by $E$.  This is given the differential $d$ defined by its action on generators:
\begin{enumerate}
\item
$d(\alpha)=0$ for each $\alpha$ of degree 0.
\item
$d(\alpha^*)=(\partial/\partial \alpha)W$, where $(\partial/\partial \alpha)$ is the noncommutative derivation with respect to $\alpha$, defined on cyclic paths as the sum of terms obtained by cyclically permuting $\alpha$ to the start, and then deleting it (see \cite{ginz}, Section 1, for a little more detail regarding this construction).
\item
$d(t_i)=e_i\sum_{\alpha}[\alpha,\alpha^*]e_i$, where the sum is over all the degree zero arrows $\alpha$.
\end{enumerate}
The result is a noncompact Calabi-Yau differential graded quiver algebra (see the Appendix of \cite{Kel09} for a proof of this), which we denote $\Gamma(Q,W)$.  We denote by $\mathcal{D}_r(Q,W)$ the full subcategory of $\rmodi\Gamma(Q,W)$ consisting of the $s_i$, for $i$ the vertices of $Q$.
By construction, this category is the $A_{\infty}$ Koszul dual of $\Gamma(Q,W)$, considered as an augmented category via the correspondence between $A_{\infty}$-quiver algebras and augmented $A_{\infty}$-categories (see \cite{Koszul} for the description of the $A_{\infty}$ Koszul dual in terms of the bar resolution).  It can, then, be described as follows:
\begin{enumerate}
\item
The objects of $\mathcal{D}_r(Q,W)$ are the vertices $i$ of $Q$.
\item
The only degree zero morphisms of $\mathcal{D}_r(Q,W)$ are the identity morphisms.
\item
The degree 1 morphisms from $i$ to $j$ are given by the vector dual $(e_i E e_j)^*$, while the degree 2 morphisms are given by $e_j E e_i$.
\item
For each vertex $i$ there is a degree 3 endomorphism $\omega_i$.
\item
These are (a $k$-basis for) all the morphisms of $\mathcal{D}_r(Q,W)$.  This category is equipped with differential zero.
\item
If $\alpha$ is a degree 1 homomorphism in $\mathcal{D}_r(Q,W)$, and $\beta$ is a degree 2 homomorphism, then $m_2(\alpha,\beta)=\alpha(\beta)=m_2(\beta,\alpha)$.
\item
The only compositions left to define are the compositions of the degree 1 part of $\mathcal{D}_r(Q,W)$.  Let $\alpha_1^*,\ldots,\alpha_n^*$ be a composable sequence of elements of $E^*$, beginning at vertex $i$ and ending at vertex $j$.  Define
\begin{align*}
\mu:\widehat{kQ_{\cyc}}:=\widehat{kQ}/[\widehat{kQ},\widehat{kQ}]\rightarrow &\widehat{kQ}\\
a_1\ldots a_m\mapsto &\frac{1}{m}(a_1\ldots a_m+a_2a_3\ldots a_ma_1+\ldots +a_ma_1\ldots a_{m-1})
\end{align*}
where here we are completing with respect to path length.  Consider $W_{n+1}$, the part of $W$ containing cyclic words of length $n+1$.  We let $b_{n+1}(\alpha_n^*,\ldots,\alpha_1^*)=(\alpha_n^*,\ldots,\alpha_1^*)\mu(W_{n+1})$, be the element of $e_iEe_j\cong \Hom^2_{\mathcal{D}_r(Q,W)}(j,i)$ defined by contraction (if $W$ is given by going allong some arrows $A$, then some arrows $B$, we write $W=BA$).
\end{enumerate}
It is straightforward to see that $\mathcal{D}_r(Q,W)$ is a Calabi-Yau category of dimension 3.  By Koszul duality there are quasi-equivalences of categories (see sections 4 and 5 of \cite{Koszul})
\begin{equation}
\label{Kosduals}
\xymatrix{
\Perf(\rModi \mathcal{D}_r(Q,W))\ar[r]^-{\sim} & \thick_{\rmodi\Gamma(Q,W)}(s_i|i\in I)\hbox{,   and }\\
\tw_r(\mathcal{D}_r(Q,W))\ar[r]^-{\sim} & \triang_{\rmodi\Gamma(Q,W)}(s_i|i\in I).
}
\end{equation}
Now by Proposition \ref{retnice} we deduce that $\triang_{\rmodi\Gamma(Q,W)}(s_i|i\in I)$ is closed under homotopy retracts, and so we have a quasi-equivalence
\begin{equation}
\xymatrix{
\tw_r(\mathcal{D}_r(Q,W))\ar[r]^-{\sim} & \Perf(\rModi \mathcal{D}_r(Q,W)).
}
\end{equation}
Summing up, we have the following
\begin{prop}
\label{sumup}
Let $\mathcal{D}_r(Q,W)$ be as above.  Then we have quasi-equivalences of categories
\[
\xymatrix{
\tw_r(\mathcal{D}_r(Q,W))\ar[r]^-{\sim} & \Perf(\rModi \mathcal{D}_r(Q,W))\ar[r]^-{\sim}&\rmodi\Gamma(Q,W),
}
\]
and
\[
\xymatrix{
\tw_l(\mathcal{D}_r(Q,W))\ar[r]^-{\sim} & \Perf(\mathcal{D}_r(Q,W)\lModi)\ar[r]^-{\sim}&\Gamma(Q,W)\lmodi,
}
\]
and the categories $\tw_r(\mathcal{D}_r(Q,W))$ and $\tw_l(\mathcal{D}_r(Q,W))$ are ind-constructible 3-dimensional Calabi-Yau categories.
\end{prop}

\chapter{What is `the stack of objects'?}
\label{stackofobjects}
\section{Sheaves of modules}
We are dealing, throughout, with categories of modules over $A_{\infty}$-algebras.  The core case of interest to us is the category of perfect modules over these algebras.  To a certain extent we follow the treatment for the `stack of objects' in such a category outlined in the introduction to \cite{KS}.
\bigbreak
Let $A$ be an $A_{\infty}$-algebra.  First we consider twisted complexes constructed by taking repeated extensions of $A$ by shifts of itself.  Then we deal with (homotopy) retracts.  The trick is to parameterise these two steps algebraically.  The solution to this problem is written down in \cite{KS}.  The first step involves us considering, for each $n$, a closed subvariety of the variety of $n\times n$ upper triangular matrices with entries in the underlying vector space of $A$, cut out by a Maurer-Cartan equation.  The next step involves taking a closed subvariety of the bundle $\End^{\bullet}(M)$, defined by a set of higher $A_{\infty}$ coherence relations.  More details of this construction are provided in Section \ref{perfconstr}.  If we restrict attention to categories of modules over compact Calabi-Yau $A_{\infty}$-quiver categories (see Definition \ref{compactCyalg}) constructed from a quiver $Q$ with superpotential $W$ (see Section \ref{NCCY} for the construction of $\mathcal{D}_r(Q,W)$), we can ignore the second step, after replacing our ind-variety of upper triangular matrices with entries in $A$ with the ind-variety $\VV_l$, of Proposition \ref{kl1} and Definition \ref{objects}, due to Proposition \ref{sumup}.
\bigbreak
Now this object is simply a countable collection of varieties, possessing an obvious bundle of $\mathcal{D}_r(Q,W)$-modules.  Note that there is a great deal of overcounting in the setup as it stands, and the data of automorphism groups is not present.  We can try to fix the first problem by giving the components of the ind-variety $\VV_l^+$ an ordering $\VV_{l}=\coprod_{n\in\mathbb{N}} \mathfrak{V}_{l,\textbf{n}}$, and throwing away objects in $\VV^+_{l,\textbf{n}}$ that appear (up to quasi-isomorphism) in $\VV^+_{l,\textbf{m}}$, for $m<n$, to give a smaller ind-constructible variety $\VV^+_{l,-}$ and we can try to fix the second problem by introducing gauge groups and taking the associated ind-Artin stack.  Since orientation data is provided by an object in the category of ind-constructible super line-bundles on moduli stacks of objects, one might suppose the thing to do is to fashion this ind-stack into a representing stack, on which universal orientation data will live, so that orientation data is simultaneously provided for arbitrary moduli stacks, by pullback.
\smallbreak
This is not quite the approach we take.  The task of turning this ind-variety $\VV_l^+=\coprod \VV^+_{l,\tau}$ into a universal stack $\VV^+_{l,-}/G$ from which we can pull back orientation data via morphisms of stacks is not really workable, or necessary.  As a sign of the difficulties incurred, note that even in the case $A=k$, there is no hope of turning $\coprod \VV^+_{l,\tau}$ into an object that represents the functor sending schemes to families of perfect modules.  Here we can work just with $\VV_l$, since the category of twisted objects is closed under homotopy retracts.  One easily sees that the components of $\VV_{l}$ in this case are just given by functions $f:\mathbb{Z}\rightarrow \mathbb{N}$ such that all but finitely many numbers are set to zero.  It follows that $\VV_{l,-}$ here would just be $\VV_l$.  The function $f$ parameterises the differential graded vector space that has homology of dimension $f(i)$ in degree $i$.  The gauge group is, then, 
\[
G=\prod \GL_{f(i)}(k)
\]
and any differential graded vector bundle that can be obtained (up to quasi-isomorphism) from pullback from $\VV_l/G$ is quasi-isomorphic to its homology.  Furthermore, in order for a differential graded vector bundle to be obtained by pulling back from $\VV_{l,-}$ we would need this homology to be a graded vector bundle.  
\bigbreak
Using the propositions of Section \ref{cvect} there is an obvious workaround in the case $A=k$: since we are only interested in moduli of objects up to constructible decomposition.  For a variety $X$ parameterising differential graded vector spaces (i.e. a differential graded vector bundle on $X$) we \textit{can} assume that the bundle is formal, with homology bundles of constant dimension.  Then our ind-stack $\VV_l/G$ \textit{is} good enough -- up to constructible decomposition every differential graded vector bundle is obtained by pullback from $\VV_l/G$.  Assuming the base space of our family is a variety we can even use constructible triviality of vector bundles one more time to forget the group $G$.  In summation, every differential graded vector bundle \textit{is} obtained by pullback from $\VV_l$, up to constructible decomposition.
\bigbreak
So we propose a solution that does not start with $\VV_l^+$ so much as end up with it: let $\mathcal{M}$ be some moduli stack of objects in $\Perf(A\lMod)$.  By definition $\mathcal{M}$ has on it a flat family $\mathcal{F}$ of perfect $A$-modules.  Our solution is to work with this sheaf.  We give here a baby example.  We would like to define, on the `stack of objects' $\ob(\mathcal{C})$, a line bundle $\sDet(\Ext^{\bullet})$, as in \cite{KS}.  This, then, would provide one of the ingredients for the discussion of orientation data restricted to $\mathcal{M}$, via pullback.  Assuming we have our stack $\ob(\mathcal{C})$, we have a morphism $\iota:\mathcal{M}\rightarrow \ob(\mathcal{C})$, and we expect the isomorphism
\[
\iota^*(\sDet(\Ext^{\bullet}))\cong \sDet(\mathcal{E}xt_A(\mathcal{F})).
\]
Note, however, that we could just as well construct $\sDet(\mathcal{E}xt_A(\mathcal{F}))$ from our original sheaf $\mathcal{F}$.  This prefigures our treatment.  We cut out, for the main body of the paper, any consideration of $\ob(\mathcal{C})$ as a kind of moduli stack, and consider only sheaves of perfect modules.  Next we explain how we `end up' back with $\VV_l^+$.
\section{Flat families of modules}
\begin{defn}
A flat family of perfect $A$-modules over a scheme $X$ is a sheaf of $\mathcal{O}_X\times A$-modules such that the fibre over every geometric point is perfect, and the underlying graded coherent sheaf is a vector bundle.
\end{defn}
Here is the proposition we will use:
\begin{prop}
\label{ltg}
Let $\mathcal{F}$ be a family of $A$-modules parameterised by an irreducible integral scheme $S$, i.e. an $\mathcal{O}_S\otimes A$-module, such that the fibre over each geometric point is a perfect module.  Then the fibre over the generic point $\Spec(K(S))$ is a perfect $A\otimes K(S)$-module.
\end{prop}
\begin{proof}
Let $\mathcal{G}_i$ be a set of families of $A\otimes K(S)$-modules.  There is a differential graded $K(S)$-vector space
\[
\mathcal{RHOM}(\mathcal{F},\mathcal{G}_i)
\]
with underlying graded vector space $\Hom_{K(S)}(\bigoplus_{n\geq 0} \forget(A)^{\otimes n}[n]\otimes \mathcal{F}[1]_{K(S)},\mathcal{G}_i[1]_{K(S)})$, and the natural differential.  Set $K:=K(S)$.  The natural inclusion 
\[
\Hom_K(\bigoplus_{n\geq 0} \forget(A)^{\otimes n}[n]\otimes \mathcal{F}[1]_K,\bigoplus_{i\in I}\mathcal{G}_i[1]_K)\rightarrow \bigoplus_{i\in I} \Hom_K(\bigoplus_{n\geq 0} \forget(A)^{\otimes n}[n]\otimes \mathcal{F}[1]_K,\mathcal{G}_i[1]_K)
\]
becomes a quasi-isomorphism after tensoring with $\overline{K}$, by Proposition \ref{cpctprop}, since $\mathcal{F}$ is perfect at every geometric point.  It follows that it is a quasi-isomorphism to start with, and so $\mathcal{F}$ is perfect over the generic point.
\end{proof}
\begin{thm}
\label{constrdec}
Let $\mathcal{C}$ be a quiver algebra, and so in particular
\[
\langle h^{\mathcal{C}}\rangle_{\triang}\rightarrow \Perf(\mathcal{C}\lModi)
\]
is a quasi-equivalence of categories.  If $\mathcal{F}$ is a flat family of perfect left $\mathcal{C}$-modules over a finite type scheme $X$, then $X$ admits a finite decomposition into subschemes $X_i$ such that on each $X_i$ there is a map $\alpha_i:X_i\rightarrow \VV_l$, where $\VV_l$ is as defined in Definition \ref{objects}, and a quasi-isomorphism of $\mathcal{C}$-modules 
\[
\xymatrix{
\mathcal{F}\rightarrow \alpha_i^*(L_{\VV_l}),
}
\]
where $L_{\VV_l}$ is the natural bundle of perfect $\mathcal{C}$-modules on $\VV_l$.
\end{thm}
\begin{proof}
Since $X$ is finite type, we may assume, after finite constructible decomposition, that it is an integral affine scheme.  Note that the definition of a quiver algebra (Definition \ref{quidef}) is stable under extension of scalars.  It follows from Proposition \ref{retnice} and Proposition \ref{ltg} that the fibre of $\mathcal{F}$ over the generic point $\Spec(K(X))$ is a closed point of the ind-variety $(\VV_{K(X)})_l$ defined as in Definition $\ref{objects}$, but with the category $\mathcal{C}$ replaced by $\mathcal{C}\otimes {K(X)}$, and that, generically, $\mathcal{F}$ is pulled back from a map to $\VV_l$.  The result then follows by Noetherian induction.
\end{proof}
The slogan that goes with this result is that Proposition \ref{sumup}, which stated that there is a quasi-equivalence 
\begin{equation}
\label{reprep}
\xymatrix{
\tw_l(\mathcal{D}_r(Q,W))\ar[r]^(.4){\sim} & \Perf(\mathcal{D}_r(Q,W)\lModi),
}
\end{equation}
is true \textit{in constructible families}.  The above quasi-equivalence enables us to begin motivic Donaldson--Thomas theory for objects $M$ in $\Perf(\mathcal{D}_r(Q,W)\lModi)$, since it tells us that we can replace $M$ with a $M'\in \tw_l(\mathcal{D}_r(Q,W))$ such that $\End_{\tw_l(\mathcal{D}_r(Q,W))}(M')$ is a Calabi-Yau category with one object, and so possesses a well behaved potential function $W$, which is the crucial input for calculating the motivic weight of $M$.  The fact that this result is true in families tells us that we can calculate motivic weights in families too.
\begin{rem}
\label{genpt}
Note the key part played by passing to the generic point and then using Noetherian induction.  The fact that we can do this means that many things can be proved (the prime example being the proof (assuming their integral identity) of Kontsevich and Soibelman of the preservation of ring structure under the integration map $\Phi_-$ -- see Definition \ref{intmap}) by considering stack functions $\nu_M$, i.e. families of objects in our chosen Calabi-Yau category consisting of a single object.
\end{rem}
\begin{rem}
We can deal with flags in a similar fashion.  For example, let $\mathcal{M}$ be a moduli stack of 2-step filtrations
\begin{equation}
\label{poh}
0\subset M_1\subset M_2,
\end{equation}
such that $M_1$ and $M_2$ are flat families of $\mathcal{D}_r(Q,W)$-modules, and for every $x$ a geometric point of $\mathcal{M}$ the factors $(M_1)_x$ and $(M_2/M_1)_x$ are perfect.  Then after passing to an atlas, and taking a smooth, affine, constructible decomposition, and passing to a generic point, we get that (\ref{poh}) is a 2-step filtration of perfect $\mathcal{D}_r(Q,W)\otimes K$-modules, for $K$ some field extension of $k$.  We deduce that the sheaf of flags on $\mathcal{M}$ is pulled back (up to constructible decomposition of an atlas and quasi-isomorphism of the family) from the natural sheaf of flags on $\VV_{l,3}$, the ind-variety of triangles of $\tw_l(\mathcal{D}_r(Q,W))$.
\end{rem}
\bigbreak
Now assume only that $\mathcal{C}$ is a finite-dimensional Calabi-Yau category, not necessarily of the form $\mathcal{D}_r(Q,W)$.  Then we no longer have the quasi-equivalence (\ref{reprep}).  Recall, however, that we do have an ind-variety $\VV_l^+$ parameterising objects of $\Perf(\mathcal{C}\lModi)$, up to quasi-equivalence.  We prove the following in exactly the same way as \ref{constrdec}.
\begin{thm}
\label{notsogood}
If $\mathcal{F}$ is a flat family of perfect $\mathcal{C}$-modules over a finite type scheme $X$, $X$ admits a finite decomposition into subschemes $X_i$ such that on each $X_i$ there is a map $\alpha_i:X_i\rightarrow \VV_l^+$ and a quasi-isomorphism of $\mathcal{C}$-modules 
\[
\xymatrix{
\mathcal{F}\rightarrow \alpha_i^*(L_{\VV_l^+}),
}
\]
where $L_{\VV_l^+}$ is the natural bundle of perfect $\mathcal{C}$-modules on $\VV_l^+$.  Furthermore, given any flat family of 2-step filtrations of perfect $\mathcal{C}$-modules, we can, up to quasi-isomorphism, obtain the same family by pulling back from $\VV^+_{l,3}$, the ind-variety parameterising triangles in $\Perf(\mathcal{C}\lModi)$ (see Corollary \ref{perflag}).
\end{thm}
As already noted, while we can use such descriptions to say meaningful things about orientation data in high generality, Theorem \ref{notsogood} is not so much use for carrying out the programme of motivic Donaldson--Thomas invariants, since we do not obtain Calabi-Yau $\mathcal{HOM}$-bundles over $\mathcal{M}$, as we have not constructed them over $\VV_l^{+}$ in the first place.  In essence, all mention of $\VV_l^{+}$ could be removed from the following work: once we have given up on the idea of modelling the category $\Perf(A\lModi)$ with something that is geometric and has Calabi-Yau $\mathcal{HOM}$-bundles, as in the case of $\VV_l$, we may as well just work with superdeterminants of families of perfect modules directly, instead of working with $\VV_l^{+}$ and pulling back via Theorem \ref{notsogood}.

\begin{rem}
There is a somewhat different approach to the question of where orientation data (and, to complete the picture, the superpotential function $W$) should live.  We could take seriously the idea that these ingredients, that determine the motivic weight functions on moduli spaces that determine their motivic DT contributions, should be pulled back from some universal representing stack of perfect $A$-modules.  Such an object exists within the framework of \textit{derived algebraic geometry}, as developed in \cite{HAGI}, \cite{HAGII}.  The moduli stack of interest is a derived stack, in the terminology of \cite{HAGII}, and is constructed in our case in \cite{dgmoduli}.  One would hope to be able to construct the orientation data in the category of sheaves for this stack.  In this context one can think of a sheaf (on a stack) simply as a morphism of (derived) sheaves on simplicial $k$-algebras.  
\medbreak
In the following chapters, for familes $\mathcal{F}$ of perfect $\mathcal{C}$-modules over varieties $X$ we will be interested in super line bundles with the fibre $\sDet(M\otimes_{\mathcal{C}} T\otimes_{\mathcal{C}} M^{\diamond})$ over a module $M$ (with $-^{\diamond}$ as in Definition \ref{dimdef}, and superdeterminant as defined in (\ref{sdetdef})).  So it is clear enough what our sheaf on the stack of perfect $\mathcal{C}$-modules will be in this case: it is just the morphism taking the family $\mathcal{F}$ to this superdeterminant.  Furthermore, complexities due to the problem of `glueing along quasi-isomorphisms' disappear due to the fact that quasi-isomorphisms induce canonical isomorphisms of superdeterminants.  In summation, the correct place for Theorem \ref{ODconstr} to live is the monoidal category of constructible super line bundles over the stack of perfect $\mathcal{C}$-modules.  This fact explains why, despite the rather uneconomical structure of $\VV_l^{+}$, as opposed to $\VV_l$, Theorem \ref{ODconstr} remains true there.
\end{rem}

\chapter{Orientation data}
\label{ODchapter}
\section{Preliminaries on bifunctors}
\label{gettinggoing}
Let $\mathcal{C}$ be a finite-dimensional unital $A_{\infty}$-category, satisfying the property that
\[
\langle h_{\mathcal{C}\otimes K}\rangle_{\triang}\rightarrow \Perf(\rModi(\mathcal{C}\otimes K))
\]
is a quasi-equivalence of categories for all field extensions $K\supset k$ (in fact it is enough to check for $K=\overline{k}$ -- a consequence of the existence of $\VV_r^+$).  Let $\mathcal{M}$ be a stack parameterising a flat family of perfect right $\mathcal{C}$-modules (it will turn out to be more convenient to work with right modules for the duration of this chapter), which we denote $L_{\mathcal{M}}$. Let
\[
\beta:U\rightarrow \mathcal{M}
\]
be an atlas.  Then by Theorem \ref{constrdec} there is a constructible decomposition $U=\coprod U_i$, and morphisms $\alpha_i:U_i\rightarrow \mathfrak{V}_r$ such that the sheaf of $\mathcal{C}$-modules on $U_i$ given by the pullback of $L_{\mathcal{M}}$ along $\beta|_{U_i}$ is quasi-isomorphic to the pullback along $\alpha_i$ of the bundle of $\mathcal{C}$-modules on $\mathfrak{V}_r$, which we denote $L_{\mathfrak{V}_r}$.  A $\mathcal{C}$-module $M$ has an underlying vector space given by $\bigoplus_{x\in\mathcal{C}}M(x)$, and flatness tells us that this forgetful functor gives rise (locally) to a differential graded vector bundle along $U$.  
\bigbreak
Given a finite-dimensional differential graded vector bundle $V$ there is an associated super line bundle (i.e. a $\mathbb{Z}_2$-graded line bundle) given by taking the superdeterminant.  This is defined as
\begin{equation}
\label{sdetdef}
\sDet(V):=\prod^{\dim(V)}\bigwedge^{\mathrm{top}}(V_{\mathrm{even}})\otimes(\bigwedge^{\mathrm{top}}(V_{\mathrm{odd}}))^*,
\end{equation}
where $\prod$ is the parity change functor on $\mathbb{Z}_2$-graded objects.  The superdeterminant satisfies the canonical isomorphism
\begin{equation}
\label{qisdet}
\sDet(V)\cong \sDet(\Ho^{\bullet}(V)).
\end{equation}
Recall that there is a differential graded vector bundle $\mathcal{HOM}^{\bullet}$ on $\mathfrak{V}_r\times\mathfrak{V}_r$, satisfying the property that $\Ho^{n}(\mathcal{HOM}^{\bullet})$ calculates $\Hom_{\Di(\mathcal{C}\lModi)}(N,M[n])$ above a point $(M,N)\in \mathfrak{V}_r\times\mathfrak{V}_r$.  We deduce from (\ref{qisdet}) that if $V^{\bullet}$ is some other finite-dimensional differential graded vector bundle calculating homomorphisms in the derived category, there is a canonical isomorphism
\[
\sDet(V^{\bullet})\cong \sDet(\mathcal{HOM}^{\bullet})
\]
as super line bundles on $\mathfrak{V}_r\times \mathfrak{V}_r$.  In general we redefine the superdeterminant of a differential graded vector bundle to be the right hand side of (\ref{qisdet}), allowing us to work with all differential graded vector bundles with locally finite-dimensional homology.
\bigbreak
Recall from Proposition \ref{dimdef} the dualizing functor $-^{\diamond}:\Perf(\mathcal{C}\lModi)\rightarrow (\Perf(\rModi \mathcal{C}))^{\op}$.
We denote by
\[
\Theta:\mathcal{C}\biModi \mathcal{C}\rightarrow \Perf(\rModi \mathcal{C})\biModi\Perf(\rModi\mathcal{C})
\]
the strict differential graded functor given by
\[
\Theta:T\rightarrow -\otimes_{\mathcal{C}} T\otimes_{\mathcal{C}} -^{\diamond}.
\]
\begin{prop}
The functor $\Theta$ is a quasi-equivalence.
\end{prop}
\begin{proof}
Via the $A_{\infty}$ Yoneda embedding $h_{-}$ there is a restriction functor
\[
\textbf{res}:\Perf(\rModi\mathcal{C})\biModi\Perf(\rModi\mathcal{C})\rightarrow \mathcal{C}\biModi \mathcal{C},
\]
and it is easy to see that it is a left inverse to $\Theta$.  Now it is enough to show that the quasi-essential image of $\Theta$ is the whole of $\Perf(\rModi\mathcal{C})\biModi\Perf(\rModi\mathcal{C})$.  So let $M$ be a bifunctor in this category.  After applying a quasi-equivalence $\eta$ we may assume that $\lperf$ is a differential graded category and $M$ is a differential graded bimodule for $\lperf$.  Then since differential graded functors preserve cones it follows that there is a quasi-equivalence $\Theta(\textbf{res}(M))|_{(\tw_r(\CC),\tw_r(\CC)^{\op})}\rightarrow M|_{(\tw_r(\CC),\tw_r(\CC)^{\op})}$. Furthermore, since differential graded functors preserve finite limits and colimits up to quasi-equivalence, this quasi-equivalence lifts to $\Theta(\textbf{res}(M))$.
\end{proof}

Specifically, we need that the bifunctor $\Hom_{\lperf}(-,-)$ is of this form.  Using the above proof and Corollary \ref{diagfact} we deduce that there is a quasi-isomorphism $\Hom_{\lperf}(-,-)\rightarrow \Theta(\mathcal{C})$.
\bigbreak
Let $\MM$ be a stack parameterising perfect $\CC$-modules, and let $T$ be a $\CC$-bimodule.  After restricting to one of the $U_i$ in the constructible decomposition of the atlas for $\mathcal{M}$ of Theorem \ref{notsogood}, we obtain (from the bar construction) a differential graded vector bundle 
\begin{equation}
\label{xidef}
\Xi_T(\alpha_i^*(L_{\mathfrak{V}^+_r}),\alpha_i^*\lambda^*(L_{\mathfrak{V}_l^+})):=\alpha_i^*(L_{\mathfrak{V}^+_r})\otimes T\otimes\alpha_i^*\lambda^*(L_{\mathfrak{V}_l^+})
\end{equation}
(recall from Proposition \ref{dimdef} that $\lambda$ is the isomorphism of ind-varieties induced by the construction of $-^{\diamond}$), which at a point $x$ parameterising a perfect module $M$ calculates $\Theta(T)(M,M)$.  Now assume that $\MM$ is furthermore a stack parameterising objects in $\tw_{r}(\CC)$.  By Proposition \ref{finmods}, we may consider instead the differential graded vector bundle $\\T_{\tw}(\alpha_i^*(L_{\mathfrak{V}_r}),\alpha_i^*\lambda_i^*(L_{\mathfrak{V}_l}))$, where these $\alpha_i$ are now the maps appearing in Theorem \ref{constrdec}, a quasi-isomorphic differential graded vector bundle with finite-dimensional fibres.  From the point of view of motivic Donaldson--Thomas theory, it is the second bundle that is useful, since in the case $T=\mathcal{C}$ it carries extra data encoding the cyclic structure of the category, enabling us to write down the minimal potential (see Definition \ref{CYcatdef} for a definition of this potential, see Theorem \ref{cycmm} for an explanation of how this minimal potential is obtained, and see Chapter \ref{example} for an explanation of why this minimal potential is such an important player in this story).  However, since these two constructible differential graded vector bundles are quasi-isomorphic, they are the same as far as the superdeterminant is concerned.  For the purposes of considering orientation data, which can be defined purely in terms of these superdeterminants, we needn't restrict ourselves to finite-dimensional vector bundles, and in particular we can extend the treatment to the case of $\MM$ parameterising objects in $\rperf$, not just $\tw_r(\mathcal{C})$.
\begin{defn}
We denote by $-^{*}$ the functor
\[
-^{*}:\mathcal{C}\biModi\mathcal{C}\rightarrow(\mathcal{C}\biModi\mathcal{C})^{\op}
\]
taking a bimodule to its vector dual.  Explicitly, a bimodule $S$ for $\mathcal{C}$ is given by a $\mathbb{N}^2$-indexed series of morphisms 
\begin{eqnarray*}
&b_{S,i,j}:\Hom_{\mathcal{C}}(x_{i-1},x_i)[1]\otimes\ldots\otimes \Hom_{\mathcal{C}}(x_0,x_1)[1]\otimes \Hom_{\mathcal{C}}(y_{j-1},y_j)[1]\otimes\ldots\\
&\ldots\otimes \Hom_{\mathcal{C}}(y_0,y_1)[1]\rightarrow \Hom(S(x_0,y_j),S(x_i,y_0))[1].
\end{eqnarray*}
We define $S^*(x,y)=S(y,x)^*$.  Applying the vector space duality functor on $\dgvect$ gives a $\mathbb{N}^2$-indexed series of maps
\begin{eqnarray*}
&b_{S^*,i,j}=b_{S,j,i}^*:\Hom_{\mathcal{C}}(y_{j-1},y_j)[1]\otimes\ldots \otimes \Hom_{\mathcal{C}}(y_0,y_1)[1]\otimes \Hom_{\mathcal{C}}(x_{i-1},x_i)[1]\otimes\ldots\\&\ldots\otimes \Hom(x_0,x_1)[1]\rightarrow \Hom(S(x_i,y_0)^*,S(x_0,y_j)^*)[1]
\end{eqnarray*}
and it is easy to check that these satisfy the bimodule compatability relations.  We deal with homomorphisms in a similar fashion.
\end{defn}
We denote by $\Theta_{\tw}$ the functor sending a $\mathcal{C}$-bimodule $T$ to the $\tw_r(\mathcal{C})$-bimodule $T_{\tw}$.  The following is clear.
\begin{prop}
There is a diagram of quasi-equivalences of categories, in which the top square commutes, and the bottom square commutes up to natural quasi-equivalence:
\[
\xymatrix{
\tw_{r}(\CC)\bimodi \tw_r(\CC)\ar[r]^{-^*}_{\cong}& (\tw_r(\CC) \bimodi \tw_r(\CC))^{\op}\\
\mathcal{C} \bimodi \mathcal{C}\ar[r]_{\cong}^{-^*}\ar[u]_{\Theta_{\tw}} & 
(\mathcal{C}\bimodi \mathcal{C})^{\op}\ar[u]_{\Theta_{\tw}}\\
\Perf(\rModi\mathcal{C})\biModi \Perf(\rModi\mathcal{C}) \ar[u]_{\textbf{res}}\ar[r]^-{-^*}_-{\sim} & (\Perf(\rModi\mathcal{C})\biModi \Perf(\rModi\mathcal{C}))^{\op}.\ar[u]_{\textbf{res}^{\op}}
}
\]
\end{prop}
\smallbreak
\begin{defn}
Recall the definition of the shift functor from Remark \ref{shiftdef}.  Following \cite{KS} we define the functor
\[
-^{\vee}:\mathcal{C}\biModi\mathcal{C}\rightarrow (\mathcal{C}\biModi \mathcal{C})^{\op}
\]
as the composition $-[-3]\circ -^*$.  Let $F$ be a bimodule for an $A_{\infty}$-category $\mathcal{C}$.  A self-duality structure of degree 3 on $F$ is a quasi-isomorphism
\[
\phi:F\rightarrow F^{\vee}.
\]
We may define a self-duality structure of degree $n$, for arbitrary $n$, similarly.
\end{defn}
From now on we assume that $\mathcal{C}$ is a finite-dimensional (as in Definition \ref{findimcat}) $A_{\infty}$-category with a 3-dimensional Calabi-Yau structure.
\begin{prop}
Let $H\in \Perf(\rModi\mathcal{C})\biModi\Perf(\rModi\mathcal{C})$ denote the $\Hom$ bifunctor.  Then there is a quasi-isomorphism
\[
H\rightarrow H^{\vee}.
\]
\end{prop}
\begin{proof}
Since $H$ is quasi-isomorphic to the image of $\mathcal{C}$ under $\Theta$, it is enough to show that there is a quasi-siomorphism of $\mathcal{C}$-bimodules
\[
\xymatrix{
\mathcal{C}\rightarrow \mathcal{C}^{\vee}.
}
\]
For this we consider the Calabi-Yau pairing $\langle\bullet,\bullet\rangle$ on $\mathcal{C}$.  We define a strict morphism 
\[
f:\mathcal{C}\rightarrow \mathcal{C}^{\vee}
\]
by setting $f_1(\phi)=\langle \phi,\bullet\rangle$.  It is straightforward to check that the cyclic invariance and symmetry of the function $W$ of Definition \ref{CYcatdef} implies that this is indeed a morphism of $\mathcal{C}$-bimodules, of the right degree.
\end{proof}
As remarked in \cite{KS}, the orientation data issue is one that arises for any bifunctor to $\dgvect$ with a self-duality structure.  Say $F$ is such a bifunctor for the category $\Perf(\rModi\mathcal{C})$.  Up to quasi-isomorphism we may assume that $F$ is $\Theta(T)$, for some $\mathcal{C}$-bimodule $T$.  In Section \ref{finmodtens} we explained how, restricted to $\tw_r(\mathcal{C})\times \tw_r(\mathcal{C})^{\op}$, the bifunctor $F$ is quasi-isomorphic to $T_{\tw}$ (see Proposition \ref{finmods}), which can be considered as a differential graded vector bundle over $\mathfrak{V}_r\times \mathfrak{V}_r$, where $\mathfrak{V}_r$ is the ind-variety paramaterizing objects of $\tw_r(\mathcal{C})$.  We define
\[
F_{bun}:=\Delta^*(T_{\tw})
\]
where $\Delta:\mathfrak{V}_r \rightarrow \mathfrak{V}_r\times \mathfrak{V}_r$ is the diagonal embedding, and we define the super line bundle
\begin{equation}
\label{basicF}
\FF:=\sDet(F_{bun}).
\end{equation}
By abuse of notation we denote by the same symbol $\FF$ the analogous super line bundle on $\VV^+$, where now we use the definition
\begin{equation}
\label{curlyF}
\FF:=\sDet(\Ho^{\bullet}(\Delta^*(\Xi_T(L_{\VV^+_r},\lambda^*(L_{\VV^+_l}))))),
\end{equation}
where $\Xi$ is as in (\ref{xidef}).  Similarly we define 
\[
\FF^{\leq a}:=\sDet(\Ho^{\leq a} (\Delta^*(\Xi_T(L_{\VV^+_r},\lambda^*(L_{\VV^+_l}))))).
\]
The self-duality structure on $F$ gives us an isomorphism of constructible super line bundles
\begin{equation}
\label{sqrtdet}
(\FF^{\leq 1})^{\otimes 2}\cong \FF,
\end{equation}
given by considering the canonical isomorphism, for any constructible vector bundle $V$ and any $n\in\mathbb{Z}$:
\begin{equation}
\label{dualcon}
\sDet(V[2n+1]^{*})\cong \sDet(V).
\end{equation}
Orientation data for $F$ can be defined to be a choice of another square root of $\FF$, satisfying some extra properties.  We will first flesh out this definition, before relating it to the foundational problem discussed in Chapter \ref{example}.

\section{The definition of orientation data}
Consider the ind-variety $\VV^+_{r,3}$ of exact triangles.  There are three projections $p_i:\VV^+_{r,3}\rightarrow \VV^+_r$ corresponding to the three terms in an exact triangle.  For $\FF$ as in the previous section, define 
\begin{equation}
\label{pulledF}
\mathcal{F}_i=p_i^*(\FF)
\end{equation}
and
\[
\mathcal{F}_{i}^{\leq a}=p_i^*(\FF^{\leq a}).
\]
We also have projections $p_{i,j}:\VV^+_{r,3}\rightarrow \VV^+_r\times\VV_r^+$ given by projection onto the $i$th and $j$th factors.  If $F=\Theta(T)$, we let
\begin{equation}
\label{pairedF}
\mathcal{F}_{i,j}=p_{i,j}^*(\sDet(\Ho^{\bullet}(L_{\VV_r^+}\otimes_{\mathcal{C}} T \otimes_{\mathcal{C}} \lambda^*(L_{\VV_l^{+}})))).
\end{equation}
Note that since $F$ has a self-duality structure we have an isomorphism
\[
\mathcal{F}_{1,3}\cong\mathcal{F}_{3,1}.
\]
It follows by basic homological algebra (and some applications of the isomorphisms above) that we have a canonical isomorphism
\begin{equation}
\label{deformiso}
\mathcal{F}_2\cong \mathcal{F}_1 \otimes \mathcal{F}_3\otimes \mathcal{F}_{1,3}^{\otimes 2}.
\end{equation}
We would like for this isomorphism to have a square root, i.e. we would like for there to be a constructible isomorphism
\begin{equation}
\label{sqrtiso}
\mathcal{F}_{2}^{\leq 1}\cong \mathcal{F}_{1}^{\leq 1} \otimes \mathcal{F}_{3}^{\leq 1}\otimes \mathcal{F}_{1,3}
\end{equation}
such that taking the tensor square of this isomorphism, and using the canonical isomorphism (\ref{sqrtdet}) between the square $(\mathcal{F}_{i}^{\leq 1})^{\otimes 2}$ and $\mathcal{F}_i$, we obtain the isomorphism (\ref{deformiso}).  We may not be able to achieve this with the ind-constructible super line bundle $\mathcal{F}^{\leq 1}$, which motivates the main definition of this chapter.  We give the original definition of \cite{KS}.
\begin{defn}
\label{ODdef}
Let $T$ be a bimodule for $\mathcal{C}$, equipped with a self-duality structure.  Let $\mathcal{F}$, $\mathcal{F}_i$, and $\mathcal{F}_{i,j}$ be defined as in (\ref{curlyF}), (\ref{pulledF}) and (\ref{pairedF}), in terms of $T$.  Then the set $\overline{\OD}^{+}(T)$ of complete orientation data for $T$ on $\rperf$ is given by triples of:
\begin{enumerate}
\item[\textbf{I:}]
An ind-constructible super line bundle $\sqrt(\mathcal{F})$ on $\VV^+_r$.
\item[\textbf{II:}]
An isomorphism of ind-constructible super line bundles
\begin{equation}
\label{sqrteq}
\sqrt(\mathcal{F})^{\otimes 2}\cong \mathcal{F}.
\end{equation}
\item[\textbf{III:}]
An isomorphism of ind-constructible super line bundles
\begin{equation}
\label{triangles}
f:\sqrt(\mathcal{F}_1)\otimes \sqrt(\mathcal{F}_2)^{-1}\otimes \sqrt(\mathcal{F}_3)\cong \mathcal{F}_{1,3}
\end{equation}
\end{enumerate}
such that the following diagram
\begin{equation}
\label{ODcomm}
\xymatrix{
(\sqrt(\mathcal{F}_1)\otimes \sqrt(\mathcal{F}_2)^{-1}\otimes \sqrt(\mathcal{F}_3)\otimes \mathcal{F}_{1,3}^{-1})^{\otimes 2}\ar[r]^(.8){f\otimes f}\ar[d]&\id\otimes \id\ar[d]\\
\mathcal{F}_{1} \otimes \mathcal{F}_{2}^{-1}\otimes \mathcal{F}_{3}\otimes (\mathcal{F}_{1,3}^{-1})^{\otimes 2}\ar[r]
& \id
}
\end{equation}
commutes.  Furthermore, we require the following data:
\begin{enumerate}
\item[\textbf{QIS:}]
For any scheme $S$ and any pair of constructible morphisms $\alpha,\beta:S\rightarrow \VV^+_r$ and any quasi-isomorphism $\gamma:\alpha^*(L_{\VV^+_r})\rightarrow \beta^*(L_{\VV^+_r})$ a constructible isomorphism $\alpha^*(\sqrt{\mathcal{F}})\rightarrow\beta^*(\sqrt{\mathcal{F}})$, commuting with all of the above isomorphisms.
\end{enumerate}
The set of orientation data for $T$ on $\tw_r(\mathcal{C})$, denoted $\overline{\OD}(T)$ is defined in the same way, but with all instances of $\VV_r^+$ replaced by $\VV_r$.
\end{defn}
\begin{examp}
\label{torsor}
Let $\zeta:\Kth(\tw_r(\mathcal{C}))\rightarrow \mathbb{Z}_2$ be a group homomorphism from the Grothendieck group of $\tw_r(\mathcal{C})$.  Then $\zeta$ induces a partition of the components of $\VV_r$ into two sets, respecting quasi-isomorphism of families.  To such a homomorphism we associate the super line bundle $L_{\zeta}$ on $\VV_r$ which is trivial, with parity determined by the value of $\zeta$.  This super line bundle comes with a canonical trivialization of its square given by the canonical trivialization of the tensor square of the trivial  odd or even super line bundle.  In particular, given orientation data
\[
\sqrt{\mathcal{F}}^{\otimes 2}\cong \mathcal{F}
\]
for $T$ on $\tw_r(\mathcal{C})$ we obtain new orientation data by replacing $\sqrt{\mathcal{F}}$ by $\sqrt{\mathcal{F}}\otimes L_{\tau}$.  Since $L_{\tau}$ comes from a homomorphism of the $\Kth$-group of $\langle h_{\mathcal{C}}\rangle_{\triang}$, it follows that condition (\textbf{III}) of Definition \ref{ODdef} is met by an essentially unchanged isomorphism.  As a result of this, we do not expect orientation data to be unique.  Note that this operation gives the set of complete orientation data a free $\Hom(\Kth(\tw_r(\CC)),\mathbb{Z}_2)$-action.
\end{examp}
\begin{rem}
In practice, we never have to worry about the condition \textbf{QIS}.  We stick, at all times, to constructions that induce canonical isomorphisms on quasi-isomorphic families, i.e. we stick to constructible super line bundles coming from superdeterminants of differential graded vector bundles arising from bifunctors.  The only examples of $\sqrt{\mathcal{F}}$ that we consider are given by taking $\sDet(\Theta(S))$, for $\mathcal{C}$-bimodules $S$, so that isomorphisms between evaluations of superdeterminants on quasi-isomorphic families come for free.  The fact that one can consider orientation data at the level of ind-constructible super line bundles on $\VV_r^+$ or $\VV_r$ is worth bearing in mind, though (e.g. consider Example \ref{torsor}).
\end{rem}
We use instead a modified definition:

\begin{defn}
\label{redod}
The category $\OD(T)$ or $\OD^+(T)$ of orientation data on $\VV_r$ or $\VV_r^+$ for a $\CC$-bimodule $T$ is given by:
\begin{enumerate}
\item
$\ob(\OD(T))$ or $\ob(\OD^+(T))$ is the set of pairs $(\mathcal{G},\phi)$ as in (\textbf{I}) and (\textbf{II}) of Definition \ref{ODdef}, such that there exist isomorphisms as in (\textbf{III}) and (\textbf{QIS}).
\item
$\Hom_{\OD(T)}((\mathcal{G},\phi),(\mathcal{H},\psi))$ or $\Hom_{\OD^+(T)}((\mathcal{G},\phi),(\mathcal{H},\psi))$ is given by the set of isomorphisms $\tau:\mathcal{G}\rightarrow \mathcal{H}$ such that the following diagram commutes
\begin{equation}
\label{ODmorph}
\xymatrix{
\mathcal{G}^{\otimes 2}\ar[rd]^-{\phi}\ar[d]^{\tau^{\otimes 2}}\\ 
\mathcal{H}^{\otimes 2}\ar[r]^-{\psi}& \FF.
}
\end{equation}
\end{enumerate}
\end{defn}
\begin{rem}

We are justified in forgetting the \textit{choice} of data in (\textbf{III}) since it plays no part in the determination of the motivic weight $w$ for objects of $\mathcal{E}$, as defined in \cite{KS} (see Definition \ref{intmap}).  The role of (\textbf{III}) is in the proof that the integration map $\Phi$ of that paper is a homomorphism, which runs independently of the particular isomorphism chosen.  Our condition that there exists \textit{some} isomorphism satisfying (\textbf{III}) amounts to the cocycle condition \ref{cocycle}.  Similarly, the existence of the isomorphisms in (\textbf{QIS}) is there merely to guarantee that the integration map is well defined -- we needn't worry what these isomorphisms actually are.
\end{rem}
\section{The link with constructible functions}
\label{thelink}
Given a square root $\sqrt(\mathcal{F})$ of $\mathcal{F}$, the existence or nonexistence of a square root of the isomorphism (\ref{deformiso}), as in (\ref{triangles}), is determined entirely by the value of a certain obstruction, denoted $l$, which is defined in terms of constructible functions.  We next describe this link.  We consider only self-dual bifunctors of degree three, since the story is unchanged in higher dimensions, which are of limited interest anyway.
\smallbreak
Say $L$ is an ind-constructible super line bundle on an ind-variety $X=\coprod X_i$ equipped with a trivialization 
\[
\psi:L^{\otimes 2}\cong \idmon.
\]  
Then after passing to each $X_i$, $L$ is a constructible super line bundle on an algebraic variety.  It follows that after stratifying further, the line bundle is in fact trivial (see Section \ref{cvect}).  Pick some $U$ in the stratification, let
\[
\tau_U:\idmon[a]|_U\cong L|_U
\]
be a local isomorphism (where $a$ takes value 0 or 1) on $U$, a locally closed subvariety of $X_i$.  Then we have the canonical isomorphism
\[
\phi:\idmon|_U\cong\idmon[a]|_U\otimes\idmon[a]|_U.
\]
Putting all this together gives an isomorphism on $U$
\[
\psi|_U \circ (\tau_U\otimes \tau_U)\circ \phi:\idmon|_U\cong \idmon|_U
\]
which gives a nonvanishing function on $U$, given by the image of the constant section $1$ of $\idmon$.  Note that changing $\tau_U$ changes this constructible function by a square.  Therefore we have a well defined element of 
\[
\Constr(U,k^*)/\Constr(U,k^*)^2,
\]
where $\Constr(U,k^*)$ is the set of $k^*$-valued constructible functions on $U$.  Since each $X_i$ is just an algebraic variety it follows that the presheaf on $X_i$,
\[
\mathfrak{F}(U)=\Constr(U,k^*)/\Constr(U,k^*)^2
\]
is in fact a sheaf, and so in the above way we build a section of it.  In this way we obtain an ind-constructible function on $X$, i.e. a constructible function on each of its components.  The data of $L$ with its trivialized square also gives a section of
\[
\Constr(X,\mathbb{Z}_2),
\]
the group of ind-constructible $\mathbb{Z}_2$-valued functions on $X$.  This function is given by the parity of $L$.  We define
\begin{equation}
\label{easyJ}
J_2(X):=\Constr(X,k^*)/(\Constr(X,k^*))^2\times \Constr(X,\mathbb{Z}_2).
\end{equation}
This agrees with the definition of \cite{KS} under the assumption $\sqrt{-1}\in k$, otherwise one sets 
\[
J_2(X)=((\Constr(X,k^*)/\Constr(X,k^*)^2)\times \Constr(X,\ZZ))/(-1,2)
\]
as in \cite{KS}.  It is in $J_2(\VV_{r,3})$ or $J_2(\VV^+_{r,3})$ that the obstruction $l$ belongs.  It is given by the fact that, if we define $\FF$ as in (\ref{basicF}) or (\ref{curlyF}), with $T=\mathcal{C}$, i.e. $\FF$ is the superdeterminant of the constructible vector bundle of self extensions, the following constructible super line bundle
\[
\mathcal{F}_{tr}=\mathcal{F}_{1}^{\leq 1}\otimes (\mathcal{F}_{2}^{\leq 1})^{-1}\otimes\mathcal{F}_{3}^{\leq 1} \otimes \mathcal{F}_{1,3}^{-1}
\]
has a canonically trivialised square given by equation (\ref{deformiso}), and so gives an element $l\in J_2(\VV_{r,3})$, alternatively $l\in J_2(\VV_{r,3}^+)$.  This obstruction should be read as the failure of $\mathcal{F}^{\leq 1}$ to provide orientation data.  Indeed clearly if it is zero we can give a constructible isomorphism as in equation (\ref{sqrtiso}) and in this way let $\mathcal{F}^{\leq 1}$ provide orientation data.
\bigbreak
This approach gives a new perspective on condition (\textbf{III}) of Definition \ref{ODdef}.  Given a $\sqrt{\mathcal{F}}$ and an isomorphism $\sqrt{\mathcal{F}}\cong\mathcal{F}$ we obtain an element $h$ of $J_2(\mathfrak{C})$, since the ind-constructible super line bundle $\sqrt{\mathcal{F}}^{-1}\otimes \mathcal{F}_{\leq 1}$ has a trivialized square.  In fact we have seen that the set of isomorphism classes of objects in $\ob(\OD(\mathcal{C}))$ embeds naturally in $J_2(\VV_r)$, and similarly there is an embedding
\[
\ob(\OD^+(\mathcal{C}))/\sim_{\mathrm{isom}} \subset J_2(\VV^+_r).
\]
As noted in \cite{KS} the existence of \textit{some} isomorphism satisfying condition (\textbf{III}), i.e. the condition met by isomorphism classes of orientation data considered as subsets of $J_2(\VV_r)$ or $J_2(\VV_r^+)$, amounts to the following condition:
\begin{cond}[\textbf{Cocycle condition}]
\label{cocycle}
There is an equality in $J_2(\VV_{r,3})$ (or $J_2(\VV_{r,3}^+$))
\[
p_1^*(h)-p_2^*(h)+p_3^*(h)=l.
\]
\end{cond}
\begin{rem}
Recall the extended example of Chapter \ref{example}.  There it was noted that if we assign to every object $E$ in the category of perfect modules the motivic weight $\LL^{\frac{\sum_{i\leq 1}(-1)^i \dim(\Ext^i(E,E))}{2}}(1-\MF(W_{\min}(E))$, the integration map, so defined, need not preserve the product.  Let $\nu_{M_1}$ and $\nu_{M_2}$ be two stack functions associated to the two modules $M_1$ and $M_2$, then we saw that, after fixing a minimal model for the category consisting of just these two objects, there is a constructible vector bundle $V$ over the space of extensions $\Ext^1(M_2,M_1)$, and a nondegenerate quadratic form $Q$ on $V$, such that the required correction over an extension $M_{\alpha}$ is $\mathbb{L}^{\frac{-\dim(V)}{2}}(1-\MF(Q))$.  Now such a pair $(V',Q')$ on a scheme $X$, in general, determines an element of $J_2(X)$, via the isomorphism induced by $Q'$ between $V'$ and $V'^*$.  Furthermore, the map from such pairs to equivalence classes of relative motives in $\overline{\Mot}^{\muhat}(X)$ (see Section \ref{mvc} and Section 4.5 of \cite{KS} for the definitions) sending $(V',Q')$ to $\mathbb{L}^{\frac{-\dim(V')}{2}}(1-\MF(Q'))$ factors through the map to $J_2(X)$ (note that for this statement to be (known to) be true, we must work with $\overline{\Mot}^{\muhat}(X)$ rather than $\Mot^{\muhat}(X)$).  Finally, one can show that the obstruction element $l$ of $J_2(\VV_{r,3})$ or $J_2(\VV^+_{r,3})$, pulled back to $\Ext^1(M_2,M_1)$, is exactly the element determined by $(V,Q)$.  So in this case if we correct the motivic weight against which we integrate by twisting by $\mathbb{L}^{\frac{-\dim(V)}{2}}(1-\MF(Q'))$, for $Q'$ a nondegenerate quadratic form on a constructible vector bundle that gives rise to the element $h\in J_2(\VV_r)$ or $h\in J_2(\VV_r^+)$, the cocycle condition above is just Condition \ref{maincon}.
\end{rem}
This last remark explains the following definition of the \textit{integration map} of Kontsevich and Soibelman.  First recall that we assume we have a map $\theta:\Kth(\langle h_{\CC}\rangle)\rightarrow \Gamma$, as in Section \ref{mott} and we define $\textbf{M}=\overline{\Mot}^{\muhat}(\Spec(k))[\LL^{1/2},[\GL_n(k)]^{-1}\hbox{ }|\hbox{ }n\in\mathbb{N}]$, as in (\ref{Mdef}).
\begin{defn} [Integration map for stack functions of objects of $\langle h_{\CC}\rangle_{\triang}$]
\label{intmap}
Let $\mathcal{C}$ be a 3-dimensional Calabi-Yau category.  Let $h\in J_2(\VV_r)$ be orientation data for the $\Hom_{\mathcal{C}}$-bimodule on $\VV_r$.  Define 
\[
\Phi_{h}(X)\in \mathrm{\textbf{M}}[x^{\gamma}|\gamma\in\Gamma],
\]
where $X$ is a variety with a flat family of objects $\mathcal{X}$ of $\langle h_{\CC}\rangle_{\triang}$ over it, satisfying $\theta([\mathcal{X}])=\gamma$, by
\begin{equation}
\label{intmapeq}
\Phi_{h}(X)=\alpha^*[-\mathbb{L}^{\frac{\ext^{\leqslant 1}-\dim{V}}{2}}\phi_{(W_{\min}+Q)}]_X\cdot x^{\gamma}
\end{equation}
where $\alpha:X\rightarrow \VV_r$ is a constructible morphism such that the pullback family of objects on $X$ is quasi-isomorphic to $\mathcal{X}$, $W_{\min}+Q$ is considered as a function on the ind-constructible vector bundle $\mathcal{EXT}^1\oplus V$, and $\mathbb{L}^{-\frac{\ext^{\leqslant 1}-\dim{V}}{2}}$ is the constructibly trivial relative motive over $\VV_r$ whose fibre over a point $x$ parameterising an object $M$, is $\mathbb{L}^{(-\dim(V)+\sum_{i\leq 1} (-1)^i \Ext^i(M,M))/2}$ (we say a relative motive is trivial if it is pulled back from the absolute ring of motives along the structure morphism).  The (equivalence class) of the pair $(V,Q)$ is the one determined by the isomorphism class of orientation data $h$, i.e. it is any pair of constructible vector bundle with nondegenerate quadratic form giving rise to $h$.
\end{defn}
\begin{rem}
Although this thesis doesn't aim to cover all of the foundational ground behind its inspiration \cite{KS}, we should at least say something more to convince the reader that the above definition is well-defined, at least forgetting the work that must be done to make sense of motivic vanishing cycles of formal functions (see the discussion following Theorem \ref{TStheorem}).  To this end, we assert that if $\mathcal{X}_1$ and $\mathcal{X}_2$ are two quasi-isomorphic families of objects pulled back from morphisms $X\rightarrow \VV_r$, and $E_1$ and $E_2$ are cyclic minimal models for the pulled back $A_{\infty}$-endomorphism bundles, then there is a cyclic $A_{\infty}$-isomorphism between them.  To prove this, one should consider the ind-constructible cyclic $A_{\infty}$-category with underlying ind-variety of objects given by $X_1\coprod X_2:=X\coprod X$, where the first copy of $X$ parameterises the objects via $\mathcal{X}_1$, and the second via $\mathcal{X}_2$.  This ind-constructible Calabi-Yau category has a minimal model $\mathcal{P}$ as in Theorem \ref{cycmm}.  We consider the full sub ind-constructible Calabi-Yau category $\mathcal{Q}$ with 
\[
\Hom_{\mathcal{Q}}(a,b)=\begin{cases} \Hom_{\mathcal{P}}(a,b) &\hbox{ if }a=b\hbox{ considered as points of }$X$\\0&\hbox{ otherwise.}\end{cases}
\]
Then by assumption there are constructible sections $\alpha_{12}$ and $\alpha_{21}$ of $\HOM_{\mathcal{Q}}|_{(X_1,X_2)}$ and $\HOM_{\mathcal{Q}}|_{(X_2,X_1)}$ arising from the original quasi-isomorphisms of modules.    Adapting the argument of \cite{KLH} to the cyclic bar complex relative to the sub bundles generated by the identity section over $(X_1,X_1)$, the identity section over $(X_2,X_2)$ and the sections $\alpha_{12}$ and $\alpha_{21}$, and observing that the inclusion of the reduced cyclic bar complex is a quasi-isomorphism (see (\cite{Lod98})), one gets that these quasi-isomorphisms can be made strict up to cyclic quasi-isomorphism of the category $\mathcal{Q}$.  It follows that there is a strict isomorphism of ind-constructible $A_{\infty}$-categories between the restrictions $\mathcal{Q}|_{X_1}$ and $\mathcal{Q}|_{X_2}$, given by pre and post composition with $\alpha_{12}$ and $\alpha_{21}$.
\end{rem}
\section{The construction of orientation data}
In this section we will work with self-dual bifunctors of arbitrary odd degree, since we are actually going to state a result, and it is easy to prove it in generality.  As noted in \cite{KS}, the problem of constructing orientation data is straightforward in at least one case.  If $F\cong H\oplus \Ho^{\vee}$, and its self-duality structure is the obvious one, then we can take $\sDet(H)$ as our square root of $\sDet(F)$, and the isomorphism (\ref{deformiso}) applied to $\sDet(H)$ demonstrates that the same isomorphism applied to $\sDet(F)$ has a square root.
\begin{defn}
\label{cansplit}
Let $T$ be a bimodule on $\mathcal{C}$ with a self-duality structure.  We say that $T$ is canonically split if there is some bimodule $H$ on $\mathcal{C}$ and an isomorphism
\[
T\cong H\oplus H^{\vee}
\]
such that the self-duality structure on $T$ is the one induced by the canonical self-duality structure on $H\oplus H^{\vee}$.
\end{defn}
In \cite{KS} orientation data is described by reducing our problem (that of constructing orientation data for the $\Hom$ bifunctor) to the easy one of providing orientation data for the image under $\Theta$ of a canonically split bimodule.  This is achieved by considering families (over $\mathbb{A}^1$) of self-duality structures for $\Hom$ which are canonically split over the zero fibre.  We adopt a slightly different, purely algebraic approach.  We describe a little more fully the link between the two approaches in Remark \ref{linkage}.
\begin{thm}
\label{ODconstr}
Let $n\geq 1$, and let $\mathcal{C}$ be either:
\begin{enumerate}
\item
The category associated to a minimal compact Calabi-Yau $A_{\infty}$-quiver algebra of dimension $2n+1$, with no morphisms of negative degree, or
\item
A finite-dimensional Calabi-Yau $A_{\infty}$-category of dimension $2n+1$ such that higher compositions vanish on degree zero morphisms, and there are no morphisms in negative degree.
\end{enumerate}
Then in either case $\mathcal{C}$ has a sub-bimodule $\mathcal{C}^{\geq n+1}$, whose underlying $\mathcal{C}_{\id}$-module is given by the pieces of $\mathcal{C}$ of degree at least $n+1$. This bimodule provides a differential graded vector bundle, via the map $\Theta$, on $\mathfrak{V}_r$ in case (1), and on $\mathfrak{V}^+_r$ in case (2), which we denote $L$ in either case, and we denote by $\mathcal{L}$ the superdeterminant of $L$.  Then there is a canonical isomorphism
\[
\mathcal{L}^{\otimes 2}\cong \mathcal{H}
\]
and this provides orientation data for $\CC$.
\end{thm}
See Definition \ref{compactCyalg} for the definition of a compact Calabi-Yau $A_{\infty}$-quiver algebra.  Of course the first case is a special case of the second, in which we are in the happy situation of not having to deal with the slightly complicated $\mathfrak{V}_r^+$.
\begin{proof}
We first prove the statement that for all $\phi\in \Hom^t(x,y)$ for $t\geq n+1$ and $x,y\in\mathcal{C}$, and for all homomorphisms $\alpha_1,\ldots,\alpha_j$, and $\beta_1,\ldots,\beta_j$ in $\mathcal{C}$,
\[
m_{\mathcal{C},i+j+1}(\alpha_1,\ldots,\alpha_i,\phi,\beta_1,\ldots,\beta_j)
\]
has degree at least $(n+1)$.  For degree reasons, it is sufficient to check this under the assumption that at least one of the $\alpha$ or $\beta$ is of degree less than 1.  The statement then follows from the $\CC^0$-unitality of $\mathcal{C}$.  Next we consider
\[
S:=\cone(L\rightarrow\mathcal{C}),
\]
the cone of the inclusion morphism.  One can check easily that a minimal model $N$ for this $\mathcal{C}$-bimodule is given as follows.   Firstly, $N(x,y)=\Hom_{\mathcal{C}}^{\leq n}(x,y)$.  It is sufficient to define the composition morphisms for $N$ just for homogeneous morphisms of $\mathcal{C}$.  On these the composition for $N$ is just the composition for $\mathcal{C}$, unless the result is of degree $(n+1)$ or more, in which case it is zero.\smallbreak
We next claim that restricting the isomorphism of $\mathcal{C}$-bimodules
\[
\tau:\mathcal{C}\rightarrow\mathcal{C}^{\vee}
\]
to $L$, we obtain a quasi-isomorphism
\[
\tau_{L}:L \rightarrow N^{\vee}.
\]
This is easy to check, for dimension reasons apart from anything else.  In fact we have a commuting diagram of strict $\mathcal{C}$-bimodule morphisms, in which all the columns are isomorphisms
\begin{equation}
\label{lagdiag}
\xymatrix{
L\ar[d]_{\tau_L}\ar[r]^i & \mathcal{C} \ar[d]_{\tau}\ar[r]^q & N\ar[d]_{\tau_L^{\vee}} \\
N^{\vee}\ar[r]^{q^{\vee}} & \mathcal{C}^{\vee}\ar[r]^{i^{\vee}} & L^{\vee}.
}
\end{equation}
Since the top row of (\ref{lagdiag}) is quasi-isomorphic to a triangle of $\mathcal{C}\biModi\mathcal{C}$, so is its image under $\Theta$ (since differential graded functors preserve cones).  Given a triangle
\[
\xymatrix{
V_1\ar[r] & V_2\ar[r]& V_3
}
\]
of differential graded vector bundles on a scheme $X$, with finite-dimensional total homology, we obtain a canonical isomorphism $\sDet(V_2)\cong \sDet(V_1)\otimes \sDet(V_3)$.  In particular, if $\mathcal{N}$ is the superdeterminant of the differential graded vector bundle induced by $\Theta(N)$ on $\mathfrak{V}_r$ or $\mathfrak{V}_r^+$, then there is a canonical isomorphism
\[
\mathcal{H}\cong \mathcal{L}\otimes\mathcal{N}
\]
as well as a canonical isomorphism
\[
\mathcal{L}\cong\mathcal{N},
\]
by (\ref{dualcon}).  Putting these together we obtain an isomorphism
\[
\mathcal{H}\cong\mathcal{L}^{\otimes 2}.
\]
On $\mathfrak{V}_{r,3}$ or $\mathfrak{V}^+_{r,3}$, the ind-variety of triangles, we set 
\[
\mathcal{L}_{i,j}=p^*_{i,j}(\sDet(L_{\VV_r}\otimes L\otimes \lambda^*(L_{\VV_l})))
\]
or
\[
\mathcal{L}_{i,j}=p^*_{i,j}(\sDet(L_{\VV^+_r}\otimes L\otimes \lambda^*(L_{\VV^+_l}))),
\]
and we define $\mathcal{N}_{i,j}$ similarly.  Since $\mathcal{L}$ comes from the bifunctor $-\otimes_{\mathcal{C}} L\otimes_{\mathcal{C}} -^{\diamond}$, there is a quasi-isomorphism
\[
\mathcal{L}_1^{-1}\otimes \mathcal{L}_2\otimes \mathcal{L}^{-1}_3\otimes \mathcal{L}_{1,3}\otimes\mathcal{L}_{3,1}\cong \idmon.
\]
Since $\mathcal{N}_{1,3}\cong \mathcal{L}_{3,1}$, we obtain a commuting square of isomorphisms
\[
\xymatrix{
\mathcal{L}_1\otimes \mathcal{L}_2^{-1}\otimes\mathcal{L}_3\ar[r]\ar[d] & \mathcal{H}_{1,3}\ar[d]\\
\mathcal{N}_1\otimes \mathcal{N}_2^{-1}\otimes\mathcal{N}_3\ar[r]&\mathcal{H}_{3,1}
}
\]
and we're done.
\end{proof}
\begin{rem}
Note that the conditions of the theorem are quite weak.  For instance any odd-dimensional Calabi-Yau category that is also a differential graded category satisfies condition $(2)$.  For a category with one object, identified as an algebra $A$, this is the condition that $A$ is $A^0$-unital.
\end{rem}
\begin{defn}
\label{lagdef}
Let $\mathcal{C}_{\id}$ be the subcategory of $\mathcal{C}$ with the same objects as $\mathcal{C}$, and with only (scalar multiples of) the identity morphisms.  A $\mathcal{C}$-bimodule $L$ fitting into a diagram as in (\ref{lagdiag}), in which the rows are strict morphisms of $\mathcal{C}$-bimodules, and short exact sequences of underlying $\mathcal{C}_{\id}$-bimodules, will be called a Lagrangian sub-bimodule.
\end{defn}
\begin{rem}
\label{linkage}
Intuitively, the approach of \cite{KS} to constructing orientation data in the situation of Theorem \ref{ODconstr}(1) involves deforming the short exact sequence arising from the inclusion of the Lagrangian bimodule $\mathcal{C}^{\geq 2}$ in $\mathcal{C}$ until it splits, and then pulling back the canonical orientation data arising from the canonical splitting of the deformed version of $\mathcal{C}$.
\end{rem}
\chapter{Examples}
\label{exchap}
\section{Integration map for nilpotent modules over a Jacobi algebra}
\label{intnilp}
Let $(Q,W)$ be, as before, a quiver with potential.  Consider the quasi-equivalence of categories of Proposition \ref{sumup}
\[
\xymatrix{
\tw_{r}(\DD_r(Q,W))\ar[r]^-{\sim} &\rmodi \Gamma(Q,W).
}
\]
There is a full sub $A_{\infty}$-category of the right hand category given by modules concentrated in degree zero.  On the left hand side, this is quasi-equivalent to the category containing $h_{\mathcal{C}}$, and closed under cones of morphisms $M\rightarrow N[1]$.  In other words there is a quasi-equivalence of $A_{\infty}$-categories between the $A_{\infty}$-category of right modules for $\Gamma(Q,W)$ concentrated in degree zero, which we denote $\mathcal{A}$, and the full sub $A_{\infty}$-category of $\tw_r(\DD_r(Q,W))$, which we denote $\mathcal{A}_{\tw}$, with objects given by the $k$-points of 
\begin{equation}
\label{AAdef}
\mathfrak{A}_r:=\coprod_{\tau\in (h_{\CC})^n,n\in\mathbb{N}} \VV_{r,\tau},
\end{equation}
where $\VV_{r,\tau}$ is as in Proposition \ref{kl1}.  Note that for degree reasons, an $A_{\infty}$-module concentrated in degree zero over $\Gamma(Q,W)$ is just an ordinary module over the ordinary algebra $\Ho^0(\Gamma(Q,W))$.  Given a $\DD_r(Q,W)$-bimodule $T$ with a self-duality structure, we define orientation data for $T$ on $\mathfrak{A}_r$ as before, but now we restrict the cocycle condition to morphisms of degree 1 in $\mathcal{A}_{\tw}$, i.e. to those morphisms such that the mapping cone remains in $\mathcal{A}_{\tw}$.  Trivially, then, orientation data for $T$ on $\VV_r$ provides orientation data for $T$ on $\mathfrak{A}_r$. 
\medbreak
We construct an integration map $\Phi_{\mathcal{A}}$ for the stack functions for $\mathcal{A}$, i.e. given a flat family of nilpotent right $\Gamma(Q,W)$-modules $\mathcal{F}$ over a scheme $X$, we are going to construct an absolute motive, which will be its contribution to the Donaldson--Thomas count.  Up to constructible decomposition we may assume that the underlying vector bundle of our family is trivial, with fibre some $n$-dimensional vector space $V$.  We decompose this vector space as
\[
V=\oplus V_i,
\]
where $V_i=Ve_i$.  We fix a basis for each $V_i$, in this way obtaining a basis for $V$.  We order this basis so that every arrow of $Q$ is sent to a family of strictly upper-triangular matrices (working, again, up to constructible decomposition), and denote this (ordered) basis $(v_1,\ldots,v_m)$, where $v_t\in V e_{i(t)}$, for some function $i$ from the numbers $\{1\ldots m\}$ to the vertices of $Q$.  Given such a decomposition, we define
\[
\Rep(\Gamma(Q,W),\oplus V_i)
\]
to be the set of right modules of $\Ho^0(\Gamma(Q,W))$ respecting this decomposition, i.e. if $a$ is an arrow from $i$ to $j$ in $Q$, and $f\in \Rep(\Gamma(Q,W),\oplus V_i)$, then $f(a)$ can be considered as a linear map from $V_j$ to $V_i$.
\bigbreak
The space $\Rep(\Gamma(Q,W),\oplus V_i)$ is a Zariski closed subspace of 
\[
\bigoplus_{i,j \hbox{ vertices of } Q}\Hom_k(e_i \Ho^0(\Gamma(Q,W)) e_j,\Hom(V_i,V_j)):=D_{\oplus V_i},
\]
and occurs as the scheme-theoretic degeneracy locus of $\tra(W)$ on $D_{\oplus V_i}$ (see e.g. \cite{Conifold}).  If $P$ is the base for our flat family of nilpotent modules, we obtain a constructible morphism $g:P\rightarrow D_{\oplus V_i}$, with image in the intersection of scheme-theoretic degeneracy locus of the function $\tra(W)$ with the strictly upper-triangular matrices.  We define the \textit{dimension vector} of a module parameterised by $P$ to be the $n$-tuple $(\dim(V_1),\ldots,\dim(V_n))$.  Given our quiver $Q$ we obtain a $n\times n$ incidence matrix $M_Q$ with $(M_Q)_{ij}$ equal to the number of arrows from $i$ to $j$, and we define $\langle t_1,t_2\rangle_Q$ for two dimension vectors $t_1$ and $t_2$ by
\[
\langle t_1,t_2\rangle_Q:=t_1 t_2^T-t_1 M_Q t_2^T.
\]
\begin{defn}
\label{abintmap}
We define the integration map $\Phi_{\mathcal{A}}$ by setting 
\[
\Phi_{\mathcal{A}}(P)\in \maM[\LL^{1/2},[\GL_n(k)]^{-1}\hbox{ }|\hbox{ }n\in\mathbb{N}][x^\gamma|\gamma\in\mathbb{Z}^n]
\]
to be
\[
\Phi_{\mathcal{A}}(P)=g^*[-\phi_{\tra(W)}\mathbb{L}^{\frac{1}{2}\langle,\rangle_Q}]_P\cdot x^{v},
\]
where $\mathbb{L}^{\frac{1}{2}\langle,\rangle_Q}$ is the constructibly trivial relative motive over $\VV_{r}$ with fibre over a module $M$ given by $\mathbb{L}^{\frac{1}{2}\langle \dim(M),\dim(M) \rangle_Q}$, and $v$ is the dimension vector of the modules parameterised by $P$.
\end{defn}
By construction (see Section \ref{NCCY}) there is a natural isomorphism 
\[
D_{\oplus V_i}\cong \bigoplus_{i,j}\Hom_{\mathcal{D}_r(Q,W)^{\oplus}}^1(s_i\otimes V_i, s_j\otimes V_j),
\]
inducing an isomorphism
\[
\mathfrak{V}_{h_{s_{i(1)}},\ldots,h_{s_{i(n)}}}\cong \Rep(\Gamma(Q,W),\oplus V_i)\cap D_{\oplus V_i, \sut},
\]
where $D_{\oplus V_i, \sut}$ is the subspace of $D_{\oplus V_i}$ consisting of strictly upper-triangular matrices.  This follows from the fact that $\mathfrak{V}_{h_{s_{i(1)}},\ldots,h_{s_{i(n)}}}$ is the zero locus of the Maurer-Cartan equations, which occurs as the scheme-theoretic degeneracy locus of $W$ (see e.g. \cite{KSdef}).  So after constructible decomposition, we can consider our family of objects in $\mathcal{A}$ as a family of objects in $\mathcal{A}_{\tw}$.  Let $g:P\rightarrow \mathfrak{V}_{h_{s_{i(1)}},\ldots,h_{s_{i(n)}}}$ be a family of objects in $\mathcal{A}_{\tw}$.  Consider the differential graded vector bundle 
\[
g^*\Delta^*(\mathcal{HOM}^1|_{\mathfrak{V}_{(h_{s_{i(1)}},\ldots,h_{s_{i(n)}})}\times \mathfrak{V}_{(h_{s_{i(1)}},\ldots,h_{s_{i(n)}})}}),
\]
which as a graded vector bundle is just the trivial vector bundle with fibre $D_{\oplus V_i}$ on the space $P$.  Let $T_P$ be the total space of this vector bundle, then there is a morphism
\begin{equation}
\label{jdef}
j_M:T_P\rightarrow D_{\oplus V_i}
\end{equation}
given by $(N,a)\mapsto g(N)+a$.  We let $W_{T_P}$ be the (non-minimal) potential on the total space of $g^*\Delta^*(\mathcal{HOM}^1|_{\mathfrak{V}_{(h_{s_{i(1)}},\ldots,h_{s_{i(n)}})}\times \mathfrak{V}_{(h_{s_{i(1)}},\ldots,h_{s_{i(n)}})}})$, i.e. the pullback of the potential on the endomorphism bundle.  The key observation is the following
\begin{prop}
\label{reveng}
There is an equality of functions on $T_P$
\begin{equation}
\label{crux}
W_{T_P}=j_M^*(\tra(W))
\end{equation}
\end{prop}
\begin{proof}
This follows from the construction of the higher compositions in $\tw_r(\mathcal{D}_r(Q,W))$ (see Definition \ref{twl}) and the Calabi-Yau structure on the graded vector bundle $\mathcal{HOM}^{\bullet}$ on $\tw_r(\mathcal{D}_r(Q,W)$ (see Theorem \ref{goup}).  At a point $(N,a)$, with $N$ an object of $\tw_r(\mathcal{D}_r(Q,W))$ with associated matrix $T_N$ the left hand side of (\ref{crux}) is equal to
\[
\sum_{i\geq 1}\frac{1}{i}\langle b_{i-1+p_0+\ldots+p_{i-1}}(T_N^{\otimes p_0},a,\ldots,a,T_N^{\otimes p_{i-1}}),a\rangle
\]
which is equal to 
\[
\sum_{i\geq 1}\frac{1}{i}\langle b_{i-1}(T_N+a,\ldots,T_N+a),T_N+a\rangle
\]
by cyclic symmetry of $\langle b_{i-1}(\bullet,\ldots,\bullet),\bullet\rangle$, which is the right hand side of (\ref{crux}).
\end{proof}

\begin{thm}
\label{BBScompare}
Let $\mathcal{D}_r(Q,W)$ be as above.  We give $\VV_r$ the orientation data $c$ of Theorem  \ref{ODconstr}.  Then there is an equality $\Phi_c|_{\mathcal{A}}=\Phi_{\mathcal{A}}$.
\end{thm}
\begin{proof}
Consider the morphism
\[
\pi:\Tot(\HOM^1|_{\mathfrak{V}_{(h_{s_{i(1)}},\ldots,h_{s_{i(n)}})}\times \mathfrak{V}_{(h_{s_{i(1)}},\ldots,h_{s_{i(n)}})}})\rightarrow \Rep(\Gamma(Q,W),\oplus V_i)
\]
defined as in (\ref{jdef}).  We have already observed that, up to constructible decomposition, we can assume that a family $P\rightarrow \Rep(\Gamma(Q,W),\oplus V_i)$ factors through the inclusion of the zero section of the space $\Tot(\HOM^1|_{\mathfrak{V}_{(h_{s_{i(1)}},\ldots,h_{s_{i(n)}})}\times \mathfrak{V}_{(h_{s_{i(1)}},\ldots,h_{s_{i(n)}})}})$.  One may use the previous proposition, and the formula (\ref{MFformula}) to show also that $\pi^*(\phi_{\tra(W)})=\phi_{W_{\mathrm{tot}}}$, where $W_{\mathrm{tot}}$ is the total potential constructed via the Calabi-Yau structure on the category of twisted objects (as in Theorem \ref{goup}).  We may split the superpotential as in Theorem \ref{cycmm}, and apply Theorem \ref{TStheorem}, the motivic Thom-Sebastiani theorem (we ignore here the problem with formal functions, that is, we assume that the left hand side of the equation in the theorem is well defined -- see the remark after Theorem \ref{TStheorem}).  We deduce that $\Phi_{\mathcal{A}}$ is given by integrating against the motivic weight $\LL^{(\ext^{\leq 1}(M)-\dim(V_2))/2}(1-\MF(W_{\min})(1-\MF(Q_2))$, where $Q_2$ is the quadratic form on $V_2:=\HOM^1|_{\mathfrak{V}_{(h_{s_{i(1)}},\ldots,h_{s_{i(n)}})}\times \mathfrak{V}_{(h_{s_{i(1)}},\ldots,h_{s_{i(n)}})}}/(\Ker(d^1))$ given by the bracket $\langle d(-),-\rangle$, while $\Phi_c|_{\mathcal{A}}$ is given by integrating with the motivic weight $\LL^{(\ext^{\leq 1}(M)-\dim(V_1))/2}(1-\MF(W_{\min})(1-\MF(Q_1))$, where $Q_1$ is a quadratic form on a constructible vector bundle $V_1$, given up to equivalence by the choice of orientation data.  
\smallbreak
We can explicitly describe the class of $Q_1$ in this case.  For an arbitrary twisted object $M$ parameterised by $\mathfrak{V}_{(h_{s_{i(1)}},\ldots,h_{s_{i(n)}})}$, the differential graded vector space $\DD_r(Q,W)^{\geq 2}_{\tw}(M,M)$ is precisely the degree 2 and 3 pieces of the differential graded vector space $\DD_r(Q,W)_{\tw}(M,M)$, calculating self extensions.  It follows that 
\[
\sDet(\DD_r(Q,W)^{\geq 2}_{\tw}(M,M))\otimes \sDet(\Ho^{\geq 2}(\DD_r(Q,W)_{\tw}(M,M)))^{-1}
\]
is canonically isomorphic to $\sDet(\Image(d^1))$, as a super vector space with trivialized square, where $d$ is the differential on $\DD_r(Q,W)_{\tw}(M,M)$.  The super line bundle $\sDet(\Image(d^1))$, along with the trivialization of its square, is in turn canonically isomorphic to $\sDet(V_2)$.  Since the classes of $\LL^{-\dim(V_i)/2}(1-\MF(Q_i))$, for $i=1,2$, are determined, in families, by the same isomorphism class of constructible super line bundles with trivialized square, the theorem follows.
\end{proof}

\section{The case of a quiver with $W=0$}
Let $Q$ be a quiver, and set $W=0$ to be the trivial linear combination of cycles of $Q$.  If one defines, as in Section \ref{basicexample}, the spaces
\begin{align*}
\Rep_n(\Gamma(Q,W))=&\Hom(\Ho^0(\Gamma(Q,W)),\Mat_{n\times n}(k))\\
\cong & \Hom(k Q, \Mat_{n\times n}(k)),
\end{align*}
then the spaces $\Rep_n(\Gamma(Q,W))$ are smooth, and after taking the quotient stack under the action of $\GL_n(k)$ on $\Rep_n(\Gamma(Q,W))$ by conjugation, one obtains a smooth stack.  In common with representation spaces of superpotential algebras these are degeneracy loci of functions on smooth ambient spaces -- the special feature here is that the representation spaces \textit{are} the ambient spaces, and the functions that they are degeneracy loci of, are identically zero.
\smallbreak
If $Q$ has no oriented cycles then the moduli spaces $\Rep_n$ and $\Rep_n/\GL_n(k)$ above are of course special cases of stack functions in the more general framework of motivic Donaldson--Thomas invariants: they are families of objects in the 3-dimensional Calabi-Yau category $\Perf(\rModi(\mathcal{D}_r(Q,W))$.  This remains true in those cases in which $Q$ has oriented cycles, if we consider only nilpotent representations in $\Rep_n(\Gamma(Q,W))$.  So one expects that the motivic Donaldson--Thomas count will be given by pulling back the motivic weight 
\[
(1-\MF(0))\LL^{-\dim(\Rep_n(\Gamma(Q,W))/\GL_n(k))/2}=(1-\MF(0))\LL^{\ext^{\leq 1}/2}
\]
which is defined to be just $\LL^{\ext^{\leq 1}/2}$.  For this to be the case, from the definition of the integration map (see Definition \ref{intmap}), we expect that the contribution from the orientation data should be trivial.  We will show that this is true in the case in which our families are of objects in the Abelian category of nilpotent $kQ$-modules.
\begin{thm}
\label{trivpot}
Let $Q$ be a quiver, endowed with the potential $W=0$, and let $\mathfrak{A}_r$ be as in Theorem \ref{BBScompare}.  Then the element $h\in J_2(\mathfrak{A}_r)$ arising from the orientation data of Theorem \ref{ODconstr} is trivial.
\end{thm}
\begin{proof}
This follows from Theorem \ref{BBScompare}.  There it was shown that the element $h$ arising from the orientation data associated to the bimodule $\mathcal{D}_r(Q,W)^{\geq 2}$ was equal to the element of $J_2(\mathfrak{A}_r)$ arising from the constructible vector bundle $\Delta^*((\mathcal{HOM}^1)/\Ker(d))$, with the nondegenerate quadratic form $\langle d(\bullet),\bullet\rangle$.  Since all compositions of degree 1 elements in $\mathcal{D}_r(Q,W)$ are zero, it follows that $d|_{\Delta^*(\mathcal{HOM}^1)}=0$ on $\mathfrak{A}_r$, and we are done.
\end{proof}
\begin{examp}
Consider the quiver $Q_{\mathbb{P}^1}$ of Diagram \ref{P1quiv}, which we (necessarily) give the superpotential $W=0$.  The Ginzburg differential graded algebra $\Gamma(Q_{\mathbb{P}^1},W)$ has 0th cohomology given by the free path algebra $kQ_{\mathbb{P}^1}$.  The Abelian category $\mathcal{A}$ of left modules over this free quiver algebra is the heart for a Bridgeland stability condition (see \cite{TB07}) on the category of $\Gamma(Q_{\mathbb{P}^1},W)$-modules, and we deduce from Theorem \ref{trivpot} that if we endow the category of $\Gamma(Q_{\mathbb{P}^1},W)$-modules with the orientation data arising from the quasi-equivalence (see Propositions \ref{retnice} and \ref{sumup})
\[
\xymatrix{
\Gamma(Q_{\mathbb{P}^1},W)\lmodi \ar[r]^-{\sim} & \Perf(\rModi \mathcal{D}_l(Q_{\mathbb{P}^1},W)),
}
\]
then moduli of objects in this Abelian category have over them the trivial orientation data.  Given another Bridgeland stability condition in the same connected component as the first (see \cite{TB09}), the heart of our category is necessarily given by either a shift of the original heart, or by a shift of the Abelian category $\mathcal{A}'$ obtained by cluster mutation (see Chapter \ref{cluststuff}) at one of the two objects $h_{v_i}$.  The quiver one gets from such a mutation is the same as $Q_{\mathbb{P}^1}$, except we swap the vertices.  It follows that the superpotential in the mutated heart is also zero, and so we may apply Theorem \ref{trivpot} again and deduce that orientation data is trivial on $\mathcal{A}'$ too.  Assuming that there is only one connected component to the space of Bridgeland stability conditions (thanks to Tom Sutherland for pointing out this subtlety), it follows that all motivic Donaldson--Thomas invariants one might care to calculate in this example have trivial contribution from orientation data.
\end{examp}

\begin{figure}
\centering
\input{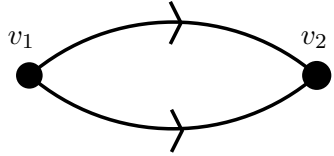}
\caption{The $\mathbb{P}^1$ quiver}
\label{P1quiv}
\end{figure}
\begin{que}
Let $(Q,W)$ be a pair of a quiver with potential, such that between any two vertices of $Q$ all of the arrows go the same way.  Assume, as above, that $W=0$.  We direct the reader to recollection of the notion of a cluster mutation of $\mathcal{D}_r(Q,W)$ in Section \ref{cluststuff}.  Note that if we mutate $Q$ at vertex $i$ we obtain a new pair $(Q',W')$, such that in general $W'\neq 0$.  By the main result of \cite{KellerMutations}, there is a quasi-equivalence of categories
\[
\xymatrix{
\Perf(\rModi\mathcal{D}_r(Q,W))\ar[r]^{\sim}& \Perf(\rModi\mathcal{D}_r(Q',W'))
}
\]
and so we obtain another choice of orientation data, coming from the Lagrangian sub-bimodule $\mathcal{D}_r(Q',W')^{\geq 2}\subset \mathcal{D}_r(Q',W')$.  By Corollary \ref{maincor} this orientation data is the same as the orientation data coming from the trivialization of the tensor square of \\$\sDet(\Theta_{\tw}(\mathcal{D}_r(Q,W)^{\geq 2}))$.  It is natural to ask, then, whether the orientation data is trivial on $\mathcal{A}'$, the heart of the category $\rmodinilp\Gamma(Q',W')$ consisting of degree zero nilpotent modules.  
\end{que}

\section{Orientation data for Hilbert schemes of $\Cp^3$}
\label{C3}
\begin{figure}
\centering
\input{3loops.pstex_t}
\caption{The quiver $Q_{\Cp^3}$ for $\Cp\langle x,y,z\rangle$}
\label{3loops}
\end{figure}
For the remainder of this chapter we work exclusively over $\Cp$.  We next consider a `global' example, and build a bridge to \cite{BBS}.  That paper provides an assignment of motivic Donaldson--Thomas counts to Hilbert schemes of $\Cp^3$, seemingly without the introduction of orientation data.  The approach is mainly the one replicated in Section \ref{basicexample}, and involves an integration map that is essentially the natural generalization of the integration map $\Phi_{\mathcal{A}}$ of Definition \ref{abintmap}; we will quickly recall how this goes in this case.\medbreak
Let $Q_{\Cp^3}$ be the quiver of Figure \ref{3loops}, with the arrows labelled $x,y,z$.  Next, let $W=xyz-xzy$.  For this quiver with superpotential, we form the \textit{uncompleted} Jacobi algebra
\begin{align*}
R:=&\Cp Q_{\Cp^3}\langle x,y,z \rangle/ \langle \frac{\partial}{\partial x}W,\frac{\partial}{\partial y}W,\frac{\partial}{\partial z}W\rangle\\
=&\Cp Q_{\Cp^3}\langle x,y,z \rangle / \langle yz-zy,zx-xz,xy-yx \rangle\\
\cong & \Cp[x,y,z],
\end{align*}
where these partial differentials are the noncommutative differentials of \cite{ginz} (see also \cite{CBEG}).  If we define
\[
\Rep_n(R):=\Hom(R,\Mat_{n\times n}(\Cp))
\]
and

\begin{equation}
\label{hy}
\Rep_n(\Cp Q_{\Cp^3}):=\Hom(\Cp Q_{\Cp^3}, \Mat_{n\times n}(\Cp))\cong \Mat_{n\times n}^3,
\end{equation}
then $\Rep_n(R)$ occurs as the critical locus of the function
\[
\tra(W):=\tra(XYZ-XZY),
\]
where the $X,Y,Z$ are given by the isomorphism of (\ref{hy}).
\medbreak
Next, consider the space $\Rep_n(R)\times \Cp^n$, the space of pairs of a representation and a vector in the underlying vector space of that representation.  There is an open subscheme
\[
U_n\subset \Rep_n(R)\times \Cp^n
\]
consisting of those pairs such that the subrepresentation generated by the vector $v\in \Cp^n$ is the whole representation.  If we define
\[
W_n\subset \Rep_n(\Cp Q_{\Cp^3})\times \Cp^n
\]
similarly, then we have that $U_n=\crit(\pi^*(\tra(W)|_{W_n})$, where $\pi$ here is the projection
\[
\pi:\Rep_n(\Cp Q_{\Cp^3})\times \Cp^n\rightarrow \Rep_n(\Cp Q_{\Cp^3}).
\]
There is an action of $\GL_n(\Cp)$ on these spaces, acting by translation on the vector $v\in \Cp^n$ and by conjugation on the representation, and this action is free on $U_n$ and $W_n$, with quotient $U_n/\GL_n(\Cp)$ the Hilbert scheme of $n$ points on $\Cp^3$.  The space $W_n/\GL_n(\Cp)$ is smooth, and $\tra(W)$ lifts to it, with $U_n/\GL_n(\Cp)$ occurring as its critical locus.  We can, then, assign $\Hilb_n(\Cp^3)$ a virtual motive, given by $[-\phi_{\tra(W)}]\LL^{-\dim(W_n)/2+ \dim(\GL_n(\Cp))/2}$, from the above setup, and this is what the authors of \cite{BBS} do.
\medbreak
\begin{figure}
\centering
\input{3loopsnew.pstex_t}
\caption{The quiver modified quiver $Q'_{\Cp^3}$ for $\Cp\langle x,y,z\rangle$}
\label{3loops'}
\end{figure}
We modify this situation slightly.  As noted in \cite{BBS} and elsewhere, the Hilbert scheme in question can be recast as a moduli space of representations (or \textit{left} modules) of the quiver $Q'_{\Cp^3}$ of Figure \ref{3loops'}.  We give this quiver the same potential $W=xyz-xzy$ and take the uncompleted Ginzburg differential graded algebra $\Gamma_{\nc}(Q'_{\Cp^3},W)$.  Note that this differential graded algebra, unlike $\Gamma_{\nc}(Q_{\Cp^3},W)$, is not concentrated in degree zero.  For the time being, consider the differential graded category $\Gamma_{\nc}(Q_{\Cp^3},W)\lmod$, which naturally forms a sub-category of $\Gamma_{\nc}(Q'_{\Cp^3},W)\lmod$.  The category $\Gamma_{\nc}(Q_{\Cp^3},W)\lmod$, in turn, contains the full subcategory $\Gamma(Q_{\Cp^3},W)\lmod$, as the category of $\Gamma_{\nc}(Q_{\Cp^3},W)$-modules with nilpotent homology modules. 
\medbreak
Let $A$ be the symmetric algebra on 3 generators, placed in degree 1 (by the Koszul sign rule, this algebra is actually isomorphic, as an ungraded algebra, to the exterior algebra on three generators).  There is an isomorphism $\mathcal{D}_l(Q_{\Cp^3},W)\cong A$, given by considering the category with one object as an algebra.  By Koszul duality (see Propositions \ref{Kosduals} and \ref{retnice}) there is a quasi-equivalence
\[
\xymatrix{
\Perf(\rModi A)\ar[r]^-{\sim}& \Gamma_{\nc}(Q_{\Cp^3},W)\lmodinilp.
}
\]
The left hand category is quasi-equivalent to a Calabi-Yau 3-dimensional category, by Theorem \ref{goup}.\medbreak
An arbitrary degree zero module in $\Gamma_{\nc}(Q_{\Cp^3},W)\lmodi$ will have homology supported above finitely many points.  For arbitrary $x\in \Cp^3$, the category of modules supported above $x$ is quasi-equivalent to $\Gamma_{\nc}(Q_{\Cp^3},W)\lmodinilp$, which is the category of finite-dimensional modules supported above the origin.  So we obtain a geometric model of our category by considering an ind-constructible category with underlying ind-variety of objects
\[
\VV_{\nc}:=\coprod_{i\geq 1}\mathfrak{V}_r^i\times M_i,
\]
where $\VV_r$ is the ind-variety parameterising objects in $\tw_r(A)$, and $M_i$ is the moduli space of $i$ distinct ordered points in $\Cp^3$.  There is a natural quasi-equivalence from this category to the category with the same description, but for which we replace $M_i$ by moduli spaces of unordered points -- we've chosen this model purely for convenience.  There is an ind-constructible graded vector bundle $\mathcal{HOM}^{\spadesuit}$ on $\VV_{\nc}\times\VV_{\nc}$ parameterising homomorphisms; over a pair $((M_1,\ldots,M_i),(z_1,\ldots,z_i))\times ((N_1,\ldots,N_j),(z'_1,\ldots,z'_j))$ it is a direct sum of copies of $\Hom_{\tw_r(A)}(M_a,N_b)$ for pairs satisfying $z_a=z'_b$.  The pairing $\langle \bullet,\bullet\rangle$ extends to this constructible vector bundle.  
\medbreak
So the category obtained by taking iterated cones of 1-dimensional modules for $\Cp[x,y,z]$ is an ind-constructible Calabi-Yau category, with objects arranged geometrically into the ind-variety
\[
\Aa_{\nc}:=\coprod_{i\geq 1}\Aa_r^i\times M_i,
\]
where $\mathfrak{A}_r$ is as in (\ref{AAdef}).  We concentrate on degree zero $\Gamma_{\nc}(Q_{\Cp^3}',W)$-modules, which form an Abelian category $\mathcal{A}_{\nc,\fram}$.  We can construct an integration map for families of objects in $\mathcal{A}_{\nc,\fram}$ in the usual way: given a map
\[
\zeta:X\rightarrow \Hom(\Cp Q'_{\Cp^3}, V_0\oplus V_{\infty})
\]
such that the image is in the scheme-theoretic degeneracy locus of the function $\tra(W)$, we let
\begin{equation}
\label{naidef}
\Phi_{\mathcal{A}_{\nc,\fram}}:=\mathbb{L}[-\phi_{\tra(W)}\mathbb{L}^{\frac{1}{2}\langle,\rangle_{Q'_{\Cp^3}}}].
\end{equation}
The point is: as well as building integration maps by supplying orientation data, we can reverse engineer integration maps to obtain orientation data.  We can use Proposition \ref{reveng} and the identity
\begin{equation}
\label{shiftW}
\tra((X+\lambda \id)(Y+\lambda \id)(Z+\lambda \id)-(X+\lambda \id)(Z+\lambda \id)(Y+\lambda \id))=\tra(XYZ-XZY)
\end{equation}
to show that in this case too, our integration map comes from integrating with respect to the nonminimal $W$ coming from the pullback of our graded vector bundle $\mathcal{HOM}^{\spadesuit}$ along the natural projection $f$:
\[
f:\Hom(\Cp Q'_{\Cp^3}, V_0\oplus V_{\infty})\rightarrow \Hom(\Cp Q'_{\Cp^3},V_{0})\cong \Hom(\Cp Q_{\Cp^3},V_0)\supset \VV_{h_{s_0},\ldots,h_{s_0}}\cong D_{V_0,slt}
\]
where here, as in Section \ref{intnilp} we assume (as we can, up to constructible decomposition), that the image of $\zeta$ lies in the subspace of strictly lower triangular matrices, with respect to some fixed basis contained in $V_0\cup V_{\infty}$, and that the image of an arrow from $i$ to $j$ is given by a morphism from $V_i$ to $V_j$ (there is some swapping here on account of our use of left modules).  We use Theorem \ref{cycmm} to split $W$ into a minimal (cubic and higher terms) part and a part arising from a nondegenerate quadratic form $Q$ on an ind-constructible vector bundle $V$.  Then we should set the orientation data $h_{\nc,\fram}$ to be the element of $J_2(\Rep(\Gamma_{\nc}(Q,W)))$ arising from this $(V,Q)$.  The statement that the cocycle condition holds then comes from the Kontsevich--Soibelman integral identity of \cite{KS}, proved (at least at the level of mixed Hodge modules) in \cite{COHA}.  
\medbreak
We give a different, more precise description of this orientation data, and prove that it does provide orientation data, at least on a category that is quasi-equivalent to the category of degree zero finite-dimensional modules on the framed quiver $Q'_{\Cp^3}$ with potential $W$.  Firstly, we must say something about the cyclic structure on the category of finite-dimensional modules for $\Gamma_{\nc}(Q'_{\Cp^3},W)$.  For this, first take the cyclic structure on the category with objects $\Cp^3\cup \{s_{\infty}\}$ that restricts to $\HOM^{\spadesuit}$ on $\Cp^3$, and where we set $\HOM|_{\{s_{\infty}\}\times\Cp^3}=\Cp[-1]$, the shift of the trivial line bundle, $\HOM|_{\Cp^3\times \{s_{\infty}\}}=\Cp[-2]$, and $\HOM|_{\{s_{\infty}\}\times\{s_{\infty}\}}=\Cp[0]\oplus\Cp[-3]$, and set all higher multiplications on these last three bundles to be trivial.  We next extend the (nondegenerate) bracket in the natural way.  We call this ind-constructible Calabi-Yau category $\CC$.  Note that the full $A_{\infty}$-Calabi-Yau subcategory of $\CC$ containing only $s_{\infty}$ and the skyscraper sheaf of a single point in $\Cp^3$ is naturally isomorphic to $\DD_l(Q',W)$.  Then the construction for $\tw_r({\CC})$ is exactly as in Definition \ref{twl}, as is the proof that $\tw_r(\CC)$ forms a new ind-constructible Calabi-Yau category, which is quasi-equivalent to the category $\Gamma_{\nc}(Q'_{\Cp^3},W)\lmodi$.  We note that, by construction, there is an inclusion of categories from the ind-constructible category with underlying ind-variety $\VV_{\nc}$, considered above.  We denote the underlying ind-variety of objects of $\tw_r(\CC)$ by $\VV_{\nc,\fram}$.  In analogy with our previous notation, we denote by $\Aa_{\nc,\fram}$ the ind-variety of objects of the sub ind-constructible Calabi-Yau category of $\tw_r(\CC)$ given by repeated degree 1 extensions of skyscraper sheaves on $\Cp^3$ and the simple module $s_{\infty}$.  This category is quasi-equivalent to $\mathcal{A}_{\nc,\fram}$.
\smallbreak
Next we specify the orientation data on $\Aa_{\nc,\fram}$.  There is a map of ind-constructible categories 
\[
\varrho_{i,t}:\VV_r^i\times M_i\rightarrow \VV_r
\]
given by projection onto the $t$th copy of $\VV_r^i$.  The ind-variety $\VV_r$ has an isomorphism class of super line bundles with trivialized square, $h\in J_2(\VV_r)$, given by the super line bundle with trivialized square coming from orientation data of Theorem \ref{ODconstr}. We define
\[
h_{\nc,\fram}=\sum_{i\geq 1,t\leq i} \varrho_{i,t}^*(h).
\]
\begin{prop}
\label{gettinglate}
The element $h_{\nc,\fram}$ provides orientation data for the $\Hom$ bifunctor on $\Aa_{\nc,\fram}$.
\end{prop}
\begin{proof}
To start with, consider the general situation in which we have just two objects $x_1$ and $x_2$ in a 3-dimensional Calabi-Yau category $\mathcal{D}$.  By the minimal model theorem there is a splitting of the four homomorphism spaces
\[
\Hom_{\mathcal{D}}(x_i,x_j)\cong \Ext_{\mathcal{D}}(x_i,x_j)\oplus V_{ij}
\]
preserved by the differential, which we'll denote $d$, such that $\Ho^{\bullet}(V_{ij})=0$.  Let $\alpha\in \Ext^1(x_2,x_1)$.  Then by the mapping cone construction of Section \ref{triangstruc} we obtain a differential $d_{\mathrm{tot},\alpha}$ on 
\[
H:=\bigoplus_{i,j\in\{1,2\}}\Hom_{\mathcal{D}}(x_i,x_j)
\]
calculating $\Ext^1(x_{\alpha},x_{\alpha})$, where $x_{\alpha}=\cone(\alpha)$.  Under the above splitting we have that $d_{\mathrm{tot},\alpha}=d+d_{\mathrm{min},\alpha}$, where $d_{\mathrm{min},\alpha}$ is the analogous differential on 
\[
E:=\bigoplus_{i,j\in\{1,2\}}\Ext_{\mathcal{D}}(x_i,x_j).
\]
We deduce that $d_{\mathrm{tot},\alpha}$ respects the splitting into two parts
\[
H\cong E \oplus \bigoplus_{i,j\in\{1,2\}} V_{ij}.
\]
So if we let $j_{\mathrm{tot}}\in J_2(\Ext^1(x_2,x_1))$ be determined by $H^1/\ker(d_{\mathrm{tot},\alpha})$, considered as a constructible vector bundle with the natural inner product, and $j_{\mathrm{min}} \in J_2(\Ext^1(x_2,x_1))$ be given by $E^1/\ker(d_{\mathrm{min}})$, then 
\begin{equation}
\label{spliteq}
j_{\mathrm{tot}}=j_{\mathrm{min}}+c,
\end{equation}
where $c$ is a constant element of $J_2(\Ext^1(x_2,x_1))$ given by the value of $j_{\mathrm{tot}}$ over $0\in \Ext^1(x_2,x_1)$.
\smallbreak
Now we need to show that our choice of orientation data satisfies the cocycle condition.  By the usual move of passing to generic points and using Noetherian induction (see Remark \ref{genpt}), it will be enough to prove this over the family parameterised by $\Ext^1(M'_1,M'_2)$, the coarse moduli space of the stack $\nu_{M'_2}\star\nu_{M'_1}$, where $M'_1$ and $M'_2$ are two $\Gamma_{\nc}(Q'_{\Cp^3},W)$-modules in $\mathcal{A}_{\nc,\fram}$.  We can construct $M'_1$ and $M'_2$ as repeated extensions by 1-dimensional $\Gamma_{\nc}(Q_{\Cp^3},W)$-modules and the simple module $s_{\infty}$.  By considering mapping cones this gives us a Calabi-Yau $A_{\infty}$-algebra $T:=\RHom(M_1\oplus s_{\infty}^{a_1}\oplus M_2\oplus s_{\infty}^{a_2},M_1\oplus s_{\infty}^{a_1}\oplus M_2\oplus s_{\infty}^{a_2})$, quasi-isomorphic to the endomorphism algebra of $M'_1\oplus M'_2$, with a sub differential graded vector space $T_{12}:=\RHom(M_1\oplus s_{\infty}^{a_1},M_2\oplus s_{\infty}^{a_2})$ calculating $\Ext^1(M'_1,M'_2)$.  Here $M_1$ and $M_2$ are direct sums of 1-dimensional $\Gamma_{\nc}(Q_{\Cp^3},W)$-modules.
\smallbreak
We denote by $\Zo^1(T_{12},d)$ the 1-cycles of $T_{12}$.  By formula (\ref{spliteq}), the obstruction $l\in J_2(\Zo^1(T_{12},d))$ above the extension defined by $\alpha$ is determined from the vector space $V=\End^1(M_1\oplus s_{\infty}^{a_1}\oplus M_2\oplus s_{\infty})/\Ker(d_{\alpha})$, where $d_{\alpha}$ is again determined by the mapping cone construction.  The obstruction $l_{\mathrm{min}}\in J_2(\Ext^1(M'_1,M'_2))$ is determined from $l$ via pullback along any inclusion $\Ext^1(M'_1,M'_2)\rightarrow \Zo^1(T_{12},d)$.  Now $V$ splits into sixteen different parts, of the form $\End^1(P,Q)$, for $P,Q\in\{M_1,s_{\infty}^{a_1},M_2,s_{\infty}^{a_2}\}$.  We note that since no terms in the potential pass through $v_{\infty}$, the differential applied to all of these summands is zero, unless in fact $P,Q\in\{M_1,M_2\}$, and also we deduce that the image of $d_{\alpha}$ lies in these same four summands.  So in fact the obstruction element $l\in J_2(\Zo^1(T_{12},d))$ is determined by the pullback along the projection $\Zo^1(T_{12},d)\rightarrow \Ext^1(M'_1,M'_2)\rightarrow \Ext^1(M_2,M_1)$.  Since our choice of orientation data is pulled back from orientation data on the underlying $\Gamma_{\nc}(Q_{\Cp^3},W)$-modules, we are done.
\end{proof}

\begin{prop}
There are equalities of $\hat{\mu}$-equivariant motives
\[
\Phi_{\nc}(\Hilb_n(\Cp^3))=\Phi_{\mathcal{A}_{\nc,\fram}}(\Hilb_n(\Cp^3))
\]
where $\Phi_{\nc}$ is the Kontsevich--Soibelman integration map with the given orientation data, and $\Phi_{\mathcal{A}_{\nc,\fram}}$ is as in (\ref{naidef}).
\end{prop}
\begin{proof}
This follows from the identity (\ref{shiftW}), Proposition \ref{reveng} and Theorem \ref{BBScompare} -- in fact the two integration maps are the same.
\end{proof}
\begin{prop}
Under the pullback map 
\[
J_2(\Aa_{\nc,\fram})\rightarrow J_2(\Aa_r),
\]
the element $h_{\nc,\fram}$ is sent to $h$, the orientation data constructed from the $A$-bimodule $A^{\geq 2}$.
\end{prop}
\begin{proof}
This follows straight from the definition of $h_{\nc,\fram}$.
\end{proof}
\section{The conifold}
\label{conif}
\begin{figure}
\centering
\input{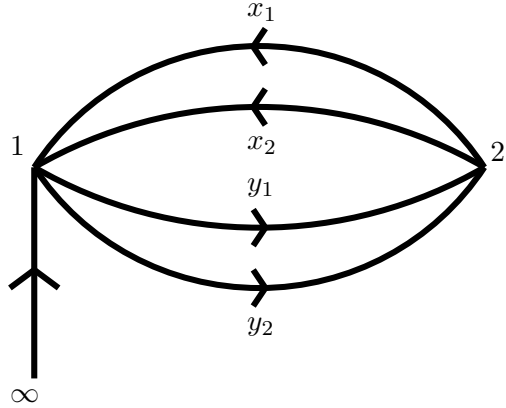}
\caption{The (framed) quiver $Q'_{\coni}$ for the noncommutative conifold}
\label{conquiv}
\end{figure}
Consider the quiver $Q'_{\coni}$ of Figure \ref{conquiv} with the potential 
\[
W=x_1y_1x_2y_2-x_1y_2x_2y_1.
\]
We will be interested in objects in $\mathcal{B}_{\nc,\fram}$, the Abelian category of finite-dimensional left modules, concentrated in degree zero, over the Jacobi algebra of this quiver with potential.  Continuing with our notation of the previous section, we consider these as degree zero modules of the (uncompleted) Ginzburg differential graded algebra $\Gamma_{\nc}(Q'_{\coni},W)$.   As with the case of the quiver of Figure \ref{3loops} we have added a vertex $\infty$ which should be thought of as providing a framing.  Note that if we delete this vertex the (uncompleted) Jacobi algebra associated to this quiver with potential is a noncommutative crepant resolution of the conifold singularity in the sense of \cite{NonComCrep}, due to \cite{FlopsAndNoncommRings}, see also \cite{Flops}; let $Y=\Spec(\Cp[x,y,z,w]/(xy-zw))$ be the conifold singularity, and let 
\[
\xymatrix{
X\ar[d]\\ Y
}
\]
be a crepant resolution, with $X$ the total space of the bundle $\mathcal{O}_{\mathbb{P}^1}(-1)^{\oplus 2}$ over $\mathbb{P}^1$.  Then if $B$ is the Jacobi algebra we obtain after deleting the vertex $\infty$, there is an isomorphism
\[
B\cong \End_{Y}(\pi^*(\mathcal{O}_{\mathbb{P}^1})\oplus \pi^*(\mathcal{O}_{\mathbb{P}^1}(1)))
\]
and an equivalence of derived categories
\[
\mathrm{D}^b_{\cpct}(Y)\rightarrow \Di(B\lmod)
\]
where the left hand side is the bounded derived category of coherent sheaves on $Y$ with compactly supported cohomology, and the right hand side is, as usual, $\Ho^0$ of the $A_{\infty}$-category of $B$-modules with finite-dimensional homology.  Accordingly we denote the quiver obtained by deleting the vertex $\infty$ from $Q'_{\coni}$ by $Q_{\coni}$.  Note that there is a forgetful functor 
\[
\mathcal{B}_{\nc,\fram}\rightarrow \mathcal{B}_{\nc},
\]
where $\mathcal{B}_{\nc}$ is the Abelian category of finite-dimensional left modules over $\Gamma_{\nc}(Q_{\coni},W)$, concentrated in degree zero.
\medbreak
If $M$ is an object of $\mathcal{B}_{\nc,\fram}$, then we define the dimension vector $\textbf{v}_M$ of $M$ to be the 3-tuple $(\dim(e_{\infty}M),\dim(e_1M),\dim(e_2M))$.  We will be considering some basic stack functions of objects of $\mathcal{B}_{\nc,\fram}$.  Firstly, $\Hilb_{n,m}(A)$ will be the scheme of objects of $\mathcal{B}_{\nc,\fram}$ with dimension vector $(1,n,m)$ such that $e_{\infty}M$ generates $M$.  Secondly, $\MM_{n,m}$ will be the stack of all representations of $\Gamma_{\nc}(Q_{\coni},W)$ concentrated in degree zero, considered as representations of $\Gamma_{\nc}(Q'_{\coni},W)$, with dimension vector $(0,n,m)$.  Finally, $\nu_{s_{\infty}}$ is the indicator function for the representation $s_{\infty}$, as usual.  Then there is an identity in the Hall algebra
\begin{equation*}
\sum_{n,m\geq 0} \MM_{n,m} \star \nu_{s_{\infty}}= \sum_{a,b\geq 0} \Hilb_{a,b}(A)\star\sum_{n,m\geq 0}\MM_{n,m}\cdot
\end{equation*}
coming from the fact that every module with dimension vector $(1,n,m)$ occurs uniquely as an extension of $s_{\infty}$ by a $\Gamma_{\nc}(Q_{\coni},W)$-module and also, uniquely, as an extension of an element of $\MM_{n-a,m-b}$ by an element of $\Hilb_{a,b}(A)$. 
\medbreak
Next take some $Z:\mathbb{Z}^2\rightarrow \Cp$ a homomorphism taking dimension vectors of finite-dimensional representations of the quiver $\Gamma_{\nc}(Q_{\coni},W)$ to points with imaginary part greater than zero (as in \cite{BridgelandM}).  We assume that $\arg(Z((1,0)))<\arg(Z((0,1)))$.  We define
\begin{equation}
\label{help}
\mathcal{C}_{\phi} =\hbox{ the full subcategory of } \mathcal{B}_{\nc} \hbox{ such that }\arg(Z(M))=\phi.
\end{equation}
As shown in \cite{NN}, the collection of simple objects of the union of these categories is given by a countable collection of spherical objects with dimension vector $(n,n+1)$, which we label $C_{n,n+1}$, a countable collection of spherical objects with dimension vector $(n+1,n)$, which we label $C_{n+1,n}$, and finally the skyscraper sheaves over points of the resolved conifold $X$, which all have dimension vector $(1,1)$ when considered as modules over our algebra $\Gamma_{\nc}(Q_{\coni},W)$.  A typical module $C_{2,3}$ is pictured in Figure \ref{23examp}.  The objects $C_{a,b}$ are all nilpotent, and can be considered as modules of $\Gamma(Q'_{\coni},W)$.
\begin{figure}
\centering
\input{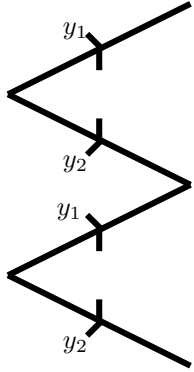}
\caption{The module $C_{2,3}$}
\label{23examp}
\end{figure}

\medbreak
By the existence and uniqueness of Harder-Narasimhan filtrations, there is an identity in the motivic Hall algebra
\begin{align}
\label{HNfilt}
\sum_{n,m\geq 0} \MM_{n,m}=&\mathcal{P}(0,1)\star\mathcal{P}(1,2)\star\ldots\star\mathcal{P}(n,n+1)\star\ldots\star\mathcal{P}(1,1)\star\ldots\star\mathcal{P}(n+1,n)\star\ldots\star\mathcal{P}(1,0).
\end{align}
where $\mathcal{P}(a,b)$ is the stack function for the objects in $\mathcal{C}_{Z((a,b))}$ which admit a filtration by simple objects in $\mathcal{C}_{Z((a,b))}$, where $\mathcal{C}_{Z((a,b))}$ is as in (\ref{help}).  Strictly speaking, in order to make sense of the right hand side we work in the ring obtained by quotienting out by the preimage under $\Phi_{\coni}$ of the ideal generated by $x^{(p,q)}$, for increasingly large $p$ and $q$ (see the discussion around (\ref{extrarel})).  For a fixed ideal the right hand side of (\ref{HNfilt}) becomes a finite product, so in particular, the infinite product makes sense in an appropriate completion of the Hall algebra.
\medbreak
Now to write down the image of these identities under the integration map we had better define the orientation data.  We first write down orientation data for the category $\mathcal{B}_{\nc}$ of degree zero finite-dimensional $\Gamma_{\nc}(Q_{\coni},W)$-modules.  Note that under the derived equivalence between $\Gamma_{\nc}(Q_{\coni},W)$-modules and sheaves on $X$, the resolved conifold, the nilpotent modules are sent to sheaves supported on the exceptional locus, while the others are sent to compactly supported sheaves with support away from the exceptional locus.  So we split our category as a direct sum
\[
\mathcal{B}_{\nc}=\Gamma_{\nc}(Q_{\coni},W)\lmod_0=\Gamma_{\nc}(Q_{\coni},W)\lmod_{0,\mathrm{\nilp}}\oplus \Coh_{\cpct}(Y'),
\]
where $Y'$ is $Y$ with the singular point removed, and the zeroes in the subscripts are because we are interested only in degree zero modules.  To construct orientation data for $\mathcal{B}_{\nc}$ it is enough to do so for each of the summands.  Under Koszul duality the left hand summand is just a subcategory of the category of perfect right $\mathcal{D}_l(Q_{\coni},W)$-modules, for which we have constructed orientation data, by Theorem \ref{ODconstr}.  The objects of the right hand summand are again arranged into the ind-constructible variety $\coprod_i \mathfrak{A}_r^i \times M'_i$, where $M'_i$ is the moduli space of $i$ distinct ordered points in $Y'$, and $\mathfrak{A}_r$ is as in Section \ref{C3}.  We give the category $\Coh_{\cpct}(Y')$ the pulled back orientation data, which on $\mathfrak{A}^i_r\times M'_i$ is given by $\pi_1^*(h)+\ldots+\pi_i^*(h)$, where $h\in J_2(\mathfrak{A}_r)$ is the orientation data arising from the fact that $\mathfrak{A}_r$ is a sub ind-variety of $\mathfrak{V}_r$, the ind-variety parameterising objects of $\tw_r(\mathcal{D}_l(Q_{\Cp^3},xyz-xzy))$, for $Q_{\Cp^3}$ as in Figure \ref{3loops}, since $Y'$ is smooth.
\medbreak
Next, in order to obtain orientation data for $\mathcal{B}_{\nc,\fram}$, the Abelian category of degree zero finite-dimensional representations for $\Gamma_{\nc}(Q'_{\coni},W)$, we pull back the orientation data above along the functor to the category of $\Gamma_{\nc}(Q_{\coni},W)$-modules, and this provides orientation data due to the proof of Proposition \ref{gettinglate}.  We denote the integration map defined using this orientation data by $\Phi_{\coni}$.
\begin{lem}
The stack function $\nu_{C_{a,b}}$ carries trivial orientation data.
\end{lem}
\begin{proof}
We prove this for $C_{2,3}$, the proof generalizes easily.  Note that $C_{2,3}$ is obtained by a single extension, i.e. it fits into a short exact sequence
\begin{equation}
\label{c23ext}
s_{2}^{\oplus 3}\rightarrow C_{2,3}\rightarrow s_1^{\oplus 2}.
\end{equation}
We define the ind-variety $\mathfrak{V}_{\coni}$ to be the ind-variety of objects of $\tw_r(\mathcal{D}_l(Q_{\coni},W))$.  Then we need to show that $l\in J_2(\mathfrak{V}_{\coni,3})$ is trivial above the extension determined by(\ref{c23ext}).  We calculate this in the usual way: first let
\[
E^{\bullet}=E_1\oplus\ldots\oplus E_4
\]
where
\begin{align*}
E_1&=\Ext_{\Gamma(Q_{\coni},W)\lmodi}(s_1^{\oplus 2},s_1^{\oplus 2})\\
E_2&=\Ext_{\Gamma(Q_{\coni},W)\lmodi}(s_1^{\oplus 2},s_2^{\oplus 3})\\
E_3&=\Ext_{\Gamma(Q_{\coni},W)\lmodi}(s_2^{\oplus 3},s_1^{\oplus 2})\\
E_4&=\Ext_{\Gamma(Q_{\coni},W)\lmodi}(s_2^{\oplus 3},s_2^{\oplus 3}).
\end{align*}
Then (\ref{c23ext}) gives an element $\alpha\in E_2$ (up to scalars), and we define $d=b_2(\alpha,\bullet)+b_2(\bullet,\alpha)+b_3(\alpha,\bullet,\alpha)$.  Since $W$ is quartic, in fact we have
\[
d|_{E^1}=b_3(\alpha,\bullet,\alpha),
\]
and $d|_{E^1}$ is only nonzero when restricted to $(E_3)^1$.  Next we decompose $(E_3)^1=U_1\oplus U_2$, where $U_i$ is spanned by elements of $\Ext_{\Gamma(Q_{\coni},W)\lmodi}^1(s_2^{\oplus 3},s_1^{\oplus 2})$ labelled by the arrow $x_i$ of Figure \ref{conif}.  Then since $W=x_1y_1x_2y_2-x_1y_2x_2y_1$ we deduce that $U_1$ is mapped to $U_2^*$ and $U_2$ is mapped to $U_1^*$ by $d$, where the duality here comes from $\langle\bullet,\bullet\rangle$.  We deduce that the quadratic form induced on $E^1/\Ker(d:E^1\rightarrow E^2)$ is split, and so $l$ is trivial above $\alpha$. 
\end{proof}
\begin{rem}
In fact since we're working over $\Cp$ it is sufficient to show that the number $\dim(E^1/\Ker(d:E^1\rightarrow E^2))$ has even parity (see the remark alongside (\ref{easyJ})).  This parity is given by $\dim(E^{\leq 1})+\dim(\Ext^{\leq 1}(C_{2,3}))=13+24+1$.  We give the above proof because it doesn't depend on the base field.
\end{rem}
One proves in exactly the same way
\begin{lem}
The modules $\mathcal{O}_x$, for $x$ a point on the exceptional locus, carry trivial orientation data.
\end{lem}
\begin{lem}
The minimal potential $W_{\min}(C_{a,b})$ is trivial.
\end{lem}
\begin{proof}
The minimal potential is a function on $\Ext^1_{\Gamma(Q_{\coni},W)\lmodi}(C_{a,b})$, which is itself trivial.
\end{proof}
Since each $C_{a,b}$ is spherical, there are equalities
\[
\mathcal{P}(n,n+1)=\sum_{i\geq 0} \nu_{C_{n,n+1}^{\oplus i}}/[\GL_i(\Cp)],
\]
and since the orientation data and the minimal potential are trivial on $C_{a,b}$, we have that
\[
\Phi_{\coni}(\mathcal{P}(n,n+1))=\sum_{i\geq 0} \mathbb{L}^{i^2/2}/[\GL_n(\Cp)]\cdot (x^{0,n,(n+1)})^i.
\]
Using a trick due to Tom Bridgeland we have the equation
\begin{equation}
\label{con1}
\Phi_{\coni}(\mathcal{P}(n,n+1))\cdot x^{1,0,0}\cdot \Phi_{\coni}(\mathcal{P}(n,n+1))^{-1}=x^{1,0,0}\cdot \sum_{0\leq i \leq n}[\mathrm{Gr}(n,i)]\cdot x^{0,in,i(n+1)}.
\end{equation}
Finally, we note that $\mathcal{P}(1,1)$ is the stack function for sheaves of $X$ supported at points, and deduce that
\begin{equation}
\label{con2}
\mathcal{P}(1,1) \star \nu_{s_{\infty}} \star \mathcal{P}(1,1)=\Hilb_{\mathrm{pts}}(X),
\end{equation}
where $\Hilb_{\mathrm{pts}}^n(X)$ is the stack function associated to length $n$ zero-dimensional coherent sheaves on $X$, and $\Hilb_{\mathrm{pts}}(X)=\coprod_{n\geq 0} \Hilb^n_{\mathrm{pts}}(X)$.  From (\ref{con1}), (\ref{con2}) and (\ref{HNfilt}) we obtain an infinite product formula for 
\begin{equation}
\label{laters}
\sum_{n,m\geq 0} \Phi_{\coni}(\Hilb_{n,m}(A)).
\end{equation}
From the construction of the minimal potential on the full subcategory of $\Coh(X)$ consisting of modules with zero-dimensional support, we deduce that
\[
\sum_{n\geq 0}\Phi_{\coni}(\Hilb_{\mathrm{pts}}(X))\cdot x^{(1,n,n)}=x^{(1,0,0)}\cdot (\sum_{n\geq 0}\Phi_c(\Hilb_{0}^n)x^{(0,n,n)})^{X}
\]
where the exponential is as defined in \cite{BBS}, $\Hilb_0^n$ is the punctual Hilbert scheme, given the orientation data $c$ it inherits as a moduli space of objects in the category of perfect modules over $\mathcal{D}_l(Q_{\Cp^3},xyz-xzy)$, and $\Phi_c$ is the Kontsevich--Soibelman integration map using the orientation data of the previous section, or equivalently of Theorem \ref{ODconstr} (see Definition \ref{intmap}).  By the main theorem of \cite{BBS} and also Theorem \ref{BBScompare}, this gives an infinite product description for the partition function (\ref{laters}), involving refined McMahon factors.

\chapter{The set of choices for orientation data}
\label{gocompare}
\section{Lagrangian sub-bimodules}
Recall Definition \ref{lagdef}, in which we defined Lagrangian sub-bimodules of the diagonal bimodule $\mathcal{C}$, for $\mathcal{C}$ a finite dimensional $A_{\infty}$-category carrying a $(2n+1)$-dimensional Calabi-Yau structure.  In Chapter \ref{ODchapter} we explained how a Lagrangian sub-bimodule of $\mathcal{C}$ gives rise to a choice of orientation data for the category of perfect modules over $\mathcal{C}$, and how there is a natural sub-bimodule of $\mathcal{C}$, denoted $\mathcal{C}^{\geq n+1}$, which is Lagrangian under quite weak assumptions.
\bigbreak
We assume that $2n+1=3$ -- this is the realm of motivic Donaldson--Thomas theory.  The natural question, given that one of the aims of the study of motivic Donaldson--Thomas invariants is to associate motives to spaces of objects in an arbitrary Calabi-Yau category, is whether our choice of orientation data is dependent on our choice of $\mathcal{C}$.  That is, say we have another finite-dimensional Calabi-Yau category $\mathcal{C}'$, and a quasi-equivalence of categories
\[
\xymatrix{
\Perf(\rModi \mathcal{C})\ar[r]^{\sim} & \Perf(\rModi\mathcal{C}'),
}
\]
then for an arbitrary stack $\mathcal{M}$ of objects in $\Perf(\rModi\mathcal{C})$ we have two choices of orientation data: one is the one constructed via tensor products with the bimodule $\mathcal{C}^{\geq 2}$, while the other comes from using the above quasi-isomorphism to consider $\mathcal{M}$ as a stack of $\mathcal{C}'$-modules, and then taking the tensor product with $\mathcal{C}'^{\geq 2}$.  In an ideal world, these two choices would always be the same, and so we would have a relatively canonical choice of orientation data: if a Calabi-Yau category can be described, up to quasi-equivalence, as $\Perf(\rModi\mathcal{C})$, for \textit{some} finite-dimensional $\mathcal{C}$, then we would pick the orientation data coming from this choice, i.e. the constructible super line bundle $\sDet(\Ho^{\bullet}(L_{\VV^+_r} \otimes \Theta(\mathcal{C}^{\geq 2})\otimes \lambda^*(L_{\VV_l^+})))$ on $\VV_r^+$ (see Section \ref{gettinggoing} for the relevant definitions).
\begin{examp}
\label{badlag}
Consider the quiver $Q$ with one vertex and one loop.  Denote the loop $a$.  We give this the superpotential $W=a^2$.  Then $\mathcal{D}_r(Q,W)$ is the category with one object, which we will call $E$, $\End_{\mathcal{D}_r(Q,W)}(E)$ has underlying graded vector space
\[
\End_{\mathcal{D}_r(Q,W)}(E)=k\oplus k[-1]\oplus k[-2] \oplus k[-3],
\]
and the differential maps the first graded piece isomorphically onto the second.  We deduce that the inclusion
\[
\Ho^{*}(S^3)\rightarrow \End^{\bullet}(E)
\]
from the cohomology algebra of the 3-sphere is a quasi-isomorphism, and so there is a quasi-isomorphism of categories
\[
\mathcal{D}_r(Q_0,0)\rightarrow \mathcal{D}_r(Q,W)
\]
where $Q_0$ is the quiver with one vertex and no edges.  This gives rise to a quasi-equivalence of categories
\[
\Perf(\rModi\mathcal{D}_r(Q,W))\rightarrow \Perf(\rModi\mathcal{D}_r(Q_0,0)).
\]
In both categories we may consider the stack function $\nu_{h_{E}}$, with one point, parameterising the spherical object $h_{E}$.  However, considered as a stack function for $\Perf(\rModi\mathcal{D}_r(Q,W))$, this stack has over it an even square root of the superdeterminant of the diagonal bifunctor, since $\mathcal{D}_r(Q,W)^{\geq 2}$ is 2-dimensional.  On the other hand, considered as a stack function for $\Perf(\rModi\mathcal{D}_r(Q_0,0))$ it has over it an odd square root of the superdeterminant of the diagonal bifunctor, since $\mathcal{D}_r(Q_0,0)^{\geq 2}$ is 1-dimensional.
\end{examp}
\bigbreak
Despite the above example, it turns out we can prove an invariance result in a wide class of situations.  This was the point of introducing the concept of a Lagrangian sub-bimodule.  The following proposition demonstrates that there are often many ways to construct Lagrangian sub-bimodules.  The main result of this section will be that they all produce the same orientation data.
\begin{prop}
\label{handylag}
Let $(Q,W)$ be a pair of a quiver with a potential.  After picking a basis for the underlying $S$-bimodule of $Q$ (i.e. a set of arrows), we identify each arrow $a:x\rightarrow y$ of $Q$ with a generator of the one-dimensional subspace of $\Hom^2(h_x,h_y)$ it represents in the category $\mathcal{D}_r(Q,W)$, and denote by $a^*\in \Hom^1(h_y,h_x)$ the dual of $a$ under the pairing $\langle\bullet,\bullet\rangle$.  Let $T$ be a subset of the arrows of $Q$ such that no cycle in $W$ contains 2 different arrows of $T$, and no cycle of $W$ contains more than a single copy of an arrow in $T$.  Then there is a Lagrangian sub-bimodule $L_T\subset\mathcal{D}_r(Q,W)$, equal to
\[
k\{ a|a\notin T,\hbox{ } a^*|a\in T\}\oplus \mathcal{D}_r(Q,W)^{3}.
\]
\end{prop}
\begin{proof}
The point is that $T$ is closed under the composition with elements of $\mathcal{D}_r(Q,W)$, due to our restrictions on $T$.  This means we can promote the exact sequence of underlying $\mathcal{D}_r(Q,W)_{\id}$-bimodules
\[
\xymatrix{
L_T\ar[r]&\mathcal{D}_r(Q,W)\ar[r]& N_T
}
\]
to a triangle of $\mathcal{D}_r(Q,W)$-bimodules, and 
\[
L_T\rightarrow \mathcal{D}_r(Q,W)
\]
and
\[
N_T^{\vee}\rightarrow \mathcal{D}_r(Q,W)^{\vee}
\]
are strict inclusions of bimodules, mapping to the same underlying $\mathcal{D}_r(Q,W)_{\id}$-bimodule, so there is an isomorphism
\[
L_T\rightarrow N_T^{\vee}
\]
making
\[
\xymatrix{
L_T\ar[r]\ar[d] & \mathcal{C}\ar[r]\ar[d] & N_T\ar[d]\\
N_T^{\vee}\ar[r] & \mathcal{C}^{\vee}\ar[r] & L_T^{\vee}
}
\]
commute.
\end{proof}
\begin{examp}
Let $(Q,W)$ be a quiver with potential, let $0$ be a vertex of $Q$, and assume that the cycles of $W$ going through $0$ only pass through it once.  Then we may pick $T$ to be the set of arrows ending at vertex $0$.  We will see in a short while that the Lagrangian sub-bimodule $L_T$ arises when we perform cluster mutations.  So the answer to the general question below, regarding the comparison of orientation data arising from different Lagrangian sub-bimodules, is related to the question of how orientation data changes under derived equivalence.
\end{examp}
So we consider the general question: how does orientation data change if we pick a different Lagrangian sub-bimodule of $\mathcal{C}$?  We start with an analogy.  Let $M$ be some manifold, and say it has on it a $2n$-dimensional graded symplectic vector bundle
\[
\xymatrix{
V\ar[d]\\
M,
}
\]
with an associated nondegenerate quadratic form on $V$ of grade $t$, for some odd number $t$.  Say now we are interested in finding a square root for the line bundle $\sDet(V)$ (this will indeed be a line bundle since $2n$ is even).  There is one situation in which this is easy, namely, let $L\subset V$ be a Lagrangian sub-bundle.  Then we have a natural isomorphism
\[
\xymatrix{
\phi_L:\sDet(L)^{\otimes 2}\ar[r]&\sDet(V).
}
\]
The analogue of the question we are interested in will be the following: let $L'\subset V$ be some other Lagrangian sub-bundle.  Assume, moreover, that $L\cap L'\subset V$ is also a sub-bundle, i.e. the dimension of the intersection between $L$ and $L'$ doesn't change along fibres (obviously in the situation in which everything is algebraic this can be arranged after taking a constructible decomposition).  We wish to find an isomorphism 
\[
\xymatrix{
\tau:\sDet(L')\ar[r] &\sDet(L)
}
\]
such that the following diagram commutes
\[
\xymatrix{
\sDet(L)^{\otimes 2}\ar[d]_{\tau^{\otimes 2}}\ar[r]^{\phi_L} &\sDet(V).\\
\sDet(L')^{\otimes 2}\ar[ur]_{\phi_{L'}}
}
\]
Such an isomorphism is given as follows:
\begin{enumerate}
\item
There is a natural isomorphism 
\[
\sDet(L)\cong \sDet(L\cap L')\otimes \sDet(L/(L\cap L')).
\]
\item
There is a natural isomorphism 
\[
\sDet(L')\cong \sDet(L\cap L')\otimes \sDet(L'/(L\cap L')).
\]
\item
There is a natural isomorphism 
\[
\sDet(L'/(L\cap L'))\cong \sDet((L'+L)/L).
\]
\item
The symplectic form on $V$ induces a natural isomorphism 
\[
\sDet((L'+L)/L)\cong \sDet(L/(L\cap L')).
\]
\end{enumerate}
Putting all this together we get exactly the isomorphism we want.  All of this is elementary linear algebra -- but the idea for what follows is essentially the same.
\begin{thm}
\label{invres}
Let $L_1$ and $L_2$ be Lagrangian sub-bimodules of $\mathcal{C}$, a finite-dimensional Calabi-Yau category satisfying one of the conditions of Theorem \ref{ODconstr}, with associated orientation data $\Upsilon_1$ and $\Upsilon_2$ on the ind-variety $\VV_r^+$.  Then there is an isomorphism in $\OD^+(\mathcal{C})$ of orientation data
\[
\Upsilon_1\cong\Upsilon_2.
\]
\end{thm}
\begin{proof}
We have a diagram of underlying $\mathcal{C}_{\id}$-bimodules
\begin{equation}
\label{brentcross}
\xymatrix{S\ar[r]& N_2\ar[r]& Q\\
L_1\ar[u]\ar[r] &\mathcal{C}\ar[u]\ar[r] & N_1\ar[u]\\
K\ar[r]\ar[u]& L_2\ar[u]\ar[r] & T\ar[u]
}
\end{equation}
in which the bottom left square is a pullback, the top right square is a pushout, and all rows and columns are short exact sequences.  We claim that, since the morphisms of the middle row and the middle column are strict, this square can be upgraded uniquely to a commutative diagram of strict $\mathcal{C}$-bimodule morphisms.  We explain this for the top left square, the others are similar.  So consider
\[
\xymatrix{
S\ar[r]^{\alpha} & N_2\\
L_1\ar[u]^{\phi}\ar[r]^{\beta} & \mathcal{C}\ar[u]^{\psi}\\
K.\ar[u]^{i}
}
\]

Let $a\in S$.  Then there exists $\tilde{a}\in L_1$ such that $\phi(\tilde{a})=a$.  Let $x_1,..,x_n$ and $y_1,\ldots,y_m$ be morphisms in $\mathcal{C}$.  Then we set 
\[
m_{S,m+n+1}(x_1,\ldots,x_n,a,y_1,\ldots,y_m)=\phi(m_{L_1,m+n+1}(x_1,\ldots,x_n,\tilde{a},y_1,\ldots,y_m)).
\]
Now pick $b\in K$.  We have
\begin{align*}
\phi(m_{m+n+1}(x_1,\ldots,x_n,i(b),y_1,\ldots,y_m))= &\pm\phi(i(m_{m+n+1}(x_1,\ldots,x_n,b,y_1,\ldots,y_m)))\\
=&0
\end{align*}
by strictness of $i$.  It follows that our definition of $m_{S}$ is well-defined.  The condition that in the upgrade of (\ref{brentcross}) to a diagram of $\mathcal{C}$-bimodules the morphisms are strict determines the module structures on the four corner objects uniquely.  It follows, from this uniqueness, that the entire diagram is self-dual under the functor $-^{\vee}$.  In particular, if we let $\mathcal{T}$ and $\mathcal{S}$ be the ind-constructible super line bundles on $\mathfrak{V}^+_r$ given by $\sDet(\Ho^{\bullet}(L_{\VV^+_r} \otimes \Theta(T)\otimes \lambda^*(L_{\VV_l^+})))$ and $\sDet(\Ho^{\bullet}(L_{\VV^+_r} \otimes \Theta(S)\otimes \lambda^*(L_{\VV_l^+})))$ respectively (see the discussion preceeding (\ref{curlyF}) for an explanation of this notation), then there is an isomorphism
\[
\phi:\mathcal{S}\cong \mathcal{T},
\]
as well as canonical isomorphisms 
\[
f:\mathcal{L}_2\cong \mathcal{K}\otimes \mathcal{T}
\]
and 
\[
g:\mathcal{L}_1\cong\mathcal{K}\otimes \mathcal{S}
\]
coming from the fact that all rows and columns are triangles of $\mathcal{C}$-bimodules (we define $\mathcal{K}=\sDet(\Ho^{\bullet}(L_{\VV^+_r} \otimes \Theta(K)\otimes \lambda^*(L_{\VV_l^+})))$).  Putting $\tau=f^{-1}(\id\otimes \phi)g$ gives the desired isomorphism of orientation data.
\end{proof}
One proves in exactly the same way the analogous statement regarding orientation data in $\OD(\mathcal{C})$ arising from pairs of Lagrangian sub-bimodules of $\mathcal{C}$.

\section{Cluster mutations}
\label{cluststuff}
We next recall the theory of cluster collections.  As will become clear to those who know this side of the story, this exposition is essentially Koszul dual to the usual treatment (for which our go-to reference is \cite{KellerMutations}, see also \cite{DWZ08}, \cite{DWZ10} and \cite{FZ02}).
\bigbreak
Let $E_0,\ldots,E_n$ be a collection of spherical objects in a 3-dimensional triangulated Calabi-Yau category $\mathcal{D}$, i.e. for each object $E_i$ we have a quasi-isomorphism
\[
\xymatrix{
\Ext_{\mathcal{D}}(E_i)\ar[r]^-{\sim} & \Ho^*(S^3).
}
\]
We further assume that for any two distinct objects $E_i$ and $E_j$, $\Ext_{\mathcal{D}}^{\bullet}(E_i,E_j)$ is concentrated in one degree only, and this degree is either 1 or 2.  Such a collection is called a cluster collection.  Under these circumstances, we are guaranteed that the full subcategory of $\mathcal{D}$ containing only these objects, is quasi-isomorphic to $\mathcal{D}_l(Q,W)$, for a quiver $Q$ whose vertices are the objects $E_i$, and such that the number of arrows from $E_i$ to $E_j$ is $\dim(\Ext^1(E_i,E_j))$ (see \cite{VdB10}).  We take the minimal category $\mathcal{D}_l(Q,W)$, which in turn fixes a model for the category $\tw_r(\mathcal{D}_l(Q,W))$, as explained in Section \ref{twistcat}, which is identified with a quasi-full subcategory of the triangulated closure of $\mathcal{D}_l(Q,W)$ in $\mathcal{D}$.  We obtain a quasi-fully faithful functor
\[
\Perf(\rModi\mathcal{D}_l(Q,W))\rightarrow \mathcal{D},
\]
and we assume that this is a quasi-equivalence, so we may as well set 
\[
\mathcal{D}=\Perf(\rModi\mathcal{D}_l(Q,W)).
\]
\bigbreak
The category $\mathcal{D}_l(Q,W)$ is a finite-dimensional Calabi-Yau category, and we can apply Theorem \ref{ODconstr} to obtain orientation data for the diagonal bimodule $\mathcal{D}_l(Q,W)$ on $\VV_r$, defined as in Definition \ref{objects} (but here we have reverted to right modules over $\mathcal{D}_l(Q,W)$).

Now we consider a new cluster collection $E'_1,\ldots,E'_n$ given as follows.  Firstly, we set
\[
E'_0=E_0[-1].
\]
Next, for $i\neq 0$ we let 
\[
E_i':=\cone(E_0[-1]\otimes \Ext^1(E_0,E_i)\rightarrow E_i)
\]
be the cone of the universal extension from $E_0$.  We let $\mathcal{C}'$ be the full subcategory of $\mathcal{D}$ with objects the $E'_i$.  By the main theorem of \cite{KellerMutations} we have a quasi-equivalence of categories
\[
\xymatrix{
\Perf(\rModi\mathcal{D}_l(Q,W))\ar[r]^-{\sim}& \Perf(\rModi\mathcal{C}')
}
\]
and so we find ourselves with two choices of orientation data.
\bigbreak
It turns out to be quite easy to write down a minimal model for $\mathcal{C}'$, after we make the simplifying assumption that for any $u$ a 3-cycle in $Q$ passing through vertex $0$, the $u$ coefficient of $W$ is zero.  There are a few types of differential graded vector spaces $\Hom_{\mathcal{C}'}(E'_i,E'_j)$ to consider.
\begin{enumerate}
\item
Say $i\neq 0 \neq j$, $\Ext^1(E_0,E_i)\neq 0\neq \Ext^1(E_0,E_j)$.  Then we have
\[
\Hom_{\mathcal{C}'}(E'_i,E'_j):=A\oplus B_2\oplus C_1 \oplus D_{0,3},
\]
where
\begin{align}
\label{pog}
A=&\Hom_{\mathcal{D}_l(Q,W)}(E_i,E_j),\nonumber\\
B=&\Hom_{\mathcal{D}_l(Q,W)}(E_i,E_0\otimes \Hom_{\mathcal{D}_l(Q,W)}^1(E_0,E_j)),\nonumber\\
C=&\Hom_{\mathcal{D}_l(Q,W)}(E_0\otimes \Hom^1_{\mathcal{D}_l(Q,W)}(E_0,E_i),E_j),\nonumber\\
D=&\Hom_{\mathcal{D}_l(Q,W)}(E_0\otimes \Hom_{\mathcal{D}_l(Q,W)}^1(E_0,E_i),E_0\otimes \Hom_{\mathcal{D}_l(Q,W)}^1(E_0,E_j)),
\end{align}
so in turn, we have a splitting $D=D^0\oplus D^3$ into graded pieces.  It is easy to see (using the Calabi-Yau pairing) that the differential maps $B$ to $D^3$ isomorphically, and (using nothing at all) that it maps $D^0$ to $C$, isomorphically.
\item
If $i\neq 0\neq j$, and either $\Ext^1(E_0,E_i)=0$ or $\Ext^1(E_0,E_j)=0$, then $\Hom_{\mathcal{C}'}(E'_i,E'_j)$ is a cone, and it is easy to check that its differential is zero.  For example let $\Ext^1(E_0,E_i)=0$.  Then
\[
\Hom_{\mathcal{C}'}(E_i',E_j')=\Hom_{\mathcal{D}_l(Q,W)}(E_i,E_0\otimes \Hom^1_{\mathcal{D}_l(Q,W)}(E_0,E_j))\oplus \Hom_{\mathcal{D}_l(Q,W)}(E_i,E_j).
\]
The differential is given by the composition $m_2$, but since we have stipulated that there are no cubic terms going through vertex $0$ in $W$, this means that any such composition is zero.
\item
Finally, assume that $i=0$ and $j\neq 0$ or $i\neq 0$ and $j=0$.  In the first case, 
\[
\Hom_{\mathcal{C}'}(E'_i,E'_j)=\Hom_{\mathcal{D}_l(Q,W)}(E_0[-1],E_0\otimes \Ext^1_{\mathcal{D}_l(Q,W)}(E_0,E_j))\oplus \Hom_{\mathcal{D}_l(Q,W)}(E_0[-1],E_i).
\]
If the first summand is nonzero then the differential on this graded vector space maps the degree minus one part onto the degree zero part, and is zero on the degree two part.  Otherwise this is just a graded vector space living in degree one.  In the second case,
\[
\Hom_{\mathcal{C}'}=\Hom_{\mathcal{D}_l(Q,W)}(E_0\otimes \Ext^1(E_0,E_i),E_0[-1])\oplus \Hom_{\mathcal{D}_l(Q,W)}(E_i,E_0[-1]).
\]
If the first summand is nonzero then the differential maps the degree three part isomorphically onto the degree four part, and is zero on the degree one part, otherwise this differential graded vector space is again just a graded vector space, living in degree two.
\end{enumerate}
So the only cases in which we have a nonvanishing differential are cases (1) and (3), and in case (1) we have shown that the inclusion $\Hom_{\mathcal{D}_l(Q,W)}(E_i,E_j)\rightarrow \Hom_{\mathcal{C}'}(E_i',E_j')$ is a quasi-isomorphism, while in the other case there is a unique choice of quasi-isomorphism of vector spaces from a minimal differential graded vector space.  Let $M$ be the underlying $\mathcal{C}'_{\id}$-bimodule of the diagonal bimodule $\mathcal{C}'$.  We have shown that we have a minimal model $M'\rightarrow M$ for $M$ (considered as a $\CC_{\id}$-bimodule), given by the inclusions 
\[
\Hom_{\mathcal{D}_l(Q,W)}(E_i,E_j)\rightarrow \Hom_{\mathcal{C}'}(E_i',E_j'),
\]
if we are in case (1) above, such that $M'(E'_i,E'_j)=M(E'_i,E'_j)$ in case (2), and the unique inclusion of the homology in case (3).
\begin{prop}
The sub $\mathcal{C}'_{\id}$-bimodule $M'$ is closed under the operations $m_{\mathcal{C}',n}$.
\end{prop}
\begin{proof}
Let $i\neq 0\neq j$ and suppose $E_i$ and $E_j$ satisfy $\Ext^1(E_0,E_i)\neq 0\neq  \Ext^1(E_0,E_j)$ (the case (3) is straightforward).  Let $x_1,\ldots,x_n$ be a series of morphisms in $\mathcal{C}'$, with the preimage of $x_1$ equal to $E_i$, and the target of $x_n$ equal to $E_j$.  Each of these morphisms is either a $1\times 1$ matrix, or a $1\times 2$, or a $2\times 1$ matrix of morphisms in $\mathcal{D}_l(Q,W)$.  We assume, to start with, that all of these matrices are of degree 1.  Then $m_{\mathcal{C}',n}(x_n,\ldots,x_1)\in A\oplus B\oplus C\oplus D$, where these are as defined in (\ref{pog}).  From the construction of the cone, the $B$ and $D$ components of $m_{\mathcal{C}',n}(x_n,\ldots,x_1)$ are zero.  The $C$ component lives in degree 2, so it is zero too.  Now say one of the $x_k$ is not of degree $1$.  By the construction of the higher compositions in $\mathcal{D}_l(Q,W)$, $m_{\mathcal{D}_l(Q,W),t}(\ldots,y_k,\ldots)=0$ for all $t\geq 3$, if $y_k$ is not of degree 1.  It follows that if $n=2$, then
\[
m_{\mathcal{C}',n}(x_n,\ldots,x_1)=m_{\mathcal{D}_l(Q,W),n}(x_n,\ldots,x_1),
\]
where the right hand multiplication is the multiplication of matrices in $\mathcal{D}_l(Q,W)$.  It follows that the multiplication of these morphisms would be the same if the $E_i'$ were not the universal extensions, but rather the cones of the zero map, for which closure is clear.  Finally we assume $n=1$.  Then $m_1(x_1)\in C^{\geq 2}$, which is zero, since we are working with a cluster collection, and by supposition $\Hom_{\mathcal{D}_l(Q,W)}^1(E_0,E_j)\neq 0$.
\end{proof}
It follows that $M'$ can be promoted to an $A_{\infty}$-category, with an inclusion 
\[
M'\rightarrow \mathcal{C}',
\]
which is a quasi-isomorphism.  So we have a minimal model for $\mathcal{C}'$, with underlying $\mathcal{C}'_{\id}$-bimodule $M'$, which we denote by $\mathcal{D}_l(Q',W')$ -- it is the category associated to the mutated quiver (see \cite{KellerMutations}).  We may consider the old diagonal bimodule $\mathcal{C}'$ as a strict bimodule for $\mathcal{D}_l(Q',W')$, and the diagonal bimodule $\mathcal{D}_l(Q',W')$ is a minimal model for this bimodule.  In general, of course, a minimal model always exists for $\CC'$, using homological perturbation -- see \cite{DWZ08} for a proof of the strictly speaking more powerful analogue under Koszul duality, stating that there is a splitting of the mutated quiver with superpotential into a quiver with a purely quadratic superpotential and a quiver with a superpotential with no quadratic terms.  We retain our simplifying assumption for now -- as we've just seen, there is a particularly neat minimal model in this case.
\bigbreak
Recall the definition of $T_{\tw}$, for $T$ a $\mathcal{C}$-bimodule and $\mathcal{C}$ an arbitrary $A_{\infty}$-category, from Section \ref{finmodtens}.  This is a bimodule for $\tw_r(\mathcal{C})$ that is quasi-isomorphic to 
\[
-\otimes_{\mathcal{C}} T\otimes_{\mathcal{C}} -^{\diamond_{\tw}}
\]
and we have $\mathcal{C}_{\tw}\cong \Hom_{\tw_r(\mathcal{C})}(-,-)$.  
\medbreak
By construction, $\mathcal{C}'\cong \mathcal{D}_l(Q,W)_{\tw}$, as a $\mathcal{C}'$-bimodule.  The bimodule $(\mathcal{D}_l(Q,W)^{\geq 2})_{\tw}$ is the Lagrangian sub-bimodule of $\mathcal{D}_l(Q,W)_{\tw}$ that gives rise to the choice of orientation data arising from the equation
\[
\mathcal{D}=\Perf(\rModi\mathcal{D}_l(Q,W)).
\]
The bimodule $\mathcal{D}_l(Q,W)^{\geq 2}_{\tw}$ forms a sub-bimodule of $\mathcal{C}'$, which is easy to describe.  For example let $i\neq 0 \neq j$, $\Ext^1(E_0,E_i)\neq 0\neq \Ext^1(E_0,E_j)$.  Then 
\begin{equation}
\label{hopp}
(\mathcal{D}_l(Q,W)_{\geq 2})_{\tw}(E_i',E_j')=\Hom_{\mathcal{D}_l(Q,W)}^{\geq 2}(E_i,E_j)\oplus B^2\oplus D^3.
\end{equation}
As before, $B^2$ maps isomorphically onto $D^3$.  If $i\neq 0$ and $j=0$ then either $\Ext^1_{\mathcal{D}_l(Q,W)}(E_0,E_i)\neq 0$ and $(\mathcal{D}_l(Q,W)^{\geq 2})_{\tw}(E'_i,E'_j)$ is a contractible vector space, or $\Ext^1_{\mathcal{D}_l(Q,W)}(E_0,E_i)= 0$ and $(\mathcal{D}_l(Q,W)^{\geq 2})_{\tw}(E'_i,E'_j)=\Ext^{\geq 2}_{\mathcal{D}_l(Q,W)}(E_i,E_0)[-1]=0$.  Similarly, if $i=0$ and $j\neq 0$ then $(\mathcal{D}_l(Q,W)^{\geq 2})_{\tw}(E'_i,E'_j)$ is either given by the degree  two part of $\Hom_{\mathcal{D}_l(Q,W)}(E_0[-1],E_0\otimes \Ext^1_{\mathcal{D}_l(Q,W)}(E_0,E_j))$ if $\Ext^1(E_0,E_j)\neq 0$, or by the degree one $\Hom_{\mathcal{D}_l(Q,W)}(E_0[-1],E_j)$ otherwise.

\smallbreak
In all other cases $(\mathcal{D}_l(Q,W)^{\geq 2})_{\tw}(E_i',E_j')$ occurs naturally as a subspace of $\Hom_{\mathcal{D}_l(Q',W')}(E_i',E_j')$, which is minimal, and so we obtain a minimal model for $(\mathcal{D}_l(Q,W)^{\geq 2})_{\tw}$ as a $\mathcal{D}_l(Q',W')$-bimodule from the inclusions
\[
\Hom_{\mathcal{D}_l(Q,W)}^{\geq 2}(E_i,E_j)\rightarrow \Hom_{\mathcal{D}_l(Q,W)}^{\geq 2}(E_i,E_j)\oplus B^2\oplus D^3,
\]
where the right hand side is as in (\ref{hopp}).  Call this minimal model $L_{\min}$.
\begin{prop}
The $\mathcal{D}_l(Q',W')$-bimodule $L_{\min}$ is a Lagrangian sub-bimodule of the diagonal bimodule $\mathcal{D}_l(Q',W')$.
\end{prop}
\begin{proof}
The triangle $\mathcal{D}_l(Q,W)^{\geq 2}\rightarrow \mathcal{D}_l(Q,W)\rightarrow \mathcal{D}_l(Q,W)^{\leq 1}$ gives a triangle of $\mathcal{C}'$-bimodules 
\[
\mathcal{D}_l(Q,W)^{\geq 2}_{\tw}\rightarrow \mathcal{C}'\rightarrow Q.
\]
The bimodule $Q$ is minimal except, firstly, in case (1), in which $i\neq 0 \neq j$, $\Ext^1(E_0,E_i)\neq 0\neq \Ext^1(E_0,E_j)$.  Then we have 
\[
Q(E_i,E_j)=\Hom_{\mathcal{D}_l'(Q,W)}^{\leq 1}(E_i,E_j)\oplus D^0\oplus C^1,
\]
with $D^0$ mapping isomorphically to $C^1$.  Secondly, if $j=0$, $i\neq 0$, then we have that $Q(E'_i,E'_j)=\cone(\Ext^{\leq 1}_{\mathcal{D}_l(Q,W)}(E_i,E_0)[-1]\rightarrow\Ext^{\leq 1}_{\mathcal{D}_l(Q,W)}(E_0\otimes \Ext^1(E_0,E_i),E_0))$, and so there is a unique minimal model for this vector space, since only one of the terms in the cone can be nonzero.  The other half of case (3) is similar.  As before, we obtain a minimal model $Q_{\min}$ by taking these injections of homology, and obtain the commutative diagram
\[
\xymatrix{
L_{\min}\ar[r]\ar[d] & \mathcal{D}_l(Q',W')\ar[d]\ar[r] & Q_{\min}\ar[d]\\
(\mathcal{D}_l(Q,W)^{\geq 2})_{\tw}\ar[r]\ar[d] & \mathcal{C}'\ar[r]\ar[d] & Q\ar[d]\\
Q^{\vee}\ar[r]\ar[d] &\mathcal{C}'^{\vee}\ar[r]\ar[d] & (((\mathcal{D}_l(Q,W)^{\geq 2})_{\tw})^{\vee}\ar[d]\\
Q_{\min}^{\vee}\ar[r] &\mathcal{D}_l(Q',W')^{\vee}\ar[r] &L_{\min}^{\vee}.
}
\]
Alternatively, we can explicitly describe the minimal bimodule $L_{\min}$.  If $i\neq 0\neq j$, then 
\[
L_{\min}(E_i',E_j')=\mathcal{D}_l(Q',W')^{\geq 2}(E_i',E_j').
\]
If $i=0=j$ then again,
\[
L_{\min}(E_i',E_j')=\mathcal{D}_l(Q',W')^{\geq 2}(E_i',E_j').
\]
If $i=0$ and $j\neq 0$, then $L_{\min}(E'_i,E'_j)$ is given by the homology of
\[
\Hom_{\mathcal{D}_l(Q,W)}^{\geq 2}(E_0,E_0\otimes \Ext_{\mathcal{D}_l(Q,W)}^1(E_0,E_j))[1]\oplus \Hom^{\geq 2}_{\mathcal{D}_l(Q,W)}(E_0,E_j)[1]
\]
and we deduce that $L_{\min}(E'_i,E'_j)=\mathcal{D}_l(Q',W')^{\geq 2}_{\tw}(E'_i,E'_j)$ if $\Ext_{\mathcal{D}_l(Q,W)}^1(E_0,E_j)\neq 0$, while $L_{\min}(E'_i,E'_j)=\mathcal{D}_l(Q',W')^{\geq 1}_{\tw}(E'_i,E'_j)$ otherwise.  Similarly, $L_{\min}(E'_j,E'_i)$ is given by the homology of 
\[
\Hom_{\mathcal{D}_l(Q,W)}^{\geq 2}(E_0\otimes \Ext_{\mathcal{D}_l(Q,W)}^1(E_0,E_j),E_0)[-1]\oplus \Hom^{\geq 2}_{\mathcal{D}_l(Q,W)}(E_j,E_0)[-1],
\]
which is zero, whatever $\Ext_{\mathcal{D}_l(Q,W)}^1(E_0,E_j)$ is.  It follows that $L_{\min}$ is the bimodule $L_T$, for $T$ the set of arrows in $Q'$ with source $0$, and so $L_{\min}$ is Lagrangian by Proposition \ref{handylag}, and the fact that $W'$ has no cycles passing through $0$ twice (see \cite{KellerMutations}).
\end{proof}
\begin{cor}
\label{maincor}
Let $\mathcal{D}_l(Q,W)$ and $\mathcal{D}_l(Q',W')$ be as above.  Let $(h)_{\mathcal{D}_l(Q,W)}$ (respectively $(h)_{\mathcal{D}_l(Q',W')}$) be the orientation data on $\Perf(\rModi \mathcal{D}_l(Q,W))$ (respectively $\\ \Perf(\rModi\mathcal{D}_l(Q',W'))$) coming from the Lagrangian bimodule $\mathcal{D}_l(Q,W)^{\geq 2}$ (respectively $\mathcal{D}_l(Q',W')^{\geq 2}$).  Then under the quasi-equivalence
\[
\phi:\Perf(\rModi \mathcal{D}_l(Q,W))\rightarrow \Perf(\rModi \mathcal{D}_l(Q',W'))
\]
there is an isomorphism $\phi^*((h)_{\mathcal{D}_l(Q',W')})\cong (h)_{\mathcal{D}_l(Q,W)}$.  In other words, Conjecture 12 of \cite{KS} is true (as long as there are no cubic terms of $W$ going through vertex $0$).
\end{cor}
\begin{proof}
We have shown that the orientation data from $\mathcal{D}_l(Q,W)^{\geq 2}$ is given by a $\mathcal{D}_l(Q',W')$-bimodule that is quasi-isomorphic to a Lagrangian sub-bimodule.  So the result follows from Theorem \ref{invres}.
\end{proof}
\begin{rem}
In fact it is not too hard to see that the above argument can be adapted for the case in which there are cubic terms going through vertex zero, so we can prove Conjecture 12 of \cite {KS} generally.  We refrain from doing so here, since in fact the conjecture follows also from Theorem \ref{outwater} below.
\end{rem}
\section{Choices of orientation data -- Examples}
The above corollary suggests that maybe there is not the plethora of choices of orientation data that one might suspect there to be.  Let $\mathcal{C}$ be a finite-dimensional Calabi-Yau category, and assume that the embedding
\[
\langle h_{\mathcal{C}}\rangle_{\triang}\rightarrow \Perf(\rModi \mathcal{C})
\]
is quasi-full and faithful, even after tensoring with arbitrary field extensions of $k$.  Recall Example \ref{torsor}, in which it was shown that the set of isomorphism classes of choices of orientation data for $\Perf(\rModi\mathcal{C})$ carries a free action of 
\[
\Hom(\Kth(\Perf(\rModi\mathcal{C})),\mathbb{Z}_2).
\]
\begin{thm}
\label{rigidity}
Let $\mathcal{C}$ be as above, and assume also that the base field $k$ is algebraically closed.  Then the action of $\Hom(\Kth(\Perf(\rModi\mathcal{C})),\mathbb{Z}_2)$ on the set of isomorphism classes of orientation data in $\OD(\CC)$, the category of Definition \ref{redod}, is transitive.
\end{thm}
\begin{proof}
Let $x_1,\ldots,x_n$ be the objects of $\mathcal{C}$.  It is sufficient to show that orientation data is determined by restriction to the $n$ stack functions $\nu_{h_{x_i}}$, the family of objects over a point consisting just of $h_{x_i}$.  Orientation data over each of these stack functions is just given by a parity, i.e. if $(\mathcal{G},\phi)$ is a pair consisting of a constructible super line bundle on $\mathfrak{V}$, with $\phi$ a trivialization of its square, then the restriction of this data to the point $h_{x_i}$ is either isomorphic to $k$ placed in degree 0, with the canonical trivialization of $k^{\otimes 2}$, or $k$ placed in degree 1, with the canonical trivialization of $k[1]^{\otimes 2}$.    In other words $J_2(\Spec(k))\cong \mathbb{Z}_2$, since we are working over an algebraically closed field (see (\ref{easyJ})).
\smallbreak
Now say $\tilde{a}=(\mathcal{G},\phi)$ gives orientation data.  Consider the obstruction element $l\in J_2(\mathfrak{V}_3)$ restricted to the triangle 
\[
0\rightarrow 0\rightarrow 0.
\]
Since it is built from a quadratic from on a subspace of $\Ext(0,0)^{\oplus 4}$, we deduce that it is trivial above this triangle, and so the cocycle condition states that $(h)_{\tilde{a}}|_0$, the element of $J_2(\mathfrak{V})$ coming from $\tilde{a}$, restricted to the zero module, is trivial.  Since zero occurs as an extension of $h_{x}$ by the shift of $h_{x}$, the cocycle condition then determines the value of $h_{\tilde{a}}$ at all stack functions $\nu_{h_x[i]}$.  Finally, all elements of $\tw_r(\mathcal{C})$ occur as extensions of these modules, and so it follows that $(h)_{\tilde{a}}$ is entirely determined by the value of $(h)_{\tilde{a}}$ at $h_{x}$, for $x\in \mathcal{C}$, and the obstruction element $l$, which is intrinsic to $\mathcal{C}$.
\end{proof}
\begin{examp}
Consider again the category $\Perf(\rModi\Ho^*(S^3,\Cp))$.  The theorem above tells us that the two non-isomorphic choices of orientation data given in Example \ref{badlag} are \textit{all} the possible choices of orientation data.
\end{examp}
\begin{examp}
Consider again the orientation data for the Abelian category of compactly supported (framed) coherent sheaves on $\Cp^3$ constructed in Section \ref{C3}.  The proof of Theorem \ref{rigidity} tells us that a choice of orientation data for this category comes from a choice of constructible super line bundle on $\Cp^3$ with trivialized square, and a choice of parity over the stack function $\nu_{s_{\infty}}$.  So as in Section \ref{thelink}, a choice of orientation data arises from a choice of an element of 
\[
\Constr(\Cp^3)^*/(\Constr(\Cp^3)^*)^2\times (\Constr(\Cp^3,\mathbb{Z}_2)\times \mathbb{Z}_2.
\]
Our construction in Section \ref{C3} corresponds to picking the trivial element.  Here, the space of choices of orientation data is uncountably infinite.
\end{examp}
Let $E_0,\ldots,E_n$ be a spherical collection in a 3-dimensional Calabi-Yau category $\mathcal{D}$.  Let $\mathcal{C}$ be the full sub-category containing these objects, and assume that we have taken a minimal model for $\mathcal{C}$.  Say we have a quasi-equivalence of categories 
\[
\langle h_{\mathcal{C}'}\rangle_{\triang}\rightarrow \langle h_{\mathcal{C}}\rangle_{\triang}
\]
where $\mathcal{C}'$ is the full subcategory of $\mathcal{D}$ containing objects $E_0',\ldots,E_n'$, formed as before, by letting $E_0'=E_0[-1]$ and replacing each other $E_i$ with the universal extension from $E_i$.  This situation is strictly more general than the situation considered in the previous section, see e.g.  \cite{Nagao} for examples.  Then, since we have picked a minimal model for $\mathcal{C}$, it follows that the resulting orientation data $\tilde{a}=(\sDet(\Ho^{\bullet}(L_{\VV^+_r} \otimes \Theta(\CC^{\geq 2})\otimes \lambda^*(L_{\VV_l^+}))),\phi)$ is trivial above each $h_{E_i}$.  If we pick a minimal model for $\mathcal{C}'$ too, then the same comment holds, regarding the orientation data $\tilde{a}'=(\sDet(\Ho^{\bullet}(L_{\VV^+_r} \otimes \Theta(\CC'^{\geq 2})\otimes \lambda^*(L_{\VV_l^+}))),\phi')$.  So to check that $\tilde{a}\cong\tilde{a}'$, it is necessary and sufficient to find out what the value of the obstruction element $l$ is at each of the universal extensions -- if it is zero then by the cocycle condition the two choices agree.  This is equivalent to checking that the super vector space with trivialized square determined by $\tilde{a}$ is trivial above the $h_{E'_i}$, by the fact that $\tilde{a}$ is trivial above the stack functions $h_{E_i}$ and it satisfies the cocycle condition.
\begin{thm}
\label{outwater}
Let $\mathcal{C}$ and $\mathcal{C}'$ be as above, and let $\tilde{a}$ and $\tilde{a}'$ be the orientation data on the category 
\[
\tw_r(\mathcal{C})
\]
coming from the two bimodules $\mathcal{C}^{\geq 2}$ and $\mathcal{C}'^{\geq 2}$.  Assume that, for all $i\neq 0$, and for all 3-tuples
\[
x_1,x_2,x_3
\]
of morphisms, with $x_1,x_3\in \Hom_{\mathcal{C}}(E_0,E_i)$ and $x_2\in \Hom_{\mathcal{C}}(E_i,E_0)$, we have the equality
\[
m_3(x_1,x_2,x_3)=0.
\]
Then there is an isomorphism of orientation data $\tilde{a}\cong\tilde{a}'$.  Alternatively, this assumption holds, regardless of the higher compositions, if $\sqrt{-1}\in k$.
\end{thm}
\begin{proof}
We need to calculate the obstruction element $l$ at each universal extension.  If this is trivial, then by the cocylce condition, and the fact that $(h)_{\tilde{a}}$ is trivial at each of the $\nu_{E_k}$, it follows that $(h)_{\tilde{a}}$ is trivial at each of the stack functions $\nu_{h_{E_i'}}$, and the two choices of orientation data are the same.\bigbreak
This obstruction element is determined in the usual way: let
\[
V=V_a\oplus V_b\oplus V_c\oplus V_d
\]
where 
\begin{align*}
V_a=&\Hom_{\mathcal{C}}(E_0\otimes \Hom_{\mathcal{C}}^1(E_0,E_i),E_0\otimes \Hom_{\mathcal{C}}^1(E_0,E_i))\\
V_b=&\Hom_{\mathcal{C}}(E_0\otimes \Hom_{\mathcal{C}}^1(E_0,E_i),E_i)\\
V_c=&\Hom_{\mathcal{C}}(E_i,E_0\otimes \Hom_{\mathcal{C}}^1(E_0,E_i))\\
V_d=&\Hom_{\mathcal{C}}(E_i,E_i).
\end{align*}
Then this graded vector space has a differential given by 
\[
d=b_2(\bullet,\lambda)+b_2(\lambda,\bullet)+b_3(\lambda,\bullet,\lambda)
\]
where $\lambda\in \Hom_{\mathcal{C}}(E_0[-1]\otimes \Hom_{\mathcal{C}}^1(E_0,E_i),E_i)$ is the universal extension.  Note that, by the assumption on $m_3$, we have the equality
\[
d=b_2(\bullet,\lambda)+b_2(\lambda,\bullet).
\]
The differential maps $V_a\oplus V_d$ to $V_b$.  It maps $V_b$ to zero, and $V_c$ to $V_a\oplus V_d$.  The spaces $V_a$ and $V_d$ are self-dual under the pairing $\langle\bullet,\bullet\rangle$ and $V_b$ and $V_c$ are dual to each other.  It follows that the quadratic form 
\[
Q=\langle d(\bullet),\bullet\rangle
\]
on 
\[
U=\frac{V}{\Ker(d:V_1\rightarrow V_2)}
\]
is split under the direct sum
\[
\frac{V_a\oplus V_d}{\Ker(d_1)}\bigoplus \frac{V_b\oplus V_c}{\Ker(d_1)}.
\]
In particular, the element $(\mathrm{Det}(Q),\mathrm{dim}(U))$ is trivial in $J_2(\Spec(k))$.  For the final remark, we note that under the assumption that $\sqrt{-1}\in k$, it is enough to show that $\dim(U)$ is even, and this will imply that $l\in J_2(\VV_{r,3})$ is trivial above the extension $\lambda$.  This follows from a dimension count, namely $\dim(V^1)=\dim(\Hom_{\mathcal{C}}^1(E_0,E_i))^2+1$ and $\dim(V^2)=\dim(\Hom_{\mathcal{C}}^1(E_0,E_i))^2$, and finally $\dim(\Ext^{\leq 1}(E_{\alpha},E_{\alpha})=1$, since the universal extension is spherical.  By basic linear algebra, the parity of $\dim(U)$ is given by the sum of these numbers.
\end{proof}
\begin{examp}
Consider the noncommutative conifold arising from the pair of a quiver with potential $(Q_{\coni},W)$ of Example \ref{conif}.  This is quite special, in that the above theorem does not apply to it (dropping temporarily the assumption that we work over $\Cp$), since the superpotential contains quartic terms, which are cycles between the two vertices.  However, it turns out that we have already calculated the value of $l$ at the universal extension, for this is just $C_{2,1}$, in the notation of Section \ref{conif}.  We deduce that the orientation data pulled back from the mutated conifold is isomorphic to the orientation data from $\mathcal{D}_l(Q_{\coni},W)^{\geq 2}$, even before we reinstate the assumption that we work over $\Cp$, which we now do.\medbreak
Let $(Q_{\coni,+},W_+)$ be the mutated quiver with superpotential.  This is in fact an isomorphic quiver to $(Q_{\coni},W)$, but the natural derived equivalence
\[
\mathrm{D}^b(\Gamma_{\nc}(Q_{\coni,+},W_+)\lmod)\rightarrow \mathrm{D}^b(\Gamma_{\nc}(Q_{\coni},W)\lmod)
\]
is not trivial, in that it doesn't come from an equivalence between the natural hearts of these two categories (see \cite{Flops}).  The heart of the left hand category should be thought of as $\Coh_{\cpct}(X_{\nc}^+)$, the category of finite modules on a noncommutative version of the flop of $X$.  If we restrict to sheaves supported away from the exceptional locus, the above quasi-equivalence does reduce to an equivalence of Abelian categories, and so the orientation data is the same after passing through the quasi-equivalence, by construction.  We have just seen that the orientation data on the exceptional locus, pushed through the derived equivalence, is also the same both sides.  We deduce that the conifold flop preserves orientation data.
\medbreak
Our comments here can be summed up by saying that natural choices of orientation data tend to glue across cluster transformations, flops, and crossings of `walls of the second kind'.  So while the fact that orientation data must be introduced at all, and is generally non-unique, implies a negative answer to Question \ref{bigq}, the good news is that orientation data is natural, well-behaved, and, we hope the reader agrees, not so frightening after all.
\end{examp}

\bibliographystyle{amsplain}
\bibliography{SPP}
\end{document}